\documentclass[english]{smfart.cls}

\usepackage{amsmath,amsxtra,amssymb}
\usepackage[all]{xy}
\usepackage{dsfont}
\usepackage{mathrsfs}
\usepackage{enumerate, vmargin}
\usepackage{color}

\theoremstyle{plain} 
\numberwithin{equation}{section}
\newtheorem{theorem}{Theorem}[section]

\newtheorem{prop}[theorem]{Proposition} 
\newtheorem{lemma}[theorem]{Lemma}
\newtheorem{cor}[theorem]{Corollary} 
\newtheorem{defi}[theorem]{Definition} 
\newtheorem{remark}[theorem]{Remark}

\newtheorem{pb}[theorem]{Problem}

\newtheorem{fact}[theorem]{Fact}

\newcommand{\re}{\begin{remark}\rm}
\newcommand{\mar}{\end{remark}}

\newcommand{\al}{\alpha}
\newcommand{\s}{\sigma}
\newcommand{\si}{\sigma}
\newcommand{\Si}{\Sigma}

\newcommand{\eps}{\varepsilon}

\newcommand{\nz}{{\mathbb N}}
\newcommand{\rz}{{\mathbb R}}
\newcommand{\Z}{{\mathbb Z}}

\newcommand{\R}{\mathbb{R}}

\newcommand{\N}{{\mathcal N}}
\renewcommand{\L}{{\mathcal L}}
\newcommand{\F}{{\mathcal F}}
\newcommand{\E}{{\mathcal E}}
\newcommand{\A}{{\mathcal A}}
\newcommand{\V}{{\mathcal V}}
\newcommand{\D}{{\mathcal D}}
\newcommand{\M}{{\mathcal M}}
\newcommand{\PP}{{\mathcal P}}
\newcommand{\Nc}{{\mathcal N}}
\newcommand{\I}{{\mathcal I}}
\newcommand{\U}{{\mathcal U}}
\newcommand{\B}{{\mathcal B}}
\newcommand{\T}{{\mathcal T}}
\newcommand{\W}{{\mathcal W}}
\newcommand{\Q}{{\mathcal Q}}

\renewcommand{\H}{{\mathcal H}}
\renewcommand{\P}{\mathcal{P}} 
\newcommand{\Rcal}{{\mathcal R}}

\renewcommand{\tilde}{\widetilde}
\renewcommand{\hat}{\widehat}

\newcommand{\h}{\mathsf{h}} 
\newcommand{\lh}{\mathsf{Lh}} 
\renewcommand{\th}{\tilde{\mathsf{h}}} 
\newcommand{\hh}{\hat{\h}}
\newcommand{\hH}{\hat{\mathcal{H}}}

\newcommand{\BMO}{\mathcal{BMO}}

\newcommand{\MO}{\mathcal{MO}}

\newcommand{\K}{\mathrm{K}}
\newcommand{\J}{\mathrm{J}}
\newcommand{\cond}{\mathrm{cond}}

\newcommand{\bmo}{\mathsf{bmo}}
\newcommand{\lbmo}{\mathsf{Lbmo}}
\newcommand{\mo}{\mathsf{mo}}

\newcommand{\x}{\tilde{x}}

\newcommand{\tp}{\tilde{p}}
\newcommand{\tq}{\tilde{q}}
\newcommand{\fin}{\mathrm{fin}}
\renewcommand{\1}{\mathds{1}} 

\newcommand{\oten}{\overline{\otimes}}
\newcommand{\ten}{\otimes}
\newcommand{\prodd}{\prod\nolimits}
\newcommand{\wst}{w^*\mbox{-}}
\newcommand{\w}{w\mbox{-}}
\newcommand{\qd}{\end{proof}}
\newcommand{\tr}{\mathrm{tr}}
\renewcommand{\AA}{\mathsf{A}}
\newcommand{\Acirc}{\stackrel{\circ}{\mathsf{A}}}
\newcommand{\st}{\mathrm{st}}
\renewcommand{\and}{\quad \mbox{ and } \quad }

\author[Marius Junge]{Marius Junge$^\dag$}
\address{Department of Mathematics\\
University of Illinois at  Urbana-Champaign, Urbana, IL 61801}
\email{junge@math.uiuc.edu}

\author[Mathilde Perrin]{Mathilde Perrin$^\star$}
\address{Instituto de Ciencias Matem\'{a}ticas\\
CSIC-UAM-UC3M-UCM\\
Consejo Superior de Investigaciones Cient\'{i}ficas\\
C/ Nicol\'{a}s Cabrera 13-15. 28049, Madrid, Spain}
\email{mathilde.perrin@icmat.es}

\title[Theory of $\H_p$-spaces for continuous filtrations] {Theory of $\H_p$-spaces for continuous filtrations in von Neumann algebras}
\alttitle{Th\'eorie des espaces $\H_p$ pour des filtrations continues dans des alg\`ebres de von Neumann}

\begin{document}





\frontmatter
\begin{abstract}
We introduce Hardy spaces for martingales with respect to continuous filtration for von
Neumann algebras. In particular we prove the analogues of the Burkholder/Gundy and
Burkholder/Rosenthal inequalities in this setting. The usual arguments using stopping
times in the commutative case are replaced by tools from noncommutative function
theory and  allow us to obtain the analogue of the Feffermann-Stein duality and prove a noncommutative Davis decomposition.
\end{abstract}

\begin{altabstract}
Nous introduisons des espaces de Hardy pour des martingales relatives \`{a} des filtrations continues d'alg\`{e}bres de von Neumann. 
En particulier, nous d\'emontrons les in\'egalit\'es de Burkholder/Gundy et de Burkholder/Rosenthal dans ce cadre. 
Les arguments usuels bas\'es sur des temps d'arr\^et dans le cas commutatif sont remplac\'es par des outils de la th\'eorie des fonctions non commutatives, 
qui nous permettent d'obtenir l'analogue de la dualit\'e de Fefferman-Stein et de prouver une d\'ecomposition de Davis non commutative.   
\end{altabstract}

\altkeywords{espaces $L_p$ non commutatifs, martingales non commutatives, espaces de Hardy, filtration continue}
\thanks{$^\dag$ Partially supported by National Science Foundation grants DMS-0901457 and DMS-1201886}
\thanks{$^\star$ Partially supported by the Agence Nationale de Recherche, ERC Grant StG-256997-CZOSQP and NSF grant DMS-1201886}

\maketitle

\tableofcontents 

\mainmatter

\section*{Introduction}

\setcounter{theorem}{0}

The theory of stochastic integrals and martingales with continuous time is a well-known theory with many applications. 
Quantum stochastic calculus is also well developed with  applications reaching into fields such as quantum optics. 
In the setting of von Neumann algebras, 
many classical martingale inequalities have been reformulated for noncommutative martingales with respect to discrete filtrations, 
see e.g. \cite{PX,JX,JD,jx-ros}.  
The aim of this paper is to study martingales with respect to continuous filtrations in von Neumann algebras. 
Our long term goal is to develop a satisfactory theory for semimartingales, including the convergence of the stochastic integrals. 
In the noncommutative setting, we cannot construct the stochastic integrals pathwise as in \cite{DM}. 
It is unimaginable to consider the path of a process of operators in a von Neumann algebra. 
However, it is well-known that in the classical case, the convergence of the stochastic integrals is closely related 
to the existence of the quadratic variation bracket $[\cdot,\cdot]$ 
via the formula 
$$X_tY_t=\int^tX_{s^-}dY_s+\int^tY_{s^-}dX_s+[X,Y]_t.$$
Here the quadratic variation bracket can be characterized as the limit in probability of the following dyadic square functions
$$[X,Y]_t=X_0Y_0+\lim_{n\to \infty} \sum_{k=0}^{2^n-1}(X_{t\frac{k+1}{2^n}}-X_{t\frac{k}{2^n}})(Y_{t\frac{k+1}{2^n}}-Y_{t\frac{k}{2^n}}).$$
Hence we will first study this quadratic variation bracket in the setting of von Neumann algebras, 
and then deal with stochastic integrals in a forthcoming paper based on the theory developed here. 
More precisely, we will focus on the $L_{p/2}$-norm of this bracket by considering the Hardy spaces $H_p$ defined in the classical case by the norm
$$\|x\|_{H_p}=\|[x,x]\|_{p/2}^{1/2}.$$
This paper develops a theory of the Hardy spaces of noncommutative martingales with respect to a continuous filtration.
One fundamental application is an interpolation theory for these noncommutative function spaces
which has already found applications in the theory of semigroups (see e.g. \cite{jm-riesz}).

Let us consider a von Neumann algebra $\M$. For simplicity, we assume that $\M$ is finite and equipped with a normal faithful normalized trace $\tau$. 
Fortunately, the theory of noncommutative $H_p$-spaces is now very well understood in the discrete setting, i.e., 
when dealing with an increasing sequence $(\M_n)_{n\geq 0}$ of von Neumann subalgebras of $\M$, whose union is weak$^*$-dense in $\M$. 
We consider the associated conditional expectations $\E_n:\M\to \M_n$. 
In the noncommutative setting it is well-known that we always encounter two different objects,
the row and column versions of the Hardy spaces:
$$\|x\|_{H_p^c} = \Big\|\Big(\sum_n |d_n(x)|^2\Big)^{1/2}\Big\|_p\quad \mbox{and} \quad
  \|x\|_{H_p^r} = \Big\|\Big(\sum_n |d_n(x^*)|^2\Big)^{1/2}\Big\|_p ,$$
where $d_n(x)=\E_n(x)-\E_{n-1}(x)$. 
Here $\|x\|_p=(\tau(|x|^p))^{1/p}$ refers to the norm in the noncommutative $L_p$-space. 
The noncommutative Burkholder-Gundy inequalities from \cite{PX} say that
\begin{equation}\label{BGI}
  L_p(\M) = H_p \quad \mbox{with equivalent norms for }  1<p<\infty,
\end{equation}
where the $H_p$-space is defined by
$$H_p=\left\{\begin{array}{cl}
H_p^c+H_p^r& \quad \mbox{for} \quad 1\leq p<2 \\
H_p^c\cap H_p^r& \quad \mbox{for}\quad 2\leq p<\infty
\end{array}\right..$$
Following the commutative theory, we should expect to define the bracket $[x,x]$ for a martingale $x$ and then define
$$ \|x\|_{\hat{\H}_p^c} = \|[x,x]\|_{p/2}^{1/2} \quad \mbox{and}
 \quad \|x\|_{\hat{\H}_p^r} = \| [x^*,x^*]\|_{p/2}^{1/2} .$$
Armed with the definition we may then attempt to prove \eqref{BGI} for a continuous filtration $(\M_t)_{t\geq 0}$. 
For simplicity, we assume that the continuous parameter set is given by the interval $[0,1]$. 
We define a candidate for the noncommutative bracket following a nonstandard analysis approach. 
For a finite partition $\s=\{0=t_0<t_1<\cdots <t_n=1\}$ of the interval $[0,1]$ and $x\in \M$, we consider the finite bracket
$$[x,x]_\si=\sum_{t\in \si} |d_t^{\si}(x)|^2,$$
where $d_t^\s(x)=\E_t(x)-\E_{t^-(\s)}(x)$.  
Then for $p>2$, \eqref{BGI} gives an a-priori bound \\
$\|[x,x]_\si]\|_{p/2}^{1/2}\leq \alpha_p\|x\|_p$. 
Hence, for a fixed ultrafilter $\U$ refining the general net of finite partitions of $[0,1]$, we may simply define
$$[x,x]_\U=\w L_{p/2} \mbox{-}\lim_{\s,\U}[x,x]_\s.$$ 
In fact, in nonstandard analysis, the weak-limit corresponds to the standard part 
and is known to coincide with the classical definition of the bracket for commutative martingales. 
However, the norm is only lower semi-continuous with respect to the weak topology and we should not expect Burkholder/Gundy inequalities for continuous
filtrations to be a simple consequence of the discrete theory of $H_p$-spaces. 
Yet, using the crucial observation that the $L_{p/2}$-norms of the discrete brackets $[x,x]_\s$ are monotonous up to a constant, 
we may show the following result.

\begin{theorem}\label{eqnormHH}
Let $1\leq p <\infty$ and $x\in \M$. Then 
$$\|[x,x]_\U\|_{p/2}\simeq \lim_{\s,\U}\|[x,x]_\s\|_{p/2}
\simeq \left\{\begin{array}{cl}
\sup_\s\|[x,x]_\s\|_{p/2}& \quad \mbox{for} \quad 1\leq p< 2 \\
\inf_\s\|[x,x]_\s\|_{p/2}& \quad \mbox{for}\quad 2\leq p<\infty
\end{array}\right..$$
\end{theorem}

In particular, this implies that the $L_{p/2}$-norm of the bracket $[x,x]_\U$ does not depend on the choice of the ultrafilter $\U$, up to equivalent norm. 
We will discuss the independence of the bracket $[x,x]_\U$ itself from the choice of $\U$ in a forthcoming paper. 
Hence for $1\leq p <\infty$ and $x \in \M$ we define the norms
$$\|x\|_{\hH_p^c}=\|[x,x]_\U\|_{p/2}^{1/2} \quad \mbox{and} \quad 
\|x\|_{\H_p^c}=\lim_{\s,\U}\|[x,x]_\s\|_{p/2}^{1/2}=\lim_{\s,\U}\|x\|_{H_p^c(\s)}.$$
We denote by $\hH_p^c$ and $\H_p^c$ respectively the corresponding completions. 
Using Theorem \ref{eqnormHH} we may show that actually 
\begin{equation}\label{introhH=H}
\hH_p^c=\H_p^c \quad \mbox{with equivalent norms for } 1\leq p <\infty.
\end{equation}
Hence this defines a good candidate for the Hardy space of noncommutative martingales with respect to the continuous filtration $(\M_t)_{0\leq t \leq 1}$. 
We now want to establish for this space the analogues of many well-known results in the discrete setting. 
For doing this, we will use the definition of the space $\H_p^c$, which will be more practical to work with. 
In particular, we may embed $\H_p^c$ into some ultraproduct space, which has an $L_p$-module structure and a $p$-equiintegrability property. 
This allows us to consider $\H_p^c$ as an intermediate space of operators between $L_2(\M)$ and $L_p(\M)$. 
Then, by complementation, we can show the following duality result. 

\begin{theorem}\label{duality}
Let $1<p<\infty$ and $\frac{1}{p}+\frac{1}{p'}=1$. Then 
$$(\H_p^c)^*=\H_{p'}^c\quad \mbox{with equivalent norms}.$$
\end{theorem}

Note that throughout this paper, following \cite{PX} we will consider the anti-linear duality, given by the duality bracket $(x|y)=\tau(x^*y)$. 
Since no confusion is possible, we will denote it by $(\H_p^c)^*$. 
With this convention, the dual space of a column space is still a column space. 
For $p=1$, we also establish the analogue of the Fefferman-Stein duality in this setting:
$$(\H_1^c)^*=\BMO^c\quad \mbox{with equivalent norms}.$$
We have to be careful when defining the space $\BMO^c$. 
A naive candidate for the $\BMO^c$ norm is given by
$$ \|x\|_{\tilde{\BMO^c}} = \lim_{\si,\U} \|x\|_{BMO^c(\si)},\quad \mbox{where} \quad 
\|x\|_{BMO^c(\si)}=  \sup_{t\in \si} \|\E_t(|x-x_{t^-}|^2)\|_\infty^{1/2}.$$
However, here our restriction to finite partitions (instead of random partitions in the classical case) is restrictive. 
Indeed, if one of the $\|x\|_{BMO^c(\si)}$'s is finite, then $x$ is already in $\M$. 
Definitively, we expect $\BMO^c$ to be larger than $\M$. 
We will therefore say that an element $x\in L_2(\M)$ belongs to the unit ball of $\BMO^c$ if it
can be approximated in $L_2$-norm by elements of the form
$$ \w L_2 \mbox{-}\lim_{\s,\U} x_\s \quad \mbox{with} \quad \lim_{\s,\U} \|x_{\s}\|_{BMO^c(\s)}\leq 1.$$
This definition gives the expected interpolation result
$$\H_p^c=[\BMO^c,\H_1^c]_{1/p} \quad \mbox{ with equivalent norms for } 1<p<\infty.$$ 
We may define the Hardy space $\H_p$ as in the discrete setting by considering the sum of the column and row Hardy spaces in $L_p(\M)$ 
for $1< p <2$, and their intersection in $L_2(\M)$ for $2\leq p<\infty$. 
The continuous analogue of \eqref{BGI} is then obtained by taking the weak limit of the discrete decompositions for $1<p<2$. 
However, the usual duality argument used to deduce the case $2<p<\infty$ may not be directly applied in this case. 
We first need to extend a stronger Burkholder-Gundy decomposition introduced by Randrianantoanina to the continuous setting. 
More precisely, we need a Burkholder-Gundy decomposition with a simultaneous control of $\H_p$ and $L_2$ norms. 
This is one of the delicate and key points of this paper. 
In fact, such decompositions with simultaneous control of norms turn out to be essential when dealing with duality in the continuous setting. 
In particular, this was one of the motivations of the recent paper \cite{RandX}. 
In this paper, we introduce another version of a sum, the $\boxplus$-sum of two spaces, which is obtained as the completion of a normed space equipped with a quotient norm. 
In classical probability, stopping time arguments allow to show that there is no ``virtual kernel" when trying to embed this abstract space in $L_1$. 
However, in functional analysis and in particular through Grothendieck's formulation of the approximation property, 
we know that hard analysis may be required to decide whether for such completions the kernel is automatically trivial. 
The same remains true in our situation, and we have to rely on Randrianantoanina's work to control these kernels in some cases. 
We show that for the Hardy space $\H_p$ we may use either the new $\boxplus$-sum or the usual sum in the definition, and we deduce the continuous analogue of \eqref{BGI}

\begin{theorem}\label{introBG}
Let $1<p<\infty$. Then
$$L_p(\M)=\H_p \quad \mbox{with equivalent norms.}$$
\end{theorem}

We are also interested in the conditioned Hardy spaces $h_p$, defined in the discrete setting by the norms
$$\|x\|_{h_p^c}=\Big\|\Big(\sum_{n}\E_{n-1}|d_n(x)|^2\Big)^{1/2}\Big\|_p,\quad 
\|x\|_{h_p^r}=\|x^*\|_{h_p^c} \quad \mbox{and} \quad
\|x\|_{h_p^d}=\Big(\sum_n\|d_n(x)\|_p^p\Big)^{1/p}.$$
Then the noncommutative Burkholder inequalities proved in \cite{JX} state that
\begin{equation}\label{BI}
L_p(\M)=h_p \quad \mbox{with equivalent norms for } 1<p<\infty,
\end{equation}
where the $h_p$-space is defined by
$$h_p=\left\{\begin{array}{cl}
h_p^d+h_p^c+h_p^r& \quad \mbox{for} \quad 1\leq p<2 \\
h_p^d\cap h_p^c\cap h_p^r& \quad \mbox{for}\quad 2\leq p<\infty
\end{array}\right..$$ 
A column version of these inequalities, which also holds true for $p=1$, 
have been discovered independently in \cite{jm-riesz} and \cite{Per}:
\begin{equation}\label{D}
H_p^c= \left\{\begin{array}{cl}
h_p^d+h_p^c& \quad \mbox{for} \quad 1\leq p<2 \\
h_p^d\cap h_p^c& \quad \mbox{for}\quad 2\leq p<\infty
\end{array}\right.. 
\end{equation}
In the commutative theory the decomposition for $1\leq p <2$ corresponds to a version of the Davis decomposition into jump part and conditioned square function. 
In the conditioned case, we still have a crucial monotonicity property, and considering the conditioned bracket
$$\langle x,x\rangle_\si=\sum_{t\in \si} \E_{t^-(\s)}|d_t^{\si}(x)|^2$$
for a finite partition $\s$, we define the conditioned Hardy spaces $\hh_p^c$ and $\h_p^c$ of noncommutative martingales 
with respect to the filtration $(\M_t)_{0\leq t \leq 1}$. 
Then we may adapt the theory developed for the $\H_p^c$-spaces to $\hh_p^c$ and $\h_p^c$ and obtain that 
\begin{equation}\label{hh=h-intro}
\hh_p^c=\h_p^c \quad \mbox{with equivalent norms for } 1\leq p <\infty.
\end{equation}
Sometimes we have to resort the theory of noncommutative functions spaces, in particular $L_p$-modules over finite von Neumann algebras 
for comparing different candidates for the $\h_p$-norms. 
Indeed, in \eqref{hh=h-intro} the construction is based 
on free amalgamated products and use the free analogue of Rosenthal inequalities.  
This complementation result implies the conditioned analogue of Theorem \ref{duality} and injectivity results for $1<p<\infty$. 
At the time of this writing we do not know if the injectivity result still holds true for $p=1$, i.e., if $\h_1^c$ embeds into $L_1(\M)$. 
We will need to consider the corresponding subspace of $L_1(\M)$, denoted by $\lh_1^c$. 
Note that in this case the space $\bmo^c$ is easier to describe. 
It is defined as the set of operators $x\in L_2(\M)$ such that 
$$\sup_{0\leq t\leq 1} \|\E_t|x-x_{t}|^2\|_\infty <\infty.$$  
We also prove the expected interpolation result. 
To obtain the continuous analogue of the decompositions \eqref{BI} and \eqref{D} for $1<p<2$ and $1\leq p <2$ respectively, 
we need to introduce another diagonal space $h_p^{1_c}\subset h_p^d$, which yields a stronger Davis decomposition, closer to the classical one. 
Then we deduce the continuous analogues of \eqref{BI} and \eqref{D} for $2\leq p<\infty$ by a dual approach. 
Unfortunately, we cannot directly describe the dual space of our continuous analogue of the diagonal space $\h_p^d$. 
We introduce a variant of the Davis decomposition for $1<p<2$ with simultaneous control of $\h_p$ and $L_2$ norms, based on a deep result of Randrianantoanina. 
Here we use again the $\boxplus$-sum and we need to show that the kernel is trivial in this situation. 
As a payoff, we find a nice description of the space $\H_1$, and the continuity of the maps defined on it can be checked on atoms. 
For open problems in this direction we refer to the appendix. 
We obtain that for the conditioned Hardy spaces, the two sums coincide. 
Moreover, it is very easy to see that in the Davis decomposition we may replace the diagonal space $\h_p^d$ by a larger, $L_2$-regularized space $K_p^d=\h_p^d+L_2(\M)$. 
That leads to a satisfactory description of the duality for the conditioned Hardy space 
$$\h_p=\left\{\begin{array}{cl}
\h_p^d+\h_p^c+\h_p^r& \quad \mbox{for} \quad 1< p<2 \\
\J_p^d\cap \h_p^c\cap \h_p^r& \quad \mbox{for}\quad 2\leq p<\infty
\end{array}\right.,$$ 
where $\J_{p}^d$ denotes the dual space of $K_{p'}^d$. 
We obtain the continuous analogue of \eqref{D} and \eqref{BI} respectively:

\begin{theorem}
Let $1\leq p<\infty$. Then
\begin{enumerate}
\item[(i)] $\H_p^c= \left\{\begin{array}{ll}
\h_1^d+\lh_1^c& \quad \mbox{for} \quad p=1 \\
\h_p^d+\h_p^c& \quad \mbox{for} \quad 1< p<2 \\
\J_p^d\cap \h_p^c& \quad \mbox{for}\quad 2\leq p<\infty
\end{array}\right.$ with equivalent norms.
\item[(ii)] For $1<p<\infty$, $$L_p(\M)=\h_p \quad \mbox{with equivalent norms.}$$
\end{enumerate}
\end{theorem}

By approximation, we deduce a new characterization of $\BMO^c$. 

\begin{theorem}
Let $1<p<\infty$. Then 
$$L_p(\M)=[\BMO,\H_1]_{\frac{1}{p}}\quad \mbox{with equivalent norms. }$$
\end{theorem}

The paper is organized as follows. 
In Section \ref{sectprel} we recall some necessary preliminaries on ultraproduct of Banach spaces in general, 
and on ultraproduct of von Neumann algebras in particular. 
We also discuss the finite case, and give some background on $L_p$-modules and free Rosenthal inequalities. 
The main part of this paper is developed in Section \ref{sectHp}, 
where we define the Hardy spaces $\hH_p^c$ and $\H_p^c$ of noncommutative martingales with respect to a continuous filtration 
and prove Theorem \ref{eqnormHH} and \eqref{introhH=H}. 
We also transfer injectivity, complementation, duality and interpolation results from the discrete setting to this case. 
The continuous analogue of the noncommutative Burkholder-Gundy inequalities (Theorem \ref{introBG}) is proved in Section \ref{sectBG}, 
where we introduce a variant way of considering the sum of two Banach spaces. 
In our setting this corresponds in some sense to focus on the decomposition at the level of $L_2(\M)$, 
and with the help of Randrianantoanina's results we extend our continuous Burkholder-Gundy decomposition to this stronger sum.  
Section \ref{secthp} is devoted to the study of the conditioned Hardy spaces $\h_p^c$. 
The Davis and Burkholder-Rosenthal inequalities are presented in Section \ref{sectD-BR}, 
in which the diagonal spaces $\h_p^d,\h_p^{1_c}, \K_p^d$ and $\J_{p'}^d$ for $1\leq p <2$ are defined.  
At the beginning of each section, we recall the discrete results that we want to reformulate in the continuous setting, 
and add some details on the discrete proofs. 
At the end of this paper, some open problems are collected in the Appendix. 

Throughout this paper, the notation $a_p\simeq b_p$ means that there exist two positive constants $c$ and $C$ such that
$$c\leq \frac{a_p}{b_p}\leq C  .$$

\vspace{1cm}

\textbf{Aknowledgments}

\vspace{0.5cm}

The second named author would like to thank the Math Department of the University of Illinois, 
where a first version of this paper was done, for its warm hospitality. 
We are grateful to Eric Ricard and Quanhua Xu for numerous fruitful discussions and useful comments, which led to many corrections and improvements.

\section{Preliminaries}\label{sectprel}

\subsection{Noncommutative $L_p$-spaces and martingales with respect to continuous filtrations} 

We use standard notation in operator algebras. We refer to \cite{kar-I,tak-I} for background on von Neumann algebra theory, 
to the survey \cite{px-survey} for details on noncommutative $L_p$-spaces, and to \cite{Haag,Te} in particular for the Haagerup noncommutative $L_p$-spaces. 
In the sequel, even if we will define some $L_p$-spaces in the type III case, 
we will mainly work with noncommutative $L_p$-spaces associated to semifinite von Neumann algebras. 
Let us briefly recall this construction. 
Let $\M$ be a semifinite von Neumann algebra equipped with a normal faithful semifinite trace $\tau$. 
For $0 < p \leq \infty$, we denote by $L_p(\M,\tau)$ or simply $L_p(\M)$ the noncommutative $L_p$-space associated with $(\M,\tau)$. 
Note that if $p=\infty$, $L_p(\M)$ is just $\M$ itself with the operator norm; also recall that for $0<p<\infty$ 
the (quasi) norm on $L_p(\M)$ is defined by
$$\|x\|_p=(\tau(|x|^p))^{1/p}, \quad x\in L_p(\M)$$
where $|x|=(x^*x)^{1/2}$ is the usual modulus of $x$. 

Following \cite{PX}, for $1\leq p <\infty$ and a finite sequence $a=(a_n)_{n\geq 0}$ in $L_p(\M)$ we set 
$$\|a\|_{L_p(\M;\ell_2^c)}=\Big\|\Big(\sum_{n\geq 0} |a_n|^2\Big)^{1/2}\Big\|_p
\quad \mbox{and} \quad 
\|a\|_{L_p(\M;\ell_2^r)}=\|a^*\|_{L_p(\M;\ell_2^c)}.$$
Then $\|\cdot\|_{L_p(\M;\ell_2^c)}$ (resp. $\|\cdot\|_{L_p(\M;\ell_2^r)}$) defines a norm on the family of finite sequences of $L_p(\M)$. 
The corresponding completion is a Banach space, denoted by $L_p(\M;\ell_2^c)$ (resp. $L_p(\M;\ell_2^r)$). 
For $p=\infty$, we define $L_{\infty}(\M;\ell_2^c)$ (respectively $L_{\infty}(\M;\ell_2^r)$) 
as the Banach space of the sequences in $L_{\infty}(\M)$ such that $\sum_{n \geq 0} x_n^*x_n$ 
(respectively $ \sum_{n \geq 0} x_nx_n^*$) converges for the weak-operator topology. 
These spaces will be denoted by  $L_p(\M;\ell_2^c(I))$ and $L_p(\M;\ell_2^r(I))$ when the considered sequences are indexed by $I$. 

Let $(\M_t)_{t\geq 0}$ be an increasing family of von Neumann subalgebras of $\M$ whose union is weak$^*$-dense in $\M$. 
Moreover, we assume that for all $t\geq 0$ there exist normal faithful conditional expectations $\E_t:\M\to \M_t$.
Throughout this paper, we assume that the filtration $(\M_t)_{t\geq 0}$ is right continuous, i.e., 
$\M_t=\bigcap_{s>t}\M_s$ for all $t\geq 0$. 
A family $x=(x_t)_{t\geq 0}$ in $L_1(\M)$ is called a noncommutative martingale with respect to $(\M_t)_{t\geq 0}$ if
$$\E_s(x_t)=x_s , \quad \forall 0\leq s \leq t.$$
If in addition all $x_t$'s are in $L_p(\M)$ for some $1\leq p\leq \infty$, then $x$ is called an $L_p$-martingale. 
In this case we set 
$$\|x\|_p=\sup_{t\geq 0} \|x_t\|_p.$$
If $\|x\|_p <\infty$, we say that $x$ is a bounded $L_p$-martingale. 

Let $x=(x_t)_{t\geq 0}$ be a noncommutative martingale with respect to $(\M_t)_{t\geq 0}$. 
We say that $x$ is a finite martingale if there exists a finite time $T \geq 0$ such that $x_t=x_T$ for all $t\geq T$. 
In this paper, we will only consider finite martingales on $[0,1]$, i.e., $T=1$. 
In this case, for a finite partition $\s=\{0=t_0<t_1<t_2<\cdots<t_n=1\}$ of $[0,1]$ we denote 
$t^+(\s)=t_{j+1}$ the successor of $t=t_j$ and $t^-(\s)=t_{j-1}$ its predecessor, 
and for $t\geq 0$ we define
$$d_t^\s(x)=
\left\{\begin{array}{ll}
x_t-x_{t^-(\s)} & \quad \mbox{for } t>0\\
x_0& \quad \mbox{for } t = 0
\end{array}\right..$$ 
In the sequel, for any operator $x\in L_1(\M)$ we denote $x_t=\E_t(x)$ for all $t\geq 0$.

\subsection{Ultraproduct techniques}\label{ultra}

\subsubsection{Ultraproduct of Banach spaces}

Our approach will be mainly based on ultraproduct constructions. 
Let us first recall the definition and some well-known results on the ultraproducts of Banach spaces.   
Let $\U$ be an ultrafilter on a directed set $\I$. 
They are fixed throughout all this subsection. 
Recall that $\U$ is a collection of subsets of $\I$ such that
\begin{enumerate}
\item[(i)] $\emptyset \notin \U$;
\item[(ii)] If $A,B \subset \I$ such that $A\subset B$ and $A\in \U$, then $B\in \U$;
\item[(iii)] If $A,B\in \U$ then $A\cap B\in \U$;
\item[(iv)] If $A\subset \I$, then either $A\in \U$ or $\I \setminus A \in \U$.
\end{enumerate} 
Let $X$ be a normed vector space. For a family $(x_i)_{i\in \I}$ indexed by $\I$ in $X$, 
we say that $x=\lim_{i,\U} x_i$ is the limit of the $x_i$'s along the ultrafilter $\U$ if 
$$\{i\in \I : \|x-x_i\|<\eps \} \in \U \quad \mbox{for all }  \eps>0.$$
Recall that this limit always exists whenever the family $(x_i)_{i\in \I}$ is in a compact space. 
If $X$ is a dual space, then its unit ball is weak$^*$-compact, 
and any bounded family in $X$ admits a weak$^*$-limit along the ultrafilter $\U$. 
If $X$ is reflexive, since the weak-topology coincide with the weak$^*$-topology, 
we deduce that any bounded family in $X$ admits a weak-limit along the ultrafilter $\U$. 

We now turn to the ultraproduct construction. 
Let us start with the ultraproduct of a family $(X_i)_{i\in \I}$ of Banach spaces. 
Let $\ell_\infty(\{X_i: i\in \I\})$ be the space of bounded families $(x_i)_{i\in\I}\in \prodd_i X_i$ equipped with the supremum norm. 
We define the ultraproduct $\prodd_\U X_i$, also denoted by $\prodd_i X_i/\U$, as the quotient space $\ell_\infty(\{X_i: i\in \I\})/\Nc^\U$, 
where $\Nc^\U$ denotes the (closed) subspace of $\U$-vanishing families, i.e.,
$$\Nc^\U=\{(x_i)_{i\in\I} \in \ell_\infty(\{X_i: i\in \I\}):\lim_{i,\U} \|x_i\|_{X_i}=0\}.$$
We will denote by $(x_i)^\bullet$ the element of $\prodd_\U X_i$ represented by the family $(x_i)_{i\in\I}$. 
Recall that the quotient norm is simply given by 
$$\|(x_i)^\bullet\|=\lim_{i,\U} \|x_i\|_{X_i}.$$
If $X_i=X$ for all $i$, then we denote by $\ell_\infty(\I;X)$ the space of bounded $X$-valued families 
and by $\prodd_\U X$ the quotient space $\ell_\infty(\I;X)/\Nc^\U$, called ultrapower in this case. 
We refer to \cite{Hein,Si} for basic facts about ultraproducts of Banach spaces.
If $(X_i)_{i\in \I},(Y_i)_{i\in \I}$ are two families of Banach spaces and $T_i:X_i\to Y_i$ are linear operators  
uniformly bounded in $i\in \I$, 
we can define canonically the ultraproduct map $T_\U=(T_i)^\bullet$ as 
$$T_\U:\left\{\begin{array}{ccc}
\prodd_\U X_i&\longrightarrow &\prodd_\U Y_i\\
(x_i)^\bullet&\longmapsto &(T_ix_i)^\bullet
\end{array}\right..$$
In the sequel we will often use the following useful fact without any further reference.

\begin{lemma}
Let $(X_i)_{i\in \I}$ be a family of Banach spaces and let $x=(x_i)^\bullet \in \prodd_\U X_i$ be such that 
$\|x\|_{\prodd_\U X_i}=\lim_{i,\U} \|x_i\|_{X_i}<1$. 
Then there exists a family $(\x_i)_{i\in \I} \in  \ell_\infty(\{X_i: i\in \I\})$ such that
$$x=(\x_i)^\bullet \quad \mbox{and} \quad  \|\x_i\|_{X_i}<1 , \forall i\in \I.$$
\end{lemma}

\begin{proof}
Setting 
$$\x_i=\left\{\begin{array}{cc}
x_i&\quad \mbox{if }  \|x_i\|_{X_i}<1\\
0& \quad \mbox{otherwise}
\end{array}\right.,$$
we get a family verifying $\|\x_i\|_{X_i}<1$ for all $i\in \I$. 
Moreover, by the definition of the limit along the ultrafilter $\U$, we have $\lim_{i,\U} \|x_i-\x_i\|_{X_i}=0$. 
Indeed, if we denote $\ell=\lim_{i,\U} \|x_i\|_{X_i}<1$, then for any $\delta>0$ we have
$$A_\delta= \{i\in \I : |\ell-\|x_i\|_{X_i}|<\delta \} \in \U .$$
Observe that for $\delta=\frac{1-\ell}{2}>0$, each $i\in A_\delta$ satisfies $\|x_i\|_{X_i}<\ell+\delta=\frac{1+\ell}{2}<1$. 
Hence for all $\eps>0$, the condition (ii) in the definition of an ultrafilter implies
$$A_{\frac{1-\ell}{2}}\subset \{i\in \I : \|x_i\|_{X_i}<1\}
\subset \{i\in \I : \|x_i-\x_i\|_{X_i}<\eps \}  \in\U .$$
This shows that $(x_i)^\bullet=(\x_i)^\bullet$ and ends the proof. 
\qd

We will need to study the dual space of an ultraproduct. 
For a family of Banach spaces $(X_i)_{i\in \I}$, 
there is a canonical isometric embedding $J$ of $\prodd_{\U}X_i^*$ into $\Big(\prodd_{\U}X_i\Big)^*$ 
defined by
$$(Jx^*|x)=\lim_{i,\U} (x_i^*|x_i)$$
for $x^*=(x_i^*)^\bullet \in \prodd_{\U}X_i^*$ and $x=(x_i)^\bullet \in \prodd_{\U}X_i$. 
Hence we may identify $\prodd_{\U}X_i^*$ with a subspace of $\Big(\prodd_{\U}X_i\Big)^*$. 
These two spaces coincide in the following case. 

\begin{lemma}[\cite{HMo}]\label{dualultrapdct}
 Let $(X_i)_{i\in \I}$ be a family of Banach spaces. 
Then $\Big(\prodd_{\U}X_i\Big)^*= \prodd_{\U}X_i^*$ if and only if $\prodd_{\U}X_i$ is reflexive.
\end{lemma}

Even in the non reflexive case, the subspace $\prodd_{\U}X_i^*$ is ``big'' in $\Big(\prodd_{\U}X_i\Big)^*$ in the sense of the following Lemma. 
This is also a well-known fact of the theory of ultraproducts (see \cite{Si}, Section $11$), 
we include a proof for the convenience of the reader.

\begin{lemma}\label{le:ultrapdct}
 Let $(X_i)_{i\in \I}$ be a family of Banach spaces.
Then the unit ball of $\prodd_{\U}X_i^*$ is weak$^*$-dense in the unit ball of $\Big(\prodd_{\U}X_i\Big)^*$.
\end{lemma}

\begin{proof}
We first prove that for two normed vector spaces $X$ and $Y$ such that $Y$ is a norming subspace of $X^*$, 
the unit ball of $Y$ is weak$^*$-dense in the unit ball of $X^*$. 
Suppose that $B_Y$ is not weak$^*$-dense in $B_{X^*}$, then by the Hahn-Banach Theorem there exist $x^*\in B_{X^*}$ and $x\in X$ such that
$|(x^*|x)|=1$ and for all $y\in B_Y$,  $|( y|x)|<\delta$, $0<\delta<1$.
Since $Y$ is a norming subspace of $X^*$ we have
$$\|x\|_X=\sup_{y\in B_Y}|(y|x)|<\delta.$$
Then
$$1=|(x^*|x)|\leq \|x^*\|_{X^*}\|x\|_X<\delta,$$
which contradicts  $\delta<1$. 
It remains to apply this general result to $X=\prodd_{\U}X_i$ and $Y=\prodd_{\U}X_i^*$. 
It suffices to see that $\prodd_{\U}X_i^*$ is a norming subspace of $\Big(\prodd_{\U}X_i\Big)^*$.
Let $x=(x_i)^\bullet \in \prodd_{\U}X_i$.
For each $i\in \I$, there exists $z_i^*\in B_{X_i^*}$ such that $\|x_i\|_{X_i}=|(z_i^*|x_i)|$.
Multiplying by a complex number of modulus $1$, we can assume that $\|x_i\|_{X_i}=(z_i^*|x_i)$.
Thus
 \begin{align*}
 \|x\|_{\prodd_{\U}X_i} &= \lim_{i,\U} \|x_i\|_{X_i} 
= \lim_{i,\U} (z_i^*|x_i)\\
&\leq \sup_{y^*=(y_i^*)^\bullet \in B_{\prodd_{\U}X_i^*}} |\lim_{i,\U}( y_i^*|x_i)|
 = \sup_{y^*\in B_{\prodd_{\U}X_i^*}}|(y^*|x)|  . 
 \end{align*}
\qd

\subsubsection{Ultraproduct of von Neumann algebras : the general case}

We now consider the ultraproduct construction for von Neumann algebras. 
For convenience we will simply consider ultrapowers, but all the following discussion remains valid for ultraproducts. 
It is well-known that if $\A$ is a C$^*$-algebra, then $\prodd_\U \A$ is still a C$^*$-algebra. 
On the other hand, the class of von Neumann algebras is not closed under ultrapowers. 
However, according to Groh's work \cite{Groh}, 
we know that the class of the preduals of von Neumann algebras is closed under ultrapowers. 
Let $\M$ be a von Neumann algebra. 
Then $\prodd_\U \M_*$ is the predual of a von Neumann algebra denoted by
$$\tilde{\M}_\U=\Big(\prodd_\U \M_*\Big)^*.$$
Moreover, $\prodd_\U \M$ identifies naturally to a weak$^*$-dense subalgebra of $\tilde{\M}_\U$. 
As detailed in \cite{Ray}, we can also see $\tilde{\M}_\U$ as the von Neumann algebra generated by $\prodd_\U \M$ in $B(\prodd_\U \H)$, 
where we have a standard $*$-representation of $\M$ over the Hilbert space $\H$.  
Following Raynaud's work \cite{Ray}, for all $p>0$ we can construct an isometric isomorphism 
$$\Lambda_p:\prodd_\U L_p(\M) \to L_p(\tilde{\M}_\U),$$
which preserves the following structures
\begin{itemize}
\item conjugation: $\Lambda_p( (x_i^*)^\bullet)=\Lambda_p((x_i)^\bullet)^*$,
\item absolute values: $\Lambda_p( (|x_i|)^\bullet)=|\Lambda_p((x_i)^\bullet)|$,
\item $\prodd_\U\M$-bimodule structure: 
$\Lambda_p( (a_i)^\bullet\cdot (x_i)^\bullet \cdot (b_i)^\bullet)
=(a_i)^\bullet\cdot \Lambda_p((x_i)^\bullet) \cdot (b_i)^\bullet$,
\item external product: 
$\Lambda_r( (x_i)^\bullet\cdot (y_i)^\bullet)=\Lambda_p((x_i)^\bullet)\cdot \Lambda_q((y_i)^\bullet)$ for $\frac{1}{r}=\frac{1}{p}+\frac{1}{q}$,
\end{itemize}
for all $(x_i)^\bullet \in \prodd_\U L_p(\M)$, $(y_i)^\bullet \in \prodd_\U L_q(\M)$ and $(a_i)^\bullet,(b_i)^\bullet \in \prodd_\U\M$. 
In the sequel we will identify the spaces $\prodd_\U L_p(\M)$ and $L_p(\tilde{\M}_\U)$ without any further reference. 

\subsubsection{Ultraproduct of von Neumann algebras : the finite case}

We now discuss the finite situation. Let $\M$ be a finite von Neumann algebra equipped with a normal faithful normalized trace $\tau$. 
In this case the usual von Neumann algebra ultrapower is $\M_\U=\ell_\infty(\I;\M)/\I^\U$, where 
$$\I^\U=\{(x_i)_{i\in\I} \in \ell_\infty(\I;\M) : \lim_{i,\U} \tau(x_i^*x_i)=0\}.$$
According to Sakai (\cite{Sakai}), $\M_\U$ is a finite von Neumann algebra when equipped with the ultrapower map of the trace $\tau$, denoted by 
$\tau_\U$ and defined by 
$$\tau_\U((x_i)^\bullet)=\lim_{i,\U} \tau(x_i).$$
Note that this definition is compatible with $\I_\U$, and defines a normal faithful normalized trace on $\M_\U$. 
We may identify $\M_\U$ as a dense subspace of $L_1(\M_\U)$ via the map $x\in \M_\U \mapsto \tau_\U(x \cdot)\in L_1(\M_\U)$. 
Then for $x=(x_i)^\bullet \in \M_\U$, we have $\|x\|_1=\lim_{i,\U} \|x_i\|_1$. 
Observe that this does not depend on the representing family $(x_i)$ of $x$. 
Let us define the map
$$\iota:\left\{\begin{array}{ccc}
\M_\U&\longrightarrow &L_1(\tilde{\M}_\U)\\
(x_i)^\bullet&\longmapsto &(\tau(x_i \cdot))^\bullet
\end{array}\right..$$
We see that this map is well-defined, and it is clear that $\|\iota((x_i)^\bullet)\|_1=\lim_{i,\U} \|x_i\|_1$. 
Hence by density we can extend $\iota$ to an isometry from $L_1(\M_\U)$ into $L_1(\tilde{\M}_\U)$. 
Since $L_1(\M_\U)$ is stable under $\tilde{\M}_\U$ actions, Theorem III.$2.7$ of  \cite{tak-I} gives a central projection $e_\U$ in $\tilde{\M}_\U$ such that 
$L_1(\M_\U)=L_1(\tilde{\M}_\U)e_\U$. 
We can see that $e_\U$ is the support projection of the trace $\tau_\U$. 
In the sequel we will identify $\M_\U$ as a subalgebra of $\tilde{\M}_\U$, by considering $\M_\U=\tilde{\M}_\U e_\U$. 
More generally we have
\begin{equation}\label{LpMU}
L_p(\M_\U)=L_p(\tilde{\M}_\U) e_\U \quad \mbox{for all }  0<p\leq \infty.
\end{equation}
The subspace $L_p(\M_\U)$ can be characterized by using the notion of $p$-equiintegrability as follows. 
Let us recall the definition of a $p$-equiintegrable subset of a noncommutative $L_p$-space introduced 
in \cite{tak-I} for $p=1$ and by Randrianantoanina in \cite{ran-kadec} for any $p$.

\begin{defi}
Let $0<p<\infty$. 
A bounded subset $K$ of $L_p(\M)$ is called $p$-equiintegrable if 
$$\lim_{n\to\infty} \sup_{x\in K} \|e_nxe_n\|_p=0$$
for every decreasing sequences $(e_n)_n$ of projections of $\M$ which weak$^*$-converges to $0$. \\
If $p=1$, we say that $K$ is uniformly integrable.
\end{defi}

Recall that finite subsets of $L_p(\M)$ are $p$-equiintegrable. 
We will use the following characterization coming from \cite[Corollary 2.7]{HRS} in the case $1\leq p <\infty$, and \cite[Lemma 1.3]{SX} for $0<p<1$. 

\begin{lemma}\label{pequiintegr}
Let $0< p<\infty$ and $(x_i)_{i\in\I}$ be a bounded family in $L_p(\M)$. 
Then the following assertions are equivalent.
\begin{enumerate}
\item[(i)] $(x_i)_{i\in\I}$ is $p$-equiintegrable;
\item[(ii)] $\displaystyle\lim_{T\to\infty} \sup_i \mathrm{dist}_{L_p}(x_i,TB_\M)=0$;
\item[(iii)] $\displaystyle\lim_{T\to\infty} \lim_{i,\U} \|x_i\1(|x_i|>T)\|_p=0$, 
\end{enumerate}
where for $a\geq 0$, $\1(a>T)$ denotes the spectral projection of $a$ corresponding to the interval $(T,\infty)$.  
\end{lemma}

Observe that \eqref{LpMU} implies that for $0<p<\infty$ and $x\in L_p(\tilde{\M}_\U)$ 
$$x\in L_p(\M_\U) \Leftrightarrow x=xe_\U.$$
Moreover, in the finite case, $e_\U$ corresponds to the projection denoted by $s_e$ in \cite{RX}.  
Hence Theorem $4.6$ of \cite{RX} yields the following characterization of $L_p(\M_\U)$. 

\begin{theorem}\label{LpMUpequi}
Let $0<p<\infty$ and $x\in L_p(\tilde{\M}_\U)$. 
Then the following assertions are equivalent.
\begin{enumerate}
\item[(i)] $x\in L_p(\M_\U)$;
\item[(ii)] $x$ admits a $p$-equiintegrable representing family $(x_i)_{i\in\I}$.
\end{enumerate}
\end{theorem}

For $0<p<\tp\leq \infty$, since $\M$ is finite we have a contractive inclusion $L_{\tp}(\M)\subset L_p(\M)$. 
Let us denote by $I_{\tp,p}:\prodd_\U L_{\tp}(\M) \to \prodd_\U L_p(\M)$ the contractive ultraproduct map of the componentwise inclusion maps.  
Note that although the componentwise inclusion maps are injective, the ultraproduct map $I_{\tp,p}$ is not. 
However, its restriction to $L_{\tp}(\M_\U)$ is injective. 
Indeed, using the weak$^*$-density of $\prodd_\U \M$ in $\tilde{\M}_\U$, 
we see that $I_{\tp,p}$ is bimodular under the action of $\tilde{\M}_\U$. 
Hence, if $x\in L_{\tilde{p}}(\tilde{\M}_\U)$ satisfies $x=xe_\U$, then $I_{\tp,p}(x)=I_{\tp,p}(xe_\U)=I_{\tp,p}(x)e_\U \in L_p(\M_\U)$. 
This shows that $I_{\tp,p}:L_{\tp}(\M_\U)\to L_{p}(\M_\U)$. 
Moreover, since $\M_\U$ is finite, the map $I_{\tp,p}$ coincides on $L_{\tp}(\M_\U)$ with the natural inclusion $L_{\tp}(\M_\U)\subset L_{p}(\M_\U)$. \\
We deduce from Theorem \ref{LpMUpequi} the following description of the space $L_p(\M_\U)$, viewed as a subspace of $L_p(\tilde{\M_\U})$. 

\begin{lemma}\label{L_p(NU)}
Let $0< p <\infty$. Then 
$$L_p(\M_\U) = \overline{\bigcup_{\tilde{p}>p}I_{\tp,p}(L_{\tilde{p}}(\tilde{\M}_\U))}^{\|\cdot\|_{L_p(\tilde{\M}_\U) }}.$$
\end{lemma}

\begin{proof}  
Let us first show that 
$I_{\tp,p}(L_{\tilde{p}}(\tilde{\M}_\U)) \subset L_p(\M_\U)$ 
for $\tilde{p}>p$.  
Let $x=(x_i)^\bullet \in L_{\tilde{p}}(\tilde{\M}_\U)$. 
By Theorem \ref{LpMUpequi}, it suffices to prove that the family $(x_i)_{i\in\I}$ is $p$-equiintegrable. 
For $T >0$ and each $i\in \I$ we have
$$ \|x_i \1(|x_i|>T)\|_p
\leq \|x_i |x_i|^{\frac{\tp}{p}-1}T^{1-\frac{\tp}{p}}\|_p\leq T^{1-\frac{\tp}{p}}\|x_i\|_{\tp}^{\frac{\tp}{p}}.$$
Taking the limit along the ultrafilter $\U$ we obtain
$$ \lim_{i,\U} \|x_i \1(|x_i|>T)\|_p
\leq T^{1-\frac{\tp}{p}}\|x\|_{L_{\tp}(\tilde{\M}_\U)}^{\frac{\tp}{p}}.$$
Since $1-\frac{\tp}{p} <0$, this tends to $0$ as $T$ goes to $\infty$. 
We conclude that $(x_i)_{i\in\I}$ is $p$-equiintegrable by using Lemma \ref{pequiintegr}. 
Conversely, let $x\in L_p(\M_\U)$. Since $\M_\U$ is finite, $L_{\tilde{p}}(\M_\U)$ is dense in $L_p(\M_\U)$ for all $\tilde{p}>p$. 
Hence for all $\eps>0$ there exists $y\in L_{\tilde{p}}(\M_\U)$ such that $\|x-y\|_{L_p(\M_\U)} <\eps$. 
Since $L_p(\M_\U)$ is isometrically embedded into $L_p(\tilde{\M}_\U)$ and $y=I_{\tp,p}(y)\in I_{\tp,p}(L_{\tilde{p}}(\tilde{\M}_\U))$, this ends the proof. 
\qd

For $p=1$, we can translate the notion of uniform integrability in terms of compactness as follows. 

\begin{theorem}[\cite{tak-I}]\label{weakcompact}
Let $K$ be a bounded subset of the predual $\M_*$ of $\M$. 
Then the following assertions are equivalent.
\begin{enumerate}
\item[(i)] $K$ is uniformly integrable;
\item[(ii)]$K$ is weakly relatively compact.
\end{enumerate}
\end{theorem}

Let us consider
$$i_\U:\left\{\begin{array}{ccc}
(\M,\tau) &\longrightarrow  &(\M_\U,\tau_\U) \\
x&\longmapsto &(x)^\bullet
\end{array}\right..$$
Since $i_\U$ is trace preserving, this yields an isometric embedding of $L_1(\M)$ into $L_1(\M_\U)$. 
Hence we get natural inclusions
$$L_1(\M)\subset L_1(\M_\U)\subset L_1(\tilde{\M}_\U),$$
where $L_1(\tilde{\M}_\U)$ represents the bounded families in $L_1(\M)$, $L_1(\M_\U)$ corresponds to the weakly converging families along $\U$ 
and $L_1(\M)$ consists of the collection of the constants families. 

We end this subsection with the introduction of a conditional expectation. 
We set 
$$\E_\U=(i_\U)^*:\M_\U \to \M.$$
Then $\E_\U$ is a normal faithful conditional expectation on $\M_\U$. 
Since $\E_\U$ is trace preserving, for all $1\leq p \leq \infty$ 
we can extend $\E_\U$ to a contraction from $L_p(\M_\U)$ onto $L_p(\M)$, still denoted by $\E_\U$. 
Moreover, for $1<p\leq \infty$ and $x=(x_i)^\bullet \in L_p(\M_\U)$ we have 
$$\E_\U(x)=\wst L_p \mbox{-} \lim_{i,\U} x_i.$$
Indeed, for $y\in L_{p'}(\M)$ and $\frac{1}{p}+\frac{1}{p'}=1$ we can write
\begin{equation}\label{EU}
\tau(\E_\U(x)^*y)= \tau_\U (x^*i_\U(y))=\lim_{i,\U}\tau(x_i^*y).
\end{equation}
Note that since in this case $L_p(\M)$ is a dual space, the weak$^*$-limit of the $x_i$'s exists for any bounded family $(x_i)$. 
Hence we may extend $\E_\U$ to $L_p(\tilde{\M}_\U)$ for $1<p\leq \infty$. 
However this extension, still denoted by $\E_\U$ in the sequel, is no longer faithful.  
For $1<p<\infty$, since $L_p(\M)$ is reflexive, the weak$^*$-limit corresponds to the weak-limit. 
Recall that by Theorem \ref{weakcompact}, $L_1(\M_\U)$ corresponds to the weakly converging families. 
Thus \eqref{EU} implies that for $1\leq p <\infty$ and $x=(x_i)^\bullet \in L_p(\M_\U)$ we have 
$$\E_\U(x)=\w  L_p \mbox{-} \lim_{i,\U} x_i.$$

\subsection{$L_p$ $\M$-modules}

We will use the theory of $L_{p}$-modules
introduced in \cite{JS}. This structure will help us to prove
duality and interpolation results for different $\H_p$-spaces. 
We may say that $L_{p}$-modules are $L_p$-versions of Hilbert $W^*$-modules. 
Let $\M$ be a von Neumann algebra. 

\begin{defi}\label{defLpmodule}
Let $1\leq p<\infty$. 
A right $\M$-module $X$ is called a right $L_p$ $\M$-module if 
it has an $L_{p/2}(\M)$-valued inner product, i.e. there is a sesquilinear map $\langle \cdot ,\cdot \rangle:X\times X\to L_{p/2}(\M)$, conjugate linear in the first variable, such that for all $x,y \in X$ and all $a\in \M$ 
 \begin{enumerate}
 \item[(i)] $\langle x,x\rangle \geq 0$, and $\langle x,x\rangle = 0 \Leftrightarrow x=0$, 
 \item[(ii)] $\langle x,y\rangle^*= \langle y,x\rangle$,
 \item[(iii)] $\langle x,ya\rangle  = \langle x,y\rangle a $,
 \end{enumerate}
and $X$ is complete in the inherited (quasi)norm 
$$\|x\| = \|\langle x,x\rangle\|_{p/2}^{1/2} .$$
We call $X$ a right $L_\infty$ $\M$-module if it has an $L_{\infty}(\M)$-valued inner product 
and is complete with respect to the strong operator topology, i.e. the topology arising from the seminorms
$$\|x\|_\varphi=(\varphi(\langle x,x\rangle))^{1/2}, \quad \varphi \in \M_*^+.$$
\end{defi}

The basic example of such a right $L_p$ $\M$-module is given by the column $L_p$-space $L_p(\M;\ell_2^c)$. 
Here for $a\in \M$ and $x=\sum_{n\geq 0} e_{n,0} \ten x_n, y=\sum_{n\geq 0} e_{n,0} \ten y_n \in L_p(\M;\ell_2^c)$ we define the right $\M$-module action by
$$x\cdot a=\sum_{n\geq 0} e_{n,0} \ten (x_n a).$$
Then we define the following  $L_{p/2}(\M)$-valued inner product 
$$\langle x,y\rangle_{L_p(\M;\ell_2^c)}=\sum_{n\geq 0} x_n^*y_n \in L_{p/2}(\M).$$
Let us highlight another important example of $L_p$-module introduced in \cite{JD}. 
Let $\E:\M\to \N$ be a normal faithful conditional expectation, where $\N$ is a von Neumann subalgebra of the finite von Neumann algebra $\M$. 
Then for $0<p\leq \infty$ and $x,y\in L_p(\M)$ we may consider the bracket
$$\langle x,y\rangle_{L_p^c(\M;\E)}=\E(x^*y) \in L_{p/2}(\N),$$
where $\E$ denotes the extension of $\E$ to $L_{p/2}(\M)$ (see \cite{JX} for details on conditional expectations). 
It is clear that this defines an $L_{p/2}(\N)$-valued inner product, 
and the associated $L_p$ $\N$-module is denoted by $L_p^c(\M;\E)$. 
This means that $L_p^c(\M;\E)$ is the completion of $\M$ with respect to the quasi-norm 
$$\|x\|_{L_p^c(\M;\E)}=\|\E(x^*x)\|_{p/2}^{1/2}.$$
For $p=\infty$ we denote by $L_\infty^{c,\st}(\M;\E)$ the closure with respect to the strong operator topology. 
Recall that for $p\geq 2$ the space $L_p^c(\M;\E)$ can also be defined as the closure of $L_p(\M)$. 
It is proved in Proposition $2.8$ of \cite{JD} that this latter example is similar to the former one.  
More precisely, this Proposition shows that $L_p^c(\M;\E)$ is isometrically isomorphic, as a module, to a complemented subspace of $L_p(\N;\ell_2^c)$. 
As a consequence, we obtain that $\|\cdot\|_{L_p^c(\M;\E)}$ is a norm.
We also deduce from the well-known duality and interpolation results for the column $L_p$-space $L_p(\N;\ell_2^c)$ 
the same results for $L_p^c(\M;\E)$.

\begin{prop}\label{dualintpolLpc(M,E)}
Let $1<p<\infty$.   
\begin{enumerate}
\item[(i)] Let $\frac{1}{p}+\frac{1}{p'}=1$. Then $(L_p^c(\M;\E))^*=L_{p'}^c(\M;\E)$ isometrically.
\item[(ii)] We have $(L_1^c(\M;\E))^*=L_{\infty}^{c,\st}(\M;\E)$ isometrically.
\item[(iii)] Let $1\leq p_1 < p_2 \leq \infty$ and $0<\theta<1$ be such that 
$\frac{1}{p}=\frac{1-\theta}{p_1}+\frac{\theta}{p_2}$.  Then
$$L_p^c(\M;\E)=[L_{p_1}^c(\M;\E),L_{p_2}^c(\M;\E)]_{\theta}\quad \mbox{with equivalent norms}.$$
\end{enumerate}
\end{prop}

\re\label{embeddingL1c(M,E)}
Since $L_p(\M)$ is dense in $L_p^c(\M;\E)$ for $p\geq 2$, 
Proposition \ref{dualintpolLpc(M,E)} (i) implies that for $1<p\leq 2$, $L_p^c(\M;\E)$ embeds into $L_p(\M)$. 
This still holds true for $p=1$. 
Indeed, $L_1^c(\M;\E)$ is described as a subspace of $L_1(\M)$ in \cite{JPa}, (c) p. 28, as follows
$$L_1^c(\M;\E)=L_2(\M)L_2(\N)\quad \mbox{with equivalent norms}.$$
Recall that $L_2(\M)L_2(\N)$ is defined as the subset of elements $x\in L_1(\M)$ which factorizes as $x=ya$ with $y\in L_2(\M)$ and $a \in L_2(\N)$. 
The norm is given by 
$$\|x\|_{L_2(\M)L_2(\N)}=\inf_{x=ya} \|y\|_{L_2(\M)}\|a\|_{L_2(\N)}.$$
\mar

Proposition $2.8$ of \cite{JD} has been extended in \cite{JS} for any $L_p$ $\M$-module. 
By Theorem $3.6$ of \cite{JS}, a right $\M$-module $X$ is a right $L_p$ $\M$-module if and only if $X$ is a "column sum of $L_p$-spaces" in the following sense.

\begin{theorem}[\cite{JS}]\label{pricipalLpmodule}
Let $X$ be a right $L_p$ $\M$-module. 
Then $X$ is isometrically isomorphic, as an $L_p$-module, to a principal $L_p$-module, i.e., 
there exists a set $(q_\alpha)_{\alpha\in I}$ of projections in $\M$ such that 
$$X\cong \Big\{(\xi_\alpha)_{\alpha\in I} : \xi_\alpha \in q_\alpha L_p(\M), \sum_\alpha \xi_\alpha^*\xi_\alpha \in L_{p/2}(\M)\Big\}.$$
\end{theorem}

This latter set is denoted by $\oplus_I q_\alpha L_p(\M)$ and endowed with the norm 
$\|(\xi_\alpha)_{\alpha}\|=\Big\|\sum_\alpha \xi_\alpha^*\xi_\alpha\Big\|_{p/2}^{1/2}$. 
In the finite case, if we have a projective system of $L_p$ $\M$-modules in the sense of the following Corollary with some density property, 
then we may represent this family by using the same set of projections. 

\begin{cor}\label{familyLpmodule}
Let $\M$ be a finite von Neumann algebra. 
Let $(X_p)_{1\leq p \leq \infty}$ be a family of right $\M$-modules such that  
\begin{enumerate}
 \item[(i)] $X_p$ is an $L_p$ $\M$-module for all $1\leq p \leq \infty$.
 \item[(ii)] There exists a family of modular maps 
$I_{q,p}:X_{q}\to X_p$ for $p\leq q$ satisfying $I_{p,p}=id_{X_p}$ and $I_{q,p}\circ I_{r,q}=I_{r,p}$ for $p\leq q \leq r$.
 \item[(iii)] The inner products are compatible with the maps $I_{q,p}$, i.e., 
$$\langle x,y\rangle_{X_q}=\langle I_{q,p}(x),I_{q,p}(y)\rangle_{X_p}$$
for $p\leq q$ and $x,y\in X_q$.
 \item[(iv)] $I_{\infty,p}(X_\infty)$ is dense in $X_p$ for all $1\leq p \leq \infty$. 
\end{enumerate}
Then there exists a set $(q_\alpha)_{\alpha\in I}$ of projections in $\M$ such that for all $1\leq p \leq \infty$, 
$X_p$ is isometrically isomorphic, as an $L_p$-module, to $\oplus_I q_\alpha L_p(\M)$. 
\end{cor}

\begin{proof}
Observe that (iii) implies that the maps $I_{q,p}$ are contractive and injective. 
Indeed, for $p\leq q$ and $x\in X_q$, since $\M$ is finite we have
\begin{align*}
 \|I_{q,p}(x)\|_{X_p}
&=\|\langle I_{q,p}(x),I_{q,p}(x)\rangle_{X_p}\|_{p/2}^{1/2}\\
&=\|\langle x,x\rangle_{X_q}\|_{p/2}^{1/2}\leq \|\langle x,x\rangle_{X_q}\|_{q/2}^{1/2}\\
&=\|x\|_{X_q}.
\end{align*}
For the injectivity, if $I_{q,p}(x)=0$ then $\langle I_{q,p}(x),I_{q,p}(x)\rangle_{X_p}=0$ in $L_{p/2}(\M)$. 
By (iii), this implies that $\langle x,x\rangle_{X_q}=0$ in $L_{q/2}(\M)$, hence $x=0$ in $X_q$ by (i) of Definition \ref{defLpmodule}. 
We now turn to the proof of the Corollary. 
We first apply Theorem \ref{pricipalLpmodule} to the $L_\infty$ $\M$-module $X_\infty$ and obtain 
a set $(q_\alpha)_{\alpha\in I}$ of projections in $\M$ and an isometric isomorphism of $L_p$-modules 
$\phi_\infty:X_\infty \to \oplus_I q_\alpha L_\infty(\M)$. 
We may extend this isomorphism to $X_p$ by density as follows. 
For $1\leq p <\infty$ and $x=I_{\infty,p}(y) \in I_{\infty,p}(X_\infty)$ we set 
$$\phi_p(x)=\phi_\infty(y) \in \oplus_I q_\alpha L_\infty(\M).$$
Since $\M$ is finite, we have a contractive inclusion $\oplus q_\alpha L_\infty(\M) \subset \oplus q_\alpha L_p(\M)$ 
and $\phi_p$ preserves the $L_{p/2}(\M)$-valued inner product. 
Indeed, for $x_1=I_{\infty,p}(y_1), x_2=I_{\infty,p}(y_2) \in I_{\infty,p}(X_\infty)$, 
the modularity of $\phi_\infty$ implies
\begin{align*}
\langle \phi_p(x_1),\phi_p(x_2)\rangle_{\oplus q_\alpha L_p(\M)}
&= \langle \phi_\infty(y_1),\phi_\infty(y_2)\rangle_{\oplus_I q_\alpha L_\infty(\M)}
=\langle y_1,y_2\rangle_{X_\infty}\\
&= \langle I_{\infty,p}(y_1),I_{\infty,p}(y_2)\rangle_{X_p} \quad \mbox{ by (iii)}\\
&=\langle x_1,x_2\rangle_{X_p}.
\end{align*}
Hence by the density assumption (iv) we can extend $\phi_p$ to an isometric homomorphism of $L_p$-modules on $X_p$ to $\oplus q_\alpha L_p(\M)$. 
Since $\oplus q_\alpha L_\infty(\M)$ is dense in $\oplus q_\alpha L_p(\M)$, by the same way we can construct $\phi_p^{-1}$. 
Thus we obtain an isometric isomorphism of $L_p$-modules $\phi_p$ which makes the following diagram commuting
$$
\xymatrix{
    X_\infty \ar@{<->}[r]^{\phi_\infty}  \ar@{->}[d]_{I_{\infty,p}}   &\oplus q_\alpha L_\infty(\M) \ar[d]^{id} \\
    X_p \ar@{<->}[r]_{\phi_p} & \oplus q_\alpha L_p(\M)
  }.
$$
\qd

In this situation, we may deduce the following results from some well-known facts on the column $L_p$-spaces $\oplus q_\alpha L_p(\M)$. 

\begin{cor}\label{Lpmodule}
Let $\M$ be a finite von Neumann algebra. 
Let $(X_p)_{1\leq p \leq \infty}$ be a family of right $\M$-modules as in Corollary \ref{familyLpmodule}.  
\begin{enumerate}
\item[(i)] Let $1\leq p < \infty$ and $\frac{1}{p}+\frac{1}{p'}=1$. Then $(X_p)^*=X_{p'}$ isometrically.
\item[(ii)] Let $1\le p_1 <p< p_2 \leq \infty$ and $0<\theta<1$ be such that 
$\frac{1}{p}=\frac{1-\theta}{p_1}+\frac{\theta}{p_2}$.  Then
$$X_p=[X_{p_1},X_{p_2}]_{\theta}.$$
\end{enumerate}
\end{cor}

\subsection{Free Rosenthal inequalities}\label{subsectfreeRos}

Amalgamated free products and the free analogue of Rosenthal inequalities (\cite{JPX}) 
will be key tools needed for the study of the continuous analogue of the conditioned Hardy spaces $h_p^c$. 
After briefly recalling the notations, we will present the Rosenthal/Voiculescu type inequality stated in \cite{JPX} in the amalgamated free product case.  
Then we will extend it by duality to the case $1\leq p<2$. 

Voiculescu introduced the notion of amalgamated free product of C$^*$-algebras in \cite{Voi}, and we refer to \cite{JPX} and \cite{JPa}
for the construction in the von Neumann setting. 
Let $\AA_0,\AA_1, \cdots, \AA_N$ be a finite family of von Neumann algebras having $\M$ as a common von Neumann subalgebra. 
Suppose that $\M$ is finite and equipped with a normal faithful normalized trace $\tau$. 
We also assume that $\E_n=\E_{\M|\AA_n}$ are faithful conditional expectations.  
Recall that the amalgamated free product $\N=\ast_\M \AA_n$ can be seen as 
$$ \ast_\M \AA_n =\Big(\M \oplus \bigoplus_{m\geq 1} \bigoplus_{j_1 \neq j_2 \neq \cdots \neq j_m} 
\Acirc_{j_1}\Acirc_{j_2}\cdots \Acirc_{j_m}\Big)'',$$
where $\Acirc_n$ denotes the mean-zero subspace
$$\Acirc_n=\{a_n \in \AA_n : \E_n(a_n)=0\}.$$
We denote by $\rho:\M \to \N$ the $*$-homomorphism which sends $\M$ to the amalgamated copy, 
and by $\E_\M:\N \to \M$ the normal faithful conditional expectation onto the amalgamated copy. 
The von Neumann algebra $\AA_n$ can be identified as von Neumann subalgebra of $\N$ via the $*$-homomorphism
$$\rho_n:\AA_n \to \N,$$
which sends $\AA_n$ to the $n$-th copy of $\N=\ast_\M \AA_n$. 
With this identification, we may use either $\E_\M$ or $\E_n$ (=$\E_\M \circ \rho_n$ rigorously) indistinctively over $\AA_n$. 
In the sequel we will always use the notation $\E_\M$. 
Moreover, we may equip the von Neumann algebras $\AA_0,\AA_1, \cdots, \AA_N$ and $\N$ with the normal faithful normalized trace defined by
$$\tr=\tau \circ \E_\M.$$
We denote by $\E_{\AA_n}:\N \to \AA_n$ the conditional expectation onto $\AA_n$. 
It turns out that $\AA_0, \cdots, \AA_N$ are freely independent over $\E_\M$. 
For a given nonnegative integer $d$, we denote the homogeneous part of degree $d$ of the algebraic free product by $\Sigma_d$
$$\Sigma_d=\bigoplus_{j_1 \neq j_2 \neq \cdots \neq j_d} 
\Acirc_{j_1}\Acirc_{j_2}\cdots \Acirc_{j_d}.$$
For $d=0$, $\Sigma_d$ is simply $\M$. 
We define $\N_d$ as the weak$^*$-closure of $\Sigma_d$ in $\N$. 
This means that $\N_d$ is the subspace of $\N$ of homogeneous free polynomials of degree $d$. 
We also define $X_p^d$ as the closure of $\Sigma_d$ in $L_p(\N)$, 
and $Y_{p,c}^d$ (resp. $Y_{p,r}^d$) as the closure of $\Sigma_d$ in $L_p^c(\N;\E_\M)$ (resp. $L_p^r(\N;\E_\M)$). 
We will need the complementation result below. 

\begin{prop}\label{XpYpcompl}
Let $1\leq p \leq \infty$ and $d$ be a nonnegative integer. 
Let 
$$\PP_d:\M \oplus \bigoplus_{m\geq 1} \bigoplus_{j_1 \neq j_2 \neq \cdots \neq j_m}\Acirc_{j_1}\Acirc_{j_2}\cdots \Acirc_{j_m} \to \Sigma_d$$
be the natural projection. 
Then
\begin{enumerate}
\item[(i)] $\PP_d$ extends to a bounded projection (of norm less than $\max(4d,1)$) from $L_p(\N)$ onto $X_p^d$.
\item[(ii)] $\PP_d$ extends to a contractive projection from $L_p^c(\N;\E_\M)$ (resp. $L_p^r(\N;\E_\M)$) onto $Y_{p,c}^d$ (resp. $Y_{p,r}^d$).
\end{enumerate}
\end{prop}

\begin{proof}
Assertion (i) is stated in \cite{JPX}. 
It can be deduced from the case $p=\infty$ proved in \cite{RiX} by transposition and complex interpolation. 
The second point follows easily from orthogonality. 
Indeed, let $x\in \M \oplus \bigoplus_{m\geq 1} \bigoplus_{j_1 \neq j_2 \neq \cdots \neq j_m}\Acirc_{j_1}\Acirc_{j_2}\cdots \Acirc_{j_m}$. 
We can write $x=\sum_{m\geq 0} \PP_m(x)$. Then by orthogonality we get
$$\E_\M(x^*x)=\sum_{m\geq 0} \E_\M\big(\PP_m(x)^*\PP_m(x)\big) \geq \E_\M\big(\PP_d(x)^*\PP_d(x)\big).$$
Hence
$$\|\PP_d(x)\|_{L_p^c(\N;\E_\M)}=\|\E_\M\big(\PP_d(x)^*\PP_d(x)\big)\|_{p/2}^{1/2}
\leq \|\E_\M(x^*x)\|_{p/2}^{1/2}=\|x\|_{L_p^c(\N;\E_\M)},$$
and (ii) is proved. 
\qd

In the sequel we will only consider the case of words of length $1$, i.e., $d=1$.
We also introduce the space $Z_p$, defined for $1\leq p \leq \infty$ as the completion of $\Sigma_1$ with respect to the norm
$$\Big\|\sum_{n=0}^N a_n\Big\|_{Z_p}=\Big(\sum_n \|a_n\|_p^p\Big)^{1/p}.$$
We may naturally define the map
$$\sum_{n=0}^N a_n \in \Sigma_1 \mapsto \sum_{n=0}^N e_{n,n} \ten a_n \in B(\ell_2^{N+1}) \overline{\ten} \N.$$
This map extends to an isometry from $Z_p$ to $L_p(B(\ell_2^{N+1}) \overline{\ten} \N)$, 
and allows to consider $Z_p$ as a subspace of $L_p(B(\ell_2^{N+1}) \overline{\ten} \N)$. 
Moreover, this inclusion is complemented. 

\begin{lemma}\label{Zpcompl}
Let $1\leq p \leq \infty$. 
Then $Z_p$ is $2$-complemented into $L_p(B(\ell_2^{N+1}) \overline{\ten} \N)$.
\end{lemma}

\begin{proof}
We consider the projection $\Q:L_p(B(\ell_2^{N+1}) \overline{\ten} \N)\to Z_p$ defined by
$$\Q\Big(\sum_{n,k=0}^N e_{n,k}\ten x_{n,k}\Big)= \sum_{n=0}^N \E_{\AA_n}(x_{n,n})-\E_\M(\E_{\AA_n}(x_{n,n})).$$
The contractivity of the conditional expectations in $L_p$ yields
\begin{align*}
\Big\|\Q\Big(\sum_{n,k=0}^N e_{n,k}\ten x_{n,k}\Big)\Big\|_{Z_p}
&=\Big(\sum_{n=0}^N \|\E_{\AA_n}(x_{n,n})-\E_\M(\E_{\AA_n}(x_{n,n}))\|_p^p\Big)^{1/p}\\
&\leq 2 \Big(\sum_{n=0}^N \|x_{n,n}\|_p^p\Big)^{1/p}
=2 \Big\|\sum_{n=0}^N e_{n,n} \ten x_{n,n}\Big\|_{L_p(B(\ell_2^{N+1}) \overline{\ten} \N)}\\
&\leq 2\|\sum_{n,k=0}^N e_{n,k}\ten x_{n,k}\|_{L_p(B(\ell_2^{N+1}) \overline{\ten} \N)}.
\end{align*}
The last inequality comes from the boundedness of the diagonal projection in $L_p(B(\ell_2^{N+1}) \overline{\ten} \N)$.
\qd
 
We now recall the Rosenthal/Voiculescu type inequality in the amalgamated free product case proved in \cite{JPX}. 
We present these inequalities as they are stated in \cite{JPa} for $d=1$ and $2\leq p \leq \infty$, and extend them by duality to the case $1\leq p<2$. 

\begin{theorem}\label{freeRos} 
Let $a_0, a_1 \cdots, a_N \in L_p(\ast_\M \AA_n)$. 
Then the following equivalence of norms holds with relevant constants independent of $p$ or $N$.
\begin{enumerate}
\item[i)] For $2\leq p\leq \infty$, if $a_n \in L_p(\Acirc_n)$ for $0\leq n\leq N$ then
$$\Big\|\sum_{n=0}^N a_n\Big\|_p \simeq 
\Big(\sum_{n=0}^N \|a_n\|_p^p\Big)^{1/p}
+ \Big\|\Big(\sum_{n=0}^N \E_\M(a_n^*a_n)\Big)^{1/2}\Big\|_p +
\Big\|\Big(\sum_{n=0}^N \E_\M(a_na_n^*)\Big)^{1/2}\Big\|_p .$$
\item[ii)] For $1\leq p \leq 2$, if $a_n \in \Acirc_n$ for $0\leq n\leq N$ then
$$\Big\|\sum_{n=0}^N a_n\Big\|_p \simeq 
\inf
\Big(\sum_{n=0}^N\|d_n\|_p^p\Big)^{1/p} +
\Big\|\Big(\sum_{n=0}^N\E_\M(c_n^*c_n)\Big)^{1/2}\Big\|_p +
\Big\|\Big(\sum_{n=0}^N\E_\M(r_nr_n^*)\Big)^{1/2}\Big\|_p ,$$
where the infimum is taken over all the decompositions $a_n=d_n+c_n+r_n$ with $d_n, c_n, r_n \in \Acirc_n$. 
\end{enumerate}
\end{theorem}

Throughout all this paper, we consider a finite von Neumann algebra $\M$ equipped with a normal faithful normalized trace $\tau$  
and we restrict ourselves to finite martingales on the interval $[0,1]$.

\section{The $\H_p^c$-spaces}\label{sectHp}

In this section we study the column Hardy space $\H_p^c$ associated to the continuous filtration $(\M_t)_{0\leq t\leq 1}$.  
We start by defining the two candidates $\hH_p^c$ and $\H_p^c$. 
The crucial monotonicity property will imply that these two candidates for the Hardy space in the continuous setting are in fact equivalent.  
In the sequel we will focus on $\H_p^c$, and embed this space into a regularized version of an ultraproduct space, called $K_p^c(\U)$. 
This larger space satisfies a $p$-equiintegrability property which gives it a structure of $L_p$-module over a finite von Neumann algebra. 
We then check that $\H_p^c$ is an intermediate space between $L_2(\M)$ and $L_p(\M)$, 
to ensure that we are well dealing with operators. 
By complementing the continuous Hardy space $\H_p^c$ in $K_p^c(\U)$, we deduce the expected duality and interpolation results for $1<p<\infty$. 
We will then describe the associated BMO spaces, and establish the analogue of the Fefferman-Stein duality in this setting. 
The end of this section is devoted to the expected interpolation result involving the column spaces $\H_1^c$ and $\BMO^c$. 

\subsection{The discrete case}

Let us first recall the definitions of the Hardy spaces of noncommutative martingales in the discrete case and some well-known results. 
Let $(\M_n)_{n\geq 0}$ be a discrete filtration of $\M$. 
Following \cite{PX}, we introduce the column and row versions of square functions 
relative to a (finite) martingale $x = (x_n)_{n\geq 0}$:
$$
S_c (x) = \Big ( \sum^{\infty}_{n = 0} |d_n(x) |^2 \Big )^{1/2}\quad \mbox{and} \quad 
S_r (x) = \Big ( \sum^{\infty}_{n = 0} | d_n(x)^* |^2 \Big)^{1/2},$$
where 
$$d_n(x)=
\left\{\begin{array}{ll}
x_n-x_{n-1} & \quad \mbox{for } n\geq 1\\
x_0& \quad \mbox{for } n = 0
\end{array}\right.$$ 
denotes the martingale difference sequence. 
For $1 \leq p < \infty$ we define $H_p^c $ (resp. $H_p^r $) as the completion of all
finite $L_p$-martingales under the norm $\| x \|_{H_p^c}=\| S_c (x) \|_p$
(resp. $\| x \|_{H_p^r}=\| S_r (x) \|_p $). 
The Hardy space of noncommutative martingales is defined by 
$$H_p=\left\{\begin{array}{cl}
H_p^c+H_p^r& \quad \mbox{for} \quad 1\leq p< 2 \\
H_p^c\cap H_p^r& \quad \mbox{for}\quad 2\leq p<\infty
\end{array}\right..$$

We now recall some known facts on the column Hardy spaces. 
For $1\leq p <\infty$, $H_p^c$ embeds isometrically into $L_p(\M;\ell_2^c)$ 
and the noncommutative Stein inequality (see \cite{PX}) implies the following complementation result. 

\begin{prop}\label{compldiscr}
Let $1<p<\infty$. Then the discrete space $H_p^c$ is $\gamma_p$-complemented in $L_p(\M;\ell_2^c)$.
\end{prop}

\re
Recall that 
$$\gamma_p\approx \max(p,p') \mbox{ as } p\to 1 \mbox{ or } p\to \infty,$$ 
where $p'$ denotes the conjugate index of $p$. 
\mar

Since $(L_p(\M;\ell_2^c))^*=L_{p'}(\M;\ell_2^c)$ isometrically for $\frac{1}{p}+\frac{1}{p'}=1$  
and the family of column $L_p$-spaces forms an interpolation scale, 
we deduce the similar duality and interpolation results for $H_p^c$. 

\begin{cor}\label{discrdual}
Let $1<p<\infty$. Then the discrete spaces satisfy
\begin{enumerate}
\item[(i)] Let $\frac{1}{p}+\frac{1}{p'}=1$. Then $$(H_p^c)^*=H_{p'}^c \quad \mbox{with equivalent norms.}$$
\item[(ii)]  Let $1<p_1,p_2<\infty$ and $0<\theta<1$ be such that
$\frac{1}{p}=\frac{1-\theta}{p_1}+\frac{\theta}{p_2}$. Then
 $$H_p^c  = [H_{p_1}^c,H_{p_2}^c]_{\theta}  \quad \quad \mbox{with equivalent norms}  .$$
\end{enumerate}
\end{cor}

In the sequel, we will always denote the conjugate of $p$ by $p'$.

For the case $p=1$, in \cite{PX} Pisier and Xu described the dual space of $H_1^c$ as a $BMO^c$-space. 
This noncommutative analogue of the Fefferman-Stein duality has been extended by the first author and Xu in \cite{JX} to the case $1< p <2$ as follows. 
Recall that for $1\leq p\leq\infty$, we say that a sequence $(x_n)_{n\geq 0}$ in $L_p(\M)$ belongs to $L_p(\M;\ell_\infty)$ 
if  $(x_n)_{n\geq 0}$ admits a factorization $x_n=ay_nb$ with $a,b \in L_{2p}(\M)$ and $(y_n)_{n\geq 0}\in \ell_\infty(L_\infty(\M))$.  
The norm of $(x_n)_{n\geq 0}$ is then defined as
$$\|(x_n)_{n\geq 0}\|_{L_p(\M;\ell_\infty)}=\inf_{x_n=ay_nb}\|a\|_{2p}\sup_{n\geq 0}\|y_n\|_\infty\|b\|_{2p}.$$
It was proved in \cite{JD,jx-erg} that if $(x_n)_{n\geq 0}$ is a positive sequence in $L_p(\M;\ell_\infty)$, then
$$\|(x_n)_{n\geq 0}\|_{L_p(\M;\ell_\infty)}
=\sup\Big\{\sum_{n\geq 0}\tau(x_ny_n) : y_n\in L_{p'}^+(\M) ,\Big\|\sum_{n\geq 0}y_n\Big\|_{p'}\leq 1\Big\}.$$
The norm of $L_p(\M;\ell_\infty)$ will be denoted by $\|\sup^+_n x_n\|_p$. 
We should warn the reader that $\|\sup^+_n x_n\|_p$ is just a notation since $\sup_n x_n$ does not make any sense in the noncommutative setting. 
For $2<p\leq \infty$ we define
$$L_p^cMO=\{x\in L_2(\M): \|x\|_{L_p^cMO} <\infty\},$$
where 
$$\|x\|_{L_p^cMO}=\|{\sup_{n\geq 0}}^+ \E_n|x-x_{n-1}|^2\|_{p/2}^{1/2}.$$
Here we use the convention $x_{-1}=0$. 
For $p=\infty$ we denote this space by $BMO^c$. 

\begin{theorem}[\cite{PX,JX}]\label{FSdiscr}
Let $1\leq p <2$. Then the discrete spaces satisfy
$$(H_p^c)^*=L_{p'}^cMO \quad \mbox{with equivalent norms. }$$
Moreover, 
$$\lambda_p^{-1}\|x\|_{L_{p'}^cMO }\leq \|x\|_{(H_p^c)^*}\leq \sqrt{2} \|x\|_{L_{p'}^cMO },$$
where $\lambda_p$ remains bounded as $p\to 1$.
\end{theorem}

Combining Corollary \ref{discrdual} (i) with Theorem \ref{FSdiscr} we obtain

\begin{prop}
Let $2<p<\infty$. Then the discrete spaces satisfy
$$H_p^c=L_{p}^cMO \quad \mbox{with equivalent norms. }$$
\end{prop}

The Burkholder-Gundy inequalities have been extended to the noncommutative setting by Pisier and Xu in \cite{PX}.

\begin{theorem}\label{BGdiscr}
Let $1<p<\infty$. Then the discrete spaces satisfy
$$L_p(\M)=H_p \quad \mbox{with equivalent norms. }$$
Moreover, 
$$\alpha_p^{-1}\|x\|_{H_p}\leq \|x\|_p\leq \beta_p \|x\|_{H_p}.$$
\end{theorem}

\re
According to \cite{Pis2}, \cite{jx-const} and \cite{ran-weak} we know that 
$$\begin{array}{lcc}
\alpha_p\approx(p-1)^{-1} & \mbox{ as } p\to 1\quad , &\quad \alpha_p \approx p \mbox{ as } p\to \infty \\  
\beta_p\approx 1 &\mbox{ as } p\to 1 \quad , &\quad \beta_p \approx p \mbox{ as } p\to \infty.  
\end{array}$$
In particular, for $p=1$ we have a bounded inclusion $H_1\subset L_1(\M)$. 
Throughout this paper we will always denote by $\gamma_p, \lambda_p, \alpha_p$ and $\beta_p$ the constants introduced previously. 
We will also frequently use the noncommutative Doob inequality 
$$\|{\sup_n}^+ \E_n(a)\|_p\leq \delta_p\|a\|_p \quad \mbox{for } 1<p\leq \infty, a\in L_p(\M), a\geq 0,$$
and its dual form 
$$\Big\|\sum_n\E_n(a_n)\Big\|_p\leq \delta'_p \Big\|\sum_n a_n\Big\|_p \mbox{for } 1\leq p< \infty,$$
for any finite sequence $(a_n)_n$ of positive elements in $L_p(\M)$. 
These inequalities were proved in \cite{JD}, and 
we will always denote by $\delta_p$ and $\delta'_p$ respectively the constants involved there. 
Recall that $\delta'_p=\delta_{p'}$ for $1\leq p<\infty$. Moreover, we have  
$$\delta_p \approx (p-1)^{-2} \quad \mbox{ as }  p\to 1 \quad  \mbox{and} \quad  
\delta_p \approx 1 \quad \mbox{ as }  p\to \infty .$$
\mar

We end this collection of results with the interpolation theorem due to Musat in \cite{M} 
(see also \cite{JM} for a different proof with better constants). 

\begin{theorem}\label{intH1cBMOc-discr}
Let $1<p<\infty$. Then the discrete spaces satisfy
\begin{enumerate}
\item [(i)] $H_p^c=[BMO^c,H^c_1]_{\frac{1}{p}}\quad \mbox{with equivalent norms}.$
\item [(ii)] $L_p(\M)=[BMO,H_1]_{\frac{1}{p}}\quad \mbox{with equivalent norms}.$
\end{enumerate}
\end{theorem}

\re
Observe that if we consider a finite filtration $(\M_n)_{n=0}^N$, then the $H_p^c$-norm is equivalent to the $L_p$-norm for $1\leq p <\infty$. 
This comes directly from the triangle inequality in $L_p(\M)$ for $2\leq p <\infty$, and from the fact that $\|\cdot\|_{p/2}$ is a $p/2$-norm for $1\leq p <2$.  
\mar

\subsection{Definitions of $\hH_p^c$ and $\H_p^c$}\label{subsectdefHpc}

We fix an ultrafilter $\U$ over the set of all finite partitions of the interval $[0,1]$, denoted by $\PP_{\mathrm{fin}}([0,1])$, 
such that for each finite partition $\s$ of $[0,1]$ 
the set 
$$ U_{\si}= \{\si' \in \PP_{\mathrm{fin}}([0,1]) : \si\subset \si'\} \in \U .$$
Let us point out that in what follows, all considered partitions will be finite. 
We start by introducing a candidate for the bracket $[\cdot,\cdot]$ in the noncommutative setting. 
For $\si \in \PP_{\mathrm{fin}}([0,1])$ fixed and $x\in \M$, we define the finite bracket
$$[x,x]_\si=\sum_{t\in \si} |d_t^{\si}(x)|^2.$$
Observe that $ \|[x,x]_\si\|_{p/2}^{1/2} = \|x\|_{H_p^c(\si)} $, 
where $H_p^c(\si)$ denotes the noncommutative Hardy space with respect to the discrete filtration $(\M_t)_{t\in \s}$. 
Hence the noncommutative Burkholder-Gundy inequalities recalled in Theorem \ref{BGdiscr} and the H\"{o}lder inequality imply for each finite partition $\s$ and $x\in \M$
\begin{equation}\label{estimateH_pcdiscr} 
\begin{array}{cccccl}
\beta_p^{-1}\|x\|_p &\leq & \|[x,x]_\si\|_{p/2}^{1/2}&\leq &\|x\|_2 & \quad \mbox{ for } 1\leq p < 2 \\
\|x\|_2&\leq& \|[x,x]_\si\|_{p/2}^{1/2} &\leq& \alpha_p\|x\|_p & \quad  \mbox{ for } 2\leq p <\infty 
\end{array}.
\end{equation}
We deduce that for $1\leq p <\infty$, $([x,x]_\s)^\bullet \in L_{p/2}(\M_\U)$. 
Indeed, we see that the family $([x,x]_\s)_\s$ is uniformly bounded in $L_{p/2}(\M)$ and in $L_{\tilde{p}/2}(\M)$ for any $\tp >\max(p,2)$ 
(by $\alpha_{\tp}\|x\|_{\tp} \leq \alpha_{\tp} \|x\|_\infty$). 
Hence by Lemma \ref{L_p(NU)} this means that the associated element in the ultraproduct is in the regularized part. 
In particular for $x\in \M$ and $1\leq p <\infty$, we have $([x,x]_\s)^\bullet \in L_{\tp/2}(\M_\U)$ for any $\tp >\max(p,2)$. 
Thus we can apply the conditional expectation $\E_\U$ to this element and set 
$$[x,x]_\U = \E_\U( ([x,x]_\s)^\bullet).$$
Since this bracket is in $L_{\tp/2}(\M)$ for any $\tp >\max(p,2)$, it is also in $L_{p/2}(\M)$ and we may define 
$$ \|x\|_{\hat{\H}_p^{c}} = \|[x,x]_\U\|_{p/2}^{1/2}.$$
Note that for any $\tp> \max(p,2)$, this coincides with the weak-limit in $L_{\tp/2}(\M)$,  
and we can write 
$$\|x\|_{\hat{\H}_p^{c}} = \|\w L_{\tp/2} \mbox{-}\lim_{\s,\U} [x,x]_\s \|_{p/2}^{1/2}.$$
In particular, for $2<p<\infty$ we simply have
$$[x,x]_\U=\w L_{p/2} \mbox{-}\lim_{\s,\U} [x,x]_\s \quad \mbox{and} \quad 
\|x\|_{\hat{\H}_p^{c}} = \|\w L_{p/2} \mbox{-}\lim_{\s,\U} [x,x]_\s \|_{p/2}^{1/2}.$$
This definition depends a priori on the choice of the ultrafilter $\U$, and we should write $\|\cdot\|_{\hat{\H}_p^{c,\U}}$. 
However, we will show in the sequel that in fact this quantity does not depend on $\U$ up to equivalent norm. 
Hence for the sake of simplicity we will omit the power $\U$ and simply denote $\|\cdot\|_{\hat{\H}_p^c}$.  

We also introduce the following natural candidate for the norm of the Hardy space in the continuous setting. 
For $x\in \M$ and $1\leq p <\infty$ we define 
$$\|x\|_{\H_p^{c}}= \lim_{\si,\U} \|[x,x]_\si\|_{p/2}^{1/2}=\lim_{\si,\U} \|x\|_{H_p^c(\si)}.$$
The family $(\|[x,x]_\si\|_{p/2}^{1/2})_\s$ is uniformly bounded by \eqref{estimateH_pcdiscr}, hence the limit with respect to the ultrafilter $\U$ exists. 
Taking the limit in \eqref{estimateH_pcdiscr} we get for $x\in \M$
\begin{equation}\label{estimateH_pc} 
\begin{array}{cccccl}
\beta_p^{-1}\|x\|_p &\leq& \|x\|_{\H_p^{c}} &\leq &\|x\|_2 & \quad  \mbox{ for } 1\leq p < 2 \\
\|x\|_2 &\leq &\|x\|_{\H_p^{c}} &\leq &\alpha_p\|x\|_p & \quad  \mbox{ for } 2\leq p <\infty 
\end{array}.
\end{equation}
This shows that $\|\cdot\|_{\H_p^{c}}$ defines a norm on $\M$. 
As for $\|\cdot\|_{\hat{\H}_p^c}$, the norm $\|\cdot\|_{\H_p^{c}}$ depends a priori on the choice of the ultrafilter $\U$, 
but we will show that it does not (up to a constant) and hence simply denote $\|\cdot\|_{\H_p^c}$. 
Moreover, the properties of the conditional expectation $\E_\U$ imply the following estimates for $x\in \M$
\begin{equation}\label{estimatehatH_pc} 
\begin{array}{cccccccl}
\beta_p^{-1}\|x\|_p &\leq & \|x\|_{\H_p^{c}} &\leq &\|x\|_{\hat{\H}_p^{c}} &\leq &\|x\|_2 & \quad  \mbox{ for } 1\leq p < 2 \\
\|x\|_2 &\leq &\|x\|_{{\hH}_p^{c}} &\leq & \|x\|_{\H_p^{c}}  &\leq &\alpha_p\|x\|_p & \quad  \mbox{ for } 2\leq p <\infty 
\end{array}.
\end{equation}

Here for $2\leq p <\infty$ we used the contractivity of $\E_\U$ for the $L_{p/2}$-norm, 
and for $1\leq p< 2$ we need the following well-known result due to Hansen.

\begin{lemma}\label{le:contr}
Let $\A$ be a semifinite von Neumann algebra and 
$T:\A\rightarrow \A$ be a trace preserving, completely positive linear contraction.  
Let $0< p\leq 1$. Then
$$T(x^p)\leq (T(x))^p\quad \mbox{and} \quad \|x\|_p\leq \|T(x)\|_p$$
for each positive element $x\in \A$.
\end{lemma}

Then \eqref{estimatehatH_pc} shows that $\|\cdot\|_{\hat{\H}_p^{c}}$ defines a quasinorm on $\M$. 

\begin{defi}
 Let $1\leq p <\infty$. 
We define the spaces $\hH_p^{c}$ and $\H_p^{c}$ as the completion of $\M$ with respect to the (quasi)norm 
$\|\cdot\|_{\hat{\H}_p^{c}}$ and $\|\cdot\|_{\H_p^{c}}$ respectively. 
\end{defi}

We may check that for $x\in \M$ and $1\leq p <\infty$, 
$\langle x,x \rangle_{\hat{\H}_p^{c}}=[x,x]_\U$ extends to an $L_{p/2}(\M)$-valued inner product on $\hat{\H}_p^{c}$, 
which endows $\hat{\H}_p^{c}$ with an $L_p$ $\M$-module structure. 
Hence Theorem \ref{pricipalLpmodule} implies that $\|\cdot\|_{\hat{\H}_p^{c}}$ is a norm for $1\leq p <\infty$. 

\re
Note that thanks to \eqref{estimatehatH_pc}, 
$L_{\max(p,2)}(\M)$ is dense in $\H_p^c$ and $\hH_p^c$ for $1\leq p <\infty$. 
\mar

By definition, we deduce from the discrete case the following

\begin{lemma}\label{reflexivityHpc}
Let $1<p<\infty$. Then $\H_p^c$ is reflexive.
\end{lemma} 

\begin{proof}
It suffices to observe that the $\H_p^c$-norm satisfies the Clarkson inequalities. 
Then we will deduce that $\H_p^c$ is uniformly convex, so reflexive. 
Note that for each $\si$, the $H_p^c(\si)$-norm satisfies the Clarkson inequalities with relevant constants depending only on $p$. 
This comes from the fact that the noncommutative $L_p$-spaces do (see \cite{px-survey}), and recall that for $x\in \M$ we have 
$$\|x\|_{H_p^c(\si)}=\Big\|\sum_{t\in \s} e_{t,0} \ten d_t^\si (x)\Big\|_{L_p(B(\ell_2(\si))\oten \M)}.$$
Taking the limit  over $\s$ yields the desired Clarkson inequalities for the $\H_p^c$-norm.
\qd

\subsection{Monotonicity and convexity properties}

The crucial observation for the study of the spaces $\hH_p^{c}$ and $\H_p^{c}$ is that 
the $H_p^c(\si)$-norms verify some monotonicity properties. 

\begin{lemma}\label{convexityHpc}
Let $1\leq p <\infty$ and $\s \in \PP_{\mathrm{fin}}([0,1])$. 
 \begin{enumerate}
  \item[(i)] Let $1\leq p \leq 2$, $x\in L_2(\M)$ and $\s' \supset \s$. Then
$$\|x\|_{H_p^c(\s)}\leq \beta_p \|x\|_{H_p^c(\s')}.$$
Hence 
$$\|x\|_{\H_p^{c}} \leq  \sup_\s \|x\|_{H_p^c(\s)} \leq \beta_p\|x\|_{\H_p^{c}} .$$
 \item[(ii)] Let $2\leq p<\infty$. 
Let $\s^1, \cdots, \s^M$ be partitions contained in $\s$, 
let $(a_m)_{1\leq m\leq M}$ be a sequence of positive numbers such that $\sum_m a_m=1$, and let $x^1, \cdots, x^M \in L_p(\M)$. 
Then for $x =\sum_m a_m x^m$ we have 
$$\|x\|_{H_p^c(\s)}\leq \alpha_p \Big\|\sum_{m=1}^M a_m [ x^m,x^m]_{\s^m}\Big\|_{p/2}^{1/2}.$$
In particular for $x\in L_p(\M)$ and $\s \subset \s'$ we have 
$$\|x\|_{H_p^c(\s')}\leq \alpha_p \|x\|_{H_p^c(\s)}.$$
Hence 
$$ \alpha_p^{-1} \|x\|_{\H_p^{c}}\leq \inf_\s \|x\|_{H_p^c(\s)} \leq \|x\|_{\H_p^{c}}.$$
 \end{enumerate}
\end{lemma}

\begin{proof}
Let $1\leq p \leq 2$, $x\in L_2(\M)$ and $\s \subset \s'$.  
Applying the noncommutative Burkholder-Gundy inequalities to 
$$y=\sum_{t\in \s} e_{t,0} \ten d_t^\s(x)$$
in $L_p(B(\ell_2(\s))\oten \M)$ for the finite partition $\s'$, we get
$$\|y\|_{L_p(B(\ell_2(\s))\oten \M)}\leq \beta_p \|y\|_{H_p^c(\s')(B(\ell_2(\s))\oten \M)}.$$
Here we consider the discrete filtration of $B(\ell_2(\s))\oten \M$ given by $(B(\ell_2(\s))\oten \M_t)_{t\in \s'}$. 
Note that 
$$\|y\|_{H_p^c(\s')(B(\ell_2(\s))\oten \M)}
=\Big\|\sum_{s\in \s'} \sum_{t\in \s} e_{s,0} \ten e_{t,0} \ten  d_s^{\s'}(d_t^\s (x)) \Big\|_{L_p(B(\ell_2(\s'))\oten B(\ell_2(\s))\oten \M)}.
$$
An easy computation gives that for $s\in \s'$, $t\in \s$ 
$$d_{s}^{\s'}(d_{t}^\s (x))=
\left\{\begin{array}{cc}
d_{s}^{\s'} (x)&\quad \mbox{if} \quad t^-(\s)\leq s^-(\s')<s\leq  t\\
0&  \quad \mbox{otherwise}
\end{array}\right..$$
Hence for $s\in \s'$ fixed, only one term does not vanish in the sum over $t\in \s$ and we get 
$$\|y\|_{H_p^c(\s')(B(\ell_2(\s))\oten \M)}=
\Big\|\sum_{s\in \s'} e_{s,0} \ten  d_s^{\s'} (x) \Big\|_{L_p( B(\ell_2(\s'))\oten \M)}
=\|x\|_{H_p^c(\s')}.$$
The result follows from the fact that $\|y\|_{L_p(B(\ell_2(\s))\oten \M)}=\|x\|_{H_p^c(\s)}$.

We now consider $2\leq p <\infty$. Let us first assume that the partitions $\s^m$ are disjoint.  
Denote $\s'$ the union of $\s^1, \cdots, \s^M$. 
As above, we apply the noncommutative Burkholder-Gundy inequalities to 
$$y=\sum_{m=1}^M \sum_{t\in \s^m}  e_{t,0} \ten \sqrt{a_m}  d_t^{\s^m}(x^m)$$
in $L_p(B(\ell_2(\s'))\oten \M)$ for the finite partition $\s$. 
We get 
$$ \|y\|_{H_p^c(\s)(B(\ell_2(\s'))\oten \M)} \leq \alpha_p \|y\|_{L_p(B(\ell_2(\s'))\oten \M)}.$$
On the one hand, since the partitions $\s^m$ are disjoint we have 
$$\|y\|_{L_p(B(\ell_2(\s'))\oten \M)}=\Big\|\sum_{m=1}^M \sum_{t\in \s^m} a_m |d_t^{\s^m}(x^m)|^2\Big\|_{p/2}^{1/2}
=\Big\|\sum_{m=1}^M a_m [ x^m,x^m]_{\s^m}\Big\|_{p/2}^{1/2}.$$
On the other hand, 
$$\|y\|_{H_p^c(\s)(B(\ell_2(\s'))\oten \M)}
=\Big\|\sum_{s\in \s} \sum_{m=1}^M \sum_{t\in \s^m}  e_{s,0} \ten e_{t,0} \ten \sqrt{a_m} d_s^{\s}(d_t^{\s^m} (x^m)) \Big\|_{L_p(B(\ell_2(\s))\oten B(\ell_2(\s'))\oten \M)}.
$$
Again, for $s\in \s$ and $m\in \{1, \cdots, M\}$ fixed, since $\s^m \subset \s$, 
only one term does not vanish in the sum over $t\in \s^m$, and it is equal to $d_s^{\s}(x^m)$. 
Hence 
$$\|y\|_{H_p^c(\s)(B(\ell_2(\s'))\oten \M)} = \Big\|\sum_{s\in \s} \sum_{m=1}^M a_m |d_s^{\s}(x^m)|^2\Big\|_{p/2}^{1/2}.$$
By the operator convexity of $|\cdot|^2$ we obtain 
$$\|x\|_{H_p^c(\s)} = \Big\|\sum_{s\in \s} \Big| \sum_{m=1}^M a_m d_s^{\s}(x^m)\Big|^2\Big\|_{p/2}^{1/2}
\leq  \alpha_p \|y\|_{H_p^c(\s)(B(\ell_2(\s'))\oten \M)},$$
which yields the required inequality. 
In the general case, when the partitions are not disjoint, the result still holds by approximation, thanks to the fact that the filtration is right continuous. 
Indeed, if there exists a common point $t$ which is both in $\s^m$ and $\s^n$ (for $n\neq m$), 
then we can replace $t$ by $t+\varepsilon$ in $\s^m$ (for $\varepsilon$ small enough), 
which does not change the considered norms when passing to the limit as $\varepsilon \to 0$.
\qd

This monotonicity property immediately implies the following crucial result, mentioned previously.

\begin{theorem}\label{indpdtU}
For $1\leq p <\infty$ the space $\H_p^{c}$ is independent of the choice of the ultrafilter $\U$, 
up to equivalent norm. 
\end{theorem}

\subsection{$\hH_p^c=\H_p^c$}\label{secthat}

In this subsection we show that the two candidates $\hH_p^{c}$ and $\H_p^c$ introduced previously for the Hardy space of noncommutative martingales 
with respect to the continuous filtration $(\M_t)_{0\leq t \leq 1}$ actually coincide. 
In particular we will deduce that, up to an equivalent constant, these spaces do not depend on the choice of the ultrafilter $\U$. 

\begin{theorem}\label{hat}
Let $1 \leq p <\infty$. Then 
$$\H_p^c=\hH_p^{c} \quad \mbox{with equivalent norms.}$$
\end{theorem}

Theorem \ref{indpdtU} yields immediately 

\begin{cor}\label{hatindpdtU}
For $1\leq p <\infty$ the space $\hH_p^{c}$ is independent of the choice of the ultrafilter $\U$, 
up to equivalent norm. 
\end{cor}

The case $2\leq p<\infty$ is an easy consequence of the convexity property proved in Lemma \ref{convexityHpc}, as detailed below. 

\begin{proof}[Proof of Theorem \ref{hat} for $2\leq p<\infty$]
It suffices to show that the $\H_p^c$-norm and the $\hH_p^c$-norm are equivalent on $\M$. 
Let $x\in \M$, by \eqref{estimatehatH_pc}  we have $\|x\|_{\hH_p^c}\leq \|x\|_{\H_p^c}$. 
Now assume that $\|x\|_{\hH_p^c}=\|[x,x]_\U\|_{p/2}^{1/2}<1$. 
Since the two spaces coincide with $L_2(\M)$ for $p=2$, we consider $2<p<\infty$. 
In that case we have $[x,x]_\U=\w L_{p/2} \mbox{-} \lim_{\s,\U}[x,x]_\s$. 
We can find a sequence of positive numbers $(\alpha_m)_{m=1}^{M}$ such that $\sum_m \alpha_m=1$ and 
finite partitions $\s^1, \cdots, \s^{M}$ satisfying 
$$\Big\|\sum_{m=1}^{M} \alpha_m  [ x,x]_{\s^m} \Big\|_{p/2}<1.$$
Applying Lemma \ref{convexityHpc} (ii) to $\s=\cup_m \s^m$ we get
$$\|x\|_{H_p^c(\s)}\leq \alpha_p.$$
Then 
$$\|x\|_{\H_p^c}\leq \alpha_p\|x\|_{H_p^c(\s)}\leq \alpha_p^2.$$
\qd

For $1\leq p<2$, it is more complicated to explicit the bracket $[x,x]_\U$. 
This is why we will use a dual approach. 
The trick is to embed $\hat{\H}_p^{c}$ into a larger ultraproduct space defined as follows. 
Let us fix $q>2$. 
We define the set 
$$\I=\PP_{\fin}(\M) \times \PP_{\fin}([0,1]) \times \rz_+^*,$$
where $\PP_{\fin}(\M)$ denotes the set of all finite families in $\M$. 
Then $\I$ is a partially ordered set by the natural order.  
We define an ultrafilter $\V$ on $\I$ as follows. 
For $G \in \PP_{\fin}(\M)$ we define 
$$S_G=\{ F \in \PP_{\fin}(\M)  : G \subseteq F\}$$
and consider the filter base on $\PP_{\fin}(\M)$
$$\T=\{ S_G : G \in \PP_{\fin}(\M)\}.$$
On $\rz_+^*$ we consider the filter base given by
$$\W=\{ ]0,\delta] : \delta >0\}.$$
Then the product $\V'=\T\times \U \times \W$ is a filter base on $\I$, and we consider $\V$ an ultrafilter on $\I$ refining $\V'$. 
Let us now fix an element $i=(F, \s_i,\varepsilon) \in \I$. 
For each $x\in F$, the Burkholder-Gundy inequalities applied to each $\s$ for $q>2$ yields 
that the family $([ x,x]_\s)_\s$ is uniformly bounded in $L_{q/2}(\M)$. Since $L_{q/2}(\M)$ is reflexive, the weak-limit exists and 
$$[x,x]_\U=\w L_{q/2}\mbox{-}\lim_{\s,\U} [ x,x]_\s.$$ 
The same holds for the finite family $F$, i.e., 
the family $([ x,x]_\s)_{x\in F}$ is uniformly bounded in $L_{q/2}(\M)\oplus \cdots  \oplus L_{q/2}(\M)$. 
By reflexivity, the weak-limit exists and can be approximated by convex combinations in $L_{q/2}$-norm. 
Hence we can find a sequence of positive numbers $(\alpha_m(i))_{m=1}^{M(i)}$ such that $\sum_m \alpha_m(i)=1$ and 
finite partitions $\s_i^1, \cdots, \s_i^{M(i)}$ satisfying for all $x\in F$
\begin{equation}\label{coeff}
\Big\|[x,x]_\U- \sum_{m=1}^{M(i)} \alpha_m(i)  [ x,x]_{\s_i^m} \Big\|_{q/2}<\varepsilon.
\end{equation}
We may assume in addition that $\s_i$ is contained in $\s_i^m$ for all $m$. 
We consider the Hilbert space $\H_i=\ell_2\Big(\bigcup_{m,t\in \s_i^m}\{t\}\Big)$
equipped with the norm
$$\|(\xi_{m,t})_{1\leq m \leq M(i), t \in \s_i^m}\|_{\H_i}=\Big(\sum_{m=1}^{M(i)}\alpha_m(i)\sum_{t\in \s_i^m}|\xi_{m,t}|^2\Big)^{1/2}.$$
For $1\leq p \leq \infty$ and $i\in \I$ we consider the column space $L_p(\M;\H_i^c)$.
Recall that for any sequence $(\xi_{m,t})_{1\leq m \leq M(i),t\in \s_i^m}$ in $L_p(\M)$ we have
$$\Big\|\sum_{m=1}^{M(i)} \sum_{t\in \s_i^m}  e_{m,0}\ten e_{t,0}  \ten \xi_{m,t}\Big\|_{L_p(\M;\H_i^c)}
=\Big\|\Big(\sum_{m=1}^{M(i)}\alpha_m(i)\sum_{t\in \s_i^m}|\xi_{m,t}|^2\Big)^{1/2}\Big\|_p.$$
Then for $1\leq p <\infty$ we have 
$$(L_p(\M;\H_i^c))^*=L_{p'}(\M;\H_i^c) \quad \mbox{isometrically},$$
via the duality bracket
$$(\xi|\eta)_{L_p(\M;\H_i^c),L_{p'}(\M;\H_i^c)}
=\sum_{m=1}^{M(i)}\sum_{t\in \s_i^m}\alpha_m(i)\tau(\xi_{m,t}^*\eta_{m,t}).$$

\begin{lemma}\label{embedhatHpc}
Let $1\leq p <2$. Then $\hH_p^{c}$ embeds isometrically into $\prodd_\V  L_p(\M;\H_i^c)$.
\end{lemma}

\begin{proof}
By density it suffices to consider an element $x\in \M$. 
We associate $x$ with $\x=(\x(i))^\bullet \in \prodd_\V  L_p(\M;\H_i^c)$ defined as follows. 
For each index $i=(F,\s_i,\varepsilon)\in \I$ such that $x\in F$ we set 
$$\x(i)=\sum_{m=1}^{M(i)} \sum_{t\in \s_i^m} e_{m,0}\ten e_{t,0} \ten d_t^{\s_i^m}(x),$$
and $\x(i)=0$ otherwise. 
Then we claim that
\begin{equation}\label{isometry}
\|\x\|_{\prodd_\V  L_p(\M;\H_i^c)} = \lim_{i,\V} \|\x(i)\|_{ L_p(\M;\H_i^c)} = \|[x,x]_\U\|_{p/2}^{1/2}= \|x\|_{\hH_p^c}.
\end{equation}
Indeed, 
for $\delta >0$, we observe that for $i=(F,\s_i,\varepsilon)$ such that $x\in F$ and $\varepsilon^{p/2} \leq \delta$ 
we have by the triangle inequality applied to the norm $\|\cdot\|_{p/2}^{p/2}$ and \eqref{coeff}
\begin{align*}
\Big|\|[x,x]_\U\|_{p/2}^{p/2}- \|\x(i)\|_{ L_p(\M;\H_i^c)}^p\Big| 
&= \Big|\|[x,x]_\U\|_{p/2}^{p/2} -
\Big\|\sum_{m=1}^{M(i)}\alpha_m(i)\sum_{t\in \s_i^m}|d_t^{\s_i^m}(x)|^2\Big\|_{p/2}^{p/2}\Big|\\
&=\Big|\|[x,x]_\U\|_{p/2}^{p/2}
-\Big\|\sum_{m=1}^{M(i)}\alpha_m(i)[ x,x]_{\s_i^m}\Big\|_{p/2}^{p/2}\Big| \\
&\leq \Big\|[x,x]_\U- \sum_{m=1}^{M(i)} \alpha_m(i)  [ x,x]_{\s_i^m} \Big\|_{p/2}^{p/2}\\
&\leq \Big\|[x,x]_\U- \sum_{m=1}^{M(i)} \alpha_m(i)  [ x,x]_{\s_i^m} \Big\|_{q/2}^{p/2}\\
&< \varepsilon^{p/2} \leq \delta.
\end{align*}
This means that 
$$S_{\{x\}}\times  \PP_{\fin}([0,1]) \times ]0,\delta^{2/p}] 
 \subset 
\{i \in \I  :  
 \Big|\|[x,x]_\U\|_{p/2}^{p/2}- \|\x(i)\|_{ L_p(\M;\H_i^c)}^p\Big| <\delta\}.$$
Since by construction, the set $S_{\{x\}}\times  \PP_{\fin}([0,1]) \times ]0,\delta^{2/p}] \in \T\times \U \times \W$ is in the ultrafilter $\V$, 
we deduce that the set in the right hand side is also in $\V$ for all $\delta>0$. 
Thus by the definition of the limit with respect to an ultrafilter we get
$$\lim_{i,\V} \|\x(i)\|^p_{ L_p(\M;\H_i^c)} = \|[x,x]_\U\|_{p/2}^{p/2}.$$
This concludes the proof of \eqref{isometry} and 
shows that the map $x\in \M\mapsto \x$ extends to an isometric embedding of $\hH_p^{c}$ into $\prod_\V L_p(\M;\H_i^c)$. 
\qd

This embedding will be useful to describe the dual space of $\hH_p^{c}$. 

\begin{lemma}\label{dualhatHpc}
Let $1\leq p <2$.
Then
$$(\hH_p^{c})^*\subset (\H_p^c)^*.$$
\end{lemma}

\begin{proof}
Let $\varphi \in (\hH_p^{c})^*$ be a functional of norm less than one.
By Lemma \ref{embedhatHpc} and the Hahn-Banach Theorem we can extend $\varphi$ to a linear functional on $\prod_\V L_p(\M;\H_i^c)$ of norm less than one, 
also denoted by $\varphi$. 
Lemma \ref{le:ultrapdct} implies that $\varphi$ is the weak$^*$-limit of elements $\xi_\lambda$
in the unit ball of $\prodd_\V (L_p(\M;\H_i^c))^*=\prodd_\V L_{p'}(\M;\H_i^c)$. 
For each $\lambda$, we will prove that there exists $z_\lambda\in L_2(\M)$ such that
$$(\xi_\lambda|\x)=\tau(z_\lambda^*x), \forall  x \in \M
\quad \mbox{and} \quad \|z_\lambda\|_{(\H_p^c)^*}\leq k_{p},$$
where $\x$ denotes the element in $\prodd_\V L_p(\M;\H_i^c)$ corresponding to $x$ via the embedding given by Lemma \ref{embedhatHpc}.
Then we will set $z=\w L_2\mbox{-} \lim_\lambda z_\lambda$ and get an element $z\in L_2(\M)$ such that
$$\varphi(x)=\lim_\lambda (\xi_\lambda|\x)
=\lim_\lambda\tau(z_\lambda^*x)
=\tau(z^*x),  \forall x\in \M 
\quad \mbox{and} \quad \|z\|_{(\H_p^c)^*}\leq k_{p}.$$
Finally we will conclude the proof using the density of $\M$ in $\hH_p^{c}$.

We now consider an element $\xi=(\xi(i))^\bullet\in \prod_\V L_{p'}(\M;\H_i^c)$ of norm less than one, with
$$\xi(i)=\sum_{m=1}^{M(i)} \sum_{t\in \s_i^m} e_{m,0}\otimes e_{t,0} \otimes \xi_{m,t}(i).$$
Fix $i=(F,\s_i, \varepsilon)\in \I$ and $1\leq m\leq M(i)$. 
Then $ \xi_m(i):=\sum_{t\in \s_i^m} e_{m,0}\otimes e_{t,0} \otimes \xi_{m,t}(i) \in L_{p'}(\M; \ell_2^c(\s_i^m))$. 
We set 
$$z_m(i)=\sum_{t\in \s_i^m} d_t^{\s_i^m}(\xi_{m,t}(i)),$$
where $d_t^{\s_i^m}(\xi_{m,t}(i))=\E_t(\xi_{m,t}(i))-\E_{t^-(\s_i^m)}(\xi_{m,t}(i))$ for $t>0$ and $d_0^{\s_i^m}(\xi_{m,0}(i))=\E_0(\xi_{m,0}(i))$.  
Note that since the partition $\s_i^m$ is finite, we have $z_m(i)\in L_{p'}(\M)$. 
Then we consider
$$z(i)=\sum_m \alpha_m(i) z_m(i) \in L_{p'}(\M).$$ 
We first show that $\|z(i)\|_{L_{p'}^cMO(\s'_i)}\leq k_{p}$ for $\s'_i=\s_i^1\cup \cdots \cup \s_i^{M(i)}$. 
Let $s\in \s'_i$. Then for $m$ fixed, we denote by $t_m(s)$ the unique element in $\s_i^m$ satisfying $t_m(s)^-(\s_i^m)\leq s^-(\s'_i)<s\leq t_m(s)$.  
The operator convexity of the square function $|\cdot|^2$ yields 
\begin{equation}\label{Heq1}
 \begin{array}{ccl}
\E_s|z(i)-\E_{s^-(\s'_i)}(z(i))|^2
&=\E_s\Big|\displaystyle\sum_m\alpha_m(i)(z_m(i)-\E_{s^-(\s'_i)}(z_{m}(i)))\Big|^2\\
&  \leq\displaystyle\sum_m\alpha_m(i)\E_s|z_m(i)-\E_{s^-(\s'_i)}(z_{m}(i))|^2. 
  \end{array}
\end{equation}
On the other hand we can write
\begin{align*}
&\E_s|z_m(i)-\E_{s^-(\s'_i)}(z_{m}(i))|^2 \\
&=\E_s\Big(\sum_{t>t_m(s), t\in \s_i^m} |d_t^{\s_i^m}(z_m(i))|^2 + |\E_{t_m(s)}(z_m(i))-\E_{s^-(\s'_i)}(z_m(i))|^2\Big)\\
&=\E_s\Big(\sum_{t>t_m(s), t\in \s_i^m} |d_t^{\s_i^m}(\xi_{m,t}(i))|^2 + |\E_{t_m(s)}(\xi_{m,t_m(s)}(i))-\E_{s^-(\s'_i)}(\xi_{m,t_m(s)}(i))|^2\Big)\\
&\leq  4 \E_s\Big(\sum_{t>t_m(s), t\in \s_i^m} |\xi_{m,t}(i)|^2\Big) 
+2\E_s|\xi_{m,t_m(s)}(i)|^2 +2\E_{s^-(\s'_i)}|\xi_{m,t_m(s)}(i)|^2 \\
&\leq  4 \E_s\Big(\sum_{t\in \s_i^m} |\xi_{m,t}(i)|^2\Big)+ 2\E_{s^-(\s'_i)}\Big(\sum_{t\in \s_i^m} |\xi_{m,t}(i)|^2\Big).
\end{align*}
Here the second identity comes from the fact that for $t \in \s_i^m$
$$
\E_{t_m(s)}\big(d_t^{\s_i^m}(\xi_{m,t}(i))\big)=
\left\{\begin{array}{ll}
d_t^{\s_i^m}(\xi_{m,t}(i)) & \mbox{ if } t\leq t_m(s) \\
0& \mbox{ if } t > t_m(s) 
\end{array}\right.
$$
and
$$
\E_{s^-(\s'_i)}\big(d_t^{\s_i^m}\big(\xi_{m,t}(i))\big)=
\left\{\begin{array}{ll}
d_t^{\s_i^m}(\xi_{m,t}(i)) & \mbox{ if } t < t_m(s) \\
\E_{s^-(\s'_i)}(\xi_{m,t}(i))-\E_{t_m(s)^-(\s_i^m)}(\xi_{m,t}(i))& \mbox{ if } t = t_m(s) \\
0& \mbox{ if } t > t_m(s) 
\end{array}\right..
$$
Then \eqref{Heq1} gives
$$\E_s|z(i)-\E_{s^-(\s'_i)}(z(i))|^2 \leq  
4 \E_s\Big(\sum_{m,t \in \s_i^m}\alpha_m(i)|\xi_{m,t}(i)|^2\Big)+ 2\E_{s^-}\Big(\sum_{m,t \in \s_i^m}\alpha_m(i)|\xi_{m,t}(i)|^2\Big).$$
By the noncommutative Doob inequality we obtain
\begin{align*}
&\|z(i)\|_{L_{p'}^cMO(\s'_i)}^2=\|{\sup_{s\in \s'_i}}^+\E_s|z(i)-\E_{s^-(\s'_i)}(z(i))|^2\|_{p'/2}\\
&\leq 4\Big\|{\sup_{s\in \s'_i}}^+ \E_s\Big(\sum_{m,t\in \s^m_i}\alpha_m(i)|\xi_{m,t}(i)|^2\Big)\Big\|_{p'/2}
+2 \Big\|{\sup_{s\in \s'_i}}^+ \E_{s^-(\s'_i)}\Big(\sum_{m,t\in \s^m_i}\alpha_m(i)|\xi_{m,t}(i)|^2\Big)\Big\|_{p'/2}\\
&\leq 6\delta_{p'/2}\Big\|\displaystyle\sum_{m,t\in \s_i^m}\alpha_m(i)|\xi_{m,t}(i)|^2\Big\|_{p'/2}
=6\delta_{p'/2}\|\xi(i)\|_{L_{p'}(\M;\H_i^c)}^2.
  \end{align*}
Hence 
\begin{equation}\label{z(i)}
\|z(i)\|_{L_{p'}^cMO(\s'_i)}\leq 3\delta_{p'/2}^{1/2}\|\xi(i)\|_{L_{p'}(\M;\H_i^c)}\leq 3\delta_{p'/2}^{1/2}.
\end{equation}
In particular, we see that the family $(z(i))_i$ is uniformly bounded in $L_2(\M)$. 
We set \\
$z=\w L_2\mbox{-}\lim_{i,\V} z(i)$. 
By the density of $L_2(\M)$ in $\H_p^c$ we have 
$$\|z\|_{(\H_p^c)^*}
=\sup_{x\in L_2(\M), \|x\|_{\H_p^c}\leq 1}|\tau(z^*y)|.$$
Then for $x\in L_2(\M), \|x\|_{\H_p^c}\leq 1$, Lemma \ref{convexityHpc} and \eqref{z(i)} imply
\begin{align*}
|\tau(z^*x)|&\leq \lim_{i,\V} |\tau(z(i)^*x)|
\leq  \sqrt{2}\lim_{i,\V}\|z(i)\|_{L_{p'}^cMO(\s'_i)}\|x\|_{H_p^c(\s'_i)} \\
&\leq 3\sqrt{2}\delta_{p'/2}^{1/2}\beta_p\|x\|_{\H_p^c}\leq 3\sqrt{2}\delta_{p'/2}^{1/2}\beta_p.
\end{align*}
Hence we get $\|z\|_{(\H_p^c)^*}\leq k_p$ with $k_p=3\sqrt{2}\delta_{p'/2}^{1/2}\beta_p$. 
Finally, it remains to check that for all $x\in \M$, $z$ satisfies
\begin{equation}\label{eqduality}
(\xi|\tilde{x})_{\prodd_\V L_{p'}(\M;\H_i^c),\prodd_\V L_{p}(\M;\H_i^c)}=\tau(z^*x).
\end{equation}
We first verify that for each $i=(F,\s_i,\varepsilon)\in \I$ such that $x\in F$ we have 
$$(\xi(i)|\tilde{x}(i))_{L_{p'}(\M;\H_i^c),L_{p}(\M;\H_i^c)}=\tau(z(i)^*x).$$
For each $m$ we have
$$\tau(z_m(i)^*x) = \sum_{t\in \s_i^m} \tau\big(d_t^{\s_i^m}(\xi_{m,t}(i))^*x\big)
=\sum_{t\in \s_i^m}\tau\big(\xi_{m,t}(i)^*d_{t}^{\s_i^m}(x)\big).$$
Then 
$$\tau(z(i)^*x)=\sum_{m=1}^{M(i)} \alpha_m(i) \tau(z_m(i)^*x)
=\sum_{m=1}^{M(i)}\sum_{t\in \s_i^m}\alpha_m(i)\tau\big(\xi_{m,t}(i)^*d_{t}^{\s_i^m}(x)\big)
=(\xi(i)|\tilde{x}(i)).$$
By the construction of the ultrafilter $\V$ this is sufficient to show that the limits along $\V$ coincide, 
and \eqref{eqduality} follows. 
This concludes the proof of the Lemma.
\qd

\begin{proof}[Proof of Theorem \ref{hat} for $1\leq p <2$]
By density, it suffices to prove the equivalence of the norms on $\M$. 
This follows from \eqref{estimatehatH_pc}, and we prove the reverse inequality by duality by using Lemma \ref{dualhatHpc}. 
\qd 

In the sequel, we will use the definition of $\H_p^c$ to transfer the results from the discrete case to the continuous setting. 
Indeed, this construction seems more natural for taking the limit in the classical results. 

\subsection{Ultraproduct spaces and $L_p$-modules}\label{Kp}

In this subsection we introduce the ultraproduct of the column $L_p$-spaces and its regularized version, 
into which we will isometrically embed the Hardy space $\H_p^c$. 
We will equip these ultraproduct spaces with some $L_p$-module structure. 


\begin{defi}
Let $1\leq p < \infty$. We define  
$$\tilde{K}_p^c(\U)=\prodd_\U L_p(\M;\ell_2^c(\s)) \quad \mbox{and} \quad  K_p^c(\U)=\tilde{K}_p^c(\U)\cdot e_\U,$$
where $\cdot$ denotes the right modular action of $\tilde{\M}_\U$ on $\tilde{K}_p^c(\U)$. \\
For $p=\infty$ we set
$$\tilde{K}_\infty^c(\U)=\overline{\prodd_\U L_\infty(\M;\ell_2^c(\s))}^{ so} \quad \mbox{and} \quad  K_\infty^c(\U)=\tilde{K}_\infty^c(\U)\cdot e_\U,$$
where the strong operator topology is taken in the von Neumann algebra generated by $\prodd_\U B(\ell_2(\s))\oten \M$, 
and coincides with the topology arising from the seminorms
$$\|\xi\|_\eta=\lim_{\s,\U} \tau\Big(\eta_\s\sum_{t\in\s}|\xi_\s(t)|^2\Big)^{1/2},\quad \mbox{for }
\eta=(\eta_\s)^\bullet \in (\tilde{\M}_\U)^+_*=\Big(\prodd_\U L_1(\M)\Big)^+.$$
\end{defi}

The right $\tilde{\M}_\U$-module structure of $\tilde{K}_p^c(\U)$ is given 
for $x=(x_\s)^\bullet\in \prodd_\U \M$ and $\xi=(\xi_\s)^\bullet \in \tilde{K}_p^c(\U)$ by 
$$\xi \cdot x=(\xi_\s\cdot x_\s)^\bullet.$$
It is easy to see that this does not depend on the chosen representing families. 
Moreover, by Proposition $5.2$ of \cite{JS}, this module action extends naturally from $\prodd_\U \M$ to $\tilde{\M}_\U$. 
Similarly, for $\xi=(\xi_\s)^\bullet, \eta=(\eta_\s)^\bullet \in \tilde{K}_p^c(\U)$ we consider the componentwise bracket
$$ \langle \xi,\eta\rangle_{\tilde{K}_p^c(\U)} 
= (\langle  \xi_\s,\eta_\s\rangle_{L_p(\M;\ell_2^c(\s))})^{\bullet}
= \Big(\sum_{t\in \s} \xi_\s(t)^*\eta_\s(t)\Big)^\bullet \in \prod_{\U} L_{p/2}(\M)\cong L_{p/2}(\tilde{\M}_\U),$$
where $\xi_\s=\sum_{t\in \s} e_{t,0}\ten \xi_\s(t), \eta_\s=\sum_{t\in \s} e_{t,0}\ten \eta_\s(t) \in L_p(\M;\ell_2^c(\s))$. 
This defines an $L_{p/2}(\tilde{\M}_\U)$-valued inner product which generates the norm of $\tilde{K}_p^c(\U)$ and is compatible with the module action. 
Hence $\tilde{K}_p^c(\U)$ is a right $L_p$ $\tilde{\M}_\U$-module for $1\leq p \leq \infty$. 
In the sequel, the regularized spaces will be crucial tools to study $\H_p^c$. 
We may equip $K_p^c(\U)$ with an $L_p$-module structure over the finite von Neumann $\M_\U$ thanks to the following observation. 

\begin{lemma}\label{Kpc}
Let $1\leq p \leq \infty$. Let $\xi \in \tilde{K}_p^c(\U)$. 
Then the following assertions are equivalent. 
\begin{enumerate}
\item[(i)] $\xi \in K_p^c(\U)$;
\item[(ii)] $\langle \xi,\xi \rangle_{\tilde{K}_p^c(\U)}  \in L_{p/2}(\M_\U)$;
\end{enumerate}
\end{lemma}

\begin{proof}
By \eqref{LpMU}, it suffices to show that for $\xi \in \tilde{K}_p^c(\U)$ we have 
$$\xi=\xi\cdot e_\U\Leftrightarrow \langle \xi,\xi \rangle_{\tilde{K}_p^c(\U)}=\langle \xi,\xi \rangle_{\tilde{K}_p^c(\U)}e_\U.$$
This comes from Definition \ref{defLpmodule} and the fact that $e_\U$ is a central projection. 
Indeed, we can write 
\begin{align*}
\xi=\xi\cdot e_\U
&\Leftrightarrow \xi\cdot (1-e_\U)=0 \\
&\Leftrightarrow \langle \xi\cdot (1-e_\U),\xi\cdot (1-e_\U)\rangle_{\tilde{K}_p^c(\U)} =0 \\
&\Leftrightarrow  (1-e_\U)^*\langle \xi,\xi\rangle_{\tilde{K}_p^c(\U)}(1-e_\U)=0 \\
&\Leftrightarrow \langle \xi,\xi\rangle_{\tilde{K}_p^c(\U)}(1-e_\U)=0\\
& \Leftrightarrow  \langle \xi,\xi \rangle_{\tilde{K}_p^c(\U)}=\langle \xi,\xi \rangle_{\tilde{K}_p^c(\U)}e_\U.
\end{align*}
\qd

Lemma \ref{Kpc} implies that $K_p^c(\U)$ is an $L_p$ $\M_\U$-module. 
Moreover, the family \\
$(K_{p}^c(\U))_{1\leq p\leq \infty}$ forms a projective system of $L_p$ $\M_\U$-modules. 
Indeed, for $1\leq p \leq q \leq \infty$ we may consider the contractive ultraproduct of the componentwise inclusion maps 
$I_{q,p}: \tilde{K}_{q}^c(\U) \to \tilde{K}_{p}^c(\U)$. 
By modularity, this map preserves the regularized spaces, i.e., $I_{q,p}:K_{q}^c(\U) \to K_{p}^c(\U)$. 
Then we observe that the assumptions (i)-(iii) of Corollary \ref{familyLpmodule} are satisfied. 
In particular, we deduce that the map $I_{q,p}$ is injective on $K_q^c(\U)$. 
Hence for $1\leq p \leq q \leq \infty$ we may identify $K_{q}^c(\U)$ with a subspace of $K_{p}^c(\U)$. 
We can prove the density assumption (iv) of Corollary \ref{familyLpmodule} by using the $p$-equiintegrability as follows. 

\begin{lemma}\label{densityKpc}
 Let $1\leq p <\infty$. Then $K_\infty^c(\U)$ is dense in $K_p^c(\U)$. 
\end{lemma}

\begin{proof}
Let $\xi \in K_p^c(\U)$, then Lemma \ref{Kpc} yields that $\langle \xi,\xi \rangle_{K_p^c(\U)}  \in L_{p/2}(\M_\U)$. 
Combining Theorem \ref{LpMUpequi} with Lemma \ref{pequiintegr} we deduce that 
\begin{equation}\label{pequibraket}
\lim_{T\to \infty}\|\langle \xi,\xi \rangle_{K_p^c(\U)}\1(\langle \xi,\xi \rangle_{K_p^c(\U)}>T)\|_{L_{p/2}(\M_\U)}=0.
\end{equation}
We set $\eta_T=\xi \cdot \1(\langle \xi,\xi \rangle_{K_p^c(\U)}\leq T)$. 
Then
$$\langle \eta_T,\eta_T\rangle_{K_p^c(\U)} = 
\langle \xi,\xi \rangle_{K_p^c(\U)}\1(\langle \xi,\xi \rangle_{K_p^c(\U)}\leq T) \in \M_\U,$$
and $\eta_T \in K_\infty^c(\U)$. 
Moreover, by \eqref{pequibraket} we have
\begin{align*}
 \|\xi-\eta_T\|_{K_p^c(\U)} &
= \|\xi \cdot \1(\langle \xi,\xi \rangle_{K_p^c(\U)}> T)\|_{K_p^c(\U)} \\
&=\|\langle\xi \cdot \1(\langle \xi,\xi \rangle_{K_p^c(\U)}> T),
\xi \cdot \1(\langle \xi,\xi \rangle_{K_p^c(\U)}> T)\rangle_{K_p^c(\U)}\|_{p/2}^{1/2}\\
&=\|\langle \xi,\xi \rangle_{K_p^c(\U)} \1(\langle \xi,\xi \rangle_{K_p^c(\U)}> T)\|_{p/2}^{1/2} 
\stackrel{T\to \infty}{\longrightarrow} 0.
\end{align*}
This ends the proof of the Lemma. 
\qd

Since $\M_\U$ is finite, we deduce duality and interpolation results from Corollary \ref{Lpmodule}.

\begin{cor}\label{dualintpolKpc}
Let $1\leq p <\infty$. Then 
\begin{enumerate}
\item[(i)] $(K_p^c(\U))^*=K_{p'}^c(\U)$ isometrically.
\item[(ii)]Let $1\le p_1 <p< p_2 \leq \infty$ and $0<\theta<1$ be such that
$\frac{1}{p}=\frac{1-\theta}{p_1}+\frac{\theta}{p_2}$.  Then
$$K_p^{c}(\U)=[K_{p_1}^{c}(\U),K_{p_2}^{c}(\U)]_{\theta}\quad \mbox{isometrically}.$$
\item[(iii)] $K_p^c(\U)=\overline{\bigcup_{\tilde{p}>p}I_{\tilde{p},p}(\tilde{K}_{\tilde{p}}^c(\U))}^{\|\cdot\|_{\tilde{K}_p^c(\U)}}.$
\end{enumerate}
\end{cor}

\begin{proof}
The assertions (i) and (ii) follow directly from Corollary \ref{Lpmodule}. 
For (iii), let $\tp>p$ and $\xi \in I_{\tp,p}(\tilde{K}_{\tilde{p}}^c(\U))$. 
There exists $\eta \in \tilde{K}_{\tilde{p}}^c(\U)$ such that $\xi=I_{\tp,p}(\eta)$. 
Then by Lemma \ref{L_p(NU)} we have 
$$\langle \xi,\xi\rangle_{\tilde{K}_{p}^c(\U)}=I_{\tp,p}(\langle \eta,\eta\rangle_{\tilde{K}_{\tp}^c(\U)})
\in  I_{\tp,p}(L_{\tp/2}(\tilde{\M}_\U)) \subset L_{p/2}(\M_\U),$$
and Lemma \ref{Kpc} yields $\xi \in K_p^c(\U)$. 
Conversely, let $\xi \in K_p^c(\U)$. 
Then by Lemma \ref{densityKpc} we can approximate $\xi$ in $K_{p}^c(\U)$-norm by an element $\eta \in K_\infty^c(\U)$, 
which is in $I_{\tp,p}(\tilde{K}_{\tilde{p}}^c(\U))$ for all $\tilde{p}>p$. This concludes the proof of the Corollary.
\qd

The finiteness of $\M_\U$ also implies the following useful result. 

\begin{lemma}\label{normKpc}
Let $1\leq p <\infty$ and $\xi \in K_p^c(\U)$. 
Then 
$$\|\xi\|_{K_p^c(\U)}=\lim_{q\to p } \|\xi\|_{K_q^c(\U)}.$$
\end{lemma}

\begin{proof}
For $\xi \in K_p^c(\U)$, we have $\langle \xi, \xi \rangle_{K_p^c(\U)} \in L_{p/2}(\M_\U)$ by Lemma \ref{Kpc}. 
Since $\M_\U$ is finite we may write
$$\|\xi\|_{K_p^c(\U)}=\|\langle \xi, \xi \rangle_{K_p^c(\U)} \|_{L_{p/2}(\M_\U)}^{1/2}
=\lim_{q\to p} \|\langle \xi, \xi \rangle_{K_p^c(\U)} \|_{L_{q/2}(\M_\U)}^{1/2}
= \lim_{q\to p} \|\xi\|_{K_q^c(\U)}.$$
\qd

We may isometrically embed $\H_p^c$ into $K_p^c(\U)$ for every $1\leq p <\infty$ via the map
$$\iota : 
\H_p^c \to K_p^c(\U)$$
defined for $x\in \M$ by
\begin{equation}\label{iota}
\iota(x)=\Big(\sum_{t\in\s}e_{t,0}\ten d_t^\s(x)\Big)^\bullet.
\end{equation}
Indeed, for $x\in \M$ there exists $\tp>p$ such that 
$$\langle \iota(x),\iota(x)\rangle_{\tilde{K}_p^c(\U)}=\Big(\sum_{t\in\s}| d_t^\s(x)|^2\Big)^\bullet \in I_{\tp/2,p/2}(L_{\tp/2}(\tilde{\M}_\U)).$$
Then $\langle \iota(x),\iota(x)\rangle_{\tilde{K}_p^c(\U)} \in L_{p/2}(\M_\U)$ by Lemma \ref{L_p(NU)}, 
which ensures that $\iota(x)\in K_p^c(\U)$ by Lemma \ref{Kpc}. 
Observe that Lemma \ref{normKpc} still holds true for the $\H_p^c$-norm. 

\begin{lemma}\label{normHpc}
Let $1\leq p <\infty$  and $x\in \M$. 
Then 
$$\|x\|_{\H_p^c}=\lim_{q\to p} \|x\|_{\H_q^c}.$$
\end{lemma}

\begin{proof}
For $x\in \M$, $\iota (x) \in K_p^c(\U)$ and by Lemma \ref{normKpc} we can write
$$\|x\|_{\H_p^c}=\|\iota (x)\|_{\K_p^c(\U)}
=\lim_{q\to p} \|\iota (x)\|_{\K_q^c(\U)}
=\lim_{q\to p}  \|x\|_{\H_q^c}.$$
\qd

\subsection{Injectivity results}

In this subsection we check that the Hardy spaces defined above are well intermediate spaces between $L_2(\M)$ and $L_p(\M)$ as expected. 
The inequalities \eqref{estimateH_pc} allow to define by density natural bounded maps from $\H_p^c$ to $L_{\min(p,2)}(\M)$ for $1\leq p <\infty$. 
Since it is not clear a priori, we need to prove that these maps are injective. 

\begin{prop}\label{injHpc2}
Let $1\leq p <\infty$. Then
$$L_{\max(p,2)}(\M)  \subset \H_p^c \subset L_{\min(p,2)}(\M),$$
i.e., $\H_p^c$ embeds into $L_{\min(p,2)}(\M)$. 
\end{prop}

We first prove the following direct consequence of the monotonicity property. 

\begin{lemma}\label{injHpc}
Let $1\leq p <2$. 
Then the space $\{x\in L_p(\M) : \|x\|_{\H_p^c}<\infty \}$ is complete with respect to the norm $\|\cdot\|_{\H_p^c}$. 
\end{lemma} 

\begin{proof}
The argument we will use to prove the completeness of the space $\{x\in L_p(\M) : \|x\|_{\H_p^c}<\infty \}$ 
relies on the fact that the discrete $H_p^c(\s)$-norms are increasing in $\s$ (up to a constant) 
for $1\leq p <2$, and on the completeness of the discrete spaces $H_p^c(\s)$.  
Let $(x_n)_{n\geq 1}\subset \{x\in L_p(\M)  :  \|x\|_{\H_p^c}<\infty \}$ be a Cauchy sequence with respect to $\|\cdot\|_{\H_p^c}$. 
Recall that for $x\in \{x\in L_p(\M)  :  \|x\|_{\H_p^c}<\infty \}$ we have $\|x\|_p\leq \beta_p \|x\|_{\H_p^c}$. 
Then we deduce that $(x_n)_{n\geq 1}$ is also a Cauchy sequence in $L_p(\M)$. 
Hence $(x_n)_{n\geq 1}$ converges in $L_p(\M)$ to an element $x\in L_p(\M)$. 
Since for a finite partition $\s$, the norms $\|\cdot\|_p$ and $\|\cdot\|_{H_p^c(\s)}$ are equivalent, 
the convergence is in $H_p^c(\s)$ for each $\s$. 
It remains to prove that the convergence is also with respect to the $\H_p^c$-norm, 
and then we will conclude that $x\in \{x\in L_p(\M) :  \|x\|_{\H_p^c}<\infty \}$.
Fix $\varepsilon >0$. By the Cauchy property with respect to the $\H_p^c$-norm, there exists $n_0 \in \nz$ such that for all $n\geq n_0$,
$$\lim_{m \to \infty} \|x_m-x_n\|_{\H_p^c} < \varepsilon.$$ 
For a fixed partition $\s$, since $x_n \to x$ in $H_p^c(\s)$ we have
$$\|x-x_n\|_{H_p^c(\s)}= \lim_{m \to \infty}\|x_m-x_n\|_{H_p^c(\s)} \leq \beta_p  \lim_{m \to \infty}\|x_m-x_n\|_{\H_p^c} < \varepsilon.$$ 
Note that here $n_0$ does not depend on the partition $\s$, hence taking the limit in $\s$ we obtain the required convergence in $\H_p^c$-norm.
\qd

\begin{proof}[Proof of Proposition \ref{injHpc2}]
For $1\leq p <2$, by Lemma \ref{injHpc} and density we can isometrically embed $\H_p^c$ into $\{x\in L_p(\M) : \|x\|_{\H_p^c}<\infty \}$, 
which is clearly a subspace of $L_p(\M)$. Hence the natural map which sends $\H_p^c$ to $L_p(\M)$ is injective. 
For $2\leq p <\infty$, 
the injectivity of $\H_p^c$ into $L_2(\M)$ directly comes from the $L_p$-module structure of the spaces $K_p^c(\U)$ introduced in subsection \ref{Kp}. 
Indeed, if $2\leq p <\infty$, this structure implies the injectivity of the map $I_{p,2}:K_p^c(\U) \to K_2^c(\U)$. 
Hence the following commuting diagram yields the required injectivity result:
$$\xymatrix{
    \H_p^c \ar@{^{(}->}[d]^{\iota} \ar@{->}[r]     & \H_2^c=L_2(\M) \ar@{^{(}->}^{\iota}[d]  \\
    K_p^c(\U) \ar@{^{(}->}[r]^{I_{p,2}} & K_2^c(\U)
  }$$
 
\qd

\subsection{Complementation results}

The aim of this subsection is to obtain the analoguous results of Proposition \ref{compldiscr} and Corollary \ref{discrdual} in the continuous setting. 
Here the space $K_p^c(\U)$ will play the role of the space $L_p(\M;\ell_2^c)$ in the discrete case. 

\begin{prop}\label{HpccomplKpc}
Let $1<p<\infty$. Then $\H_p^c$ is complemented in $K_p^c(\U)$.  
\end{prop}

\begin{proof}
We first describe the projection $\D:K_p^c(\U) \to \H_p^c$ for an element 
$$\xi=(\xi_\s)^\bullet=\Big(\sum_{t\in \s} e_{t,0}\ten \xi_\s(t)\Big)^\bullet \in K_\infty^c(\U).$$
For each $\s$ we set
$$x_\s=\sum_{t\in\s} d_t^\s(\xi_\s(t)),$$
where $d_t^\s(\xi_\s(t))=\E_t(\xi_\s(t))-\E_{t^-(\s)}(\xi_\s(t))$ for $0<t \in \s$ and $d_0^\s(\xi_\s(t))=\E_0(\xi_\s(t))$. 
Since $K_\infty^c(\U)\subset K_2^c(\U)$, we have $x_\s \in L_2(\M)$ and $\|x_\s\|_2=\|\xi_\s\|_{L_2(\M;\ell_2^c(\s))}$ is uniformly bounded. 
Hence we can consider the weak-limit in $L_2(\M)$ of the $x_\s$'s and we set
$$\D(\xi)=\w L_2 \mbox{-}\lim_{\s,\U}x_\s.$$
We will show that this construction defines well an element $\D(\xi) \in \H_p^c$ with $\|\D(\xi)\|_{\H_p^c}\leq k_p \|\xi\|_{K_p^c(\U)}$ for $1<p<\infty$.  
Then we will conclude by density (see Lemma \ref{densityKpc}) that $\H_p^c$ is complemented in $K_p^c(\U)$. 

Let $\xi \in K_\infty^c(\U)$ be such that 
$\|\xi\|_{K_p^c(\U)}=\lim_{\s,\U}\|\xi_\s\|_{L_p(\M;\ell_2^c(\s))}<1$. 
We may assume that $\|\xi_\s\|_{L_p(\M;\ell_2^c(\s))}<1$ for all $\s$. 
Let $1<p<2$ and fix a finite partition $\s_0$. 
In this case we clearly have $\D(\xi) \in L_2(\M) \subset \H_p^c$. 
By Lemma \ref{convexityHpc} and the noncommutative Stein inequality in the discrete case (Proposition \ref{compldiscr}), 
for $\s_0 \subset \s$ we have
$$\|x_\s\|_{H_p^c(\s_0)}
\leq \beta_p \|x_\s\|_{H_p^c(\s)}
\leq \gamma_p \beta_p \|\xi_\s\|_{L_p(\M;\ell_2^c(\s))}
\leq \gamma_p \beta_p.$$
We see that $(x_\s)_{\s\supset \s_0}$ is uniformly bounded in $H_p^c(\s_0)$, 
and we deduce that the weak-limit of the $x_\s$'s (for $\s \supset \s_0$) exists in $H_p^c(\s_0)$. 
This weak-limit coincides with the weak-limit in $L_2$, i.e., 
$$\D(\xi)=\w H_p^c(\s_0) \mbox{-}\lim_{\s\supset \s_0,\U}x_\s.$$
Then we can write
$$\|\D(\xi)\|_{H_p^c(\s_0)}\leq \lim_{\s\supset \s_0,\U}\|x_\s\|_{H_p^c(\s_0)}\leq \gamma_p \beta_p.$$
Since this holds true for every partition $\s_0$, taking the limit we obtain 
$$\|\D(\xi)\|_{\H_p^c}\leq  \gamma_p \beta_p \|\xi\|_{K_p^c(\U)}.$$
Let $2\leq p <\infty$. 
Lemma \ref{convexityHpc} and Proposition \ref{compldiscr} imply that for each $\s$, 
$$\|x_\s\|_{\H_p^c}\leq \alpha_p \|x_\s\|_{H_p^c(\s)}
\leq \gamma_p \alpha_p \|\xi_\s\|_{L_p(\M;\ell_2^c(\s))}
\leq \gamma_p \alpha_p.$$ 
Hence the family $(x_\s)_{\s}$ is uniformly bounded in $\H_p^c$. 
The reflexivity of $\H_p^c$ (Lemma \ref{reflexivityHpc}) yields that the weak-limit of the $x_\s$'s exists in $\H_p^c$. 
Since $\H_p^c$ embeds into $L_2(\M)$ by Proposition \ref{injHpc2}, we deduce that these two weak-limits coincide, i.e., 
$$ \D(\xi)=\w \H_p^c\mbox{-}\lim_{\s,\U} x_\s \in \H_p^c.$$ 
Then we can write
$$\|\D(\xi)\|_{\H_p^c}\leq \lim_{\s,\U}\|x_\s\|_{\H_p^c}\leq \gamma_p \alpha_p \|\xi\|_{K_p^c(\U)}.$$
This concludes the proof. 
\qd

This complementation result allows to transfer the duality and interpolation results proved for $K_p^c(\U)$ in Corollary \ref{dualintpolKpc} to the spaces $\H_p^c$. 

\begin{cor}\label{dual-intpolHpc}
Let $1<p<\infty$. 
\begin{enumerate}
\item[(i)] Then $$(\H_p^c)^*=\H_{p'}^c \quad \mbox{with equivalent norms.}$$
\item[(ii)]  Let $1<p_1,p_2<\infty$ and $0<\theta<1$ be such that
$\frac{1}{p}=\frac{1-\theta}{p_1}+\frac{\theta}{p_2}$. Then
 $$\H_p^c  = [\H_{p_1}^c,\H_{p_2}^c]_{\theta}  \quad \quad \mbox{with equivalent norms}  .$$
\end{enumerate}
\end{cor}

A direct approach of the duality between $\H_p^c$ and $\H_{p'}^c$, by simply using the discrete duality, 
yields instead the following duality result. 

\begin{lemma}\label{dualHpc-direct}
Let $1<p<2$. Then 
$$(\H_{p'}^c)^*=\{x\in L_p(\M) : \|x\|_{\H_p^c}<\infty \}\quad \mbox{with equivalent norms.}$$
\end{lemma}

\begin{proof}
Let $x \in \{x\in L_p(\M)  : \|x\|_{\H_p^c}<\infty \}$ and $y\in \M$. 
For a fixed partition $\s$, the H\"{o}lder inequality implies
\begin{align*}
|\tau(x^*y)|=\Big|\sum_{t\in \s} \tau(d_t^\s(x)^*d_t^\s(y))\Big| 
& \leq \Big\|\Big(\sum_{t\in \s} |d_t^\s(x)|^2\Big)^{1/2}\Big\|_p  \Big\|\Big(\sum_{t\in \s} |d_t^\s(y)|^2\Big)^{1/2}\Big\|_{p'}\\
&=\|x\|_{H_p^c(\s)}\|y\|_{H_{p'}^c(\s)}.
\end{align*}
Passing to the limit yields
$$|\tau(x^*y)|\leq \|x\|_{\H_p^c}\|y\|_{\H_{p'}^c}.$$
Since $\M$ is dense in $\H_{p'}^c$, this shows that $x\in (\H_{p'}^c)^*$ and 
$$\|x\|_{(\H_{p'}^c)^*} \leq  \|x\|_{\H_p^c}.$$
Conversely, let $\varphi \in (\H_{p'}^c)^*$ be of norm less than one. 
Since $L_{p'}(\M)$ is dense in $\H_{p'}^c$, $\varphi$ is represented by an element $x\in L_p(\M)$ such that 
$\varphi(y)=\tau(x^*y)$ for all $y\in L_{p'}(\M)$. It remains to show that $\|x\|_{\H_p^c} <\infty$. 
For a fixed partition $\s$, by Corollary \ref{discrdual} and the density of $L_{p'}(\M)$ in $H_{p'}^c(\s)$, we get
$$\|x\|_{H_p^c(\s)} \leq \sqrt{2}\gamma_p \|x\|_{(H_{p'}^c(\s))^*} 
= \sqrt{2}\gamma_p \sup_{y\in L_{p'}(\M), \|y\|_{H_{p'}^c(\s)}\leq 1} |\tau(x^*y)|.$$
For $y\in L_{p'}(\M)$ with $\|y\|_{H_{p'}^c(\s)}\leq 1$ we have 
$$|\tau(x^*y)| =|\varphi(y)| \leq \|y\|_{\H_{p'}^c} \leq \alpha_{p'} \|y\|_{H_{p'}^c(\s)}\leq  \alpha_{p'} .$$
Hence we get 
$$\|x\|_{\H_p^c}\leq  \sqrt{2}\gamma_p \alpha_{p'} \|x\|_{(\H_{p'}^c)^*},$$
and deduce that $x \in \{x\in L_p(\M) : \|x\|_{\H_p^c}<\infty \}$. 
\qd

Hence, combining Lemma \ref{dualHpc-direct} with assertion (i) of Corollary \ref{dual-intpolHpc} we obtain 

\begin{cor}\label{HpXp}
Let $1<p<2$. Then $\H_{p}^c=\{x\in L_p(\M) : \|x\|_{\H_p^c}<\infty \}$.
\end{cor}

\re\label{dualitybracket}
We have to be careful when considering the duality bracket between $\H_p^c$ and $\H_{p'}^c$. 
Indeed, we can not always write it explicitly. 
For $1<p<\infty$ and $x\in \H_p^c, y\in \H_{p'}^c$, we can write 
$$(x|y)_{\H_p^c,\H_{p'}^c}=\tau(x^*y)$$
when either $x\in L_{\max(p,2)}(\M)$ or $y\in L_{\max(p',2)}(\M)$. 
We first suppose that $x\in L_{\max(p,2)}(\M)$ and $y\in L_{\max(p',2)}(\M)$. 
Then 
\begin{align*}
(x|y)_{\H_p^c,\H_{p'}^c}
&= (\iota(x)|\iota(y))_{K_p^c(\U),K_{p'}^c(\U)}\\
&=\Big(\Big(\sum_{t\in \s}e_{t,0}\ten d_t^\s(x)\Big)^\bullet\Big|\Big(\sum_{t\in \s}e_{t,0}\ten d_t^\s(y)\Big)^\bullet\Big)_{K_p^c(\U),K_{p'}^c(\U)}\\
&=\lim_{\s,\U} \sum_{t\in \s}\tau( d_t^\s(x)^*d_t^\s(y)) =\lim_{\s,\U} \tau(x^*y)\\
&=\tau(x^*y).
\end{align*}
We now consider $x\in L_{\max(p,2)}(\M)$ and $y=\H_{p'}^c\mbox{-}\lim_n y_n$ with $y_n\in L_{\max(p',2)}(\M)$, the other case being similar. 
Then 
$$(x|y)_{\H_p^c,\H_{p'}^c}=\lim_n(x|y_n)_{\H_p^c,\H_{p'}^c}
=\lim_n\tau(x^*y_n).$$
Since $\H_{p'}^c\subset L_{\min(p',2)}(\M)$, the sequence $(y_n)_n$ also converges in $L_{\min(p',2)}$ to $y$ and we get $(x|y)_{\H_p^c,\H_{p'}^c}=\tau(x^*y)$. 
\mar

\subsection{Fefferman-Stein duality}\label{sectFSc}

In this subsection we establish the analogue of the Fefferman-Stein duality in the continuous setting. 
The difficulty here is to find the right description of the spaces $L_p^cMO$ to get the expected duality. 

\begin{defi}
\begin{enumerate}
\item[(i)] Let $2 < p <\infty$. We define $L_p^c\MO$ as the space of all elements $x\in L_2(\M)$ of the form 
$x=\w L_2\mbox{-}\lim_{\s,\U} x_\s$ with $\lim_{\s,\U} \|x_{\si}\|_{L_p^cMO(\s)}<\infty$. 
We equip $L_p^c\MO$ with the norm 
$$\|x\|_{L_p^c\MO} 
=\inf \lim_{\s,\U} \|x_{\si}\|_{L_p^cMO(\s)}, $$
where the infimum is taken over all the descriptions $x=\w L_2\mbox{-}\lim_{\s,\U} x_\s$. 
\item[(ii)] We define the space $\BMO^c$ as the space whose closed unit ball is given by the absolute convex set
$$B_{\BMO^c}=\overline{\{x=\w L_2\mbox{-}\lim_{\si,\U} x_\s  : \lim_{\s,\U} \|x_{\si}\|_{BMO^c(\s)} \leq 1\}}^{\|\cdot\|_2}.$$
We equip $\BMO^c$ with the norm  
$$\|x\|_{\BMO^c}=\inf\{ C \geq 0 :  x \in CB_{\BMO^c}\}.$$
\end{enumerate}
\end{defi}

\begin{lemma}\label{LpMObanach}
The spaces $L_p^c\MO$ for $2<p<\infty$ and $\BMO^c$ are Banach spaces.  
\end{lemma}

To prove that $\BMO^c$ is complete we will need the following general fact. 

\begin{lemma}\label{banach}
Let $X$ be a Banach space and $B$ be an absolutely convex subset of $X$ satisfying
\begin{enumerate}
\item[(i)] $B$ is continuously embedded into the unit ball of $X$, i.e., there exists $D>0$ such that $$B\subset DB_X;$$
\item[(ii)] $B$ is closed with respect to the norm $\|\cdot\|_X$. 
\end{enumerate}
Then the space $Y$ whose unit ball is $B$, equipped with the norm 
$$\|x\|_{Y}=\inf\{C \geq 0 : x \in CB\}$$
is a Banach space.
\end{lemma}

\begin{proof}
It is a well-known fact that $\|\cdot\|_Y$ defines a norm. 
Let $\sum_n x_{n\geq 1}$ be an absolutely converging series in $(Y,\|\cdot\|_Y)$. We may assume that $\|x_n\|_Y\leq \frac{1}{2^n}$ for all $n\geq 1$.  
We want to show that this series converges in $Y$. 
We first remark that the series $\sum_n x_n$ is absolutely converging, and hence converging, in $X$. 
Then there exists $x\in X$ such that $x=\sum_n x_n$, where the convergence is with respect to $\|\cdot\|_X$.  
Thus 
$$\sum_{n=1}^N x_n \stackrel{N\to \infty}{\longrightarrow } x \quad \mbox{in } X 
\quad \mbox{and} \quad \sum_{n=1}^N x_n \in B.$$
Indeed, we have
$$\Big\|\sum_{n=1}^N x_n\Big\|_Y\leq \sum_{n=1}^N \|x_n\|_Y\leq \sum_{n=1}^N \frac{1}{2^n}\leq 1.$$
Using (ii), this shows that $x\in B$. 
It remains to see that the convergence also holds for the norm $\|\cdot\|_Y$. 
Let $\eps>0$. Let $N_0$ be such that $2^{N_0}\geq \eps^{-1}$. 
We claim that for all $M > N\geq N_0$
$$y_{N,M}=\frac{1}{\eps}\Big(\sum_{n=1}^M x_n - \sum_{n=1}^N x_n\Big) \in B.$$
Indeed, we have
$$\|y_{N,M}\|_Y=\frac{1}{\eps}\Big\|\sum_{n=N+1}^M x_n\Big\|_Y
\leq \frac{1}{\eps}\sum_{n=N+1}^M \|x_n\|_Y
\leq \frac{1}{\eps}\sum_{n=N+1}^M \frac{1}{2^n}
\leq \frac{1}{\eps 2^N}
\leq \frac{1}{\eps 2^{N_0}}\leq 1.$$
Moreover, for $N\geq N_0$ fixed we have 
$$y_{N,M}\stackrel{M\to \infty}{\longrightarrow } \frac{1}{\eps}\Big(x - \sum_{n=1}^N x_n\Big) \quad \mbox{in } X .$$
Hence (ii) yields that 
$\frac{1}{\eps}\Big(x - \displaystyle\sum_{n=1}^N x_n\Big) \in B$, i.e., 
$$\Big\|x - \sum_{n=1}^N x_n\Big\|_Y\leq \eps \quad \mbox{for all  } N\geq N_0.$$
This proves that the series converges with respect to $\|\cdot\|_Y$ and ends the proof.
\qd

\begin{proof}[Proof of Lemma \ref{LpMObanach}]
We start with the case $p=\infty$. 
We apply Lemma \ref{banach} to $X=L_2(\M)$ and $B=B_{\BMO^c}$. 
Then by the definition of $B_{\BMO^c}$, it is clear that the condition (ii) of Lemma \ref{banach} is satisfied.  
Moreover, since for $x\in L_2(\M)$ and each $\s$ we have $\|x\|_2\leq \sqrt{2}\|x\|_{BMO^c(\s)}$, the condition (i) holds for $D=\sqrt{2}$. 
Hence the construction of the space $\BMO^c$ defines a Banach space. \\
For $2<p<\infty$, we observe that $L_p^c\MO$ is the range of the bounded map 
$$\phi:\left\{\begin{array}{ccc}
\prodd_\U L_p^cMO(\s) & \longrightarrow &L_2(\M) \\
(x_\s)^\bullet & \longmapsto & \w L_2\mbox{-}\lim_{\s,\U} x_\s
\end{array}\right..$$
Indeed, since for each $\s$ we have 
$\|x_\s\|_2\leq \sqrt{2}\|x_\s\|_{L_p^cMO(\s)}$, the family $(x_\s)_\s$ is uniformly bounded in $L_2$ and hence the weak-limit in $L_2$ exists. 
Moreover, it is easily checked that the map $\phi$ is well-defined, i.e., if $ (x_\s)^\bullet=(y_\s)^\bullet \in \prodd_\U L_p^cMO(\s)$, then 
$\w L_2\mbox{-}\lim_{\s,\U} x_\s=\w L_2\mbox{-}\lim_{\s,\U} y_\s$. 
Since $\prodd_\U L_p^cMO(\s)$ is a Banach space, 
the boundedness of $\phi$ implies that $L_p^c\MO=\phi(\prodd_\U L_p^cMO(\s))$ equipped with the quotient norm 
$$\|x\|_{L_p^c\MO}=\inf_{x=\phi((x_\s)^\bullet)} \|(x_\s)^\bullet\|_{\prodd_\U L_p^cMO(\s)}$$ 
is a Banach space.
\qd

We can now state the continuous analogue of the noncommuative Fefferman-Stein duality. 

\begin{theorem}\label{fsduality}
Let $1\leq p <2$. 
Then 
$$(\H_p^c)^*=L_{p'}^c\MO \quad \mbox{with equivalent norms}.$$
Moreover, 
$$\lambda_p^{-1} \|x\|_{L_{p'}^c\MO } \leq \|x\|_{(\H_p^c)^*} \leq \sqrt{2} \|x\|_{L_{p'}^c\MO }.$$
\end{theorem}

\begin{proof}
We first prove the inclusion $L_{p'}^c\MO  \subset (\H_p^c)^*$ for $1<p<2$. 
We consider $x \in L_{p'}^c\MO$ with $\|x\|_{L_{p'}^c\MO}<1$. 
Then there exists a sequence $(x_\s)_\s$ such that  
$\lim_{\s,\U} \|x_\s\|_{L_{p'}^cMO(\s)} <1$ 
and $x=\w L_2 \mbox{-}\lim_{\s,\U} x_\s$. 
Hence for $y\in \M$ we have $\tau(x^*y) = \lim_{\s,\U} \tau(x_\s^*y)$. 
Recall that the discrete Fefferman-Stein duality for a fixed partition $\s$ implies
$$|\tau(x_\s^*y)|\leq \sqrt{2}\|x_\s\|_{L_{p'}^cMO(\s)}\|y\|_{H_{p}^c(\s)}.$$
Taking the limit we get
$$|\tau(x^*y)| 
\leq \sqrt{2}\lim_{\s,\U} (\|x_\s\|_{L_{p'}^cMO(\s)} \|y\|_{H_p^c(\s)})=  \sqrt{2}(\lim_{\s,\U} \|x_\s\|_{L_{p'}^cMO(\s)} )(\lim_{\s,\U} \|y\|_{H_p^c(\s)})
\leq \sqrt{2}\|y\|_{\H_p^c}.
$$
Since $\M$ is dense in $\H_p^c$, this shows that $x\in (\H_p^c)^*$ and 
$$\|x\|_{(\H_p^c)^*} \leq \sqrt{2}\|x\|_{L_{p'}^c\MO }.$$
The proof of this inclusion in the case $p=1$ is similar. 
Indeed, let $x\in B_{\BMO^c}$. Then there exists a sequence $(x^\lambda)_\lambda$ such that $x=L_2\mbox{-}\lim_\lambda x^\lambda$ and 
$x^\lambda=\w L_2 \mbox{-}\lim_{\s,\U} x_\s^\lambda$, with $\lim_{\s,\U} \|x_\s^\lambda\|_{\BMO^c(\s)} \leq 1$ for all $\lambda$. 
For $y\in \M$ we have
$$|\tau(x^*y)|=|\lim_\lambda  \tau((x^\lambda)^*y)|
\leq \sqrt{2} \lim_\lambda (\lim_{\s,\U} \|x_\s^\lambda\|_{BMO^c(\s)} )(\lim_{\s,\U} \|y\|_{H_1^c(\s)})
\leq \sqrt{2}\|y\|_{\H_1^c}.$$
We deduce that $\BMO^c  \subset (\H_1^c)^*$ by density as above. \\
To prove the converse inclusion, we first embed $\H_p^c$ into an ultraproduct space as follows. 
Note that the map $i_\U:x \in \M \mapsto (x)^\bullet \in \prodd_\U H_p^c(\s)$ is isometric 
with respect to the norms $\|\cdot\|_{\H_p^c}$ and $\|\cdot\|_{\prodd_\U H_p^c(\s)}$. 
Hence by the density of $\M$ in $\H_p^c$ we can extend $i_\U$ to an isometric embedding of $\H_p^c$ into $\prodd_\U H_p^c(\s)$. 
Let $\varphi \in (\H_{p}^c)^*$ be a functional of norm less than one. 
By the Hahn-Banach Theorem we can extend $\varphi$ to a functional on $\prodd_\U H_p^c(\s)$ of norm less than one. 
We now need to consider the dual space of an ultraproduct. 
Recall that the situation is much easier in the reflexive case (see subsection \ref{ultra}), 
hence we start with the case $1<p<2$. 
In this situation $\prodd_\U H_p^c(\s)$ is reflexive, and Lemma \ref{dualultrapdct} gives
$$\Big(\prodd_\U H_p^c(\s)\Big)^*=\prodd_\U (H_p^c(\s))^*\cong\prodd_\U L_{p'}^cMO(\s),$$
where the constants in the equivalence of the norms come from the discrete case (see Theorem \ref{FSdiscr}). 
Then there exists $z=(z_\s)^\bullet \in \prodd_\U L_{p'}^cMO(\s)$ of norm $\leq \lambda_p$ such that 
$$\varphi(y)=(z|i_\U(y)) , \quad \forall y \in \H_p^c.$$
Applying this to $y\in \M$ we get 
$$\varphi(y)= ( z|(y)^\bullet)
=\lim_{\s,\U} \tau(z_\s^*y)
=\tau(x^*y),$$
where $x=\w L_2\mbox{-}\lim_{\s,\U}z_\s$ is in $L_{p'}^c\MO$. By the density of $\M$ in $\H_p^c$ this proves that $\varphi$ is represented by $x$ and 
$$\|x\|_{L_{p'}^c\MO} \leq \lim_{\s,\U}\|z_\s\|_{L_{p'}^cMO(\s)} \leq \lambda_p \|x\|_{(\H_{p}^c)^*}.$$
For $p=1$, Lemma \ref{le:ultrapdct} says that the unit ball of $\prodd_\U (H_1^c(\s))^* \cong \prodd_\U BMO^c(\s)$ 
is weak$^*$-dense in the unit ball of $(\prodd_\U H_1^c(\s))^*$. 
Then there exists a sequence $(z^\lambda)_\lambda$, 
where $z^\lambda=(z^\lambda_\s)^\bullet \in \prodd_\U BMO^c(\s)$ is of norm less than $\sqrt{3}$ for all $\lambda$, such that 
$$\varphi(y)=\lim_\lambda (z^\lambda|i_\U(y)), \quad \forall y \in \H_1^c.$$
Applying this to $y\in \M$ we get 
$$\varphi(y)= \lim_\lambda (z^\lambda |(y)^\bullet)
=\lim_\lambda  \lim_{\s,\U} \tau((z^\lambda_\s)^*y)
=\lim_\lambda  \tau((x^\lambda)^*y),$$
where $x^\lambda=\w L_2\mbox{-}\lim_{\s,\U} z^\lambda_\s$ is in $\BMO^c$ of norm less than $\sqrt{3}$. 
Since for all $\lambda$ we have $\|z^\lambda \|_{\prodd_\U L_2(\M)} \leq \sqrt{2} \|z^\lambda \|_{\prodd_\U BMO^c(\s)} \leq \sqrt{6}$, 
the sequence $(x^\lambda)_\lambda$ is uniformly bounded in $L_2(\M)$. 
Setting $x=\w L_2\mbox{-}\lim_\lambda x^\lambda$ we obtain $\varphi(y)= \tau(x^*y)$ for all $y\in \M$. 
We can approximate the weak-limit $x$ by convex combinations of the $x^\lambda$'s in the $L_2$-norm. 
Since $x^\lambda \in \sqrt{3}B_{\BMO^c}$ for all $\lambda$, the convexity of the unit ball of $\BMO^c$ implies that 
any convex combination $\sum_m\alpha_m x^{\lambda_m}$ is still in $\sqrt{3}B_{\BMO^c}$. 
Thus by the definition of $B_{\BMO^c}$, we obtain that $x \in \sqrt{3}B_{\BMO^c}$. 
By the density of $\M$ in $\H_1^c$ this proves that $\varphi$ is represented by $x$ and 
$$\|x\|_{\BMO^c} \leq \sqrt{3} \|x\|_{(\H_{1}^c)^*}.$$
\qd

This duality implies the following property.

\begin{cor}\label{LpcMOclosed}
Let $2<p\leq \infty$. 
Let $(x_\lambda)_\lambda$ be a sequence in $L_2(\M)$ such that $\|x_\lambda\|_{L_{p}^c\MO}\leq 1$ for all $\lambda$ and
$x=\w L_2 \mbox{-} \lim_\lambda x_\lambda$.
Then $x\in L_{p}^c\MO$ with $\|x\|_{L_{p}^c\MO} \leq \sqrt{2}\lambda_{p}$.
\end{cor} 

\begin{proof}
Using Theorem \ref{fsduality} 
and the density of $L_2(\M)$ in $\H_{p'}^c$, we can write 
$$\|x\|_{L_{p}^c\MO} \leq \lambda_{p} \sup_{y\in L_2(\M), \|y\|_{\H_{p'}^c}\leq 1} |\tau(x^*y)|.$$
Note that for all $y\in L_2(\M), \|y\|_{\H_{p'}^c}\leq 1$ we have
$$|\tau(x^*y)|\leq \lim \sup_\lambda |\tau(x_\lambda^*y)|\leq \sqrt{2}\lim \sup_\lambda \|x_\lambda\|_{L_{p}^c\MO}\|y\|_{\H_{p'}^c}\leq \sqrt{2}.$$
Thus $x\in L_{p}^c\MO$ with $\|x\|_{L_{p}^c\MO} \leq \sqrt{2}\lambda_{p}$.
\qd

Combining Theorem \ref{fsduality} with Corollary \ref{dual-intpolHpc} (i) we immediately get the

\begin{cor}\label{LpcMOHpc}
Let $2<p<\infty$. Then 
$$L_p^c\MO=\H_p^c \quad \mbox{with equivalent norms}.$$
\end{cor}

\re\label{normLpcMO}
In particular, we deduce the following properties for $L_p^c\MO$, $2<p<\infty$:
\begin{enumerate}
 \item[(i)] $L_{p}^c\MO$ is independent of the choice of the ultrafilter $\U$, up to equivalent norm.
 \item[(ii)] $L_{p}(\M)$ is norm dense in $L_{p}^c\MO$. 
 \item[(iii)] For $x\in L_{p}(\M)$,
$$ \|x\|_{L_{p}^c\MO} = \lim_{q\to p}  \|x\|_{L_{q}^c\MO} \simeq  \lim_{\si,\U} \|x\|_{L_{p}^cMO(\si)} $$
for every ultrafilter $\U$. 
In particular, up to equivalent norms, $L_{p}^c\MO$ is the completion of  $L_{p}(\M)$ with respect to the norm $\lim_{\s,\U} \|\cdot\|_{L_{p}^cMO(\s)}$.
 \item[(iv)] The $\|\cdot \|_{L_{p}^cMO(\si)}$-norm is decreasing in $\s$ (up to a constant). 
\end{enumerate}
Note that for (iii), if $x \in L_p(\M)$ the fact that $\|x\|_{L_{p}^c\MO}\simeq \|x\|_{\H_p^c}$ combined with Lemma \ref{normHpc} ensures that 
$\lim_{q\to p}  \|x\|_{L_{q}^c\MO}$ exists. Since for $2<q<p<r$ we have $\|x\|_{L_{q}^c\MO}\leq \|x\|_{L_{p}^c\MO} \leq \|x\|_{L_{r}^c\MO}$, 
sending $q$ and $r$ to $p$ we obtain that the limit is in fact equal to $\|x\|_{L_{p}^c\MO}$. 
\mar 

Concerning the case $p=\infty$, we can also deduce some nice properties for $\BMO^c$ from Theorem \ref{fsduality}. 

\begin{cor}\label{normBMOc}
\begin{enumerate}
 \item[(i)] $\BMO^c$ is independent of the choice of the ultrafilter $\U$, up to equivalent norm.
 \item[(ii)] $\M$ is weak$^*$-dense in $\BMO^c$. 
 \item[(iii)] For $x\in \M$,
$$ \|x\|_{\BMO^c} \simeq  \sup_{2<p<\infty}  \|x\|_{L_{p}^c\MO} \leq \lim_{\si,\U} \|x\|_{BMO^c(\si)} $$
for every ultrafilter $\U$. 
 \item[(iv)] The $\|\cdot \|_{BMO^c(\si)}$-norm is decreasing in $\s$ (up to a constant). \\
 More precisely, for $x\in \M$ and $\s\subset \s'$ we have  
$$\|x\|_{BMO^c(\s')}\leq 2 \|x\|_{BMO^c(\s)}.$$
\end{enumerate}
\end{cor}

\begin{proof}
Assertions (i) and (ii) follow directly from Theorem \ref{fsduality} and Theorem \ref{indpdtU}, Proposition \ref{injHpc2} respectively.
For $x\in \M$ and $2<p<\infty$, Corollary \ref{LpcMOclosed} gives $\|x\|_{L_{p}^c\MO} \leq \sqrt{2}\lambda_p \|x\|_{\BMO^c}$. 
Conversely, by the density of $\M$ in $\H_1^c$ we have  
$$\|x\|_{\BMO^c} \leq \sqrt{3} \|x\|_{(\H_1^c)^*}=\sqrt{3} \sup_{y\in \M, \|y\|_{\H_1^c}\leq 1}|\tau(x^*y)|.$$
Let $\varepsilon >0$. 
By Lemma \ref{normHpc}, for each $y\in \M, \|y\|_{\H_1^c}\leq 1$ there exists $1<p(y)<2$ such that $\|y\|_{\H_{p(y)}^c}\leq 1 +\varepsilon$. 
Applying Theorem \ref{fsduality} to $\frac{1}{p(y)}+\frac{1}{p(y)'}=1$ we get 
$$|\tau(x^*y)| \leq \sqrt{2}\|x\|_{L_{p(y)'}^c\MO}\|y\|_{\H_{p(y)}^c} 
\leq \sqrt{2}(1+\varepsilon)  \sup_{2<p<\infty}  \|x\|_{L_{p}^c\MO}. $$
Sending $\eps$ to $0$, we obtain 
$$ (\sqrt{2}\lambda_p)^{-1}\sup_{2<p<\infty}  \|x\|_{L_{p}^c\MO} \leq \|x\|_{\BMO^c} \leq \sqrt{6}  \sup_{2<p<\infty}  \|x\|_{L_{p}^c\MO} .$$
Then by Remark \ref{normLpcMO} (iii) we deduce 
$$ \|x\|_{\BMO^c} \simeq \sup_{2<p<\infty}  \|x\|_{L_{p}^c\MO} \simeq \sup_{2<p<\infty}  \lim_{\si,\U} \|x\|_{L_{p}^cMO(\si)} 
\leq \lim_{\si,\U} \|x\|_{BMO^c(\si)}.$$
Finally, (iv) comes from the reversed monotonicity result for the $H_1^c(\s)$-norms by duality. 
But this approach yields a constant $\sqrt{12}$, which can be improved by a direct proof that we include below. 
Let $x\in \M$ and $\s\subset \s'$. 
Fix $u\in \s'$, there exists a unique element $s(u)\in \s$, satisfying $s(u)^-(\s)\leq u^-(\s')<u\leq s(u)$. 
Observe that for $b \in \B(\ell_2(\s))\oten \M$ we have by contractivity of the conditional expectation
$$\|\E_u|b-\E_{u^-(\s')}(b)|^2\|_\infty 
\leq 2   (\|\E_u|b|^2\|_\infty +\|\E_u|\E_{u^-(\s')}(b)|^2\|_\infty) 
\leq 4 \|\E_u(\E_{s(u)}|b|^2)\|_\infty
\leq 4 \|\E_{s(u)}|b|^2\|_\infty. $$
Applying this to 
$$b=\sum_{s \in \s, s\geq {s(u)}} e_{s,0}\ten d_s^\s(x) \in \B(\ell_2(\s))\oten \M$$
we get
\begin{align*}
\|\E_u|b-\E_{u^-(\s')}(b)|^2\|_{\B(\ell_2(\s))\oten \M} 
&=\Big\|\E_u\Big(\sum_{v \in \s', v\geq u } |d_v^{\s'}(b)|^2\Big)\Big\|_{\B(\ell_2(\s))\oten \M} \\
&\leq 4 \|\E_{s(u)}|b|^2\|_{\B(\ell_2(\s))\oten \M}
=4\Big\|\E_{s(u)}\Big(\sum_{s \in \s, s\geq {s(u)}}|d_s^\s(x)|^2\Big)\Big\|_\infty.
\end{align*}
Recall that 
$$d_{v}^{\s'}(d_{s}^\s (x))=
\left\{\begin{array}{cc}
d_{v}^{\s'} (x)&\quad \mbox{if} \quad s^-(\s)\leq v^-(\s')<v\leq  s\\
0&  \quad \mbox{otherwise}
\end{array}\right..$$
Note that $s(\cdot)$ is monotonous, i.e., for $u,v \in \s'$, $v\geq u$ implies $s(v)\geq s(u)$. 
Hence 
$$d_v^{\s'}(b)=\sum_{s \in \s, s\geq {s(u)}} e_{s,0}\ten d_v^{\s'}(d_s^\s(x))=e_{s(v),0}\ten d_v^{\s'}(x),$$
and
$$\Big\|\E_u\Big(\sum_{v \in \s', v\geq u } |d_v^{\s'}(b)|^2\Big)\Big\|_{\B(\ell_2(\s))\oten \M}
=\Big\|\E_u\Big(\sum_{v \in \s', v\geq u } |d_v^{\s'}(x)|^2\Big)\Big\|_\infty .$$
At the end we showed that for each $u\in \s'$, 
$$\|\E_u|x-\E_{u^-(\s')}(x)|^2\|_\infty^{1/2} \leq 2 \|x\|_{BMO^c(\s)},$$
which yields the required result by taking the supremum over $u\in \s'$. 
\qd

We end this subsection with the following characterization of the $L_p^c\MO$-spaces. 
Observe that this characterization also holds true for $p=\infty$, 
hence this allows us to consider the spaces $L_p^c\MO$ and $\BMO^c$ in a similar way. 

\begin{prop}\label{ballLpcMO}
Let $2<p\leq \infty$. 
Then the unit ball of $L_p^c\MO$ is equivalent to 
$$\mathcal{B}_p=\overline{\{x \in L_2(\M) : \lim_{\s,\U} \|x\|_{L_p^cMO(\s)}\leq 1\}}^{\|\cdot\|_2}.$$
\end{prop}

\begin{proof}
For $p=\infty$, it is obvious that $\mathcal{B}_\infty \subset B_{\BMO^c}$. 
For $2<p<\infty$, Corollary \ref{LpcMOclosed} implies that $ \mathcal{B}_p\subset \sqrt{2}\lambda_p B_{L_p^c\MO}$. 
Conversely, let $x\in B_{L_p^c\MO}$ for $2<p\leq \infty$. 
It suffices to consider $x=\w L_2\mbox{-}\lim_{\s,\U} x_\s$ with $\lim_{\s,\U} \|x_\s\|_{L_p^cMO(\s)}\leq 1$. 
Let $\eps>0$. We may assume that $\|x_\s\|_{L_p^cMO(\s)}\leq 1+\eps$ for each $\s$. 
For a fixed partition $\s'$, since the $L_p^cMO(\s)$-norms are decreasing we have 
$$\lim_{\s,\U} \|x_{\s'}\|_{L_p^cMO(\s)}\leq k_p  \|x_{\s'}\|_{L_p^cMO(\s')}\leq k_p(1+\eps).$$
Moreover, the family $(x_\s)_\s$ is uniformly bounded in $L_2(\M)$. 
Then $x$ is the limit in $L_2$-norm of convex combinations of the $x_\s$'s. 
Let $y=\sum_m \alpha_m x_{\s^m}$ be such a convex combination, then 
$$\lim_{\s,\U} \|y\|_{L_p^cMO(\s)}
\leq \sum_m \alpha_m  \lim_{\s,\U} \|x_{\s^m}\|_{L_p^cMO(\s)}\leq k_p(1+\eps).$$
Sending $\eps$ to $0$ we get $x\in k_p\mathcal{B}_p$. 
\qd

\subsection{Interpolation}

We end the study of the $\H_p^c$ spaces with the continuous analogue of the interpolation Theorem \ref{intH1cBMOc-discr} (i). 
We will deal with the complex method of interpolation, and we refer to \cite{BL} for informations on interpolation. 
This interpolation result has already been used in the literature and is particularly important in abstract semigroup theory. 

\begin{theorem}\label{intH1cBMOc}
Let $1<p<\infty$. Then
$$\H_p^c=[\BMO^c,\H^c_1]_{\frac{1}{p}}\quad \mbox{with equivalent norms}.$$
\end{theorem}

\begin{proof} 
By definition $\BMO^c \subset L_2(\M) \subset \H_1^c$, hence the couple $[\BMO^c,\H^c_1]$ is compatible. 
Recall that $\H_p^c$ embeds isometrically into $K_p^c(\U)$ via the map $\iota$ defined by \eqref{iota} for $1\leq p <\infty$, 
and this inclusion is complemented for $1<p<\infty$ by Proposition \ref{HpccomplKpc}. 
Then the fact that the spaces $K_p^c(\U)$ form an interpolation scale for $1\leq p\leq \infty$ 
(by Corollary \ref{dualintpolKpc} (ii)) clearly implies the inclusion 
\begin{equation}\label{intpolH1cL2}
[\H_1^c,L_2(\M)]_{\theta} \subset \H_p^c
\end{equation} 
for $1<p<2$ and $\frac{1}{p}=1-\frac{\theta}{2}$. 
Conversely, we will prove that 
\begin{equation}\label{intpolBMOL2}
[\BMO^c,L_2(\M)]_{2/p}\subset \H_p^c \quad \mbox{ with equivalent norms for } 2<p<\infty.
\end{equation}
In fact, we will first show that 
\begin{equation}\label{intpoltildeBMOL2}
[\tilde{\BMO}^c,L_2(\M)]_{2/p}\subset \H_p^c \quad \mbox{ with equivalent norms for } 2<p<\infty,
\end{equation}
where 
$$\tilde{\BMO}^c=\overline{\{x\in L_2(\M) : \|x\|_{\tilde{\BMO}^c}=\lim_{\s,\U} \|x\|_{BMO^c(\s)}<\infty\}}^{\|\cdot\|_{\tilde{\BMO}^c}}
=\overline{\M}^{\|\cdot\|_{\tilde{\BMO}^c}}.$$
Then we will use the following fact from \cite{BL}.
\begin{fact}
Let $A_0,A_1$ be a compatible couple such that $A_0\cap A_1$ is dense in $A_0$ and $A_1$. 
Let $\overline{A}_0$ be such that $B_{\overline{A}_0}=\overline{B_{A_0}}^{\|\cdot\|_{A_0+A_1}}$. Then for any $0\leq \theta \leq 1$ we have
$$[A_0,A_1]_{\theta}=[\overline{A}_0,A_1]_{\theta} \quad \mbox{ isometrically}.$$
\end{fact}
We will apply this fact to $A_0=\tilde{\BMO}^c$, $A_1=L_2(\M)$, and \eqref{intpolBMOL2} will follow from \eqref{intpoltildeBMOL2}. 
Indeed, $\tilde{\BMO}^c\cap L_2(\M)\supset \M$ is clearly dense in $\tilde{\BMO}^c$ and in $L_2(\M)$. 
Moreover, $\overline{A}_0=\BMO^c$ with equivalent norm. More precisely, we have
\begin{equation}\label{BMOc-tildeBMOc}
\overline{B_{\tilde{\BMO}^c}}^{\|\cdot\|_2}\subset B_{\BMO^c}\subset 2\overline{B_{\tilde{\BMO}^c}}^{\|\cdot\|_2}.
\end{equation}
Indeed, since $\overline{B_{\tilde{\BMO}^c}}^{\|\cdot\|_2}=\overline{\{x\in L_2(\M)  :  \|x\|_{\tilde{\BMO}^c} \leq 1\}}^{\|\cdot\|_2}$, 
the first inclusion of \eqref{BMOc-tildeBMOc} is obvious by the definition of $\BMO^c$. 
Conversely, if $x=\w L_2\mbox{-}\lim_{\si,\U} x_\s$ with $\lim_{\s,\U} \|x_{\si}\|_{BMO^c(\s)} < 1$ 
(we may assume that $\|x_{\si}\|_{BMO^c(\s)} < 1$ for every $\s$), 
then for all $\eps>0$ we can find a convex combination $x_\eps=\sum_m\alpha_m x_{\s^m} \in L_2(\M)$ such that $\|x-x_\eps\|_2<\eps$. 
By Corollary \ref{normBMOc}, for $\s \supset \cup_m \s^m$ we may write
$$\|x_\eps\|_{BMO^c(\s)}
\leq \sum_m \alpha_m \|x_{\s_m}\|_{BMO^c(\s)}
\leq 2 \sum_m \alpha_m \|x_{\s_m}\|_{BMO^c(\s^m)} <2.$$ 
Then $\|x_\eps\|_{\tilde{\BMO}^c}<2$ and $x\in 2\overline{B_{\tilde{\BMO}^c}}^{\|\cdot\|_2}$, which proves the second inclusion of \eqref{BMOc-tildeBMOc}. 
It remains to prove \eqref{intpoltildeBMOL2}. 
Observe that we have an isometric embedding 
$i_\U:\tilde{\BMO}^c \rightarrow \prodd_\U BMO^c(\s)$
given by $i_\U(x)=(x)^\bullet$ for $x\in \M$. 
This map satisfies $\phi \circ i_\U(x)=x$ for all $x\in \M$, where $\phi$ is defined by 
$$\phi:\left\{\begin{array}{ccc}
\prodd_\U BMO^c(\s) & \longrightarrow &L_2(\M) \\
(x_\s)^\bullet & \longmapsto & \w L_2\mbox{-}\lim_{\s,\U} x_\s
\end{array}\right..$$ 
Let $x\in \M$ be such that $\|x\|_{[\tilde{\BMO}^c,L_2(\M)]_{2/p}}\leq 1$. 
Then there exists an analytic function $f\in \mathcal{F}(\tilde{\BMO}^c,L_2(\M))$ 
such that $x=f\big(\frac{2}{p}\big)$ and 
$$\|f\|_{\mathcal{F}(\tilde{\BMO}^c,L_2(\M))}
=\max \big\{ \sup_{t\in \R} \|f(it)\|_{\tilde{\BMO}^c} , \sup_{t\in \R} \|f(1+it)\|_{2}\big\}
\leq 1.$$ 
By setting $g=i_\U \circ f$, since $i_\U$ is also isometric from $L_2(\M)$ to $\prodd_\U L_2(\M)$, 
we get a function $g\in \mathcal{F}(\prodd_\U BMO^c(\s),\prodd_\U L_2(\M))$ of norm $\leq 1$ such that $i_\U(x)=g\big(\frac{2}{p}\big)$. 
Hence by using ultraproduct techniques and the discrete case we may write
$$i_\U(x) \in \Big[\prodd_\U BMO^c(\s),\prodd_\U L_2(\M)\Big]_{2/p}\subset \prodd_\U [BMO^c(\s),L_2(\M)]_{2/p}=\prodd_\U L_p^cMO(\s),$$
with $\|i_\U(x)\|_{\prodd_\U L_p^cMO(\s)}\leq C_p$.
Recall that in the proof of Lemma \ref{LpMObanach} we have seen that $L_p^c\MO=\phi(\prodd_\U L_p^cMO(\s))$, hence 
$x=\phi \circ i_\U(x) \in L_p^c\MO=\H_p^c$ by Corollary \ref{LpcMOHpc}. 
Moreover we obtain 
$$\|x\|_{\H_p^c}\simeq \|\phi \circ i_\U(x)\|_{L_p^c\MO}
\leq \|i_\U(x)\|_{\prodd_\U L_p^cMO(\s)}
\leq C_p \|x\|_{[\tilde{\BMO}^c,L_2(\M)]_{2/p}}.$$
By density this shows \eqref{intpoltildeBMOL2} and ends the proof of \eqref{intpolBMOL2}. 
By duality, since $L_2(\M)$ is reflexive and $(\H_1^c)^*=\BMO^c$ by Theorem \ref{fsduality}, 
by combining \eqref{intpolH1cL2} and \eqref{intpolBMOL2} we obtain that $[\H_1^c,L_2(\M)]_{\theta} = \H_p^c$ with equivalent norms 
for $1<p<2$ and $\frac{1}{p}=1-\frac{\theta}{2}$. 
Finally, by using the reiteration theorem, Wolff's theorem, duality and Corollary \ref{dual-intpolHpc} (ii), 
we conclude the proof with the usual interpolation techniques. 
\qd

\section{Burkholder-Gundy inequalities}\label{sectBG}

The aim of this section is to establish the analogue of the noncommutative Burkholder-Gundy inequalities in the continuous setting. 
The theory developed previously for the column spaces still holds true for the row spaces. 
Indeed, by considering the adjoint we may define the row Hardy space $\H_p^r$ and obtain the analoguous results.  
By Proposition \ref{injHpc2} we can naturally define the Hardy space for continuous filtrations $\H_p$ as follows. 

\begin{defi}
Let $1 < p<\infty$. We define
$$\H_p=\left\{\begin{array}{cl}
\H_p^c+\H_p^r& \quad \mbox{for} \quad 1< p< 2 \\
\H_p^c\cap \H_p^r& \quad \mbox{for}\quad 2\leq p<\infty
\end{array}\right.,$$
where the sum is taken in $L_p(\M)$ and the intersection in $L_2(\M)$.  
\end{defi}

Observe that for $2\leq p<\infty$, by applying the noncommutative Burkholder-Gundy inequalities in the discrete case for each partition $\s$ and taking the limit in $\s$ 
we immediately obtain 
$$\|x\|_p \simeq \max (\|x\|_{\H_p^c},\|x\|_{\H_p^r}) \quad \mbox{for } x\in L_p(\M).$$
This means that 
$$L_p(\M)=\overline{L_p(\M)}^{\|\cdot\|_{\H_p^c\cap \H_p^r}} \quad \mbox{for } 2 \leq p<\infty.$$
However this result is too weak, we would like to prove that $L_p(\M)= \H_p^c\cap \H_p^r$ for $2\leq p<\infty$. 
To obtain this stronger result, we use a dual approach and first consider the case $1<p<2$. 
The discrete noncommutative Burkholder-Gundy inequalities (Theorem \ref{BGdiscr}) applied to each partition and the monotonicity Lemma \ref{convexityHpc} 
immediately imply the required Burkholder-Gundy inequalities in the continuous setting, as detailed in subsection \ref{subsectBGdirect}. 
However, this result won't be sufficient to apply the classical duality argument 
and get the continuous analogue of the noncommutative Burkholder-Gundy inequalities for $2\leq p <\infty$. 
We will need a stronger decomposition introduced by Randrianantoanina and recalled in subsection \ref{subsectBG-Rand}, 
which will be formalized in subsection \ref{sectsums} by defining another construction for the sum of Banach spaces. 
After extending Randrianantoanina's result for $1<p<2$ to the continuous setting, 
we will be able to deduce by duality the continuous analogue of the noncommutative Burkholder-Gundy inequalities for $2\leq p <\infty$. 
We then discuss the case $p=1$, and establish a Fefferman-Stein duality result for $\H_1$. 
We end this section with the expected interpolation result involving our spaces $\H_1$ and $\BMO$. 

\subsection{Burkholder-Gundy inequalities for $1<p<2$}\label{subsectBGdirect}

We may obtain the Burkholder-Gundy inequalities for $1<p<2$ by a direct approach, as we will detail below. 
Indeed, the proof presented here only uses the discrete Burkholder-Gundy inequalities and the crucial monotonicity property proved in Lemma \ref{convexityHpc}.   
Let us first state the result in this case. 

\begin{theorem}\label{BG1}
Let $1<p<2$. Then
$$L_p(\M)=\H_p=\H_p^c+\H_p^r \quad \mbox{with equivalent norms}.$$
\end{theorem}

\begin{proof}
The inclusion $\H_p\subset L_p(\M)$ is obvious, and for $x\in \H_p$ we have $\|x\|_p\leq \beta_p \|x\|_{\H_p}$. 
Now let $x\in L_p(\M)$. 
Then by the discrete Burkholder-Gundy inequalities, for each $\s$ we may decompose $x=a_\s+b_\s$ where $a_\s\in H_p^c(\s), b_\s\in H_p^r(\s)$ and 
$$\|a_\s\|_{H_p^c(\s)}+\|b_\s\|_{H_p^r(\s)} \leq \alpha_p \|x\|_p.$$
Moreover, for each $\s$ we have 
$$\|a_\s\|_{p}\leq \beta_p\|a_\s\|_{H_p(\s)} \leq  \beta_p\|a_\s\|_{H_p^c(\s)} \leq \beta_p \alpha_p \|x\|_p.$$ 
Hence the family $(a_\s)_\s$ is uniformly bounded in $L_p(\M)$, and since $1<p<2$ the weak-limit of the $a_\s$'s exists in $L_p$. 
The same holds for $(b_\s)$, and we set
$$a=\w L_p \mbox-\lim_{\s,\U} a_\s \quad \mbox{and} \quad b=\w L_p \mbox-\lim_{\s,\U} b_\s.$$
Then $x=a+b$. 
It remains to prove that $a\in \H_p^c$ and $b\in \H_p^r$. 
Recall that by Corollary \ref{HpXp}, since $a\in L_p(\M)$ it suffices to estimate $\|a\|_{\H_p^c}=\lim_{\s,\U}\|a\|_{H_p^c(\s)}$ 
and $\|b\|_{\H_p^r}=\lim_{\s,\U}\|b\|_{H_p^r(\s)}$. 
Fix $\eps>0$ and a finite partition $\s$ of $[0,1]$. 
We can find positive numbers $(\alpha_m)_{m=1}^M$ verifying $\sum_m \alpha_m=1$ 
and partitions $\s^1, \cdots, \s^M$ containing $\s$ such that
\begin{equation}\label{approxa}
\Big\|a-\sum_{m=1}^M \alpha_m a_{\s^m}\Big\|_p<\eps
\quad \mbox{and} \quad 
\Big\|b-\sum_{m=1}^M \alpha_m b_{\s^m}\Big\|_p<\eps
.
\end{equation}
On the one hand, note that for $y\in L_p(\M)$ we have
\begin{equation}\label{HpcLp}
\|y\|_{H_p^c(\s)}\leq 2 |\s|^{1/p}\|y\|_p.
\end{equation}
Indeed, we may write
\begin{align*}
\|y\|_{H_p^c(\s)}^p
&=\Big\|\sum_{t\in \s} |d_t^{\s}(y)|^2\Big\|_{p/2}^{p/2}\\
&\leq \sum_{t\in \s}\||d_t^{\s}(y)|^2\|_{p/2}^{p/2}
=\sum_{t\in \s}\|d_t^{\s}(y)\|_{p}^{p}\\
&\leq \sum_{t\in \s} (2\|y\|_p)^p
\leq 2^p|\s| \|y\|_p^p.
\end{align*}
Taking the adjoint we obtain 
\begin{equation}\label{HprLp}
\|y\|_{H_p^r(\s)}\leq 2 |\s|^{1/p}\|y\|_p.
\end{equation}
Then combining \eqref{approxa} with \eqref{HpcLp} and \eqref{HprLp} we get
\begin{equation}\label{BGeq1}
\Big\|a-\sum_{m=1}^M \alpha_m a_{\s^m}\Big\|_{H_p^c(\s)}\leq 2 |\s|^{1/p}\eps
\quad \mbox{and} \quad 
\Big\|b-\sum_{m=1}^M \alpha_m b_{\s^m}\Big\|_{H_p^r(\s)}\leq 2 |\s|^{1/p}\eps.
\end{equation}
On the other hand, since $\s\subset \s^m$ for all $m$, Lemma \ref{convexityHpc} yields
\begin{equation}\label{BGeq2}
\begin{array}{cl}
&\Big\|\displaystyle\sum_{m=1}^M \alpha_m a_{\s^m}\Big\|_{H_p^c(\s)}
+
\Big\|\displaystyle\sum_{m=1}^M \alpha_m b_{\s^m}\Big\|_{H_p^r(\s)}\\
\leq & \displaystyle\sum_{m=1}^M \alpha_m (\|a_{\s^m}\|_{H_p^c(\s)}
+ \|b_{\s^m}\|_{H_p^r(\s)}) \\
 \leq & \beta_p \displaystyle\sum_{m=1}^M \alpha_m (\|a_{\s^m}\|_{H_p^c(\s^m)}+\|b_{\s^m}\|_{H_p^r(\s^m)})
\leq \beta_p \alpha_p \|x\|_p.
\end{array}
\end{equation}
Finally by \eqref{BGeq1} and \eqref{BGeq2} we get
\begin{align*}
&\|a\|_{H_p^c(\s)}+\|b\|_{H_p^r(\s)} \\
\leq & \Big\|a-\sum_{m=1}^M \alpha_m a_{\s^m}\Big\|_{H_p^c(\s)}+ \Big\|\sum_{m=1}^M \alpha_m a_{\s^m}\Big\|_{H_p^c(\s)}
+\Big\|b-\sum_{m=1}^M \alpha_m b_{\s^m}\Big\|_{H_p^r(\s)}+ \Big\|\sum_{m=1}^M \alpha_m b_{\s^m}\Big\|_{H_p^r(\s)}\\
\leq & 4 |\s|^{1/p}\eps + \beta_p \alpha_p \|x\|_p.
\end{align*}
Sending $\eps$ to $0$ we obtain $\|a\|_{H_p^c(\s)}+\|b\|_{H_p^r(\s)}\leq\beta_p \alpha_p \|x\|_p$ for all $\s$. 
Taking the limit over $\s$ we get
$$\|a\|_{\H_p^c}+\|b\|_{\H_p^r} \leq \beta_p \alpha_p \|x\|_p.$$
\qd

Recalling that $(\H_p^c)^*=\H_{p'}^c$ by Corollary \ref{dual-intpolHpc} (i), 
we would like to deduce, as usual, the Burkholder-Gundy inequalities for $2<p'<\infty$ by duality from the case $1<p<2$. 
However, as detailed in Remark \ref{dualitybracket}, the duality bracket between $\H_p^c$ and $\H_{p'}^c$ is not always explicit. 
At one point we will need that the elements in the decomposition $L_p(\M)=\H_p^c+\H_p^r $ lie in $L_2(\M)$ when $x\in L_2(\M)$. 
This is why we need a stronger result due to Randrianantoanina in the discrete setting.

\subsection{Randrianantoanina's result in the discrete case}\label{subsectBG-Rand}

Let $(\M_n)_{n\geq 0}$ be a discrete filtration. 
In \cite{ran-weak}, Randrianantoanina gives another proof of the Burkholder-Gundy inequalities based on weak-type $(1,1)$ estimates. 
This approach yields a better decomposition at the $L_2$-level in the sense of the following Theorem. 

\begin{theorem}\label{RaBGdiscr} 
Let $1<p<2$ and $x\in L_2(\M)$. 
Then there exist $a,b \in L_2(\M)$ such that 
\begin{enumerate}
\item[(i)] $x=a+b$,
\item[(ii)]$ \|a\|_{H_p^c}+\|b\|_{H_p^r} \leq C(p) \|x\|_p $,
\item[(iii)] $\max\{\|a\|_2,\|b\|_2\} \leq f(p,\|x\|_p,\|x\|_2)$.
\end{enumerate}
Here $C(p)\leq C(p-1)^{-1}$ as $p\to 1$. 
\end{theorem}

\begin{proof} 
We derive the estimate of the $L_2$-norms (iii) from Randrianantoanina's construction. 
The main tool is the real interpolation, more precisely the J-method, to deduce this decomposition from a weak type $(1,1)$-inequality. 
We refer to \cite{BL} for details on interpolation. 
Let $x\in L_2(\M)$ and $1<p<2$. 
Let $0<\theta<1$ be such that $\frac1p=1-\theta +\frac{\theta}{2}$. 
We know that $L_p(\M)=[L_1(\M), L_2(\M)]_{\theta,p;J}$, hence we may write
\begin{equation}\label{dec0}
x = \sum_{\nu \in \Z} u_\nu
\end{equation}
where
\begin{equation}\label{Jmethod}
\Big(\sum_{\nu \in \Z} (2^{-\nu \theta} \max\{\|u_\nu\|_1,2^\nu\|u_\nu\|_2\})^p\Big)^{1/p}\leq C(p)\|x\|_p.
\end{equation}
We claim that we may in addition suppose that 
\begin{equation}\label{L2estimate}
\sum_{\nu \in \Z} \|u_\nu\|_2\leq f(p,\|x\|_p,\|x\|_2).
\end{equation}
For each $\nu \in \Z$ we set
$$e_\nu=\1(\mu_{4^\nu}(x)<|x|\leq \mu_{4^{\nu-1}}(x)),$$
where for $t>0$, $\mu_t(x)$ denote the generalized singular numbers of $x$. 
We refer to \cite{FK} for details on these generalized numbers. 
Since $\mu_t(x)\to \|x\|$ as $t\to 0$ and  $\mu_t(x)\to 0$ as $t\to \infty$, we see that $\sum_{\nu\in \Z}e_\nu=s(|x|)$, 
where $s(|x|)$ denotes the support projection of $x$. 
Hence we can write
\begin{equation}\label{dec1}
x = \sum_{\nu\in \Z} xe_{\nu}.
\end{equation}
Let us first show that the sequence $u_\nu=xe_{\nu}$ satisfy \eqref{Jmethod} with $C(p)=\Big(\frac{16}{3}\Big)^{1/p}$. 
Note that by the definition of $\mu_t(x)$ we have for all $\nu$
\begin{equation}\label{tauenu}
\tau(e_\nu)\leq \tau \big(\1(\mu_{4^\nu}(x)<|x|)\big) \leq 4^\nu.
\end{equation}
On the other hand, since $\mu_t(x)$ is decreasing we have
\begin{align*}
\|x\|_p^p
&=\int_0^\infty \mu_t(x)^p dt 
=\sum_{\nu\in \Z} \int_{4^{\nu-2}}^{4^{\nu-1}} \mu_t(x)^p dt \\
&\geq \sum_{\nu\in \Z} (4^{\nu-1}-4^{\nu-2})\mu_{4^{\nu-1}}(x)^p 
= \sum_{\nu\in \Z} 3.4^{\nu-2}\mu_{4^{\nu-1}}(x)^p.
\end{align*}
The inequality \eqref{tauenu} gives
$$\|xe_\nu\|_1=\tau\big(|x|\1(\mu_{4^\nu}(x)<|x|\leq \mu_{4^{\nu-1}}(x))\big)
\leq\mu_{4^{\nu-1}}(x) \tau(e_\nu) \leq \mu_{4^{\nu-1}}(x)4^\nu.$$
Using $p(2- \theta)=2$ we get
$$\sum_{\nu \in \Z} (2^{-\nu \theta} \|x e_\nu\|_1)^p
\leq \sum_{\nu \in \Z} 2^{\nu p(2- \theta)}\mu_{4^{\nu-1}}(x)^p
=\sum_{\nu \in \Z} 4^{\nu }\mu_{4^{\nu-1}}(x)^p
\leq \frac{16}{3} \|x\|_p^p.$$
The $L_2$-norm can be estimated by 
$$\|xe_\nu\|_2
=\tau\big(|x|^2\1(\mu_{4^\nu}(x)<|x|\leq \mu_{4^{\nu-1}}(x))\big)^{1/2}
\leq\mu_{4^{\nu-1}}(x) \tau(e_\nu)^{1/2} \leq \mu_{4^{\nu-1}}(x)2^\nu,$$
hence
$$\sum_{\nu \in \Z} (2^{\nu(1- \theta)} \|x e_\nu\|_2)^p
\leq \sum_{\nu \in \Z} 2^{\nu p(2- \theta)}\mu_{4^{\nu-1}}(x)^p
\leq \frac{16}{3} \|x\|_p^p.$$
Let us now consider $\nu_0 \in \Z$. 
Then, to obtain \eqref{L2estimate}, we replace decomposition \eqref{dec1} by 
$$x = \sum_{\nu\geq \nu_0} x\tilde{e}_{\nu},$$
where $\tilde{e}_{\nu}=e_{\nu}$ for $\nu>\nu_0$ and $\tilde{e}_{\nu_0}=\sum_{\nu\le \nu_0} e_{\nu}=\1(\mu_{4^{\nu_0}}(x)<|x|)$.
For a good choice of $\nu_0$, we can show that this decomposition still satisfy \eqref{Jmethod} with $C(p)=\Big(\frac{19}{3}\Big)^{1/p}$. 
Note that
 $$ 2^{-\nu_0\theta}\|x\tilde{e}_{\nu_0}\|_1=2^{-\nu_0\theta}\|x\1(\mu_{4^{\nu_0}}(x)<|x|)\|_1
 \leq 2^{-\nu_0\theta}\|x\|_2 \tau(\1(\mu_{4^{\nu_0}}(x)<|x|))^{1/2}
 \leq  2^{\nu_0(1-\theta)} \|x\|_2 $$
and 
$$2^{\nu_0(1-\theta)}\|x\tilde{e}_{\nu_0}\|_2 \leq 2^{\nu_0(1-\theta)}\|x\|_2.$$ 
We can find $\nu_0=\nu_0(p,\|x\|_p,\|x\|_2)$ such that
$$2^{\nu_0(1-\theta)}\|x\|_2 \leq \|x\|_p \Leftrightarrow \nu_0\leq (1-\theta)^{-1}\ln\Big(\frac12\Big)\ln\Big(\frac{\|x\|_p}{\|x\|_2}\Big).$$
We then obtain 
$$\Big(\sum_{\nu \geq \nu_0} (2^{-\nu \theta} \max\{\|x\tilde{e}_{\nu}\|_1,2^\nu\|x\tilde{e}_{\nu}\|_2\})^p\Big)^{1/p}\leq \Big(\frac{19}{3}\Big)^{1/p}\|x\|_p.$$
The inequality \eqref{L2estimate} follows from the H\"older inequality 
$$\sum_{\nu \geq \nu_0} \|x\tilde{e}_{\nu}\|_2
 \leq \Big(\sum_{\nu \geq \nu_0} (2^{\nu(1-\theta)}\|x\tilde{e}_{\nu}\|_2)^p\Big)^{1/p}
 \Big(\sum_{\nu \geq \nu_0} 2^{-\nu(1-\theta)p'}\Big)^{1/p'}
 \leq  f(p,\|x\|_p,\|x\|_2),$$
where 
$$f(p,\|x\|_p,\|x\|_2)=\frac{2^{-\nu_0(1-\theta)}}{(1-2^{-(1-\theta)p'})^{1/p'}}\Big(\frac{19}{3}\Big)^{1/p}\|x\|_p.$$
Now we apply Randrianantoanina's decomposition to the sequence $(u_\nu)_\nu$ satisfying \eqref{dec0}, \eqref{Jmethod} and \eqref{L2estimate}. 
For each $\nu \in \Z$, by Theorem $3.1$ of \cite{ran-weak}, 
we may find an absolute constant $K>0$ and two martingales $a^{(\nu)}=(a_n^{(\nu)})_n$ and $b^{(\nu)}=(b_n^{(\nu)})_n$ such that 
$$\E_n(u_\nu)=a_n^{(\nu)}+ b_n^{(\nu)}, \quad \forall  n\geq 0$$
and
$$
\|da^{(\nu)}\|_{L_2(\M;\ell_2^c)}+\|db^{(\nu)}\|_{L_2(\M;\ell_2^r)}
\leq 2 \|u_\nu\|_2,$$
$$\Big\|\Big(\sum_{n\geq 0}|d_n(a^{(\nu)})|^2\Big)^{1/2}\Big\|_{1,\infty} 
+\Big\|\Big(\sum_{n\geq 0}|d_n(b_n^{(\nu)})^*|^2\Big)^{1/2}\Big\|_{1,\infty} 
\leq K \|u_\nu\|_1.
$$
Recall that $\|x\|_{1,\infty}=\sup_{t>0} t\mu_t(x)$. 
Then we set
$$a=\sum_{\nu\in\Z} a^{(\nu)} \quad  \mbox{and} \quad b=\sum_{\nu\in\Z} b^{(\nu)},$$
and obtain two martingales $a$ and $b$ with $x=a+b$.  
Using the following interpolation result of noncommutative $L_p$-spaces associated to a semifinite von Neumann algebra $\Nc$
$$L_p(\Nc)=[L_{1,\infty}(\Nc), L_2(\Nc)]_{\theta,p;J},$$
and \eqref{Jmethod} we can show that
$$\Big\|\Big(\sum_{n\geq 0}|d_n(a)|^2\Big)^{1/2}\Big\|_p
+\Big\|\Big(\sum_{n\geq 0}|d_n(b)^*|^2\Big)^{1/2}\Big\|_p\leq C(p-1)^{-1}\|x\|_p.$$
It remains to prove the $L_2$-estimate (iii). 
This comes from \eqref{L2estimate} as follows
$$\|a\|_2\leq  \sum_{\nu\in \Z} \|a^{(\nu)}\|_{2} =\sum_{\nu\in \Z}  \|da^{(\nu)}\|_{L_2(\M;\ell_2^c)} 
\leq 2\sum_{\nu\in \Z} \|u_{\nu}\|_2 \leq 2f(p,\|x\|_p,\|x\|_2).$$
The estimate for $b$ is similar.
\qd

\subsection{Sums of Banach spaces}\label{sectsums}

In this subsection we introduce a notation to formalize the notion of ``decomposition at the $L_2$-level" mentioned previously. 
To do this, we discuss two competing constructions of the sum of Banach spaces in a general case. 
Let $X$ and $Y$ be two Banach spaces both embedded into a Banach space $A_1$, i.e., 
the inclusion maps $X\subset A_1$ and $Y\subset A_1$ are continuous and injective. 
In interpolation theory one considers the sum
 $$ X+Y = \{z\in A_1: \exists x\in X, y\in Y\mbox{ such that } z=x+y\}$$
equipped with the norm
 $$ \|z\|_{X+Y}= \inf_{z=x+y}\|x\|_X+\|y\|_Y .$$
The second method we will consider depends on a fourth space $A_0$, which is also injectively embedded into $A_1$. 
We assume that 
\begin{equation}\label{density}
A_0\cap X \mbox{ is dense in } X \mbox{ and } A_0 \mbox{ is dense in } Y.
\end{equation}
For $z\in A_0$ we set
$$\|z \|_{X\boxplus_{A_0} Y}=\inf_{\substack{z=x+y, \\ x\in A_0\cap X, \\ y\in A_0}}\|x\|_X+\|y\|_Y .$$
We clearly have 
\begin{equation}\label{eqsums}
\|z\|_{X+Y} \leq \|z \|_{X\boxplus_{A_0} Y} \quad \mbox{for} \quad z\in A_0,
\end{equation}
and $\|\cdot\|_{X\boxplus_{A_0} Y}$ defines a norm on $A_0$. 
We define the $A_0$-sum 
$$ X\boxplus_{A_0}Y $$
as the completion of $A_0$ with respect to the norm $\|\cdot\|_{X\boxplus_{A_0} Y}$. 
In our context we will always consider $A_0=L_2(\M)$, and simply denote $X\boxplus Y $. 
Let us state the following basic fact.

\begin{lemma}\label{quot1}  
Let $A_0,X,Y,A_1$ be four Banach spaces as above. 
Then there exists a surjective quotient map $q:X\boxplus Y\to X+Y$.
\end{lemma}

\begin{proof} 
By \eqref{eqsums} we can consider the contractive map $q:X\boxplus Y\to X+Y$ defined by $q(z)=z$ for $z\in A_0$. 
Let us show that $q$ is a quotient map. 
Let $z\in X+Y$ be of norm $<1$. 
We can find $x\in X$ and $y\in Y$ such that $z=x+y$ and $\|x\|_X=\lambda, \|y\|_Y=\mu$ with $\lambda+\mu<1$. 
Since $A_0\cap X$ is dense in $X$, we can find a sequence $(x_n)_n$ in $A_0\cap X$ 
such that the series is absolutely converging and 
$$\sum_n\|x_n\|_X\leq \lambda + \frac{1-(\lambda+\mu)}{4}\quad , \quad 
x=\sum_n x_n \mbox{ in } X.$$ 
Similarly, there exists a sequence $(y_n)_n$ in $A_0$ 
such that the series is absolutely converging and 
$$\sum_n\|y_n\|_Y\leq \mu+\frac{1-(\lambda+\mu)}{4}\quad , \quad 
y=\sum_n y_n \mbox{ in } Y.$$
Then $z_n=x_n+y_n \in X\boxplus Y$ for all $n$, and 
$$\sum_n \|z_n\|_{ X\boxplus Y}\leq \sum_n \|x_n\|_X+ \|y_n\|_Y \leq \frac{1+\lambda+\mu}{2} <1.$$
Hence the series $(z_n)_n$ is absolutely converging in $X\boxplus Y$ and we have 
$$q\Big(\sum_n z_n \Big)=z.$$
This ends the proof. 
\qd

The two sums coincide in the following cases. 

\begin{lemma}\label{sumeq}  
Let $A_0,X,Y,A_1$ be four Banach spaces as above. 
Then the following assertions are equivalent.
\begin{enumerate} 
\item[(i)] $X+Y=X\boxplus Y$ with equivalent norms;
\item[(ii)] $X+Y=X\boxplus Y$ isometrically;
\item[(iii)]$X\boxplus Y$ embeds injectively into $A_1$.
\end{enumerate}
\end{lemma}

\begin{proof}
By Lemma \ref{quot1}, we see that the two sums coincide with equivalent norms if and only if they coincide isometrically if and only if the quotient map $q$ is injective. 
Let us consider the following commuting diagram
$$\xymatrix{
    X\boxplus Y  \ar@{->}[r]^{q}  \ar@{->}[rd]^{f}   & X+Y \ar@{^{(}->}[d]^{id}  \\
    & A_1
  }.$$ 
It is clear that $q$ is injective if and only if $f$ is injective. This proves the Lemma. 
\qd
 
\re
These two sums may be seen as quotient of Banach spaces. 
Indeed, on the one hand $X+Y$ is isometrically isomorphic to the quotient space $X\oplus_1 Y/L$, where
$$L=\ker \phi =\{(x,-x)\in X\oplus_1 Y : x\in X\cap Y\}$$
and 
$$\phi:\left\{\begin{array}{ccc}
X\oplus_1 Y &\longrightarrow &A_1 \\
(x,y)&\longmapsto &x+y
\end{array}\right..$$
On the other hand, $X\boxplus Y$ is isometrically isomorphic to the completion of the quotient
$$ \big((A_0\cap X)\oplus_1 A_0\big)/L_0,$$
where
$$L_0=\ker (\phi_{|(A_0\cap X)\oplus_1 A_0})
=L \cap \big((A_0\cap X)\oplus_1 A_0\big)
= \{(x,-x)\in X\oplus_1 Y : x\in A_0\cap X\} .$$
The density assumption \eqref{density} then implies that
\begin{equation}\label{boxplusquotient}
X\boxplus Y=\overline{\big((A_0\cap X)\oplus_1 A_0\big)/L_0}
=X\oplus_1 Y / \overline{L_0}.
\end{equation}
Hence we can write
$$X+Y=X\boxplus Y \Leftrightarrow L=\overline{L_0}\Leftrightarrow L_0\mbox{ is dense in } L\Leftrightarrow 
A_0\cap X\mbox{ is dense in } X\cap Y.$$  
\mar

As mentioned previously, the introduction of this $\boxplus$-sum is motivated by some dual arguments. 
It is well known that the dual of the usual sum $X+Y$ is $X^*\cap Y^*$ whenever $X\cap Y$ is dense in $X$ and $Y$, but this is not true in general. 
In some cases, the dual space of $X \boxplus Y$ is easier to describe than the dual space of the usual sum $X+Y$. 
More precisely, the dual spaces of these two constructions are described in the following Lemma. 

\begin{lemma}\label{dualboxplus}
Let $A_0,X,Y$ and $A_1$ be such that $A_0\cap X $ is dense in $X$, $A_0$ is dense in $Y$ and $X,Y$ embed into $A_1$. Then 
\begin{enumerate}
\item[(i)] $(X+Y)^*=\{(x^*,y^*) \in X^* \oplus_\infty Y^* : x^*_{|X\cap Y}=y^*_{|X\cap Y}\}$.
\item[(ii)] $(X\boxplus Y)^*=\{(x^*,y^*) \in X^* \oplus_\infty Y^* : x^*_{|A_0\cap X}=y^*_{|A_0\cap X}\}$.
\end{enumerate}
\end{lemma}

\begin{proof}
Since $X+Y=X\oplus_1 Y/L$, we deduce that $(X+Y)^*=L^\perp\subset (X\oplus_1 Y)^*=X^* \oplus_\infty Y^*$ and we obtain (i). 
By \eqref{boxplusquotient} we can write $(X  \boxplus Y)^*=(\overline{L_0})^\perp=L_0^\perp$ and (ii) follows. 
\end{proof}

\re
Observe that the definition of the sum $X \boxplus_{A_0}Y$ only relies on the space $A_0$ and not on $A_1$. 
In fact, we do not need that $X$ and $Y$ are embedded into a common space $A_1$ to define $X \boxplus_{A_0}Y$. 
However, in that situation we cannot define the usual sum $X+Y$. 
\mar

Theorem \ref{RaBGdiscr} can be reformulated by using the $\boxplus$-sum as follows.
 
\begin{cor}\label{BGboxplusdiscr}
Let $1<p<2$. Then 
$$L_p(\M)=H_p^c \boxplus H_p^r \quad \mbox{with equivalent norms}.$$
\end{cor}
 
\begin{proof}
In this application we consider $A_0=L_2(\M), X=H_p^c, Y=H_p^r$ and $A_1=L_p(\M)$. 
The density assumption \eqref{density} is clearly satisfied. 
By the density of $L_2(\M)$, it suffices to see that the norm $\|\cdot\|_p$ is equivalent to the norm $\|\cdot\|_{H_p^c\boxplus H_p^r}$ 
defined for $x \in L_2(\M)$ by 
$$\| x\|_{H_p^c\boxplus H_p^r} = \inf_{\substack{x=a+b, \\ a, b \in L_2(\M)}} \|a\|_{H_p^c}+\|b\|_{H_p^r}.$$ 
Theorem \ref{RaBGdiscr} means that 
$$ \| x\|_{H_p^c\boxplus H_p^r} \leq C(p) \|x\|_p \quad \mbox{for} \quad x\in L_2(\M),$$
and Theorem \ref{BGdiscr} gives the reverse inequality
$$\|x\|_p\leq \beta_p\| x\|_{H_p^c+H_p^r}  \leq \beta_p \| x\|_{H_p^c\boxplus H_p^r} .$$ 
\qd
 
\subsection{Burkholder-Gundy inequalities for $2<p<\infty$}\label{BGp>2} 

As mentioned previously, we need a stronger version of the Burkholder-Gundy inequalities for $1<p<2$ stated in Theorem \ref{BG1} before proving the case $2<p<\infty$ by duality. 
We may extend Randrianantoanina's result recalled in Theorem \ref{RaBGdiscr} to the continuous setting as follows. 

\begin{prop}\label{RaBG}
Let $1<p<2$ and $x\in L_2(\M)$. 
Then there exist $a,b \in L_2(\M)$ such that 
\begin{enumerate}
\item[(i)] $x=a+b$,
\item[(ii)]$ \|a\|_{\H_p^c}+\|b\|_{\H_p^r} \leq C(p) \|x\|_p $,
\item[(iii)] $\max\{\|a\|_2,\|b\|_2\} \leq f(p,\|x\|_p,\|x\|_2)$.
\end{enumerate}
Here $C(p)\leq C(p-1)^{-1}$ as $p\to 1$. 
\end{prop}

\begin{proof}
The proof is similar to that of Theorem \ref{BG1}. 
In this case, for each $\s$ we apply Theorem \ref{RaBGdiscr} to the discrete Hardy spaces $H_p^c(\s)$ and $H_p^r(\s)$. 
We obtain a decomposition $x=a_\s+b_\s$ with 
$$ \|a_\s\|_{H_p^c(\s)}+\|b_\s\|_{H_p^r(\s)} \leq C(p) \|x\|_p \quad \mbox{and} \quad 
\max\{\|a_\s\|_2,\|b_\s\|_2\} \leq f(p,\|x\|_p,\|x\|_2).$$
Hence the families $(a_\s)_\s$ and $(b_\s)_\s$ are uniformly bounded in $L_2$, and we can consider 
$$a=\w L_2 \mbox-\lim_{\s,\U} a_\s \quad \mbox{and} \quad b=\w L_2 \mbox-\lim_{\s,\U} b_\s.$$
We obtain $x=a+b$ where $a,b \in L_2(\M)$ satisfy (iii). 
The proof of the estimate (ii) is even simpler than in the proof of Theorem \ref{BG1}.  
We use the fact that for $y\in L_2(\M)$, $\|y\|_{H_p^c(\s)}\leq \|y\|_2$ and the result follows similarly. 
\qd

\begin{cor}\label{Hpboxplus}
Let $1<p<2$. Then 
$$\H_p=\H_p^c\boxplus \H_p^r \quad \mbox{isometrically}.$$
\end{cor}

\begin{proof}
In terms of $\boxplus$-sum, Proposition \ref{RaBG} means that 
$$L_p(\M)=\H_p^c\boxplus \H_p^r \quad \mbox{with equivalent norms}.$$
Here we consider $A_0=L_2(\M), X=\H_p^c, Y=\H_p^r$ and $A_1=L_p(\M)$. 
Moreover, we know by Theorem \ref{BG1} that 
$$L_p(\M)=\H_p=\H_p^c+\H_p^r \quad \mbox{with equivalent norms}.$$
We deduce that $\H_p^c+\H_p^r=\H_p^c\boxplus \H_p^r$ with equivalent norms, hence the two sums coincide isometrically by Lemma \ref{sumeq}. 
\qd

We can now apply the duality argument to get the remaining case $2<p<\infty$. 

\begin{theorem}\label{BG2}
Let $2<p<\infty$. Then
$$L_p(\M)=\H_p=\H_p^c \cap \H_p^r \quad \mbox{with equivalent norms.}$$
\end{theorem}

\begin{proof}
In this case the non-obvious inclusion is $\H_p \subset L_p(\M)$. 
We detail the argument to highlight the need of the decomposition in $L_2(\M)$. 
Let $y \in \H_p=\H_p^c \cap \H_p^r \subset L_2(\M)$ and $x\in L_{2}(\M)$ be such that $\|x\|_{p'}\leq 1$. 
By Proposition \ref{RaBG}, there exist $a,b\in L_2(\M)$ such that $x=a+b$ and 
$$\|a\|_{\H_{p'}^c}+\|b\|_{\H_{p'}^r} \leq C(p').$$
Then $\tau(y^*x)=\tau(y^*a)+\tau(y^*b).$ 
Moreover, since $y\in \H_p^c$ and $a\in L_2(\M)$ we can write by Corollary \ref{dual-intpolHpc} (i) and Remark \ref{dualitybracket}
$$|\tau(y^*a)|=|(y|a)|\leq c(p)\|y\|_{\H_p^c}\|a\|_{\H_{p'}^c} \leq c(p)\|y\|_{\H_p}\|a\|_{\H_{p'}^c}.$$
The same estimate holds true for $b$ and we get
$$|\tau(y^*x)|\leq c(p)\|y\|_{\H_p}(\|a\|_{\H_{p'}^c}+\|b\|_{\H_{p'}^r}) \leq C(p') c(p)\|y\|_{\H_p}\|x\|_{p'}.$$
By density of $L_2(\M)$ in $L_{p'}(\M)$ we deduce
$$\|y\|_{p}\leq C(p') c(p)\|y\|_{\H_p}.$$
\qd

\re
Observe that in this proof, the fact that the decomposition $x=a+b$ is in $L_2$ is crucial. 
Indeed, if $a$ and $b$ do not lie in $L_2(\M)$, then the quantities $\tau(y^*a)$ and $\tau(y^*b)$ may not exist, and the duality argument does not work. 
\mar

\subsection{The space $\H_1$}\label{subsectH1}

We end this section with a discussion on the case $p=1$. 
Inspired by Lemma \ref{normHpc}, we define for $x\in \M$
$$\|x\|_{\H_1}=\lim_{p\to 1} \|x\|_{\H_p}.$$
Since the $\H_p$-norm is decreasing in $p$, the limit is in fact an infimum, which exists for 
$(\|x\|_{\H_p})_{p>1}$ is then a decreasing family bounded by below. 
Moreover, the inequalities
$$\beta_1^{-1} \|x\|_1\leq \|x\|_{\H_1} \leq \|x\|_2$$
ensure that this defines a norm on $\M$. 

\begin{defi}
We define the space $\H_1$ as the completion of $\M$ with respect to the norm $\|\cdot\|_{\H_1}$.
\end{defi}

By approximation we can extend Corollary \ref{Hpboxplus} to the case $p=1$. 

\begin{prop}\label{boxplusH1}
We have
$$\H_1=\H_1^c\boxplus \H_1^r \quad \mbox{isometrically}.$$
\end{prop}

\begin{proof}
In this application we consider $A_0=L_2(\M), X=\H_1^c, Y=\H_1^r$ and $A_1=L_1(\M)$. 
The density assumption \eqref{density} is satisfied. 
By the density of $L_2(\M)$, it suffices to see that the norms $\|\cdot\|_{\H_1}$ and $\|\cdot\|_{\H_1^c\boxplus \H_1^r }$ are equivalent on $L_2(\M)$. 
Let $x\in L_2(\M)$. By Corollary \ref{Hpboxplus} we may write
\begin{align*}
\|x\|_{\H_1}
&=\lim_{p\to 1} \|x\|_{\H_p}
=\lim_{p\to 1} \| x\|_{\H_p^c\boxplus \H_p^r}\\
&=\lim_{p\to 1} \inf_{\substack{x=a+b, \\a, b \in L_2(\M)}} \|a\|_{\H_p^c}+\|b\|_{\H_p^r}\\
&\geq \| x\|_{\H_1^c\boxplus \H_1^r }. 
\end{align*}
On the other hand, assume that $\| x\|_{\H_1^c\boxplus \H_1^r }<1$. 
Then there exist $a, b \in L_2(\M)$ such that $x=a+b$ and $\|a\|_{\H_1^c}+\|b\|_{\H_1^r}<1$. 
Observe that Lemma \ref{normHpc} still holds true for $a,b \in L_2(\M)$, thus
$$\|a\|_{\H_1^c}=\lim_{p\to 1} \|a\|_{\H_p^c} \quad \mbox{and} \quad \|b\|_{\H_1^r}=\lim_{p\to 1} \|b\|_{\H_p^r}.$$
Hence we can find $p>1$ such that $\|x\|_{\H_p}\leq \|a\|_{\H_p^c}+\|b\|_{\H_p^r}<1$. 
Since the $\H_p$-norm is decreasing in $p$ we get the reverse inequality $\|x\|_{\H_1}\leq \| x\|_{\H_1^c\boxplus \H_1^r }$. 
\qd


With this decomposition at the $L_2$-level we can describe the dual space of $\H_1$ as follows.

\begin{theorem}\label{dualityH1BMO}
We have 
$$(\H_1)^*=\BMO \quad \mbox{with equivalent norms},$$
where $\BMO=\BMO^c\cap \BMO^r$. 
\end{theorem}

\begin{proof}
The proof directly follows from Proposition \ref{boxplusH1} and Theorem \ref{fsduality} 
by using the same argument than the one detailed in the proof of Theorem \ref{BG2}. 
\qd

\subsection{Interpolation}

We end this section with the continuous analogue of the interpolation Theorem \ref{intH1cBMOc-discr} (ii) 
involving the spaces $\H_1$ and $\BMO$ introduced in subsection \ref{subsectH1}. 

\begin{theorem}\label{intH1BMO}
Let $1<p<\infty$. Then
$$L_p(\M)=[\BMO,\H_1]_{\frac{1}{p}}\quad \mbox{with equivalent norms}.$$
\end{theorem}

\begin{proof}
The inclusions $\BMO \subset L_2(\M) \subset \H_1$ ensure that the couple $[\BMO,\H_1]$ is compatible. 
As in the proof of Theorem \ref{intH1cBMOc}, we only need to prove that $[\BMO,L_2(\M)]_{2/p}=L_p(\M)$ for $2<p<\infty$, 
and we will conclude by using the duality $(\H_1)^*=\BMO$ 
established in Theorem \ref{dualityH1BMO}. 
On the one hand, by Theorem \ref{intH1cBMOc} we can write
\begin{align*}
[\BMO,L_2(\M)]_{2/p}
&=[\BMO^c\cap \BMO^r,L_2(\M)]_{2/p}
\subset [\BMO^c,L_2(\M)]_{2/p}\cap [\BMO^r,L_2(\M)]_{2/p}\\
&=\H_p^c\cap \H_p^r =L_p(\M),
\end{align*}
where the last equality is the Burkholder-Gundy Theorem \ref{BG2} for $2<p<\infty$. 
On the other hand, the continuous inclusion $\M \subset \BMO$ yields the reverse inclusion  
$$L_p(\M)=[\M,L_2(\M)]_{2/p}\subset [\BMO,L_2(\M)]_{2/p},$$
and finishes the proof. 
\qd

\section{The $\h_p^c$-spaces}\label{secthp}

In this section we consider the conditioned version of Hardy spaces, and study their continuous analogue. 
Following the case of the $\H_p^c$-spaces studied in Section \ref{sectHp}, 
we define two conditioned column Hardy spaces $\hh_p^c$ and $\h_p^c$ in the continuous setting. 
In this case we still have a crucial monotonicity property, with a reversed monotonicity, 
which will also imply that the two conditioned candidates $\hh_p^c$ and $\h_p^c$ coincide. 
However, the conditioned case is more complicated than the $\H_p^c$-case in the sense that we can not prove the injectivity results directly, 
as we did in Section \ref{sectHp}. 
In particular, the fact that $H_p^c(\s)=L_p(\M)$ with equivalent norms for a finite partition $\s$ is no longer true in the conditioned case. 
This is why we will first need to complement the space $\h_p^c$ into some larger space, 
which also have an $L_p$-module structure over a finite von Neumann algebra. 
The construction will be based on free amalgamated products. 
Then we will deduce duality, injectivity and interpolation results for $1<p<\infty$. 
We also establish the continuous analogue of the Fefferman-Stein duality for $\h_p^c$, 
where the description of the $L_{p}^c\mo$ spaces will be easier than the one of the $L_p^c\MO$ spaces in subsection \ref{sectFSc}. 
The end of this section is devoted to the expected interpolation result involving the column spaces $\h_1^c$ and $\bmo^c$.  

\subsection{The discrete case}\label{discretehpc}

As in Section \ref{sectHp}, we start by recalling the definitions of the conditioned Hardy spaces of noncommutative martingales in the discrete case 
and some well-known results. 
Let $(\M_n)_{n\geq 0}$ be a discrete filtration. 
Following \cite{JX}, we introduce the column and row conditioned square functions 
relative to a (finite) martingale $x = (x_n)_{n\geq 0}$ in $L_\infty(\M)$:
$$
s_c (x) = \Big ( \sum^{\infty}_{n = 0} \E_{n-1}|d_n(x) |^2 \Big )^{1/2}\quad \mbox{and} \quad 
s_r (x) = \Big ( \sum^{\infty}_{n = 0} \E_{n-1}| d_n(x)^* |^2 \Big)^{1/2},$$
where by convention we set $\E_{-1}=\E_0$. 
For $1 \leq p < \infty$ we define $h_p^c $ (resp. $h_p^r $) as the completion of all
finite $L_\infty$-martingales under the norm $\| x \|_{h_p^c}=\| s_c (x) \|_p$
(resp. $\| x \|_{h_p^r}=\| s_r (x) \|_p $). 
Let us also introduce the diagonal space $h_p^d$, defined as the subspace of $\ell_p(L_p(\M))$ consisting of all martingale difference sequences. 
Recall that $\ell_p(L_p(\M))$ is the space of all sequences $a=(a_n)_{n\geq 0}$ in $L_p(\M)$ such that
$$\|a\|_{\ell_p(L_p(\M))}=\Big(\sum_{n= 0}^\infty\|a_n\|_p^p\Big)^{1/p} <\infty ,$$
with the usual modification for $p=\infty$. 
The conditioned Hardy space of noncommutative martingales is defined by 
$$h_p=\left\{\begin{array}{cl}
h_p^d+h_p^c+h_p^r& \quad \mbox{for} \quad 1\leq p< 2 \\
h_p^d\cap h_p^c\cap h_p^r& \quad \mbox{for}\quad 2\leq p<\infty
\end{array}\right..$$

It was proved in \cite{JD} that for each $n$ and $0<p\leq \infty$, there exists an isometric right $\M_n$-module map 
$u_{n,p}:L_p^c(\M ;\E_n) \to L_p(\M_n; \ell_2^c)$ with complemented range such that 
\begin{equation}\label{u}
u_{n,p}(x)^*u_{n,q}(y)=\E_{n}(x^*y),
\end{equation}
for all $x\in L_p^c(\M;\E_n)$ and $y\in L_q^c(\M;\E_n)$. 
More precisely, for $0<p<\infty$ there exists a contractive projection $\Q_{n,p}$ defined from $L_p(\M_n; \ell_2^c)$ onto the image of $u_{n,p}$ 
such that for all $\xi \in L_p(\M_n; \ell_2^c)$
\begin{equation}\label{Q}
\Q_{n,p}(\xi)^*\Q_{n,p}(\xi)\leq \xi^*\xi.
\end{equation}
For $1<p<\infty$ we know that 
\begin{equation}\label{Q2}
\Q_{n,p}^*=\Q_{n,p'}.
\end{equation} 
In the sequel for the sake of simplicity we will drop the subscript $p$ in $u_{n,p}$ and $Q_{n,p}$. 
This proves that $h_p^c$ isometrically embeds into $L_p(\M; \ell_2^c(\nz^2))$ via the map 
$$u:\left\{\begin{array}{ccc}
h_p^c & \longrightarrow & L_p(\M; \ell_2^c(\nz^2)) \\
 x &\longmapsto & \displaystyle\sum_{n\geq 0} e_{n,0} \ten u_{n-1}(d_n(x))
\end{array}\right..$$
Furthermore, $h_p^c$ is a complemented subspace of $L_p(\M; \ell_2^c(\nz^2))$ for $1<p<\infty$. 
Indeed, we can define a projection 
$$P: L_p(\M; \ell_2^c(\nz^2)) \to h_p^c$$ as follows. 
For $\xi=\sum_n e_{n,0}\ten \xi_n \in  L_p(\M; \ell_2^c(\nz^2))$, for all $n\geq 0$ we have 
$\E_{n-1}(\xi_n) \in  L_p(\M_{n-1}; \ell_2^c(\nz))$. 
We may apply the projection $\Q_{n-1}$ and obtain for each $n$ an element $y_n \in L_p^c(\M;\E_{n-1})$ satisfying
\begin{equation}\label{yk}
\Q_{n-1}(\E_{n-1}(\xi_n))=u_{n-1}(y_n).
\end{equation}
Then we set 
$$P(\xi)=\sum_{n\geq 0} d_n(y_n).$$
It is clear that $P \circ u = id_{h_p^c}$, i.e., that $P$ is a projection from $L_p(\M; \ell_2^c(\nz^2))$ onto $h_p^c$. 
Moreover, we can show that this projection is bounded for $1<p<\infty$.  

\begin{lemma}\label{complhpdiscr}
Let $1<p<\infty$. Then the discrete space $h_p^c$ is $\gamma_p$-complemented in $ L_p(\M; \ell_2^c(\nz^2))$. 
\end{lemma}

\begin{proof}
Let $\xi=\sum_n e_{n,0}\ten \xi_n \in  L_p(\M; \ell_2^c(\nz^2))$. 
First observe that for all $n\geq 0$ we have
\begin{equation}\label{eq1}
\E_{n-1}|d_n(y_n)|^2\leq \E_{n-1}|y_n|^2.
\end{equation}
Indeed, for $n=0$, since by convention $\E_{-1}=\E_0$ and $d_0(y_0)=\E_0(y_0)$, we have
$$\E_{0}|d_0(y_0)|^2=|\E_0(y_0)|^2\leq \E_0|y_0|^2.$$ 
For $n\geq 1$, we can write
\begin{align*}
\E_{n-1}|d_n(y_n)|^2&
=\E_{n-1}(|\E_n(y_n)|^2-|\E_{n-1}(y_n)|^2)\\
&\leq \E_{n-1}(|\E_n(y_n)|^2)
\leq \E_{n-1}(\E_n|y_n|^2)=\E_{n-1}|y_n|^2.
\end{align*}
Moreover by \eqref{yk} and \eqref{Q}, we have for all $n\geq 0$
\begin{equation}\label{eq2}
\E_{n-1}|y_n|^2= |u_{n-1}(y_n)|^2 =  |\Q_{n-1}(\E_{n-1}( \xi_n))|^2
\leq |\E_{n-1}(\xi_n)|^2 .
\end{equation}
Combining \eqref{eq1} with \eqref{eq2} we obtain
\begin{equation}\label{eq3}
\E_{n-1}|d_n(P(\xi))|^2=\E_{n-1}|d_n(y_n)|^2 \leq |\E_{n-1}(\xi_n)|^2, \quad \forall  n\geq 0.
\end{equation}
The noncommutative Stein inequality implies 
\begin{align*}
\|P(\xi)\|_{h_p^c}
&=\Big\| \Big(\sum_{n\geq 0} \E_{n-1}|d_n(y_n)|^2\Big)^{1/2}\|_p 
\leq  \Big\| \Big(\sum_{n\geq 0} |\E_{n-1}(\xi_n)|^2\Big)^{1/2}\Big\|_p \\ 
&\leq \gamma_p  \Big\| \Big(\sum_{n\geq 0} |\xi_n|^2\Big)^{1/2}\Big\|_p =\gamma_p \|\xi\|_{L_p(\M; \ell_2^c(\nz^2))}.
\end{align*}
\qd

We deduce the following duality and interpolation results.  

\begin{cor}\label{discrdualhpc}
Let $1<p<\infty$. Then the discrete spaces satisfy
\begin{enumerate}
\item[(i)] Let $\frac{1}{p}+\frac{1}{p'}=1$. Then $$(h_p^c)^*=h_{p'}^c \quad \mbox{with equivalent norms.}$$
\item[(ii)]  Let $1<p_1,p_2<\infty$ and $0<\theta<1$ be such that
$\frac{1}{p}=\frac{1-\theta}{p_1}+\frac{\theta}{p_2}$. Then
 $$h_p^c  = [h_{p_1}^c,h_{p_2}^c]_{\theta}  \quad \quad \mbox{with equivalent norms}  .$$
\end{enumerate}
\end{cor}

\re\label{adjointPu}
Observe that for $1<p\leq \infty$ we have $P=u^*$. 
Indeed, for $x\in h_p^c$ and $\xi \in L_{p'}(\M; \ell_2^c(\nz^2))$ we may write
\begin{align*}
(P(\xi)|x)&=\sum_n\tau(d_n(y_n)^*d_n(x)) 
=\sum_n\tau(y_n^*d_n(x)) \\
&=\sum_n\tau(\E_{n-1}(y_n^*d_n(x))) \\
&=\sum_n\tau(u_{n-1}(y_n)^*u_{n-1}(d_n(x))) \quad  \quad \mbox{by \eqref{u}}\\  
&=\sum_n\tau(\Q_{n-1}(\E_{n-1}(\xi_n))^*u_{n-1}(d_n(x))) \quad \quad \mbox{by \eqref{yk}}\\
&=\sum_n\tau(\E_{n-1}(\xi_n)^*\Q_{n-1}(u_{n-1}(d_n(x))))  \quad \quad \mbox{by \eqref{Q2}}\\  
&=\sum_n\tau(\E_{n-1}(\xi_n)^*u_{n-1}(d_n(x)))
= \sum_n\tau(\xi_n^*u_{n-1}(d_n(x)))\\
&=(\xi|u(x)).
\end{align*}
\mar

The analogue of the Fefferman-Stein duality for the conditioned case was established independently in \cite{jm-riesz} and \cite{bcpy}. 
For $2<p\leq \infty$ we introduce
$$L_p^cmo=\{x\in L_2(\M): \|x\|_{L_p^cmo} <\infty\},$$
where 
$$\|x\|_{L_p^cmo}=\max\big(\|\E_0(x)\|_p ,  \|{\sup_n}^+ \E_n|x-x_{n}|^2\|_{p/2}^{1/2}\big).$$
For $p=\infty$ we denote this space by $bmo^c$. 

\begin{theorem}\label{fsdualityhpcdiscr}
Let $1\leq p <2$. Then the discrete spaces satisfy
$$(h_p^c)^*=L_{p'}^cmo \quad \mbox{with equivalent norms. }$$
Moreover, 
$$\nu_p\|x\|_{L_{p'}^cmo }\leq \|x\|_{(h_p^c)^*}\leq \sqrt{2} \|x\|_{L_{p'}^cmo },$$
where $\nu_p$ remains bounded as $p\to 1$.
\end{theorem}

Combining these two latter results we obtain

\begin{prop}\label{hpcLpcmo}
Let $2<p<\infty$. Then the discrete spaces satisfy
$$h_p^c=L_{p}^cmo \quad \mbox{with equivalent norms. }$$
\end{prop}

Observe that we can extend Lemma \ref{complhpdiscr} to the case $p=\infty$ in the following sense.

\begin{lemma}\label{Pbddbmo}
Let $2<p\leq \infty$. 
Then $P : L_p(\M; \ell_2^c(\nz^2)) \to L_p^c mo$ is bounded.
\end{lemma}

\begin{proof}
Let $\xi=\sum_n e_{n,0}\ten \xi_n \in  L_p(\M; \ell_2^c(\nz^2))$ and $x=P(\xi)$. 
On the one hand, by \eqref{eq3} for $n=0$ we have
$$\|\E_0(x)\|_p\leq \|\E_0(\xi_0)\|_p \leq \|\xi_0\|_p=\|(|\xi_0|^2)^{1/2}\|_p 
\leq \Big\|\Big(\displaystyle\sum_{n\geq 0} |\xi_n|^2\Big)^{1/2}\Big\|_p=\|\xi\|_{L_p(\M; \ell_2^c(\nz^2))}.$$
On the other hand, note that by \eqref{eq3}, for each $n\geq 0$ we have
\begin{equation}\label{eq4}
\begin{array}{cl}
\E_n|x-x_n|^2&=\E_n\Big(\displaystyle\sum_{k>n} \E_{k-1}|d_k(x)|^2\Big)
\leq \E_n\Big(\displaystyle\sum_{k>n} |\E_{k-1}(\xi_k)|^2\Big)\\
&\leq \E_n\Big(\displaystyle\sum_{k>n} \E_{k-1}|\xi_k|^2\Big) =  \E_n\Big(\displaystyle\sum_{k>n} |\xi_k|^2\Big) \\
&\leq   \E_n\Big(\displaystyle\sum_{k\geq 0} |\xi_k|^2\Big).
\end{array}
\end{equation}
Since $1<\frac{p}{2}\leq \infty$, the noncommutative Doob inequality gives
$$ \|{\sup_n}^+\E_n|x-x_n|^2\|_{p/2}
\leq \Big\|{\sup_n}^+\E_n\Big(\displaystyle\sum_{k\geq 0} |\xi_k|^2\Big)\Big\|_{p/2}
\leq \delta_{p/2}  \Big\|\displaystyle\sum_{k\geq 0} |\xi_k|^2\Big\|_{p/2}
=\delta_{p/2} \|\xi\|_{L_p(\M; \ell_2^c(\nz^2))}^2.$$
Thus we get 
$$\|P(\xi)\|_{L_p^cmo}=\|x\|_{L_p^cmo}=\max\big(\|\E_0(x)\|_p ,  \|{\sup_n}^+ \E_n|x-x_{n}|^2\|_{p/2}^{1/2}\big)\leq \max(1,\delta_{p/2}^{1/2})\|\xi\|_{L_p(\M; \ell_2^c(\nz^2))}.$$
\qd

The noncommutative Burkholder-Rosenthal inequalities were obtained by the first named author and Xu in \cite{JX}.

\begin{theorem}\label{Bdiscr}
Let $1<p<\infty$. Then the discrete spaces satisfy
$$L_p(\M)=h_p \quad \mbox{with equivalent norms. }$$
Moreover, 
$$\kappa_p^{-1}\|x\|_{h_p}\leq \|x\|_p\leq \eta_p \|x\|_{h_p}.$$
\end{theorem}

\re
It is important to note that $\eta_p$ remains bounded as $p\to 1$, i.e., for $p=1$ we have a bounded inclusion $h_1\subset L_1(\M)$. 
\mar

We end this subsection with the conditioned analogue of Theorem \ref{intH1cBMOc-discr} proved in \cite{bcpy}. 

\begin{theorem}\label{inth1cbmoc-discr}
Let $1<p<\infty$. Then the discrete spaces satisfy
$$h_p^c=[bmo^c,h^c_1]_{\frac{1}{p}}\quad \mbox{with equivalent norms}.$$
\end{theorem}

\subsection{Definitions of $\hh_p^c$, $\h_p^c$ and basic properties}

Following Section \ref{sectHp}, we start by fixing an ultrafilter $\U$. 
For $\si \in \PP_{\mathrm{fin}}([0,1])$ and $x\in \M$, we define the finite conditioned bracket
$$\langle x,x \rangle_\si=\sum_{t\in \si} \E_{t^-(\s)}|d_t^{\si}(x)|^2$$
(recalling our convention that $\E_{0^-(\s)}=\E_0$). 
Observe that $ \|\langle x,x \rangle_\si\|_{p/2}^{1/2} = \|x\|_{h_p^c(\si)} $, 
where $h_p^c(\si)$ denotes the noncommutative conditioned Hardy space with respect to the discrete filtration $(\M_t)_{t\in \s}$. 
Hence the noncommutative Burkholder-Rosenthal inequalities recalled in Theorem \ref{Bdiscr} and the H\"{o}lder inequality imply for each finite partition $\s$ and $x\in \M$
\begin{equation}\label{estimateh_pc} 
\begin{array}{cccccl}
\eta_p^{-1}\|x\|_p &\leq & \|\langle x,x \rangle_\si\|_{p/2}^{1/2}&\leq &\|x\|_2 & \quad \mbox{ for } 1\leq p < 2 \\
\|x\|_2&\leq& \|\langle x,x \rangle_\si\|_{p/2}^{1/2} &\leq& \kappa_p\|x\|_p & \quad  \mbox{ for } 2\leq p <\infty 
\end{array}.
\end{equation}
Then, adapting the discussion detailed in subsection \ref{subsectdefHpc}, for $x\in \M$ and $1\leq p <\infty$ we may define
$$\langle x,x \rangle_\U = \E_\U( (\langle x,x \rangle_\s)^\bullet)  \quad , \quad 
\|x\|_{\hat{\h}_p^{c}} = \|\langle x,x \rangle_\U\|_{p/2}^{1/2}
 \quad \mbox{and} \quad 
\|x\|_{\h_p^{c}}= \lim_{\si,\U} \|x\|_{h_p^c(\si)}.$$
The properties of the conditional expectation $\E_\U$ imply the analogue of \eqref{estimatehatH_pc} 
\begin{equation}\label{estimatehath_pc} 
\begin{array}{cccccccl}
\eta_p^{-1}\|x\|_p &\leq & \|x\|_{\h_p^{c}} &\leq &\|x\|_{\hat{\h}_p^{c}} &\leq &\|x\|_2 & \quad  \mbox{ for } 1\leq p < 2 \\
\|x\|_2 &\leq &\|x\|_{{\hh}_p^{c}} &\leq & \|x\|_{\h_p^{c}}  &\leq &\kappa_p\|x\|_p & \quad  \mbox{ for } 2\leq p <\infty 
\end{array}.
\end{equation}
Hence $\|\cdot\|_{\hh_p^{c}}$ and $\|\cdot\|_{\h_p^{c}}$ define two (quasi)norms on $\M$.  
As for $\hH_p^c$ and $\H_p^c$, these (quasi)norms a priori depend on the choice of the ultrafilter $\U$. 
We will show that they actually do not, up to equivalent norm, and simply denote $\|\cdot\|_{\hat{\h}_p^c}$ and $\|\cdot\|_{\h_p^c}$. 

\begin{defi}
 Let $1\leq p <\infty$. 
We define the spaces $\hh_p^{c}$ and $\h_p^{c}$ as the completion of $\M$ with respect to the (quasi)norms 
$\|\cdot\|_{\hat{\h}_p^{c}}$ and $\|\cdot\|_{\h_p^{c}}$ respectively. 
\end{defi}

As we did for $\hat{\H}_p^c$, we may equip $\hh_p^c$ with an $L_p$ $\M$-module structure 
and show that $\|\cdot\|_{\hat{\h}_p^{c}}$ is a norm for $1\leq p <\infty$. 

\re
In this case we also note that $L_{\max(p,2)}(\M)$ is dense in $\h_p^c$ and $\hh_p^c$ for $1\leq p < \infty$. 
\mar

The conditioned version of Lemma \ref{reflexivityHpc} holds true.  

\begin{lemma}
Let $1<p<\infty$. Then $\h_p^c$ is reflexive.
\end{lemma} 

\subsection{Monotonicity and convexity properties}

In the conditioned case we still have some monotonicity properties of the discrete norms, but the monotonicity is reversed. 

\begin{lemma}\label{convexityhpc} 
Let $1\leq p <\infty$ and $\s\in\mathcal{P}_{\fin}([0,1])$. 
 \begin{enumerate}
  \item[(i)] Let $1\leq p<2$. 
Let $\s^1, \cdots, \s^M$ be partitions contained in $\s$, 
let $(\alpha_m)_{1\leq m\leq M}$ be a sequence of positive numbers such that $\sum_m\alpha_m=1$, and 
let $x^1, \cdots, x^M\in L_2(\M)$. 
Then for $x=\sum_m \alpha_m x^m$ we have 
$$\|x\|_{h_p^c(\s)}\leq 2^{1/p} \Big\|\sum_{m=1}^M \alpha_m \langle x^m,x^m\rangle_{\s^m}\Big\|_{p/2}^{1/2}.$$
In particular for $x\in L_2(\M)$ and $\s \subset \s'$ we have 
$$\|x\|_{h_p^c(\s')}\leq 2^{1/p} \|x\|_{h_p^c(\s)}.$$
Hence 
$$ \inf_\s \|x\|_{h_p^c(\s)} \leq \|x\|_{\h_p^{c}} \leq 2^{1/p} \inf_\s \|x\|_{h_p^c(\s)}.$$
 \item[(ii)] Let $2\leq p<\infty$. 
Let $\s^1, \cdots, \s^M$ be partitions containing $\s$, 
let $(\alpha_m)_{1\leq m\leq M}$ be a sequence of positive numbers such that $\sum_m\alpha_m=1$, and 
let $x^1, \cdots, x^M\in L_p(\M)$.
Then for $x=\sum_m \alpha_m x^m$ we have 
$$\|x\|_{h_p^c(\s)}\leq {\delta'}_{p/2}^{1/2} \Big\|\sum_{m=1}^M \alpha_m \langle x^m,x^m\rangle_{\s^m}\Big\|_{p/2}^{1/2}.$$
In particular for $x\in L_p(\M)$ and $\s \subset \s'$ we have 
$$\|x\|_{h_p^c(\s)}\leq {\delta'}_{p/2}^{1/2} \|x\|_{h_p^c(\s')}.$$
Hence 
$${\delta'}_{p/2}^{-1/2} \sup_\s \|x\|_{h_p^c(\s)} \leq \|x\|_{\h_p^{c}} \leq \sup_\s \|x\|_{h_p^c(\s)}.$$
 \end{enumerate}
\end{lemma}

\begin{proof}
We first consider $1\leq p <2$. 
On the one hand, the operator convexity of $|\cdot|^2$ yields
$$\|x\|_{h_p^c(\s)}^2 = 
\Big\|\displaystyle\sum_{s\in \s}\E_{s^-(\s)} \Big|\sum_m  \alpha_m d_{s}^\s(x^m)\Big|^2\Big\|_{p/2}
\leq \Big\|\displaystyle\sum_{ m, s\in \s} \alpha_m \E_{s^-(\s)}|d_{s}^\s(x^m)|^2\Big\|_{p/2}.$$
On the other hand, for $1\leq m\leq M$ and $t\in \s^m$ fixed we denote by $I_t$ the collection of $s\in \s$ such that $t^-(\s^m)\leq s^-(\s)<s\leq t$. 
Then for $m$ fixed, $\bigcup_{t\in \s^m} I_t = \s$. 
Note that for $1\leq m\leq M$ and $t\in \s^m$, 
we can split up the interval
$[t^-(\s^m),t]$ in the subintervals $[s^-(\s),s]$ with $s\in I_t$ and by the martingale property (and $t^-(\s^m)\leq s^-(\s)$) we have
\begin{equation}\label{eq:cond}
\E_{t^-(\s^m)}|d_{t}^{\s^m}(x^m)|^2
=\E_{t^-(\s^m)}\Big|\sum_{s\in I_t}d_s^\s(x^m)\Big|^2
=\E_{t^-(\s^m)}\Big(\sum_{s\in I_t}\E_{s^-(\s)}|d_{s}^\s(x^m)|^2\Big).
\end{equation}
Then \eqref{eq:cond} implies
\begin{align*}
\sum_{m} \alpha_m \langle x^m,x^m\rangle_{\s^m}&=
\sum_m \alpha_m \sum_{t\in \s^m}
\E_{t^-(\s^m)}\Big(\sum_{s\in I_t}\E_{s^-(\s)}|d_{s}^\s(x^m)|^2\Big)\\
&=\sum_{m, s\in \s }  \E_{t_m(s)^-(\s^m)}\big(\alpha_m \E_{s^-(\s)}|d_{s}^\s(x^m)|^2\big),
\end{align*}
where $t_m(s)$ denotes the unique $t\in \s^m$ which satisfies $t^-(\s^m)\leq s^-(\s)<s\leq t$. 
We can rearrange the set $\{1, \cdots, M\} \times \s$ 
so that $\Big(\M_{t_m(s)^-(\s^m)}\Big)_{(m,s)}$ becomes an increasing sequence of von Neumann algebras.
Thus we can apply the dual form of the reverse noncommutative Doob inequality for $0<\frac{p}{2}<1$ (Theorem $7.1$ of \cite{JX}), and obtain
\begin{align*}
\|x\|_{h_p^c(\s)}^2&\leq 
\Big\|\displaystyle\sum_{ m, s\in \s} \alpha_m \E_{s^-(\s)}|d_{s}^\s(x^m)|^2\Big\|_{p/2}\\
&\leq  2^{2/p}\Big\|\displaystyle\sum_{m, s\in \s}  \E_{t_m(s)^-(\s^m)}\big(\alpha_m \E_{s^-(\s)}|d_{s}^\s(x^m)|^2\big)\Big\|_{p/2}\\
&= 2^{2/p}\Big\|\displaystyle\sum_m \alpha_m \langle x^m,x^m\rangle_{\s^m}\Big\|_{p/2}.
\end{align*}
We now turn to assertion (ii). 
In this case, since $\s \subset \s^m$, for $t\in \si$ and $m$ fixed we denote by $I_t^m$
the collection of $s\in \si^m$ such that $t^-(\s)\leq s^-(\s^m)<s\leq t$. Then for $m$ fixed, $\bigcup_{t\in \s} I_t = \s^m$.
We observe that
 \begin{align*}
  \E_{t^-(\s)}|d_t^\s(x)|^2
  &= \sum_{m,l=1}^ M\al_m \al_l
  \E_{t^-(\s)}(d_t^\s(x^m)^*d_t^\s(x^l))  .
  \end{align*}
By Cauchy-Schwarz, we deduce  that
 \begin{align*}
\|x\|_{h_p^c(\s)}^2 &= \Big\|\sum_{t\in \s} \E_{t^-(\s)}(|d_t^\s(x)|^2)\Big\|_{p/2}\\
 &\leq  \Big\|\sum_{t\in \s,m,l} \al_{l}\al_m \E_{t^-(\s)}(|d_t^\s(x^m)|^2)\Big\|_{p/2}^{1/2}
    \Big\|\sum_{t\in \s,m,l} \al_{l}\al_m
 \E_{t^-(\s)}(|d_t^\s(x^l)|^2)\Big\|_{p/2}^{1/2} \\
  &= \Big\|\sum_{t\in \s,m} \al_m \E_{t^-(\s)}(|d_t^\s(x^m)|^2)\Big\|_{p/2}.
 \end{align*}
Note that in the first term the summation over $l$ disappears by using
$\sum_l \al_l=1$, and in the second one the summation over $m$
disappears similarly. For $t\in \s$ and $m$ as \eqref{eq:cond} we can write 
 $$ \E_{t^-(\s)}(|d_t^\s(x^m)|^2)
 = \sum_{s\in I_t^m} \E_{t^-(\s)}(|d_s^{\s^m}(x^m)|^2).$$
By the dual version of the noncommutative Doob inequality for $1\leq \frac{p}{2} <\infty$, we deduce
that
 \begin{align*}
\Big\|\sum_{t\in \s,m} \al_m \E_{t^-(\s)}(|d_t^\s(x^m)|^2)\Big\|_{p/2} 
&= \Big\|\sum_{t\in \s,m,s\in I_t^m } \al_m \E_{t^-(\s)}( |d_s^{\s^m}(x^m)|^2)\Big\|_{p/2} \\
&= \Big\|\sum_{t\in \s} \E_{t^-(\s)}\Big(\sum_{m,s\in I_t^m}\al_m \E_{s^-(\s^m)} (|d_s^{\s^m}(x^m)|^2)\Big)\Big\|_{p/2}\\
&\leq  \delta'_{p/2} \Big\|\sum_{t\in \s} \sum_{m,s\in I_t^m} \al_m \E_{s^-(\s^m)} (|d_s^{\s^m}(x^m)|^2) \Big\|_{p/2} \\
&= \Big\|\sum_{m=1}^M \alpha_m \langle x^m,x^m\rangle_{\s^m}\Big\|_{p/2}.
 \end{align*}
This ends the proof.
\qd

The independence (up to a constant) of $\h_p^{c}$ on $\U$ follows immediately. 

\begin{theorem}\label{indpdtUh}
For $1\leq p <\infty$ the space $\h_p^{c}$ is independent of the choice of the ultrafilter $\U$, up to equivalent norm.
\end{theorem}

\subsection{$\hh_p^c=\h_p^c$}

In this subsection we show that in the conditioned case the two spaces $\hh_p^{c}$ and $\h_p^c$ also coincide.
In particular we will deduce that, up to an equivalent constant, these two spaces do not depend on the choice of the ultrafilter $\U$. 

\begin{theorem}\label{hat2}
Let $1 \leq p <\infty$. Then 
$$\h_p^c=\hh_p^{c}  \quad \mbox{with equivalent norms.}$$
\end{theorem}

Theorem \ref{indpdtUh} immediately yields 

\begin{cor}\label{hatindpdtUh}
For $1\leq p <\infty$ the space $\hh_p^{c}$ is independent of the choice of the ultrafilter $\U$, 
up to equivalent norm. 
\end{cor}

\begin{proof}[Proof of Theorem \ref{hat2}]
As in proof of Theorem \ref{hat}, we start with the case $2\leq p <\infty$, 
which is an easy consequence of the convexity property proved in Lemma \ref{convexityhpc}. 
It suffices to show that the $\h_p^c$-norm and the $\hh_p^c$-norm are equivalent on $\M$. 
Let $x\in \M$, by \eqref{estimatehath_pc}  we have $\|x\|_{\hh_p^c}\leq \|x\|_{\h_p^c}$. 
Now assume that $\|x\|_{\hh_p^c}=\|\langle x,x\rangle_\U\|_{p/2}^{1/2}<1$ and fix $\s$. 
Since the two spaces coincide with $L_2(\M)$ for $p=2$, we consider $2<p<\infty$. 
In that case we have $\langle x,x\rangle_\U=\w L_{p/2} \mbox{-} \lim_{\s,\U}\langle x,x\rangle_\s$. 
We can find a sequence of positive numbers $(\alpha_m)_{m=1}^{M}$ satisfying $\sum_m \alpha_m=1$ and 
finite partitions $\s^1, \cdots, \s^{M}$ containing $\s$ such that 
$$\Big\|\sum_{m=1}^{M} \alpha_m  \langle x,x\rangle_{\s^m} \Big\|_{p/2}<1.$$
Lemma \ref{convexityhpc} (ii) gives for all $\s$
$$\|x\|_{h_p^c(\s)}\leq {\delta'}_{p/2}^{1/2}.$$
Taking the limit over $\s$ we get
$$\|x\|_{\h_p^c}\leq {\delta'}_{p/2}^{1/2}.$$
We now turn to the case $1\leq p<2$. 
We will use the same trick as in the proof of Lemma \ref{dualhatHpc}. 
Let us adapt this argument for $\hh_p^{c}$. 
We consider the same index set 
$$\I=\PP_{\fin}(\M) \times \PP_{\fin}([0,1]) \times \rz_+^*$$
and construct similarly the ultrafilter $\V$ on $\I$. As in subsection \ref{secthat}, for each $i=(F, \s_i,\varepsilon) \in \I$ 
we can find a sequence of positive numbers $(\alpha_m(i))_{m=1}^{M(i)}$ such that $\sum_m \alpha_m(i)=1$ and 
finite partitions $\s_i^1, \cdots, \s_i^{M(i)}$ containing $\s_i$ and satisfying for all $x\in F$
$$\Big\|\langle x,x\rangle_\U- \sum_{m=1}^{M(i)} \alpha_m(i)  \langle x,x\rangle_{\s_i^m} \Big\|_{q/2}<\varepsilon.$$
In this case we consider the Hilbert space $\H_i=\ell_2\Big(\bigcup_{m,t\in \s_i^m}\{t\} \times \nz \Big)$
equipped with the norm
$$\|(\xi_{m,t,j})_{1\leq m\leq M(i),t\in \s_i^m,j\in \nz}\|_{\H_i}
=\Big(\sum_{m=1}^{M(i)}\alpha_m(i)\sum_{t\in \s_i^m,j\in \nz}|\xi_{m,t,j}|^2\Big)^{1/2}.$$
Then $\hh_p^{c}$ embeds isometrically into $\prodd_\V  L_p(\M;\H_i^c)$ via the map $x\in \M \mapsto \x=(\x(i))^\bullet$, where 
$$\x(i)=\left\{\begin{array}{cc}
\displaystyle\sum_{m=1}^{M(i)} \displaystyle\sum_{t\in \s_i^m} e_{m,0}\ten e_{t,0} \ten u_{t^-(\s_i^m)}(d_t^{\s_i^m}(x))
& \mbox{ if } i=(F,\s_i,\varepsilon) \mbox{ such that } x\in F \\
0& \mbox{otherwise}
\end{array}\right..$$
We will show that 
\begin{equation}\label{hathpc}
(\hh_p^{c})^*\subset (\h_p^c)^*.
\end{equation}
Let $\varphi \in (\hh_p^{c})^*$ be a functional of norm less than one. 
We may assume that $\varphi$ is given by an element $\xi=(\xi(i))^\bullet\in \prod_\V L_{p'}(\M;\H_i^c)$ of norm less than one, 
with
$$\xi(i)=\sum_{m=1}^{M(i)} \sum_{t\in \s_i^m} e_{m,0}\otimes e_{t,0} \otimes \xi_{m,t}(i),$$
where $\xi_{m,t}(i)\in L_{p'}(\M;\ell_2^c(\nz))$.  
Fix $i=(F,\s_i, \varepsilon)\in \I$ and $1\leq m\leq M(i)$.  
We set 
$$z_m(i)=P_{\s_i^m}(\xi_m(i)) \in L_{p'}(\M),$$
where  
$ \xi_m(i):=\sum_{t\in \s_i^m} e_{m,0}\otimes e_{t,0} \otimes \xi_{m,t}(i) \in L_{p'}(\M; \ell_2^c(\s_i^m\times\nz))$ 
and $P_{\s_i^m}$ denotes the projection from $L_{p'}(\M; \ell_2^c(\s_i^m\times\nz))$ onto $h_{p'}^c(\s_i^n)$ described in subsection \ref{discretehpc}. 
Then we consider
$$z(i)=\sum_m \alpha_m(i) z_m(i) \in L_{p'}(\M).$$
We claim that $z(i)$ is a martingale in $L_{p'}^c mo(\s_i)$. 
The crucial point here is that by Lemma \ref{Pbddbmo} the map 
$P_{\s_i^m}:L_{p'}(\M; \ell_2^c(\s_i^m\times\nz)) \to L_{p'}^c mo(\s_i^m)$ is bounded for $2<p'\leq \infty$. 
More precisely, on the one hand, \eqref{eq3} for $n=0$ implies
\begin{equation}\label{eq5.0}
|\E_0(z_m(i))|^2\leq |\E_0(\xi_{m,0}(i))|^2\leq \E_0|\xi_{m,0}(i)|^2.
\end{equation}
On the other hand, by \eqref{eq4} we have for all $s\in \s_i^m$ (and in particular for all $s\in \s_i \subset \s_i^m$)
\begin{equation}\label{eq5}
 \E_s|z_m(i)-\E_s(z_{m}(i))|^2
\leq \E_s\Big(\displaystyle\sum_{t\in \s^m}|\xi_{m,t}(i)|^2\Big).
\end{equation}
The operator convexity of the square function $|\cdot|^2$ yields 
$$|\E_0(z(i))|^2=\Big|\sum_m\alpha_m(i)\E_0(z_m(i))\Big|^2
\leq \sum_m\alpha_m(i)|\E_0(z_m(i))|^2,$$
and for each $s\in \s_i$ we get
\begin{equation}\label{eq6}
 \E_s|z(i)-\E_s(z(i))|^2=\E_s\Big|\sum_m\alpha_m(i)(z_m(i)-\E_s(z_{m}(i)))\Big|^2
  \leq\sum_m\alpha_m(i)\E_s|z_m(i)-\E_s(z_{m}(i))|^2. 
\end{equation}
Then using \eqref{eq5.0} we obtain 
$$|\E_0(z(i))|^2\leq \E_0\Big( \sum_m\alpha_m(i)|\xi_{m,0}(i)|^2\Big),$$
and the contractivity of the conditional expectation $\E_0$ on $L_{p'/2}$ implies
\begin{align*}
\|\E_0(z(i))\|_{p'}
&\leq \Big\|\E_0\Big( \sum_m\alpha_m(i)|\xi_{m,0}(i)|^2\Big)\Big\|_{p'/2}^{1/2}
\leq \Big\|\sum_m\alpha_m(i)|\xi_{m,0}(i)|^2\Big\|_{p'/2}^{1/2}\\
&\leq \Big\|\displaystyle\sum_{m,t\in \s_i^m}\alpha_m(i)|\xi_{m,t}(i)|^2\Big\|_{p'/2}^{1/2}
=\|\xi(i)\|_{L_{p'}(\M;\H_i^c)}.
 \end{align*}
Moreover \eqref{eq5} gives 
$$ \E_s|z(i)-\E_s(z(i))|^2 \leq \E_s\Big(\sum_{m,t \in \s_i^m}\alpha_m(i)|\xi_{m,t}(i)|^2\Big).$$
By the noncommutative Doob inequality we obtain
\begin{align*}
\|{\displaystyle\sup_{s\in \s_i}}^+\E_s|z(i)-\E_s(z(i))|^2\|_{p'/2}
&\leq \Big\|{\displaystyle\sup_{s\in \s_i}}^+ \E_s\Big(\displaystyle\sum_{m,t\in \s^m_i}\alpha_m(i)|\xi_{m,t}(i)|^2\Big)\Big\|_{p'/2}\\
&\leq \delta_{p'/2}\Big\|\displaystyle\sum_{m,t\in \s_i^m}\alpha_m(i)|\xi_{m,t}(i)|^2\Big\|_{p'/2}\\
&=\delta_{p'/2}\|\xi(i)\|_{L_{p'}(\M;\H_i^c)}^2.
  \end{align*}
Hence 
$$\|z(i)\|_{L_{p'}^c mo(\s_i)} \leq \max(1,\delta_{p'/2}^{1/2})\|\xi(i)\|_{L_{p'}(\M;\H_i^c)}.$$
In particular, we see that the family $(z(i))_i$ is uniformly bounded in $L_2(\M)$. 
We set 
$z=\w L_2\mbox{-}\lim_{i,\V} z(i)$. 
We claim that $z\in (\h_p^c)^*$ with 
\begin{equation}\label{z}
 \|z\|_{ (\h_p^c)^*} \leq \sqrt{2}\max(1,\delta_{p'/2}^{1/2})\|\xi\|_{\prodd_\V L_{p'}(\M;\H_i^c)}.
 \end{equation}
By the density of $L_2(\M)$ in $\h_p^c$ it suffices to estimate 
$|\tau(z^*x)|$ for all $x\in L_2(\M)$ with $\|x\|_{\h_p^c}\leq 1$. 
Note that 
\begin{equation}\label{eqx}
\|x\|_{\h_p^c}=\lim_{i,\V}\|x\|_{h_p^c(\s_i)}.
\end{equation}
Indeed, for all $\delta >0$ and $x\in L_2(\M)$, by definition of the $\h_p^c$-norm we have 
$$A_\delta=\{\s \in  \PP_{\fin}([0,1]) :  |\|x\|_{\h_p^c}-\|x\|_{h_p^c(\s)}|<\delta \} \in \U.$$
Hence the set $\PP_{\fin}(\M) \times A_\delta \times \rz_+^* \in \T\times \U \times \W \subset \V$, and since 
$$\PP_{\fin}(\M) \times A_\delta \times \rz_+^* \subset 
\{i \in \I  :  |\|x\|_{\h_p^c}-\|x\|_{h_p^c(\s_i)}|<\delta \}$$
we deduce that the set in the right hand side is also in $\V$ for all $\delta$, which proves  \eqref{eqx}. 
We conclude that for $x\in L_2(\M)$ with $\|x\|_{\h_p^c}\leq 1$ we have 
\begin{align*}
|\tau(z^*x)| &\leq \lim_{i,\V} |\tau(z(i)^*x)| 
\leq \sqrt{2} \lim_{i,\V} (\|z(i)\|_{L_{p'}^c mo(\s_i)}\|x\|_{h_p^c(\s_i)})\\
&=\sqrt{2} (\lim_{i,\V} \|z(i)\|_{L_{p'}^c mo(\s_i)})(\lim_{i,\V}\|x\|_{h_p^c(\s_i)})
\leq \sqrt{2} \max(1,\delta_{p'/2}^{1/2})\|\xi\|_{\prodd_\V L_{p'}(\M;\H_i^c)} \|x\|_{\h_p^c}\\
&\leq \sqrt{2} \max(1,\delta_{p'/2}^{1/2})\|\xi\|_{\prodd_\V L_{p'}(\M;\H_i^c)}.
\end{align*}
This proves \eqref{z}. 
Finally, it remains to check that for all $x\in L_q(\M)$, $z$ satisfies
\begin{equation}\label{eqduality2}
(\xi|\tilde{x})_{\prodd_\V L_{p'}(\M;\H_i^c),\prodd_\V L_{p}(\M;\H_i^c)}=\tau(z^*x).
\end{equation}
We first verify that for each $i=(F,\s_i,\varepsilon)\in \I$ such that $x\in F$ we have 
$$(\xi(i)|\tilde{x}(i))_{L_{p'}(\M;\H_i^c),L_{p}(\M;\H_i^c)}=\tau(z(i)^*x).$$
For all $1\leq m\leq M(i)$, Remark \ref{adjointPu} gives
$$\tau(z_m(i)^*x) = (P_{\s^m_i}(\xi_m(i))|x)
=(\xi_m(i)|u_{\s^m_i}(x))
=\sum_{t\in \s_i^m}\tau\big(\xi_{m,t}(i)^*u_{t^-(\s_i^m)}(d_{t}^{\s_i^m}(x))\big).$$
Then 
$$\tau(z(i)^*x)=\sum_{m=1}^{M(i)} \alpha_m(i) \tau(z_m(i)^*x)
=\sum_{m=1}^{M(i)}\sum_{t\in \s_i^m}\alpha_m(i)\tau\big(\xi_{m,t}(i)^*u_{t^-(\s_i^m)}(d_{t}^{\s_i^m}(x))\big)
=(\xi(i)|\tilde{x}(i)).$$
As in the proof of \eqref{eqduality}, this is sufficient to show \eqref{eqduality2}. 
The end of the proof of Theorem \ref{hat2} is similar to that of Theorem \ref{hat}. 
\qd

In the sequel, we will work with the space $\h_p^c$.

\subsection{Complementation results}\label{sectcomplhpc}

The aim of this subsection is to complement the spaces $\h_p^c$ for $1<p<\infty$ in some nice spaces, 
that means in some spaces which have an $L_p$-module structure over a finite von Neumann algebra.  
We would like to deduce the continuous analogue of Corollary \ref{discrdualhpc}. 
However, in the conditioned case, we can not extend the complementation result stated in Lemma \ref{complhpdiscr} to the continuous setting, 
as we did for the spaces $\H_p^c$. 
Hence we first need to complement $h_p^c$ into another nice space in the discrete case, 
and then we will extend this complementation result to the continuous setting. 
This construction is based on free amalgamated products and will use the Rosenthal/Voiculescu type inequality recalled in subsection \ref{subsectfreeRos}.   

\subsubsection{Complementation of $h_p^c$ in the discrete case}

Let $(\M_n)_{n=0}^N$ be a finite discrete filtration and $(\E_n)_{n=0}^N$ be the associated conditional expectations. 
The idea is to construct a larger finite von Neumann algebra $\N\supset \M$ 
and then complement $h_p^c$ in the space $L_p^c(\N;\E_\M)$. 
We set
$$\AA_0=\M, \; \AA_n=\M \ast_{\M_{n-1}} \M_n \mbox{ for } 1\leq n \leq N 
\quad  \mbox{and} \quad \N=\ast_\M \AA_n,$$
where we amalgamate over the first copy of $\M$ in $\AA_n$. 
Following the notations introduced in subsection \ref{subsectfreeRos} we consider the $*$-homomorphisms
$$\rho:\M \to \N \quad \mbox{and} \quad \rho_n:\AA_n\to \N,$$
which send respectively $\M$ to the amalgamated copy and $\AA_n$ to the $n$-th copy. 
We denote by
$$\E_\M:\N\to \M  \quad \mbox{and} \quad \E_{\AA_n}:\N\to \AA_n$$
the associated normal faithful conditional expectations. 
For each $0\leq n\leq N$ we consider the $*$-homomorphism $\pi_{n,2}:\M_n \to \AA_n$ which sends $\M_n$ to the second copy of $\AA_n$, and 
$\phi_{n,2}:\AA_n \to \M_n$ the associated conditional expectation. 
If $n=0$ then $\pi_{0,2}$ is the natural inclusion $\M_0 \subset \M$ and $\phi_{0,2}$ is simply the conditional expectation $\E_0$. 
As in subsection \ref{subsectfreeRos}, we consider the spaces $\Sigma_1, X_p^1, Y_{p,c}^1, Y_{p,r}^1$ and $Z_p$ associated to the free amalgamated product $\N$. 
We will use the following easy fact.

\begin{lemma}\label{tangentdilation}
For all $0\leq n\leq N$ we have
$$ (\E_{n-1})_{|\M_n}=\E_\M\circ \rho_n \circ \pi_{n,2},$$
where by convention we set $\E_{-1}=\E_0$. 
\end{lemma}

\begin{proof}
The equality is obvious for $n=0$. 
For $1\leq n \leq N$ and $x\in \M_n$ we write $x=\E_{n-1}(x)+(x-\E_{n-1}(x))$. 
Observe that $x-\E_{n-1}(x) \in \stackrel{\circ}{\M}_n$ in $\AA_n$, and hence by freeness
$$\E_\M\circ \rho_n \circ \pi_{n,2}(x-\E_{n-1}(x))=0.$$
We get
$$\E_\M\circ \rho_n \circ \pi_{n,2}(x)= \E_\M\circ \rho_n \circ \pi_{n,2}(\E_{n-1}(x))=\E_{n-1}(x).$$
\qd

\re
This shows that the construction detailed above gives a tangent dilation for $\M$ associated to the filtration $(\M_n)_{n=0}^N$. 
Actually this also holds in the case of any (non necessarily finite) discrete filtration. 
Let us recall the notion of a tangent dilation, which was introduced in \cite{jm-riesz}. 
For a von Neumann algebra $\M$ and a filtration $(\M_n)_{n\geq 0}$, 
a tangent dilation is given by a von Neumann algebra $\N$ and trace-preserving homomorphisms
$\pi_n:\M_n\to \N$, $\rho: \M \to \N$ such that
 \begin{enumerate}
 \item[(i)] The conditional expectation $\E_{\rho}:\N\to \rho(\M)$ satisfies
 $$\rho\circ \E_{n-1}= \E_{\rho}\circ \pi_n \quad \mbox{for all } n\geq 0;$$
 \item[(ii)] The von Neumann algebras $\N_n=\pi_n(\M_n)$ are successively independent over $\rho(\M)$.
 \end{enumerate}
The first named author and Mei constructed a tangent dilation for any group von Neumann algebras. 
More generally the construction described previously gives a tangent dilation for every von Neumann algebra and every filtration.  
Indeed by setting $\pi_n=\rho_n \circ \pi_{n,2}$, we get two trace-preserving homorphisms satisfying (i) by Lemma \ref{tangentdilation}. 
Condition (ii) is also verified by construction, and thus $\N$, $\pi_n$ and $\rho$ give a tangent dilation of $\M$. 
\mar

\begin{lemma}\label{visometry}
For $x\in \M$ we set
$$v(x)=\sum_{n=0}^N \rho_n \circ \pi_{n,2}(d_n(x)).$$
Let $1\leq p <\infty$. 
Then $v$ extends to an isometric embedding 
\begin{enumerate}
\item[(i)] from $h_p^c$ into $Y_{p,c}^1$;
\item[(ii)] from $h_p^d$ into $Z_p$.
\end{enumerate}
We will denote these isometries by $v_p^c$ and $v_p^d$ respectively.
\end{lemma}

\begin{proof}
Observe that $d_n(x) \in \stackrel{\circ}{\M}_n$ in $\AA_n$, hence $\rho_n \circ \pi_{n,2}(d_n(x)) \in \Acirc_n$. 
This means that if $x\in \M$ then $v(x) \in \Sigma_1$. 
By orthogonality and Lemma \ref{tangentdilation} we have
\begin{align*}
\E_\M(v(x)^*v(x))
=\sum_{n=0}^N \E_\M (|\rho_n \circ \pi_{n,2}(d_n(x))|^2)
=\sum_{n=0}^N \E_\M \circ \rho_n \circ \pi_{n,2} |d_n(x)|^2
=\sum_{n=0}^N \E_{n-1}|d_n(x)|^2.
\end{align*}
This means that for $x\in \M$ and $1\leq p <\infty$
$$\|v(x)\|_{L_p^c(\N;\E_\M)}=\|x\|_{h_p^c},$$
and (i) is proved. 
For the second assertion we write
$$\|v(x)\|_{Z_p}=\Big(\sum_{n=0}^N \|\rho_n \circ \pi_{n,2}(d_n(x))\|_p^p\Big)^{1/p}
=\Big(\sum_{n=0}^N \|d_n(x)\|_p^p\Big)^{1/p}=\|x\|_{h_p^d}.$$
\qd

Considering the adjoint we get the following complementation results. 

\begin{prop}\label{Rbdd}
For $y=\displaystyle\sum_{n=0}^N a_n\in \Sigma_1$ (i.e., $a_n \in \Acirc_n$ for all $0\leq n\leq N$) we set
$$\Rcal(y)=\sum_{n=0}^N \phi_{n,2}(a_n)-\E_{n-1}(\phi_{n,2}(a_n)).$$
Let $1<p<\infty$. 
Then $\Rcal$ extends to a bounded projection
\begin{enumerate}
\item[(i)] from $Y_{p,c}^1$ onto $h_p^c$;
\item[(ii)] from $Z_p$ onto $h_p^d$.
\end{enumerate}
We denote these projections by $\Rcal_p^c$ and $\Rcal_p^d$ respectively.
\end{prop}

\begin{proof}
We claim that for $x\in \M$ and $y=\sum_{n=0}^N a_n\in \Sigma_1$ we have
\begin{equation}\label{vRadjoint}
(v(x)|y) =(x|\Rcal(y)).
\end{equation}
Since $\pi_{n,2}\circ \phi_{n,2}:\AA_n\to \pi_{n,2}(\AA_n)$ is a conditional expectation on $(\AA_n,\tr\circ \rho_n)$, it is trace-preserving and we may write
$$\tr\circ \rho_n\circ \pi_{n,2}\circ \phi_{n,2}=\tr\circ \rho_n.$$
Thus
\begin{align*}
\Big(v(x)|\sum_{n=0}^N a_n\Big)
&=\sum_{n=0}^N\tr(\rho_n \circ \pi_{n,2}(d_n(x))\rho_n(a_n)^*)
=\sum_{n=0}^N\tr\circ \rho_n (\pi_{n,2}(d_n(x))a_n^*)\\
&=\sum_{n=0}^N\tr \circ \rho_n \circ \pi_{n,2}\circ \phi_{n,2} (\pi_{n,2}(d_n(x))a_n^*)\\
&=\sum_{n=0}^N\tau\circ \E_\M \circ \rho_n  \circ \pi_{n,2} (d_n(x) \phi_{n,2}(a_n)^*)\\
&=\sum_{n=0}^N\tau \circ \E_{n-1}(d_n(x) \phi_{n,2}(a_n)^*),
\end{align*}
where the last equality comes from Lemma \ref{tangentdilation}. 
Since $\E_{n-1}(\pi_{n,2}(d_n(x)))=0$ and $\E_{n-1}$ is trace-preserving, we obtain
$$\Big(v(x)|\sum_{n=0}^N a_n\Big)
=\sum_{n=0}^N\tau (d_n(x)( \phi_{n,2}(a_n)-\E_{n-1}( \phi_{n,2}(a_n)))^*)
=\Big(x|\Rcal(\sum_{n=0}^N a_n)\Big),$$
and \eqref{vRadjoint} is proved. 
Recall that for $1<p<\infty$ we have $(h_{p'}^c)^*=h_{p}^c$, $(h_{p'}^d)^*=h_{p}^d$, $(Y_{p',c}^1)^*=Y_{p,c}^1$ and $(Z_{p'})^*=Z_{p}$. 
Since $\M$ is dense in $h_p^c, h_p^d$ and $\Sigma_1$ is dense in $Y_{p,c}^1, Z_p$, 
we deduce from Lemma \ref{visometry} that 
$$(v_{p'}^c)^*=\Rcal_{p}^c:Y_{p,c}^1\to h_p^c \quad \mbox{and} \quad (v_{p'}^d)^*=\Rcal_{p}^d:Z_p\to h_p^d$$
are bounded projections. 
\qd

The free Rosenthal inequalities are a crucial tool to prove the similar results for the space $h_p$. 

\begin{prop}\label{vextendhp}
Let $1\leq p <\infty$. 
\begin{enumerate}
\item[(i)] The map $v$ extends to bounded map from $h_p$ into $X_p^1$, which is injective for $1<p<\infty$.
\item[(ii)] The map $\Rcal$ extends to a bounded projection from $X_p^1$ onto $h_p$ for $1<p<\infty$. 
\end{enumerate}
\end{prop}

\begin{proof}
Let $x\in \M$. 
We will show that $\|v(x)\|_p\approx \|x\|_{h_p}$ for $1< p <\infty$. 
We first consider the case $2\leq p <\infty$. Then Theorem \ref{freeRos} (i) yields
$$\|v(x)\|_p \approx 
\max\{\|v(x)\|_{Z_p},\|v(x)\|_{L_p^c(\N,\E_\M)},\|v(x)\|_{L_p^r(\N,\E_\M)}\}.$$
Then by Lemma \ref{visometry} we deduce
$$\|v(x)\|_p \approx 
\max\{\|x\|_{h_p^d},\|x\|_{h_p^c},\|x\|_{h_p^r}\}=\|x\|_{h_p}.$$
We now consider $1\leq p <2$. 
In that case Theorem \ref{freeRos} (ii) gives 
$$\|v(x)\|_p\approx 
\inf_{v(x)=d+c+r} \|d\|_{Z_p}+\|c\|_{L_p^c(\N,\E_\M)}+\|r\|_{L_p^r(\N,\E_\M)},$$
where the infimum runs over all the decompositions $v(x)=d+c+r$ with $d,c,r \in \Sigma_1$. 
Note that any decomposition $x=D+C+R$ of $x$ with $D\in h_p^d, C\in h_p^c$ and $R\in h_p^r$ 
yields a decomposition $v(x)=v(D)+v(C)+v(R)$. 
Hence Lemma \ref{visometry} gives 
$$ \|v(x)\|_p
\leq C \big(\|v(D)\|_{Z_p}+\|v(C)\|_{L_p^c(\N,\E_\M)}+\|v(R)\|_{L_p^r(\N,\E_\M)}\big)
=C \big(\|D\|_{h_p^d}+\|C\|_{h_p^c}+\|R\|_{h_p^r}\big).$$
Taking the infimum over all the decompositions $x=D+C+R$ we get
$$\|v(x)\|_p\leq C \|x\|_{h_p}.$$
Conversely, for any decomposition $v(x)=d+c+r$ with $d,c,r \in \Sigma_1$ we can write
$$x=\Rcal(v(x))= \Rcal(d)+ \Rcal(c)+ \Rcal(r).$$
Then Proposition \ref{Rbdd} implies for $1<p<2$
$$\|x\|_{h_p}
\leq \|\Rcal(d)\|_{h_p^d}+\|\Rcal(c)\|_{h_p^c}+\|\Rcal(r)\|_{h_p^r}
\leq C_p\big(\|d\|_{Z_p}+\|c\|_{L_p^c(\N,\E_\M)}+\|r\|_{L_p^r(\N,\E_\M)}\big).$$
Taking the infimum over all the decompositions $v(x)=d+c+r$ we get
$$\|x\|_{h_p}\leq C_p \|v(x)\|_{p}.$$
This ends the proof of (i). 
We deduce (ii) by duality, by using the fact that for $1<p<\infty$, $(X_p^1)^*=X_{p'}^1$ and $(h_p)^*=h_{p'}$. 
\qd


\subsubsection{Complementation of $\h_p^c$ in the continuous case}

We now extend this construction to the continuous setting. 
For any finite partition $\s$ of $[0,1]$ we set
$$\AA_0(\s)=\M, \; \AA_t(\s)=\M \ast_{\M_{t^-(\s)}} \M_t \mbox{ for } 0<t\in \s
\quad  \mbox{and} \quad \N(\s)=\ast_{\M,t\in \s} \AA_t(\s),$$
where we amalgamate over the first copy of $\M$ in $\AA_t(\s)$. 
We denote by $\rho_\s:\M\to \N(\s)$ the $*$-homomorphism which sends $\M$ to the amalgamated copy, 
and by $\E_\M^\s:\N(\s)\to \M$ the associated conditional expectation. 
We equip $\N(\s)$ with the finite normal faithful trace $\tr_\s=\tau\circ \E_\M^\s$.  
We consider the ultraproduct von Neumann algebra
$$\tilde{\N}_\U=\Big(\prodd_\U \N(\s)_*\Big)^*$$ 
and the associated finite von Neumann algebra
$$\N_\U= \tilde{\N}_\U f_\U,$$
where $f_\U$ denotes the support projection of the trace $\tr_\U=(\tr_\s)^\bullet$.  
Since we may extend the $*$-homomorphism $\rho_\s$ to an isometry $\rho_\s:L_p(\M)\to L_p(\N(\s))$, 
the ultraproduct map $\rho_\U=(\rho_\s)^\bullet$ is the natural inclusion 
$$\rho_\U:L_1(\M_\U)\to L_1(\N_\U).$$
Taking the adjoint we obtain a normal faithful conditional expectation
$$(\rho_\U)^*=\E_{\M_\U}:\left\{\begin{array}{ccc}
\N_\U &\longrightarrow  &\M_\U\\
(x_\s)^\bullet&\longmapsto &(\E_\M^\s(x_\s))^\bullet
\end{array}\right..$$
Hence we may consider the $L_p$ $\M_\U$-module $L_p^c(\N_\U,\E_{\M_\U})$.

\begin{lemma}\label{embeddinghpc}
Let $1\leq p<\infty$. 
Then $\h_p^c$ embeds isometrically into $L_p^c(\N_\U,\E_{\M_\U})$. 
\end{lemma}

\begin{proof}
For each $\s$, we denote by $v_\s$ the map defined in Lemma \ref{visometry} for the finite filtration $(\M_t)_{t\in \s}$. 
For $x\in \M$ we define
$$v_\U(x)=(v_\s(x))^\bullet.$$
By Proposition \ref{vextendhp} and the noncommutative Burkholder-Rosenthal inequalities (Theorem \ref{Bdiscr}), 
for all $1< p<\infty$ we have
$$\|v_\s(x)\|_{L_p(\N(\s))}\leq C_p \|x\|_{h_p(\s)}\leq C_p\kappa_p\|x\|_p\leq C_p\kappa_p\|x\|_\infty.$$
This means that $v_\U(x)\in L_p(\tilde{\N}_\U)$ for all $1<p<\infty$. 
Lemma \ref{L_p(NU)} implies that $v_\U(x)\in L_p(\N_\U)$ for all $1\leq p<\infty$. 
By Lemma \ref{visometry} we get for $1\leq p<\infty$
\begin{align*}
\|v_\U(x)\|_{L_p^c(\N_\U,\E_{\M_\U})}
&=\|\E_{\M_\U}(v_\U(x)^*v_\U(x))\|_{L_{p/2}(\M_\U)}^{1/2}
=\|(\E_\M^\s(v_\s(x)^*v_\s(x)))^\bullet\|_{L_{p/2}(\M_\U)}^{1/2}\\
&=\lim_{\s,\U}\|v_\s(x)\|_{L_p^c(\N(\s),\E_\M^\s)}
=\lim_{\s,\U}\|x\|_{h_p^c(\s)}=\|x\|_{\h_p^c}.
\end{align*}
This proves that $v_\U$ extends to an isometry from $\h_p^c$ into $L_p^c(\N_\U,\E_{\M_\U})$ for $1\leq p<\infty$. 
\qd

\begin{prop}\label{complementhpc}
Let $1<p<\infty$. 
Then $\h_p^c$ is complemented in $L_p^c(\N_\U,\E_{\M_\U})$.
\end{prop}

\begin{proof}
Let $x=(x_\s)^\bullet \in \N_\U$ be such that $\|x\|_{L_p^c(\N_\U,\E_{\M_\U})}\leq 1$. 
This means that 
\begin{equation}\label{normxLpc}
\|\E_{\M_\U}(x^*x)\|_{L_{p/2}(\M_\U)}^{1/2}
=\|(\E_\M^\s(x_\s^*x_\s))^\bullet\|_{L_{p/2}(\M_\U)}^{1/2}
=\lim_{\s,\U}\|x_\s\|_{L_p^c(\N(\s),\E_\M^\s)}\leq 1.
\end{equation}
Observe that for all $2\leq p <\infty$, we have by Proposition \ref{vextendhp} and Proposition \ref{XpYpcompl}
\begin{align*}
\|\Rcal^\s\circ \P_1^\s(x_\s)\|_{L_p(\M)}
&\leq \eta_p \|\Rcal^\s\circ \P_1^\s(x_\s)\|_{h_p(\s)}
\leq  \eta_p C_p \|\P_1^\s(x_\s)\|_{L_p(\N(\s))}\\
&\leq 4 \eta_p C_p \|x_\s\|_{L_p(\N(\s))}
\leq 4 \eta_p C_p \|x_\s\|_{\N(\s)}.
\end{align*}
Hence the family $(\Rcal^\s\circ \P_1^\s(x_\s))_\s$ is uniformly bounded in $L_p(\M)$ for all $2\leq p <\infty$. 
For $1<p<\infty$, we may consider
$$\Rcal_\U(x)=\w L_{\max(2,p)}\mbox{-}\lim_{\s,\U}\Rcal^\s\circ \P_1^\s(x_\s).$$
It remains to estimate $\|\Rcal_\U(x)\|_{\h_p^c}$. 
Proposition \ref{Rbdd} (i) and Proposition \ref{XpYpcompl} (ii) yield for each $\s$
$$\|\Rcal^\s\circ \P_1^\s(x_\s)\|_{h_p^c(\s)}
\leq C_p\|\P_1^\s(x_\s)\|_{L_p^c(\N(\s),\E_\M^\s)}
\leq C_p\|x_\s\|_{L_p^c(\N(\s),\E_\M^\s)}.$$
Taking the limit in $\s$, \eqref{normxLpc} gives 
\begin{equation}\label{RcircPsigma}
\lim_{\s,\U}\|\Rcal^\s\circ \P_1^\s(x_\s)\|_{h_p^c(\s)} \leq C_p.
\end{equation}
Let $1<p<2$ and $\eps>0$. 
We may find a sequence of positive numbers $(\alpha_m)_{m=1}^M$ such that $\sum_m \alpha_m=1$ and finite partitions $\s^1,\cdots,\s^M$ satisfying
$$\| \Rcal_\U(x) -\sum_m\alpha_m \Rcal^{\s^m}\circ \P_1^{\s^m}(x_{\s^m})\|_2\leq \eps 
 \quad \mbox{and}  \quad
\|\Rcal^{\s^m}\circ \P_1^{\s^m}(x_{\s^m})\|_{h_p^c(\s^m)}\leq C_p+\eps.$$
Since $\|z\|_{\h_p^c}\leq \|z\|_2$, by Lemma \ref{convexityhpc} (i) we get
\begin{align*}
\|\Rcal_\U(x)\|_{\h_p^c}
&\leq \Big\|\Rcal_\U(x) -\sum_m\alpha_m \Rcal^{\s^m}\circ \P_1^{\s^m}(x_{\s^m})\Big\|_{\h_p^c}
+\Big\| \sum_m\alpha_m \Rcal^{\s^m}\circ \P_1^{\s^m}(x_{\s^m})\Big\|_{\h_p^c}\\
&\leq\eps
+\Big\| \sum_m\alpha_m \Rcal^{\s^m}\circ \P_1^{\s^m}(x_{\s^m})\Big\|_{\h_p^c}\\
&\leq \eps 
+2^{1/p}\sum_m\alpha_m \|\Rcal^{\s^m}\circ \P_1^{\s^m}(x_{\s^m})\|_{h_p^c(\s^m)}\\
&\leq \eps + 2^{1/p}(C_p+\eps).
\end{align*}
Sending $\eps$ to $0$ we obtain 
$$\|\Rcal_\U(x)\|_{\h_p^c}\leq 2^{1/p}C_p\|x\|_{L_p^c(\N_\U,\E_{\M_\U})}.$$
We now consider $2\leq p<\infty$ and fix a partition $\s_0$. 
By Lemma \ref{convexityhpc} (ii) we have for all $\s \supset \s_0$
\begin{equation}\label{RcircPsigma2}
\|\Rcal^\s\circ \P_1^\s(x_\s)\|_{h_p^c(\s_0)}\leq {\delta'}_{p/2}^{1/2}\|\Rcal^\s\circ \P_1^\s(x_\s)\|_{h_p^c(\s)}.
\end{equation}
Thus \eqref{RcircPsigma} implies that the family $(\Rcal^\s\circ \P_1^\s(x_\s))_{\s \supset \s_0}$ is uniformly bounded in the reflexive space $h_p^c(\s_0)$. 
We deduce that the weak$^*$-limit of the $\Rcal^\s\circ \P_1^\s(x_\s)$'s exists in $h_p^c(\s_0)$, 
and coincides with the weak$^*$-limit in $L_p$:
$$\Rcal_\U(x)= \w h_p^c(\s_0)\mbox{-}\lim_{\s\supset \s_0,\U}\Rcal^\s\circ \P_1^\s(x_\s).$$
By using \eqref{RcircPsigma2} and \eqref{RcircPsigma} we get
$$\|\Rcal_\U(x)\|_{h_p^c(\s_0)}
\leq \lim_{\s\supset \s_0,\U}\|\Rcal^\s\circ \P_1^\s(x_\s)\|_{h_p^c(\s_0)}
\leq {\delta'}_{p/2}^{1/2}C_p.$$
Since this holds true for all partition $\s_0$, by taking the limit we obtain
$$\|\Rcal_\U(x)\|_{\h_p^c}\leq {\delta'}_{p/2}^{1/2}C_p\|x\|_{L_p^c(\N_\U,\E_{\M_\U})}.$$
This ends the proof of the Proposition. 
\qd

We deduce from Proposition \ref{dualintpolLpc(M,E)} the corresponding duality and interpolation results for the spaces $\h_p^c$.

\begin{cor}\label{dualhpc}
Let $1<p<\infty$. 
\begin{enumerate}
\item[(i)] Let $\frac{1}{p}+\frac{1}{p'}=1$. Then $$(\h_p^c)^*=\h_{p'}^c \quad \mbox{with equivalent norms.}$$
\item[(ii)]  Let $1<p_1,p_2<\infty$ and $0<\theta<1$ be such that
$\frac{1}{p}=\frac{1-\theta}{p_1}+\frac{\theta}{p_2}$. Then
 $$\h_p^c  = [\h_{p_1}^c,\h_{p_2}^c]_{\theta}  \quad \mbox{with equivalent norms}  .$$
\end{enumerate}
\end{cor}

\subsection{Injectivity results}

By using Corollary \ref{dualhpc} (i), 
it is now easy to prove that the conditioned Hardy spaces defined above are well intermediate spaces between $L_2(\M)$ and $L_p(\M)$ for $1<p<\infty$ as expected. 

\begin{prop}\label{injhpc}
Let $1 < p <\infty$. Then
$$L_{\max(p,2)}(\M)  \subset \h_p^c \subset L_{\min(p,2)}(\M),$$
i.e., $\h_p^c$ embeds into $L_{\min(p,2)}(\M)$. 
\end{prop}

Actually, the injectivity for $2\leq p <\infty$ can be proved directly as a consequence of the monotonicity Lemma \ref{convexityhpc}. 
Indeed, since the monotonicity in the conditioned case is inverse to that of $\H_p^c$, 
the conditioned analogue of Lemma \ref{injHpc} concerns the case $2\leq p<\infty$.

\begin{lemma}\label{injhpc2}
Let $2 \leq p <\infty$. 
Then the space $\{x\in L_2(\M) : \|x\|_{\h_p^c}<\infty \}$ is complete with respect to the norm $\|\cdot\|_{\h_p^c}$. 
\end{lemma} 
 
\begin{proof} 
Recall that in the conditioned case, by Lemma \ref{convexityhpc} 
the norms $\|\cdot\|_{h_p^c(\s)}$ are increasing in $\s$ (up to a constant) for $2\leq p<\infty$. 
Then the completeness of each discrete $h_p^c(\s)$-space yields the result as in the proof of Lemma \ref{injHpc}.
\qd  

It then directly follows that $\h_p^c$ embeds into $L_2(\M)$ for $2 \leq p <\infty$. 
Moreover, by simply using the discrete $h_p^c(\s)-h_{p'}^c(\s)$ duality, 
we can prove the conditioned analogue of Lemma \ref{dualHpc-direct} with the same argument. 

\begin{lemma}\label{dualhpc-direct}
Let $2\leq p <\infty$. Then 
$$(\h_{p'}^c)^*=\{x\in L_2(\M)  :  \|x\|_{\h_p^c}<\infty \}\quad \mbox{with equivalent norms.}$$
\end{lemma}

Then, combining Lemma \ref{dualhpc-direct} with assertion (i) of Corollary \ref{dualhpc} we obtain  

\begin{cor}\label{hpXp}
Let $2\leq p<\infty$. Then $\h_{p}^c=\{x\in L_2(\M)  : \|x\|_{\h_p^c}<\infty \}$.
\end{cor}

However, the injectivity in the case $1<p<2$ is highly non-trivial and we really need the complementation result stated in subsection \ref{sectcomplhpc} 
to prove it. 
This approach does not include the case $p=1$, and at the time of this writing we do not know if the natural map from $\h_1^c$ to $L_1(\M)$ is injective (see Problem \ref{h1csubsetL1}). 
For the sequel we need to introduce another candidate for the continuous analogue of the conditioned Hardy space $h_1^c$, which is embedded in $L_1(\M)$. 
We denote by 
$$\varphi:\h_1^c \to L_1(\M)$$
the natural map defined by $\varphi(x)=x$ for $x\in \M$, and set
$$\lh_1^c=\varphi(\h_1^c)\subset L_1(\M).$$ 
Since $\varphi$ is bounded, $\lh_1^c$ equipped with the norm
$$\|x\|_{\lh_1^c}=\inf_{x=\varphi(y)} \|y\|_{\h_1^c}$$
is a Banach space. 
Moreover, note that $L_2(\M)$ is still dense in $\lh_1^c$.

\re\label{hpcregular}
Considering $\h_p^c$ as a subspace of $L_{\min(p,2)}(\M)$ for $1<p<\infty$ thanks to Proposition \ref{injhpc}, we can write
$$\h_q^c\subset \h_p^c \quad \mbox{for } 1<p\leq q <\infty 
\quad \mbox{and} \quad 
 \h_q^c\subset \lh_1^c \quad \mbox{for } 1< q <\infty  .$$
Moreover, for $1<q<\infty$, the commuting diagram
$$\xymatrix{
    \h_q^c \ar@{^{(}->}[d]^{v_\U} \ar@{->}[r]     & {\h}_1^c \ar@{^{(}->}^{v_\U}[d]  \\
    L_q^c(\N_\U,\E_{\M_\U}) \ar@{^{(}->}[r] & L_1^c(\N_\U,\E_{\M_\U})
  }$$
implies that we may also consider
$$ \h_q^c\subset \h_1^c \quad \mbox{for } 1< q <\infty  .$$
\mar

\subsection{Fefferman-Stein duality}

This subsection deals with the analogue of the Fefferman-Stein duality for the conditioned Hardy spaces. 
First observe that in the discrete case, the space $L_p^cmo$ is simpler than the space $L_p^cMO$ for $2<p\leq \infty$. 
Indeed, recall that for a finite partition $\s$ and $x \in L_2(\M)$ we have 
$$\|x\|_{L_p^cMO(\s)}=\|{\sup_{t\in \s}}^+ \E_t|x-x_{t^-(\s)}|^2\|_{p/2}^{1/2}
\quad \mbox{and} \quad
\|x\|_{L_p^cmo(\s)}=\max\big(\|\E_0(x)\|_p, \|{\sup_{t\in \s}}^+ \E_t|x-x_{t}|^2\|_{p/2}^{1/2}\big).$$
The crucial point is that the index ``$t^-(\s)$", which depends on the partition $\s$, does not appear in the definition of $L_p^cmo(\s)$. 
Hence it is natural to introduce the following definition of $L_p^c\mo$ in the continuous setting.

\begin{defi}
Let $2< p\leq \infty$. We define 
$$L_{p}^c\mo=\{x\in L_2(\M) : \|x\|_{L_{p}^c\mo}<\infty\}$$
where
$$\|x\|_{L_{p}^c\mo}=\max\big(\|\E_0(x)\|_p , \|{\sup_{0\leq t \leq 1}}^+ \E_t|x-x_t|^2\|_{p/2}^{1/2}\big).$$
For $p=\infty$ we denote this space by $\bmo^c$.
\end{defi}

Recall that for a family $(x_t)_{0\leq t\leq 1}$ in $L_q(\M)$, $1\leq q \leq \infty$, we define
$$\|{\sup_{0\leq t \leq 1}}^+ x_t\|_q=\|(x_t)_{0\leq t\leq 1}\|_{L_q(\M;\ell_\infty([0,1]))}=\inf \|a\|_{2q}\sup_t\|y_t\|_\infty\|b\|_{2q},$$
where the infimum runs over all factorizations $x_t=ay_tb$ with $a,b \in L_{2q}(\M)$ and $(y_t)\in \ell_\infty(L_\infty([0,1]))$. 
The space $L_p^c\mo$ obviously does not depend on $\U$. 
Note that by Proposition $2.1$ of \cite{jx-erg} we have
$$\sup_{\sigma}\|{\sup_{t\in \si}}^+ \E_{t}|x-x_t|^2\|_{p/2}^{1/2}=\|{\sup_{0\leq t \leq 1}}^+ \E_t|x-x_t|^2\|_{p/2}^{1/2},$$
thus we obtain
\begin{equation}\label{monotonicityLpcmo}
\|x\|_{L_{p}^c\mo}=\sup_{\sigma}\|x\|_{L_{p}^cmo(\sigma)}.
\end{equation}
Since by definition $\|\cdot\|_{L_{p}^c mo(\sigma)}$ is increasing in $\s$,
for $2< p\leq \infty$ we may write
$$\|x\|_{L_{p}^c\mo}= \lim_{\s,\U}\|x\|_{L_{p}^c mo(\sigma)}$$
for every ultrafilter $\U$. This ensures that we define well a complete space. 

The discrete Fefferman-Stein duality in the conditioned case easily implies the following continuous analogue.

\begin{theorem}\label{th:dualityhp}
Let $1\leq p <2$. 
Then 
$$(\h_p^c)^*=L_{p'}^c\mo \quad \mbox{with equivalent norms}.$$
Moreover,
\begin{equation}\label{eq:dualhp}
\nu_{p}^{-1}\|x\|_{L_{p'}^c\mo}\leq \|x\|_{(\h_p^c)^*}\leq \sqrt{2}\|x\|_{L_{p'}^c\mo}.
\end{equation}
\end{theorem} 

\begin{proof}
The proof is similar to that of Lemma \ref{dualHpc-direct}, by using the discrete $h_p^c(\s)-L_{p'}^cmo(\s)$ duality. 
This argument can also be adapted for $p=1$. 
\qd

Moreover, we deduce from Proposition \ref{hpcLpcmo} that for $2\leq p<\infty$,
$$L_p^c\mo = \{x\in L_2(\M)  : \|x\|_{\h_p^c}<\infty \} \quad \mbox{with equivalent norms}.$$
Hence Corollary \ref{hpXp} yields

\begin{cor}\label{Lpmo}
Let $2< p<\infty$. Then 
$$L_p^c\mo=\h_p^c \quad \mbox{with equivalent norms}.$$
\end{cor}

As a consequence of Theorem \ref{th:dualityhp} we can characterize the space $L_{p}^c\mo$ similarly to the definition of $L_p^c\MO$.

\begin{lemma}\label{le:lim}
Let $2<p\leq \infty$.
Then 
\begin{enumerate}
\item[(i)] The unit ball of $L_{p}^c\mo$ is equivalent to 
$$\mathrm{B}_p=\{x=\w L_2 \mbox{-}\lim_{\s,\U} x_\s : \lim_{\s,\U} \|x_\s\|_{L_p^cmo(\s)}\leq 1\}.$$
More precisely, we have 
$$B_{L_{p}^c\mo}\subset \mathrm{B}_p \subset \sqrt{2}\nu_pB_{L_{p}^c\mo}.$$
\item[(ii)] Let $(x_\lambda)_\lambda$ be a sequence in $L_2(\M)$ such that $\|x_\lambda\|_{L_{p}^c\mo}\leq 1$ for all $\lambda$ and
$x=\w L_2 \mbox{-}\lim_\lambda x_\lambda $.
Then $x\in L_{p}^c\mo$ with $\|x\|_{L_{p}^c\mo} \leq \sqrt{2}\nu_{p}$.
\end{enumerate}
\end{lemma}

\begin{proof}
It is clear that 
$B_{L_{p}^c\mo}\subset \mathrm{B}_p.$
Conversely, let $x=\w L_2 \mbox{-}\lim_{\s,\U} x_\s$ be such that $\lim_{\s,\U} \|x_\s\|_{L_p^cmo(\s)}\leq 1$. 
By Theorem \ref{th:dualityhp} 
and the density of $L_2(\M)$ in $\h_{p'}^c$ we can write
$$\|x\|_{L_{p}^c\mo} \leq \nu_{p} \sup_{y\in L_2(\M), \|y\|_{\h_{p'}^c}\leq 1} |\tau(x^*y)|.$$
Note that for all $y\in L_2(\M), \|y\|_{\h_{p'}^c}\leq 1$ we have
\begin{align*}
|\tau(x^*y)|
&\leq \lim_{\s,\U} |\tau(x_\s^*y)|\leq \sqrt{2}\lim_{\s,\U} \big(\|x_\s\|_{L_{p}^cmo(\s)}\|y\|_{h_{p'}^c(\s)}\big)\\
&=\sqrt{2}\big(\lim_{\s,\U} \|x_\s\|_{L_{p}^cmo(\s)}\big)\big(\lim_{\s,\U}\|y\|_{h_{p'}^c(\s)}\big)
\leq \sqrt{2}.
\end{align*}
Thus $x\in \sqrt{2}\nu_p B_{L_{p}^c\mo}$, and this proves (i). 
The proof of (ii) is similar to that of Corollary \ref{LpcMOclosed}.
\qd

We end this subsection with the description of the dual space of $\lh_1^c$. 

\begin{theorem}\label{fsdualLh1}
We have $(\lh_1^c)^*=\lbmo^c$ with equivalent norms, where
\begin{align*}
\lbmo^c=
&\{x\in L_2(\M) : \|x\|_{\bmo^c}<\infty \mbox{ and } \lim_n \tau(x^*y_n)=0\\
&\forall \mbox{ sequence } (y_n)_n \subset \M \mbox{ such that } (y_n)_n \mbox{ converges in } \h_1^c 
\mbox{ and } y_n \to 0 \mbox{ in } L_1 \}.
\end{align*}
\end{theorem}

\begin{proof}
By definition, $\lh_1^c$ is isomorphic to the quotient space $\h_1^c / \ker\varphi$. 
Hence 
$$(\lh_1^c)^*=(\ker \varphi)^\perp \subset (\h_1^c)^*=\bmo^c.$$
This means that 
$$(\lh_1^c)^*=\{x \in \bmo^c : (x|y)_{\bmo^c,\h_1^c}=0 , \forall y \in \ker \varphi\}.$$
By definition, an element $y\in \ker \varphi$ is the limit in $\h_1^c$ of a sequence $(y_n)_n \subset \M$ such that 
$\varphi(y)=L_1\mbox{-} \lim_n \varphi(y_n)=L_1\mbox{-}\lim_n y_n =0$. 
In that case we have $(x|y)_{\bmo^c,\h_1^c}=\lim_n (x|y_n)_{\bmo^c,\h_1^c}=\lim_n \tau(x^*y_n)$, and this ends the proof. 
\qd

\re
Observe that since by definition the space $\lh_1^c$ embeds into $L_1(\M)$, then $L_\infty(\M)$ is weak-$*$ dense in $\lbmo^c$ by Theorem \ref{fsdualLh1}. 
\mar

\subsection{Interpolation}

The end of this section is devoted to the continuous analogue of Theorem \ref{inth1cbmoc-discr}.

\begin{theorem}\label{inthpcbmoc}
Let $1<p<\infty$. Then
$$\h_p^c=[\bmo^c,\h^c_1]_{\frac{1}{p}}\quad \mbox{with equivalent norms}.$$
\end{theorem}

\begin{proof}
Observe that by Remark \ref{hpcregular}, we may write $\bmo^c\subset L_2(\M)\subset \h_1^c$. 
This ensures that the couple $[\bmo^c,\h^c_1]$ is compatible. 
As in \cite{bcpy}, we first show that Corollary \ref{dualhpc} (ii) still holds true for $p_1=1$, i.e.,
\begin{equation}\label{inthpc}
\h_q^c  = [\h_{1}^c,\h_{p}^c]_{\theta}  \quad  \mbox{with equivalent norms} 
\end{equation}
for $1<p<\infty, 0<\theta<1$ and $1-\theta+\frac{\theta}{p}=\frac{1}{q}$.
Then, as in the proof of \cite[Theorem $4.1$]{bcpy}, 
we will deduce the required interpolation result by using duality (Theorem \ref{th:dualityhp} and Corollary \ref{dualhpc} (i)) and Wolff's theorem. 
Note that it suffices to prove \eqref{inthpc} for $1<q<p\leq 2$. 
Indeed, Corollary \ref{dualhpc} (ii) combined with an application of Wolff's theorem will yield \eqref{inthpc} for $1<p<\infty$. 
The inclusion $[\h_{1}^c,\h_{p}^c]_{\theta}\subset \h_q^c$ follows easily from Lemma \ref{embeddinghpc} and Proposition \ref{complementhpc}. 
Let $x\in [\h_{1}^c,\h_{p}^c]_{\theta}$ be of norm $<1$. 
Then there exists a function $f\in \F(\h_1^c,\h_p^c)$ such that $f(\theta)=x$ and 
$$\|f\|_{\F(\h_1^c,\h_p^c)}=\max(\sup_t \|f(it)\|_{\h_1^c}, \sup_t \|f(1+it)\|_{\h_p^c}) <1.$$
Since $v_\U$ is isometric by Lemma \ref{embeddinghpc}, we deduce that the function $v_\U \circ f\in \F( L_1^c(\N_\U,\E_{\M_\U}), L_p^c(\N_\U,\E_{\M_\U}))$ 
with $\|v_\U \circ f\|_{\F}=\|f\|_{\F}$. 
Hence $v_\U\circ f(\theta)=v_\U(x) \in [L_1^c(\N_\U,\E_{\M_\U}), L_p^c(\N_\U,\E_{\M_\U})]_\theta$ with norm $<1$. 
Proposition \ref{dualintpolLpc(M,E)} (iii) implies that $v_\U(x)\in L_q^c(\N_\U,\E_{\M_\U})$ for $1-\theta+\frac{\theta}{p}=\frac{1}{q}$. 
Then $x=\Rcal_\U\circ v_\U(x) \in \h_q^c$ by Proposition \ref{complementhpc}. 
Observe that this argument still works for $1<p<\infty$. 
However, we need the restriction to the case $1<p\leq 2$ to prove the reverse inequality by duality. 
We will show that
\begin{equation}\label{intLpcmo}
[\bmo^c,L_p^c\mo]_{\frac{p}{q}}\subset L_q^c\mo \quad  \mbox{with equivalent norms,} 
\end{equation}
for $2\leq p<q<\infty$, and Theorem \ref{th:dualityhp} will yield the remaining inclusion by duality (since $L_p^c\mo$ is reflexive). 
This comes directly from the discrete result and the monotonicity property \eqref{monotonicityLpcmo}. 
Let $x\in [\bmo^c,L_p^c\mo]_{\frac{p}{q}}$ be of norm $<1$. 
Then there exists $f\in \F(\bmo^c,L_p^c\mo)$ such that $f(\frac{p}{q})=x$ and
$\|f\|_{\F(\bmo^c,L_p^c\mo)}<1$. 
By \eqref{monotonicityLpcmo}, we deduce that $f\in \F(bmo^c(\s),L_p^cmo(\s))$ with norm $<1$ for each $\s$. 
Hence the discrete interpolation result gives that $x \in L_q^cmo(\s)$ for each $\s$ with 
$$\|x\|_{L_q^cmo(\s)}\leq C_q \|f\|_{\F(bmo^c(\s),L_p^cmo(\s))} \leq  C_q\|f\|_{\F(\bmo^c,L_p^c\mo)}\leq C_q.$$ 
Taking the supremum over $\s$ we obtain that $x\in L_q^c\mo$ with 
$$\|x\|_{L_q^c\mo}\leq C_q \|x\|_{ [\bmo^c,L_p^c\mo]_{\frac{p}{q}}}.$$
This ends the proof of \eqref{intLpcmo} and the Theorem follows.
\qd

\section{Davis and Burkholder-Rosenthal inequalities}\label{sectD-BR}

We continue our investigation of the Hardy spaces of noncommutative martingales in the continuous setting by studying some decompositions 
of $\H_p^c$ and $\H_p$ involving the conditioned Hardy space $\h_p^c$. 
By considering the adjoint in Section \ref{secthp} we may define the row conditioned Hardy space $\h_p^r$ and obtain the analoguous results. 
After recalling the noncommutative Davis inequalities in the discrete case, 
we will discuss three variants of this decomposition in the case $1\leq p<2$. 
The first one is a regular version of the Davis decomposition involving another diagonal space $h_p^{1_c}$ instead of $h_p^d$. 
The second version, presented in subsection \ref{subsectDRadiscr}, is a Davis decomposition in Randrianantoanina's style 
with simultaneous control of $\h_p$ and $L_2$ norms for $1<p<2$. 
The last variant is a mixed version of the two first ones, i.e., a Davis decomposition in Randrianantoanina's style involving the diagonal space $h_p^{1_c}$. 
Then we will turn to the continuous setting and define the analogue of the diagonal spaces. 
We will extend the three versions of Davis' decomposition to the continuous case for $1<p<2$, 
and, as usual, deduce the inequalities for $2<p<\infty$ by duality. 
However, we will meet some difficulty to describe the dual space of our continuous analogue of the diagonal space. 
Hence the continuous analogue of the Davis and Burkholder-Rosenthal inequalities for $2<p<\infty$ stated in Theorem \ref{Dp>2} is slightly different from the expected result. 

\subsection{The discrete case}\label{subsectDdiscr}

We first recall the analogue of the Davis decomposition for noncommutative martingales in the discrete case, 
then we discuss three stronger versions of this decomposition which will be useful for extending it to the continuous setting. 
Let $(\M_n)_{n\geq 0}$ be a discrete filtration. 

Observe that by combining the noncommutative Burkholder-Gundy inequalities (Theorem \ref{BGdiscr}) with the noncommutative Burkholder-Rosenthal inequalities (Theorem \ref{Bdiscr}) we get 
$$H_p=h_p \quad \mbox{with equivalent norms} \quad \mbox{for } 1<p<\infty.$$
By a dual approach, it was proved in \cite{jm-riesz} and \cite{Per} that this equality still holds true for $p=1$, i.e., 
\begin{equation}\label{H1h1discr}
H_1=h_1 \quad \mbox{with equivalent norms}.
\end{equation}
Moreover, we can show a column version of this result.

\begin{theorem}\label{Davishddisc}
Let $1\leq p <\infty$. Then the discrete spaces satisfy
$$H_p^c=\left\{\begin{array}{cl}
h_p^d+h_p^c& \quad \mbox{for} \quad 1\leq p<2 \\
h_p^d\cap h_p^c& \quad \mbox{for}\quad 2\leq p<\infty
\end{array}\right.\quad \mbox{with equivalent norms}.$$ 
\end{theorem}

\subsubsection{The ''regular`` version of the discrete Davis decomposition}\label{subsectDregdiscr}

The Davis decomposition stated in Theorem \ref{Davishddisc} can be refined to get a stronger decomposition, 
involving another diagonal space called $h_p^{1_c}$. 
The regularity properties satisfied by this space make it a good tool for the sequel. 

To see how we may refine Theorem \ref{Davishddisc}, we briefly recall the strategy of its proof. 
We first show the decomposition for $1\leq p <2$, then the case $2<p<\infty$ is deduced by duality. 
For $1\leq p <2$ the inclusion $h_p^d+h_p^c \subset H_p^c$ is easy, and the reverse inclusion is proved by a dual approach. 
More precisely, we can show that  
$$(h_p^d+h_p^c)^*= h_{p'}^d \cap L_{p'}^c mo\subset L_{p'}^cMO =(H_p^c)^*.$$
A close look at the dual spaces yields a stronger decomposition. 
Indeed, observe that  for $2< p' \leq \infty$ and $x\in L_2(\M)$, by the triangle inequality in $L_{p'/2}(\M; \ell_\infty)$ we can write 
$$\Big\|{\sup_{n\geq 0}}^+\E_n\Big(\sum_{k\geq n} |d_k(x)|^2\Big)\Big\|_{p'/2}
\simeq \|{\sup_{n\geq 0}}^+|d_n(x)|^2\|_{p'/2} + \Big\|{\sup_{n\geq 0}}^+\E_n\Big(\sum_{k > n} |d_k(x)|^2\Big)\Big\|_{p'/2}
.$$
Hence we get 
\begin{equation}\label{decLpcMO}
\|x\|_{L_{p'}^cMO} \simeq \max \big( \|(d_n(x))_n\|_{L_{p'}(\M;\ell_\infty^c)} ,  \|x\|_{L_{p'}^cmo}\big).
\end{equation}
Recall that for  $2< p' \leq \infty$, $L_{p'}(\M;\ell_\infty^c)$ is defined in \cite{JD,M} as
the space of all sequences $x=(x_n)_{n\geq 0}$ in $L_{p'}(\M)$ such that
$$\|(x_n)_{n\geq 0}\|_{L_{p'}(\M;\ell_\infty^c)}=\|(|x_n|^2)_{n\geq 0}\|_{L_{p'/2}(\M;\ell_\infty)}^{1/2}
=\|{\sup_{n\geq 0}}^+|x_n|^2\|_{p'/2}^{1/2}<\infty.$$
Note that a sequence $x=(x_n)_{n\geq 0}$ in $L_{p'}(\M)$ belongs to $L_{p'}(\M;\ell_\infty^c)$ if and only if there exist 
$a \in L_{p'}(\M)$ and $y=(y_n)_{n\geq 0} \subset L_\infty(\M)$ such that $x_n=y_n a$ for all $n \geq 0$. 
Moreover, 
$$\|x\|_{L_{p'}(\M;\ell_\infty^c)}=\inf\{\sup_{n\geq 0}\|y_n\|_\infty \|a\|_{p'}\},$$
where the infimum runs over all factorizations as above.

Inspired by the duality between $L_p(\M;\ell_1)$ and $L_{p'}(\M;\ell_\infty)$ proved in \cite{JD}, 
we define its predual space $L_p(\M;\ell_1^c)$ as follows. 
Let $1\leq p < 2$ and $\frac1p=\frac12+\frac1q$. 
A sequence $x=(x_n)_{n\geq 0}$ is in $L_p(\M;\ell_1^c)$ if there exist $b_{k,n} \in L_2(\M)$ and $a_{k,n}\in L_q(\M)$ such that 
\begin{equation}\label{decLp1}
x_n = \sum_{k\geq 0} b_{k,n}^*a_{k,n} 
\end{equation}
for all $n\geq 0$ and 
$$\sum_{k,n\geq 0} |b_{k,n}|^2 \in L_1(\M), \quad \sum_{k,n\geq 0} |a_{k,n}|^2 \in L_{q/2}(\M).$$
We equip $L_p(\M;\ell_1^c)$ with the norm
$$\|x\|_{L_p(\M;\ell_1^c)}=\inf\Big\{\Big(\sum_{k,n\geq 0} \|b_{k,n}\|_2^2 \Big)^{1/2}
\Big\|\Big(\sum_{k,n\geq 0} |a_{k,n}|^2\Big)^{1/2}\Big\|_q\Big\},$$
where the infimum is taken over all factorizations \eqref{decLp1}. 
In fact this space can be described in an easier way.

\begin{lemma}\label{ballLp1}
 Let $1\leq p< 2$ and $\frac1p=\frac12+\frac1q$. 
Then the unit ball of $L_p(\M;\ell_1^c)$ is the set of all sequences $(b_na_n)_{n\geq 0}$ such that 
 \begin{equation}\label{l1ataom}
\Big(\sum_{n\geq 0} \|b_n\|_2^2\Big)^{1/2} \Big\|\Big(\sum_{n\geq 0} |a_n|^2\Big)^{1/2}\Big\|_q \leq 1.
\end{equation}
\end{lemma}

\begin{proof}
It is clear that a sequence $(b_na_n)_{n\geq 0}$ satisfying \eqref{l1ataom} is in the unit ball of $L_p(\M;\ell_1^c)$. 
Conversely, let $x=(x_n)_{n\geq 0}$ be such that $x_n = \sum_{k\geq 0} b_{k,n}^*a_{k,n}$ with 
$$\Big(\sum_{k,n\geq 0} \|b_{k,n}\|_2^2 \Big)^{1/2}
\Big\|\Big(\sum_{k,n\geq 0} |a_{k,n}|^2\Big)^{1/2}\Big\|_q\leq 1.$$ 
We first set 
$a'_n= \Big(\sum_{k\geq 0} |a_{k,n}|^2\Big)^{1/2}$. 
By approximation, we may assume that the $a'_n$'s are invertible. 
Then considering 
$$v_{k,n}=a_{k,n}{a'}_n^{-1}
\quad \mbox{and} \quad 
b'_n=\sum_{k\geq 0} b_{k,n}^*v_{k,n},$$
we can write 
$ x_n =b'_na'_n$ for all $n\geq 0$. 
Moreover, 
$$\Big\|\Big(\sum_{n\geq 0} |a'_n|^2\Big)^{1/2}\Big\|_q =\Big\|\Big(\sum_{n,k\geq 0 } |a_{k,n}|^2\Big)^{1/2}\Big\|_q$$ 
and since $\sum_{k\geq 0} |v_{k,n}|^2= 1$ we get 
$$\begin{array}{ccl}
  \displaystyle\sum_{n\geq 0}\|b'_n\|_2^2 &= &\displaystyle\sum_{n\geq 0}\Big\|\displaystyle\sum_{k\geq 0} b_{k,n}^*v_{k,n}\Big\|_2^2\\
  &\leq&\displaystyle\sum_{n\geq 0}\Big\|\Big(\displaystyle\sum_{k\geq 0} b_{k,n}^*b_{k,n}\Big)^{1/2}\Big\|_2^2
\Big\|\Big(\displaystyle\sum_{k\geq 0} v_{k,n}^*v_{k,n}\Big)^{1/2}\Big\|_\infty^2\\
  &=&\displaystyle\sum_{k,n\geq 0}\|b_{k,n}\|_2^2.
\end{array}$$
Hence $(a'_n)$ and $(b'_n)$ satisfy \eqref{l1ataom}. 
\qd

\re
This implies that we have a bounded map
$$\left\{\begin{array}{ccc}
L_p(\M;\ell_1^c)&\longrightarrow &L_p(\M;\ell_2^c)\\
(b_na_n)_{n\geq 0}&\longmapsto & \displaystyle\sum_{n\geq 0} e_{n,0} \ten b_na_n
\end{array}\right..$$
Indeed, we can write
$$\sum_n e_{n,0} \ten b_na_n=\Big(\sum_n e_{n,n} \ten b_n\Big)\Big(\sum_n e_{n,0} \ten a_n\Big)$$
and the H\"{o}lder inequality gives for $\frac1p=\frac12+\frac1q$
$$\Big\|\sum_n e_{n,0} \ten b_na_n\Big\|_p \leq \Big\|\sum_n e_{n,n} \ten b_n\Big\|_2\Big\|\sum_n e_{n,0} \ten a_n\Big\|_q
=\Big(\sum_n \|b_n\|_2^2\Big)^{1/2} \Big\|\Big(\sum_n |a_n|^2\Big)^{1/2}\Big\|_q .$$
\mar

We can now state the following duality. 

\begin{prop}\label{dualLp1}
 Let $1\leq p <2$. 
Then 
$$(L_p(\M;\ell_1^c))^*=L_{p'}(\M;\ell_\infty^c)\quad \mbox{isometrically}.$$
\end{prop}

\begin{proof}
 Let $x$ be in the unit ball of $L_p(\M;\ell_1^c)$ and $y\in L_{p'}(\M;\ell_\infty^c)$. 
By Lemma \ref{ballLp1}, for all $n\geq 0$ we can decompose $x_n=b_na_n$ where $(b_n)$ and $(a_n)$ satisfy \eqref{l1ataom}. 
Then we deduce from the Cauchy-Schwarz inequality and the duality between $L_s(\M;\ell_1)$ and $L_{p'/2}(\M;\ell_\infty)$ that
\begin{align*}
 \sum_{n\geq 0} \tau(y_n^*x_n)
&=\sum_{n\geq 0} \tau(y_n^*b_na_n)=\sum_{n\geq 0} \tau((y_na_n^*)^*b_n)\\
&\leq \Big(\sum_{n\geq 0}\|y_na_n^*\|_2^2\Big)^{1/2} \Big(\sum_{n\geq 0}\|b_n\|_2^2\Big)^{1/2}
=\Big(\sum_{n\geq 0}\tau(|y_n|^2|a_n|^2)\Big)^{1/2} \Big(\sum_{n\geq 0}\|b_n\|_2^2\Big)^{1/2}\\
&\leq \|{\sup_n}^+ |y_n|^2\|_{p'/2}^{1/2} \Big\|\sum_n |a_n|^2\Big\|_s^{1/2} \Big(\sum_{n\geq 0}\|b_n\|_2^2\Big)^{1/2},
\end{align*}
where $s$ denotes the conjugate index of $\frac{p'}{2}$. 
An easy calculation gives $s=\frac{q}{2}$, 
and this yields the contractive inclusion $L_{p'}(\M;\ell_\infty^c) \subset (L_p(\M;\ell_1^c))^*$. 

Conversely, let $\varphi$ be a norm one functional on $L_p(\M;\ell_1^c)$. 
We observe that 
\begin{equation}\label{l1Lpembed}
\ell_1(L_p(\M)) \subset L_p(\M;\ell_1^c) \quad \mbox{contractively}.
\end{equation}
Indeed, for a finite sequence $x=(x_n)_{n=0}^N$ we can write
$$x=\sum_{i=1}^N x^i,$$
where $x^i=(x^i_n)_{n\geq 0}$ with $x^i_n=\delta_{n,i}x_i$. 
By setting  
$$b^i_n=\delta_{n,i} u_i|x_i|^{p/2} 
\quad \mbox{and} \quad
a^i_n=\delta_{n,i} |x_i|^{p/q},$$
where $\frac{1}{p}=\frac{1}{2}+\frac{1}{q}$ and 
$x_i=u_i|x_i|$ denotes the polar decomposition of $x_i$, 
we obtain that 
$$\|x^i\|_{L_p(\M;\ell_1^c)}
\leq \Big(\sum_{n\geq 0}\|b_n^i\|_2^2\Big)^{1/2}\Big\|\Big(\sum_{n\geq 0}|a_n^i|^2\Big)^{1/2}\Big\|_q
=\|u_i|x_i|^{p/2}\|_2\||x_i|^{p/q}\|q
\leq \|x_i\|_p,$$
for all $1\leq i \leq N$. 
Then 
$$\|x\|_{L_p(\M;\ell_1^c)}
\leq \sum_{i=1}^N \|x^i\|_{L_p(\M;\ell_1^c)}
\leq \sum_{i=1}^N \|x_i\|_p = \|x\|_{\ell_1(L_p(\M))}.$$
Since the family of finite sequences is dense in $\ell_1(L_p(\M))$, 
this shows \eqref{l1Lpembed}. 
Moreover, the density of the family of finite sequences in $L_p(\M;\ell_1^c)$ implies that $\ell_1(L_p(\M))$ is dense in  $L_p(\M;\ell_1^c)$. 
Hence there exists a sequence $y=(y_n)$ in $L_{p'}(\M)$ 
such that $\varphi(x)=\sum_n\tau(y_n^*x_n)$ for all $x=(x_n)$ in $\ell_1(L_p(\M))$. 
Then
\begin{align*}
 \|y\|_{L_{p'}(\M;\ell_\infty^c)}
 &= \|{\sup_n}^+|y_n|^2\|_{p'/2}^{1/2}\\
 &= \sup\Big\{\Big(\displaystyle\sum_n \tau\Big(|y_n|^2c_n\Big)\Big)^{1/2} : 
 c_n \in L_{q/2}^+(N), 
\Big\|\displaystyle\sum_n c_n\Big\|_{q/2}\leq 1\Big\}\\
 &= \sup\Big\{\Big(\displaystyle\sum_{n} \|y_nc_n^{1/2}\|_2^2\Big)^{1/2} : 
  c_n \in L_{q/2}^+(N), 
\Big\|\displaystyle\sum_n c_n\Big\|_{q/2}\leq 1 \Big\}\\
 &=\sup\Big\{\displaystyle\sum_{n} \tau((y_nc_n^{1/2})^*b_n) : 
  c_n \in L_{q/2}^+(N), 
\Big\|\displaystyle\sum_n c_n\Big\|_{q/2}\leq 1 , 
 \displaystyle\sum_{n}\|b_{n}\|_2^2\leq 1\Big\} \\
 &= \sup\Big\{ \varphi(x) : x_n=b_nc_n^{1/2},  c_n \in L_{q/2}^+(N), 
\Big\|\displaystyle\Big(\sum_n |c_n^{1/2}|^2\Big)^{1/2}\Big\|_{q}\leq 1 , 
 \displaystyle\sum_{n}\|b_{n}\|_2^2\leq 1\Big\} \\
&\leq 1.
 \end{align*}
Thus $y \in L_{p'}(\M;\ell_\infty^c)$. 
By density, 
the functional $\varphi$ is uniquely determined by the sequence $(y_n)$ and the duality is proved. 
\qd

Let $h_p^{1_c}$ (resp. $h_{p'}^{\infty_c}$) be the subspace of $L_p(\M;\ell_1^c)$ (resp. $L_{p'}(\M;\ell_\infty^c)$) 
consisting of all martingale difference sequences.

\begin{lemma}\label{hp1comp}
Let $1\leq p \leq \infty$. Then the discrete spaces satisfy
\begin{enumerate}
 \item[(i)] For $1\le p< 2$, $h_p^{1_c}$ is a complemented subspace of $L_p(\M;\ell_1^c)$.
 \item[(ii)] For $2<p\leq \infty$, $h_{p}^{\infty_c}$ is a complemented subspace of $L_{p}(\M;\ell_\infty^c)$.
\end{enumerate}
\end{lemma}

\begin{proof}
We first show that the Stein projection 
$$\D ((x_n)_{n\geq 0} ) = (d_n(x_n))_{n\geq 0} $$
is bounded on $L_p(\M;\ell_1^c)$ for $1\le p< 2$.
Let $(x_n)_n$ be in the unit ball of $L_p(\M;\ell_1^c)$ and let $x_n=b_na_n$ be the decomposition of $x_n$ given by Lemma \ref{ballLp1}. 
Then for each $n$ we can write
$$\E_n(x_n) = u_{n}(b_n^*)^*u_n(a_n)=\sum_{n,k} u_{n}(b_n^*)(k)^*u_n(a_n)(k),$$
where $ u_{n}(b_n^*)(k) \in L_2(\M)$ and $u_n(a_n)(k)\in L_q(\M)$. 
On the one hand, the trace preserving property of the conditional expectation gives
 $$ \sum_{n,k} \|u_{n}(b_n^*)(k)\|_2^2
 = \sum_n \tau(\E_n(b_n^*b_n)) = \sum_n \|b_n\|_2^2 .$$
On the other hand, since we have $2\leq q <\infty$ for $1\leq p<2$, the dual form of the Doob inequality yields
 $$ \Big\|\sum_{n,k} |u_n(a_n)(k)|^2 \Big\|_{q/2}
 = \Big\|\sum_n \E_n|a_n|^2\Big\|_{q/2} \leq \delta'_{q/2}
 \Big\|\sum_n |a_n|^2\Big\|_{q/2} .$$
Hence $(\E_n(x_n))_n \in L_p(\M;\ell_1^c)$ with $\|(\E_n(x_n) )_n\|_{L_p(\M;\ell_1^c)}\leq {\delta'}_{q/2}^{1/2}$,
where $\delta'_{q/2}\approx q^2$ as $q\to \infty$, $p\to 2$. 
This shows that $h_p^{1_c}$ is  $2{\delta'}_{q/2}^{1/2}$-complemented in $L_p(\M;\ell_1^c)$ for $1\leq p<2$. 

For the second assertion, the noncommutative Doob inequality and the fact that $|\E_n(x_n)|^2\leq \E_n|x_n|^2$ immediately imply that 
$h_{p}^{\infty_c}$ is $2\delta_{p/2}^{1/2}$-complemented in $L_{p}(\M;\ell_\infty^c)$.
\qd

Combining Proposition \ref{dualLp1} with Lemma \ref{hp1comp} we get the duality between $h_p^{1_c}$ and $h_{p'}^{\infty_c}$.

\begin{cor}\label{dualhp1c}
 Let $1\leq p <2$. Then the discrete spaces satisfy
$$(h_p^{1_c})^*=h_{p'}^{\infty_c} \quad \mbox{ with equivalent norms}.$$
\end{cor}

Then \eqref{decLpcMO} means that for $1\leq p <2$, we have by Corollary \ref{dualhp1c}
$$(H_p^c)^*=L_{p'}^cMO=h_{p'}^{\infty_c} \cap L_{p'}^c mo  = (h_p^{1_c} + h_p^c)^*.$$
This yields the following stronger Davis decomposition.

\begin{theorem}\label{Davish1cdisc}
Let $1\leq p <\infty$. Then the discrete spaces satisfy
$$
H_p^c=\left\{\begin{array}{cl}
h_p^{1_c}+h_p^c& \quad \mbox{for} \quad 1\leq p<2 \\
h_p^{\infty_c}\cap h_p^c& \quad \mbox{for}\quad 2\leq p<\infty
\end{array}\right.
\quad \mbox{with equivalent norms}.$$
\end{theorem}

\re\label{hp1reg} 
\begin{enumerate}
\item Observe that by interpolation between the cases $p=1$ and $p=2$ we have a contractive inclusion 
$L_p(\M;\ell_1^c)\subset \ell_p(L_p(\M))$ for $1\leq p \leq 2$. Thus, considering the martingale difference sequences, we get 
$$h_p^{1_c}\subset h_p^d \quad \mbox{contractively for } 1\leq p<2.$$ 
Hence the decomposition of Theorem \ref{Davish1cdisc} is stronger than the usual decomposition 
stated in Theorem \ref{Davishddisc}. 
\item The advantage of working with the spaces $h_p^{1_c}$ is that,
since $\M$ is finite, they satisfy the following regularity property
$$h_{\tp}^{1_c}\subset h_{p}^{1_c} \quad \mbox{contractively  for } 1\leq p \leq \tp <2 ,$$
whereas the $h_p^d$ spaces \emph{do not}.  However we loose the reflexivity property.
\end{enumerate}
\mar

\subsubsection{The version of the discrete Davis decomposition in Randrianantoanina's style}\label{subsectDRadiscr}

As for the Burkholder-Gundy inequalities in Section \ref{sectBG}, 
we will need a result due to Randrianantoanina to apply duality in the continuous setting. 
In \cite{ran-cond}, Randrianantoanina proves the following Burkholder-Rosenthal decomposition at the $L_2$-level, with simultaneous control of norms. 

\begin{theorem}\label{RaBRdiscr} 
Let $1<p<2$ and $x\in L_2(\M)$. 
Then there exist $a,b,c \in L_2(\M)$ such that 
\begin{enumerate}
\item[(i)] $x=a+b+c$,
\item[(ii)]$ \|a\|_{h_p^d}+\|b\|_{h_p^c}+\|c\|_{h_p^r} \leq C(p) \|x\|_p $,
\item[(iii)] $\max\{\|a\|_2,\|b\|_2,\|c\|_2\} \leq f(p,\|x\|_p,\|x\|_2)$.
\end{enumerate}
Here $C(p)\leq C(p-1)^{-1}$ as $p\to 1$. 
\end{theorem}

\begin{proof} 
The proof is similar to that of Theorem \ref{RaBGdiscr}.  
Let $x\in L_2(\M)$, $1<p<2$ and $0<\theta<1$ be such that $\frac1p=1-\theta +\frac{\theta}{2}$. 
As in the proof of Theorem \ref{RaBGdiscr} we may write
\begin{equation}\label{dec0'}
x = \sum_{\nu \in \Z} u_\nu
\end{equation}
with
\begin{equation}\label{Jmethod'}
\Big(\sum_{\nu \in \Z} (2^{-\nu \theta} \max\{\|u_\nu\|_1,2^\nu\|u_\nu\|_2\})^p\Big)^{1/p}\leq C(p)\|x\|_p
\end{equation}
and
\begin{equation}\label{L2estimate'}
\sum_{\nu \in \Z} \|u_\nu\|_2\leq f(p,\|x\|_p,\|x\|_2).
\end{equation}
We apply Randrianantoanina's decomposition to this sequence $(u_\nu)_\nu$.  
For each $\nu \in \Z$, by Theorem $3.1$ of \cite{ran-cond}, we may find an absolute constant $K>0$ and three adapted sequences $a^{(\nu)}, b^{(\nu)}$ and $c^{(\nu)}$ such that 
$$d_n(u_\nu)=a_n^{(\nu)}+ b_n^{(\nu)}+c_n^{(\nu)}, \quad \forall n\geq 0$$
and
$$
\|a^{(\nu)}\|_{L_2(\M;\ell_2^c)}+\|b^{(\nu)}\|_{L_2(\M;\ell_2^c)}+\|c^{(\nu)}\|_{L_2(\M;\ell_2^r)}
\leq K \|u_\nu\|_2,$$
\begin{align*}
\Big\|\sum_{n\geq 0}e_{n,n}\ten a_n^{(\nu)}\Big\|_{L_{1,\infty}(\B(\ell_2)\oten \M)}
+\Big\|\Big(\sum_{n\geq 0}\E_{n-1}|b_n^{(\nu)}|^2\Big)^{1/2}\Big\|_{1,\infty} 
+\Big\|\Big(\sum_{n\geq 0}\E_{n-1}|(c_n^{(\nu)})^*|^2\Big)^{1/2}\Big\|_{1,\infty} \\
\leq K \|u_\nu\|_1.
\end{align*}
Then we set
$$a_n=\sum_{\nu\in\Z} a_n^{(\nu)},\quad b_n=\sum_{\nu\in\Z} b_n^{(\nu)} \quad  \mbox{and} \quad c_n=\sum_{\nu\in\Z} c_n^{(\nu)},$$
and obtain three adapted sequences $a=(a_n)_n, b=(b_n)_n$ and $c=(c_n)_n$.  
Using the fact that for any semifinite von Neumann algebra $\Nc$ we have
$$L_p(\Nc)=[L_{1,\infty}(\Nc), L_2(\Nc)]_{\theta,p;J},$$
and \eqref{Jmethod}, we can show that
$$\Big(\sum_{n\geq 0} \|a_n\|_p^p\Big)^{1/p} + \Big\|\Big(\sum_{n\geq 0}\E_{n-1}|b_n|^2\Big)^{1/2}\Big\|_p
+\Big\|\Big(\sum_{n\geq 0}\E_{n-1}|c_n^*|^2\Big)^{1/2}\Big\|_p\leq C(p-1)^{-1}\|x\|_p.$$
Applying the Stein projection $\D$ to the sequences $a,b$ and $c$ we obtain three martingales. 
We set 
$$a'=\D(a),\quad b'=\D(b) \quad  \mbox{and} \quad c'=\D(c).$$
Then we have 
$$d_n(x)=d_n(a')+d_n(b')+d_n(c') \quad \forall n\geq 0.$$
Moreover, since any conditional expectation $\E$ is a contractive projection in $L_p(\M)$ 
and satisfies $\E(y)^*\E(y)\leq \E(y^*y)$, we get 
$$ \|a'\|_{h_p^d}+\|b'\|_{h_p^c}+\|c'\|_{h_p^r} \leq C'(p-1)^{-1} \|x\|_p.$$
It remains to prove the $L_2$-estimate (iii). 
This comes from \eqref{L2estimate'} by writing
$$\|a'\|_2=\|\D(a)\|_2\leq 2 \|a\|_2 \leq 2 \sum_{\nu\in \Z} \|a^{(\nu)}\|_{L_2(\M;\ell_2^c)} 
\leq 2K\sum_{\nu\in \Z} \|u_{\nu}\|_2 \leq 2Kf(p,\|x\|_p,\|x\|_2).$$
The estimates for $b'$ and $c'$ are similar.
\qd

We can derive a column version of Theorem \ref{RaBRdiscr}, which is the following version of the Davis decomposition at the $L_2$-level. 

\begin{cor}\label{RaDdiscr}
Let $(\M_n)_{n=0}^m$ be a finite filtration of $\M$. 
Let $1<p<2$ and $x\in L_2(\M)$. 
Then there exist $a,b \in L_2(\M)$ such that 
\begin{enumerate}
\item[(i)] $x=a+b$,
\item[(ii)]$ \|a\|_{h_p^d}+\|b\|_{h_p^c}\leq C(p) \|x\|_{H_p^c} $,
\item[(iii)] $\max\{\|a\|_2,\|b\|_2\} \leq f(p,\|x\|_{H_p^c},\|x\|_2)$,
\end{enumerate}
where $C(p)\leq C(p-1)^{-1}$ as $p\to 1$. 
\end{cor}

\begin{proof}
We apply Theorem \ref{RaBRdiscr} to the element
$$y=\sum_{n=0}^m e_{n,0}\ten d_n(x).$$
Here we consider the finite von Neumann algebra $\Nc=\B(\ell_2^{m+1})\oten \M$ equipped with the filtration $\Nc_n=\B(\ell_2^{m+1})\oten \M_n$. 
We have to be careful with the trace we consider on $\Nc$. 
The natural trace on $\Nc$ is $\tr_\Nc=\tr \ten \tau$, where $\tr$ denotes the usual trace on  $\B(\ell_2^{m+1})$. 
This trace is finite, but not normalized. Since Theorem $3.1$ of \cite{ran-cond} have been proved for a normalized trace, 
we will also need to consider the normalized trace $\tau_\Nc=\frac{\tr}{m+1}\ten \tau$. 
Observe that 
$$\|y\|_{L_2(\Nc,\tr_\Nc)}=\|x\|_2 \quad \mbox{and} \quad 
\|y\|_{L_p(\Nc,\tr_\Nc)}=\|x\|_{H_p^c}.$$ 
As in the proof of Theorem \ref{RaBG}, we can find a sequence $(u_\nu)_\nu$ such that $y = \sum_{\nu \in \Z} u_\nu$ with
\begin{equation}\label{Jmethody}
\Big(\sum_{\nu \in \Z} (2^{-\nu \theta} \max\{\|u_\nu\|_{L_1(\Nc,\tr_\Nc)},2^\nu\|u_\nu\|_{L_2(\Nc,\tr_\Nc)}\})^p\Big)^{1/p}
\leq C(p)\|y\|_{L_p(\Nc,\tr_\Nc)}=C(p)\|x\|_{H_p^c}
\end{equation}
and 
\begin{equation}\label{L2estimatey}
\sum_{\nu \in \Z} \|u_\nu\|_{L_2(\Nc,\tr_\Nc)}\leq f(p,\|y\|_{L_p(\Nc,\tr_\Nc)},\|y\|_{L_2(\Nc,\tr_\Nc)})
=f(p,\|x\|_{H_p^c},\|x\|_{2}).
\end{equation}
Applying Theorem $3.1$ of \cite{ran-cond} in $(\Nc,\tau_\Nc)$ for each $\nu \in \Z$, 
we may find an absolute constant $K>0$ and three adapted sequences $a^{(\nu)}, b^{(\nu)}$ and $c^{(\nu)}$ such that 
$$d_n(u_\nu)=a_n^{(\nu)}+ b_n^{(\nu)}+c_n^{(\nu)}, \quad \forall n\geq 0$$
and
\begin{equation}\label{L2estimatetauN}
\|a^{(\nu)}\|_{L_2(\Nc,\tau_{\Nc};\ell_2^c)}+\|b^{(\nu)}\|_{L_2(\Nc,\tau_\Nc;\ell_2^c)}+\|c^{(\nu)}\|_{L_2(\Nc,\tau_\Nc;\ell_2^r)}
\leq K \|u_\nu\|_{L_2(\Nc,\tau_\Nc)},
\end{equation}
\begin{equation}\label{L1inftyestimatetauN}
\begin{array}{clc}
\Big\|\displaystyle\sum_{n\geq 0}e_{n,n}\ten a_n^{(\nu)}\Big\|_{L_{1,\infty}(\B(\ell_2)\oten \Nc,\tr\ten \tau_\Nc)}
&+\Big\|\Big(\displaystyle\sum_{n\geq 0}\E_{n-1}|b_n^{(\nu)}|^2\Big)^{1/2}\Big\|_{L_{1,\infty}( \Nc,\tau_\Nc)} \\
+\Big\|\Big(\displaystyle\sum_{n\geq 0}\E_{n-1}|(c_n^{(\nu)})^*|^2\Big)^{1/2}\Big\|_{L_{1,\infty}( \Nc,\tau_\Nc)} 
&\leq K \|u_\nu\|_{L_1(\Nc,\tau_\Nc)}.
\end{array}
\end{equation}
We would like to obtain the same estimates with respect to the trace $\tr_\Nc$ to use the interpolation argument and \eqref{Jmethody}. 
Note that for $z\in L_1(\Nc)$, we have 
$$\|z\|_{L_1(\Nc,\tr_\Nc)}=(m+1)\|z\|_{L_1(\Nc,\tau_\Nc)} \quad,\quad  \|z\|_{L_{1,\infty}(\Nc,\tr_\Nc)}=(m+1)\|z\|_{L_{1,\infty}(\Nc,\tau_\Nc)}$$
and for $z\in L_2(\Nc)$ we have
$$\|z\|_{L_2(\Nc,\tr_\Nc)}=\sqrt{m+1}\|z\|_{L_2(\Nc,\tau_\Nc)}.$$
Hence multiplying \eqref{L2estimatetauN} and \eqref{L1inftyestimatetauN} by $\sqrt{m+1}$ and $(m+1)$ respectively, 
we get the same estimates with respect to the trace $\tr_\Nc$. 
Thus we may control the J-functionals for $a^{(\nu)}, b^{(\nu)}$ and $c^{(\nu)}$ in $(L_{1,\infty}(\Nc,\tr_\Nc),L_2(\Nc,\tr_\Nc))$ 
by the J-functional of $u_\nu$ in $(L_1(\Nc,\tr_\Nc),L_2(\Nc,\tr_\Nc))$, which is bounded by $C(p)\|x\|_{H_p^c}$ by \eqref{Jmethody}. 
Then applying the Stein projection we get three elements $a,b,c$ in $L_2(\Nc)$ such that 
$$y=a+b+c$$
and
$$\|a\|_{h_p^d(\Nc,\tr_\Nc)}+\|b\|_{h_p^c(\Nc,\tr_\Nc)}+\|c\|_{h_p^r(\Nc,\tr_\Nc)} \leq C(p)\|x\|_{H_p^c},$$
$$\max\{\|a\|_{L_2(\Nc,\tr_\Nc)},\|b\|_{L_2(\Nc,\tr_\Nc)},\|c\|_{L_2(\Nc,\tr_\Nc)}\} \leq f(p,\|x\|_{H_p^c},\|x\|_{2}).$$ 
Now we deduce a decomposition of $x$ satisfying (ii) and (iii) as follows. 
We consider the following projections in $\Nc$
$$e=\sum_{n\geq 0} e_{n,n}\ten 1 \quad \mbox{and} \quad f=e_{0,0}\ten 1.$$
Since $y$ has a column structure we have $y=eyf$, hence $y=eaf+ebf+ecf$.  
Writing 
$$a=\sum_{k,n\geq 0}e_{k,n}\ten a_{k,n}\quad , \quad 
b=\sum_{k,n\geq 0}e_{k,n}\ten b_{k,n}\quad \mbox{and} \quad
c=\sum_{k,n\geq 0}e_{k,n}\ten c_{k,n},$$
we have
$$eaf=\sum_{n\geq 0}e_{n,0}\ten a_{n,0}\quad , \quad 
ebf=\sum_{n\geq 0}e_{n,0}\ten b_{n,0}\quad \mbox{and} \quad
ecf=\sum_{n\geq 0}e_{n,0}\ten c_{n,0}.$$
Since $d_n(y)=e_{n,0}\ten d_n(x)$ we get 
$$d_n(x)=d_n(a_{n,0})+d_n(b_{n,0})+d_n(c_{n,0})\quad \forall n\geq 0.$$
Finally we set
$$\alpha=\sum_{n\geq 0}d_n(a_{n,0})
\quad , \quad 
\beta=\sum_{n\geq 0}d_n(b_{n,0})
\quad , \quad
\gamma=\sum_{n\geq 0}d_n(c_{n,0}),$$
and obtain three elements in $L_2(\M)$ such that $x=\alpha+\beta+\gamma$. 
It is clear that $\alpha, \beta$ and $\gamma$ verify the $L_2$-estimate (iii). 
Note that here we want a decomposition of $x$ in two elements. 
We will show that $\alpha \in h_p^d$, $\beta\in h_p^c$ and that the third element $\gamma$ is in the diagonal space $h_p^d$. 
Let us first observe that since $e,f \in \Nc_0=\B(\ell_2^{m+1})\oten \M_0$, we deduce from the module property that
\begin{equation}\label{eaf}
\|eaf\|_{h_p^d(\Nc,\tr_\Nc)}+\|ebf\|_{h_p^c(\Nc,\tr_\Nc)}+\|ecf\|_{h_p^r(\Nc,\tr_\Nc)} \leq C(p)\|x\|_{H_p^c}.
\end{equation}
Indeed, the estimate of the first term comes from the fact that $e$ and $f$ are projections, and for the second term we write
\begin{align*}
\E_{n-1}|d_n(ebf)|^2
&=\E_{n-1}|ed_n(b)f|^2
=\E_{n-1}(fd_n(b)^*ed_n(b)f)\\
&=f\E_{n-1}(d_n(b)^*ed_n(b))f\leq f\E_{n-1}|d_n(b)|^2f.
\end{align*}
Then $\|ebf\|_{h_p^c(\Nc,\tr_\Nc)}\leq \|b\|_{h_p^c(\Nc,\tr_\Nc)}$. The third term is similar. 
For the term $\alpha$ we have
\begin{align*}
\|\alpha\|_{h_p^d}
&=\Big(\sum_n \|d_n(a_{n,0})\|_p^p\Big)^{1/p}
=\Big(\sum_n \|(|d_n(a_{n,0})|^2)^{1/2}\|_p^p\Big)^{1/p}\\
&\leq \Big(\sum_n \Big\|\Big(\sum_{k} |d_n(a_{k,0})|^2\Big)^{1/2}\Big\|_p^p\Big)^{1/p} 
=\Big(\sum_n \Big\|\sum_{k} e_{k,0}\ten d_n(a_{k,0})\Big\|_{L_p(\Nc,\tr_\Nc)}^p\Big)^{1/p}\\
&= \Big(\sum_n \|d_n(eaf)\|_{L_p(\Nc,\tr_\Nc)}^p\Big)^{1/p}
=\|eaf\|_{h_p^d(\Nc,\tr_\Nc)}.
\end{align*}
We proceed similarly for the term $\beta$
\begin{align*}
\|\beta\|_{h_p^c}
&=\Big\|\Big(\sum_{n} \E_{n-1}|d_n(b_{n,0})|^2\Big)^{1/2}\Big\|_p
\leq \Big\|\Big(\sum_{n,k} \E_{n-1}|d_n(b_{k,0})|^2\Big)^{1/2}\Big\|_p\\
&=\Big\|\Big(\sum_{n} \E_{n-1}\Big|\sum_{k} e_{k,0}\ten d_n(b_{k,0})\Big|^2\Big)^{1/2}\Big\|_{L_p(\Nc,\tr_\Nc)}\\
&=\Big\|\Big(\sum_{n} \E_{n-1}|d_n(ebf)|^2\Big)^{1/2}\Big\|_{L_p(\Nc,\tr_\Nc)}
=\|ebf\|_{h_p^c(\Nc,\tr_\Nc)}.
\end{align*}
Finally for the term $\gamma$ we write 
\begin{align*}
\|\gamma\|_{h_p^d}
&=\Big(\sum_n \|d_n(c_{n,0})\|_p^p\Big)^{1/p}
=\Big\|\sum_n e_{n,n}\ten d_n(c_{n,0})\Big\|_{L_p(\Nc,\tr_\Nc)}\\
&=\Big\|\mathrm{Diag}\Big(\sum_{k,n} e_{k,n}\ten d_n(c_{k,0})\Big)\Big\|_{L_p(\Nc,\tr_\Nc)},
\end{align*}
where $\mathrm{Diag}$ denotes the diagonal projection in $\Nc$. 
Since the diagonal projection is bounded on $L_p(\Nc,\tr_\Nc)$, it remains to estimate 
\begin{align*}
\Big\|\sum_{k,n} e_{k,n}\ten d_n(c_{k,0})\Big\|_{L_p(\Nc,\tr_\Nc)}
&=\Big\|\sum_{k,n} e_{0,n}\ten e_{k,0}\ten d_n(c_{k,0})\Big\|_{L_p(\B(\ell_2^{m+1})\oten \Nc,\tr \ten \tr_\Nc)}\\
&=\Big\|\sum_{k,n} e_{0,n}\ten d_n(ecf)\Big\|_{L_p(\B(\ell_2^{m+1})\oten \Nc,\tr \ten \tr_\Nc)}\\
&=\|ecf\|_{h_p^r(\Nc,\tr_\Nc)} .
\end{align*}
Then, using  \eqref{eaf}, we deduce (ii) and the Theorem follows for the decomposition 
$$x=(\alpha+\gamma) + \beta.$$
\qd

As in Corollary \ref{BGboxplusdiscr}, Corollary \ref{RaDdiscr} can be translated by using the $\boxplus$-sum as follows.

\begin{cor}\label{BRDboxplusdiscr}
Let $1<p<2$. Then the discrete spaces satisfy
$$H_p^c=h_p^d \boxplus h_p^c\quad 
\mbox{with equivalent norms.}$$ 
\end{cor}

\subsubsection{The ''mixed`` version of the discrete Davis decomposition}\label{subsectDmixdiscr}

It is natural to wonder whether the Davis decomposition involving the regular diagonal space $h_p^{1_c}$ established in the subsection \ref{subsectDregdiscr} 
can be done with a simultaneous control of norms, in the spirit of Randrianantoanina's decompositions. 
In term of $\boxplus$-sum, we can easily establish that

\begin{theorem}\label{BRDboxplusdiscr-hp1c}
Let $1\leq p<2$. Then the discrete spaces satisfy
$$H_p^c=h_p^{1_c} \boxplus h_p^c\quad 
\mbox{with equivalent norms.}$$ 
\end{theorem}

\begin{proof}
We first look at the dual spaces and claim that if $x=(x_1,x_2)\in (h_p^{1_c}\boxplus h_p^c)^*$, then  
\begin{equation}\label{x2LpcMO}
x_2 \in (h_p^{1_c})^*\cap (h_p^c)^*=L_{p'}^cMO=(H_p^c)^*=(h_p^{1_c}+ h_p^c)^*.
\end{equation}
Then we can deduce that the quotient map $q:h_p^{1_c}\boxplus h_p^c \to h_p^{1_c}+ h_p^c$ is injective. 
Hence the two sums $h_p^{1_c}\boxplus h_p^c$ and $h_p^{1_c}+ h_p^c$ coincide isometrically, 
and the result follows from Theorem \ref{Davish1cdisc} with the same constant in the equivalence of the norms. 
To see \eqref{x2LpcMO}, we consider $x=(x_1,x_2)\in (h_p^{1_c}\boxplus h_p^c)^*$. 
Then by Lemma \ref{dualboxplus} we have $x_1\in (h_p^{1_c})^*$, 
$x_2 \in (h_p^c)^*=L_{p'}^cmo \subset L_2(\M)$ and $\langle x_1,y\rangle =\langle x_2,y\rangle$ for all $y\in L_2(\M)\cap h_p^{1_c}$. 
Hence for $y\in L_2(\M)\cap h_p^{1_c}$ we have 
$$|\langle x_2,y\rangle|=|\langle x_1,y\rangle|\leq \|x_1\|_{(h_p^{1_c})^*}\|y\|_{h_p^{1_c}}.$$
By density of $L_2(\M)\cap h_p^{1_c}$ in $ h_p^{1_c}$ we conclude that $x_2 \in (h_p^{1_c})^*$ and \eqref{x2LpcMO} follows. 
\qd

The continuous case will be more complicated, and we need to introduce some notations and prove some preliminary results in the discrete case 
to extend Theorem \ref{BRDboxplusdiscr-hp1c} to the continuous setting later in subsection \ref{subsectDmixed}. 
We can view $H_p^c$ as a subset of the conditioned column space $L_p^{\cond}(\M;\ell_2^c)$ introduced in \cite{JD}. 
Recall that for $1\leq p <\infty$ and any finite sequence $x=(x_n)_{n\geq 0}$ in $\M$, we set
$$\|x\|_{L_p^{\cond}(\M; \ell_2^c)}=\Big\|\Big(\sum_{n\geq 0}\E_n|x_n|^2\Big)^{1/2}\Big\|_p.$$
Then $\|\cdot\|_{L_p^{\cond}(\M; \ell_2^c)}$ defines a norm on the family of finite sequences of $\M$. 
We denote by $L_p^{\cond}(\M; \ell_2^c)$ the corresponding completion, 
and $H_p^c$ clearly embeds isometrically into $L_p^{\cond}(\M; \ell_2^c)$.
The $L_p^{\cond}(\M; \ell_2^c)$-norm can be characterized in an atomic way.

\begin{lemma}\label{char-Lpl2c}
Let $1\leq p <2$, $\frac{1}{p}=\frac{1}{2}+\frac{1}{q}$ and $x=(x_n)_{n\geq 0}$ be a finite sequence of $\M$. Then 
\begin{align*}\Big(\frac{p}{2}\Big)^{1/2}\inf_{\substack{x_n=b_nw_n \\ w_n \in L_q^+(\M_n)}}
\Big( \sum_{n\geq 0}\|b_n\|_2^2\Big)^{1/2}\|{\sup_n}^+ w_n^2\|_{q/2}^{1/2}
&\leq \|x\|_{L_p^{\cond}(\M; \ell_2^c)}\\
&\leq \inf_{\substack{x_n=b_nw_n \\ w_n \in L_q^+(\M_n) }}
\Big( \sum_{n\geq 0}\|b_n\|_2^2\Big)^{1/2}\|{\sup_n}^+ w_n^2\|_{q/2}^{1/2}.
\end{align*}
\end{lemma}

\begin{proof}
Recall that $\|{\sup_n}^+ w_n^2\|_{q/2}=\inf\big\{\|w\|_{q/2} : w_n^2\leq w, \forall n\geq 0\big\}$. 
We first consider a decomposition $x_n=b_nw_n$ such that $w_n \in L_q^+(\M_n)$ for all $n\geq 0$. 
Let $w\in L_{q/2}^+(\M)$ be such that $w_n^2\leq w$ for all $n\geq 0$. 
Then we may write $w_n=v_nw^{1/2}$ for all $n$, with $\|v_n\|_{\infty}\leq 1$. We obtain
\begin{align*}
&\|x\|^2_{L_p^{\cond}(\M; \ell_2^c)}\\
=&\Big\|\sum_{n\geq 0} \E_n(w_n|b_n|^2w_n)\Big\|_{p/2}
=\Big\|\sum_{n\geq 0} w_n\E_n(|b_n|^2)w_n\Big\|_{p/2}
=\Big\|w^{1/2}\Big(\sum_{n\geq 0} v_n^*\E_n|b_n|^2v_n\Big)w^{1/2}\Big\|_{p/2}\\
\leq & \Big\|\sum_{n\geq 0} v_n^*\E_n|b_n|^2v_n\Big\|_{1}\|w^{1/2}\|_{q}^2
=\Big(\sum_{n\geq 0} \tau \big((\E_n|b_n|^2)^{1/2}v_nv_n^*(\E_n|b_n|^2)^{1/2}\big)\Big)\|w\|_{q/2}\\
\leq & \Big(\sum_{n\geq 0}\|b_n\|_2^2\Big) \|w\|_{q/2}.
\end{align*}
Thus taking the infimum yields the second inequality. 
For the first one we consider a finite sequence $x=(x_n)_{n\geq 0}$ in $\M$. 
By approximation, we may assume that $B=\Big(\sum_{n\geq 0}\E_n|x_n|^2\Big)^{1/2}$ is invertible. 
For $n\geq 0$ we set $B_n=\Big(\displaystyle\sum_{0\leq k\leq n}\E_k|x_k|^2\Big)^{1/2} \in L_{p}^+(\M_n)$. 
Following \cite{bcpy}, we can show that 
\begin{equation}\label{estbcpy}
\tau(B_n^p-B_{n-1}^p)\geq \frac{p}{2}\tau\big(B_n^{p-2}(B_n^2-B_{n-1}^2)\big)
=\frac{p}{2}\tau\big(B_n^{p-2}\E_n|x_n|^2\big).
\end{equation}
Setting $w_n=B_n^{1-\frac{p}{2}}$ and $b_n=x_n B_n^{\frac{p}{2}-1}$ we get $x_n=b_nw_n$ with $w_n \in L_q^+(\M_n)$. 
Moreover, since $0<1-\frac{p}{2}\leq \frac{1}{2}$ and 
$ B_n^2\leq B^2$, we have 
$w_n^2 \leq B^{2-p}$. 
We deduce that 
$$\|{\sup_n}^+ w_n^2\|_{q/2}^{1/2}\leq \|B^{2-p}\|_{q/2}^{1/2}=\|B\|_p^{p/q}=\|x\|_{L_p^{\cond}(\M; \ell_2^c)}^{p/q}.$$
The other estimates comes from \eqref{estbcpy}
\begin{align*}
\Big(\sum_{n\geq 0}\|b_n\|_2^2\Big)^{1/2}
&= \Big(\sum_{n\geq 0}\tau\big( B_n^{\frac{p}{2}-1}|x_n|^2 B_n^{\frac{p}{2}-1}\big)\Big)^{1/2}
=\Big(\sum_{n\geq 0}\tau\big(B_n^{p-2}\E_n|x_n|^2\big)\Big)^{1/2}\\
&\leq \Big(\frac{2}{p}\Big)^{1/2}  \Big(\sum_{n\geq 0}\tau(B_n^p-B_{n-1}^p)\Big)^{1/2}
=\Big(\frac{2}{p}\Big)^{1/2} \|x\|_{L_p^{\cond}(\M; \ell_2^c)}^{p/2}.
\end{align*}
This proves the first inequality. 
\qd

We will give an explicit decomposition of $H_p^c=h_p^{1_c}+h_p^c$ by using this characterization. 
To establish the control of the norms, for technical reasons we need to recall the definition of the space $L_p(\M;\ell_1)$ introduced in \cite{JD}. 
For $1\leq p\leq 2$, a sequence $x=(x_n)_{n\geq 0}$ belongs to $L_p(\M;\ell_1)$ if there are $b_{k,n},a_{k,n}\in L_{2p}(\M)$ such that 
$x_n=\sum_{k\geq 0} b_{k,n}^*a_{k,n}$ for all $n$ and 
$$\sum_{k,n\geq 0} b_{k,n}^*b_{k,n} \in L_p(\M)\quad ,\quad \sum_{k,n\geq 0} a_{k,n}^*a_{k,n} \in L_p(\M).$$ 
Then $L_p(\M;\ell_1)$ is a Banach space when equipped with the norm 
$$\|x\|_{L_p(\M;\ell_1)} 
=\inf \Big\{\Big\|\sum_{k,n\geq 0} b_{k,n}^*b_{k,n} \Big\|_p^{1/2} \Big\|\sum_{k,n\geq 0} a_{k,n}^*a_{k,n} \Big\|_p^{1/2}\Big\},$$
where the infimum is taken over all $(b_{k,n})$ and $(a_{k,n})$ as above. 
Recall that for a positive sequence $x=(x_n)_{n\geq 0}$ we have 
$$\|x\|_{L_p(\M;\ell_1)} =\Big\|\sum_{n\geq 0} x_n \Big\|_p.$$ 
We will use the following inclusion. 

\begin{lemma}\label{Lql1subLql2c}
Let $1\leq q\leq \infty$. Then 
$$L_q(\M;\ell_1)
\subset L_q(\M;\ell_2^c) \quad \mbox{ contractively}.$$
\end{lemma}

\begin{proof}
Since the spaces $L_q(\M;\ell_1)$ and $L_q(\M;\ell_2^c)$ interpolate, it suffices to prove the result for $q=1$ and $q=\infty$. 
The case $q=1$ is clear.
For $q=\infty$, let $x_n=\sum_{k\geq 0} b_{k,n}^*a_{k,n}$ be such that 
$$\Big\|\sum_{k,n\geq 0} b_{k,n}^*b_{k,n} \Big\|_\infty\leq 1
\and \Big\|\sum_{k,n\geq 0} a_{k,n}^*a_{k,n} \Big\|_\infty\leq 1.$$
We set $a_n=\Big(\sum_k a_{k,n}^*a_{k,n}\Big)^{1/2}$. By approximation, we may assume that the $a_n$'s are invertible. 
Then considering $v_{k,n}=a_{k,n}a_n^{-1}$ and $b_n=\sum_k b_{k,n}^*v_{k,n}$, 
we can write $x_n=b_na_n$. 
Note that $\sum_{k} v_{k,n}^*v_{k,n}=1$ and 
$$\|b_n\|_{\infty}\leq \Big\|\sum_{k\geq 0} b_{k,n}^*b_{k,n}\Big\|_{\infty}^{1/2}\Big\|\sum_{k\geq 0} v_{k,n}^*v_{k,n}\Big\|_{\infty}^{1/2}\leq 1.$$ 
Then
$$\sum_{n\geq 0} x_n^*x_n =\sum_{n\geq 0} a_n^*b_n^*b_na_n 
\leq \sum_{n\geq 0} a_n^*a_n=\sum_{k,n\geq 0} a_{k,n}^*a_{k,n} \leq 1,$$
which proves the result for $q=\infty$. 
\qd

We can now establish the decomposition of an element $x\in L_2(\M)$ in $h_p^{1_c}+h_p^{c}$. 
However, in this case we cannot get directly such a decomposition in $L_2(\M)$ with a simultaneous control of $h_p$ and $L_2$ norms, 
but we are able to approximate $x$ with elements for which we have such a simultaneous control of norms. 

\begin{prop}\label{BRDexpldicr-hp1c}
Let $1\leq p<2$, $p<p_0<\frac{4}{4-p}$ and $x\in L_2(\M)$. 
Then there exist two families $(a_T)_{T\geq 0}$ and $(b_T)_{T\geq 0}$ in $L_2(\M)$ such that  
\begin{enumerate}
\item[(i)] $x=\displaystyle\lim_{T\to \infty} a_T+b_T $ in $H_{p_0}^c$,
\item[(ii)] $\|a_T\|_{h_p^{1_c}}+\|b_T\|_{h_p^{c}}\leq C(p)\|x\|_{H_p^c}$ for all $T\geq 0$,
\item[(iii)] $\max\{\|a_T\|_2,\|b_T\|_2\}\leq g(p,\|x\|_{H_p^c},T)$ for all $T\geq 0$.
\end{enumerate} 
\end{prop}

\begin{proof}
Let $x\in L_2(\M)$. Following the proof of Lemma \ref{char-Lpl2c}, we set 
$$B=\Big(\sum_{n\geq 0}|d_n(x)|^2\Big)^{1/2}, \quad 
B_n=\Big(\sum_{0\leq k\leq n}|d_k(x)|^2\Big)^{1/2}
\and w_n=B_n^{1-\frac{p}{2}}.$$
By approximation, we may assume that $B$ is invertible and set $b_n=d_n(x) B_n^{\frac{p}{2}-1}$ so that $d_n(x)=b_nw_n$ for all $n\geq 0$ and
$$\Big(\sum_{n\geq 0}\|b_n\|_2^2\Big)^{1/2}
\leq \Big(\frac{2}{p}\Big)^{1/2} \|x\|_{H_p^c}^{p/2}
\and 
\|{\sup_n}^+ w_n^2\|_{q/2}^{1/2}\leq \|x\|_{H_p^c}^{p/q}.$$
Fix $T>0$. We consider the spectral projections 
$$e_T^{(1)}=\1\Big(\Big(\sum_{n\geq 0} (w_n-w_{n-1})^2\Big)^{1/2}\leq T\Big) , 
e_T^{(2)}=\1\Big(B^{1-\frac{p}{2}}\leq T\Big) \and e_T=e_T^{(1)}\wedge e_T^{(2)}.$$ 
We set
$$a_T=\sum_{n\geq 0} d_n\big(b_n(w_n-w_{n-1})\E_n(e_T)\big) \and 
b_T=\sum_{n\geq 0} d_n\big(b_n w_{n-1}\E_{n-1}(e_T)\big) .$$
We first check that $a_T$ and $b_T$ satisfy the estimates (ii) and (iii). 
Since $h_p^{1_c}$ is complemented in $L_p(\M;\ell_1^c)$ by Lemma \ref{hp1comp}, we have for $\frac{1}{p}=\frac{1}{2}+\frac{1}{q}$
\begin{align*}
\|a_T\|_{h_p^{1_c}} 
&\leq C_p \|(b_n(w_n-w_{n-1})\E_n(e_T))_n\|_{L_p(\M;\ell_1^c)}\\
&\leq C_p \Big(\sum_{n\geq 0}\|b_n\|_2^2\Big)^{1/2} \Big\|\Big(\sum_{n\geq 0} |(w_n-w_{n-1})\E_n(e_T)|^2\Big)^{1/2}\Big\|_q.
\end{align*}
Since $w_n \in L_q^+(\M_n)$ and $w_{n-1}\leq w_n$ we have by Stein's inequality 
\begin{align*}
&\Big\|\Big(\sum_{n\geq 0} |(w_n-w_{n-1})\E_n(e_T)|^2\Big)^{1/2}\Big\|_q
=\Big\|\Big(\sum_{n\geq 0} |\E_n((w_n-w_{n-1})e_T)|^2\Big)^{1/2}\Big\|_q \\
\leq & \gamma_q \Big\|\Big(\sum_{n\geq 0} |(w_n-w_{n-1})e_T|^2\Big)^{1/2}\Big\|_q
=\gamma_q \Big\|e_T\Big(\sum_{n\geq 0} (w_n-w_{n-1})^2\Big)e_T\Big\|_{q/2}^{1/2} \\
\leq & \gamma_q \Big\|\sum_{n\geq 0} (w_n-w_{n-1})^2\Big\|_{q/2}^{1/2} 
\leq \gamma_q \Big\|\sum_{n\geq 0} (w_n-w_{n-1})\Big\|_{q} 
=\gamma_q \|B^{1-\frac{p}{2}}\|_{q}=\gamma_q \|x\|_{H_p^c}^{p/q},
\end{align*}
where the last inequality comes from Lemma \ref{Lql1subLql2c}. 
We deduce that 
$\|a_T\|_{h_p^{1_c}} \leq C_p\Big(\frac{2}{p}\Big)^{1/2}\gamma_q  \|x\|_{H_p^c}$. 
For estimating $b_T$, we will use the well-known fact 
\begin{align}\label{En-1(dn)}
\E_{n-1}|\E_n(a_n)-\E_{n-1}(a_n)|^2
&=\E_{n-1}|\E_n(a_n)|^2-|\E_{n-1}(a_n)|^2\\\nonumber
&\leq \E_{n-1}|\E_n(a_n)|^2
\leq \E_{n-1}(\E_n|a_n|^2)=\E_{n-1}|a_n|^2
\end{align}
to write
$$\|b_T\|_{h_p^{c}}\leq \Big\|\Big(\sum_{n\geq 0} \E_{n-1} |b_n w_{n-1}\E_{n-1}(e_T)|^2\Big)^{1/2}\Big\|_p.$$
Then by the same argument than the one we used in the first part of the proof of Lemma \ref{char-Lpl2c} we obtain
\begin{align*}
\|b_T\|_{h_p^{c}}
&\leq \Big(\sum_{n\geq 0}\|b_n\|_2^2\Big)^{1/2}\|{\sup_n}^+ |w_{n-1}\E_{n-1}(e_T)|^2\|_{q/2}^{1/2}\\
& \leq \Big(\sum_{n\geq 0}\|b_n\|_2^2\Big)^{1/2}\|{\sup_n}^+ w_{n-1}^2\|_{q/2}^{1/2}
\leq \Big(\frac{2}{p}\Big)^{1/2} \|x\|_{H_p^c}.
\end{align*}
This proves (ii). 
We now turn to the estimate of the $L_2$-norms. 
By definition of $e_T$ we can write
\begin{align*}
\|(w_n-w_{n-1})\E_n(e_T)\|_\infty
&=\|\E_n((w_n-w_{n-1})e_T)\|_\infty \\
& \leq \|(w_n-w_{n-1})e_T\|_\infty 
=\|e_T(w_n-w_{n-1})^2e_T\|_\infty^{1/2}\\
&\leq \|e_T^{(1)}(w_n-w_{n-1})^2e_T^{(1)}\|_\infty^{1/2}
\leq T
\end{align*}
and
\begin{align*}
\|w_{n-1}\E_{n-1}(e_T)\|_\infty
&=\|\E_{n-1}(w_{n-1}e_T)\|_\infty
\leq \|w_{n-1}e_T\|_\infty 
=\|e_Tw_{n-1}^2e_T\|_\infty^{1/2}\\
&\leq \|e_T^{(2)}w_{n-1}^2e_T^{(2)}\|_\infty^{1/2}
\leq \|e_T^{(2)}B^{2-p}e_T^{(2)}\|_\infty^{1/2}
\leq T.
\end{align*}
Thus
$$\|a_T\|_2 \leq 2 \Big(\sum_{n\geq 0}\|b_n\|_2^2\Big)^{1/2}\sup_n \|(w_n-w_{n-1})\E_n(e_T)\|_\infty
\leq 2\Big(\frac{2}{p}\Big)^{1/2} \|x\|_{H_p^c}^{p/2}T$$
and 
$$\|b_T\|_2 \leq 2 \Big(\sum_{n\geq 0}\|b_n\|_2^2\Big)^{1/2}\sup_n \|w_{n-1}\E_{n-1}(e_T)\|_\infty
\leq 2\Big(\frac{2}{p}\Big)^{1/2} \|x\|_{H_p^c}^{p/2}T. $$
We obtain (iii) with $g(p,\|x\|_{H_p^c},T)=2\Big(\frac{2}{p}\Big)^{1/2}\|x\|_{H_p^c}^{p/2}T$. 
It remains to prove the convergence (i) in $H_{p_0}^c$. 
We set
$$y_T=\sum_{n\geq 0} d_n(b_n(w_n-w_{n-1})\E_n(1-e_T)) \and 
z_T=\sum_{n\geq 0} d_n(b_n w_{n-1}\E_{n-1}(1-e_T)) .$$
Then $x-(a_T+b_T)=y_T+z_T$ and Theorem \ref{Davish1cdisc} implies
\begin{equation}\label{limT}
\|x-(a_T+b_T)\|_{H_{p_0}^c}\leq C(p_0) \big(\|y_T\|_{h_{p_0}^{1_c}}+\|z_T\|_{h_{p_0}^{c}}\big).
\end{equation}
Observe that
\begin{equation}\label{1-e}
\tau(1-e_T)\leq 2T^{-q}\|x\|_{H_p^c}^p.
\end{equation}
Indeed, since $1-e_T=1-(e_T^{(1)}\wedge e_T^{(2)})=(1-e_T^{(1)})\vee (1- e_T^{(2)})$, we have 
$$\tau(1-e_T)\leq \tau(1-e_T^{(1)})+\tau(1-e_T^{(2)}).$$
Moreover, Lemma \ref{Lql1subLql2c} yields 
\begin{align*}
\tau(1-e_T^{(1)})
&=\tau\Big(\1\Big(\Big(\sum_{n\geq 0} (w_n-w_{n-1})^2\Big)^{1/2}> T\Big)\Big) 
\leq T^{-q}\tau\Big(\Big(\sum_{n\geq 0} (w_n-w_{n-1})^2\Big)^{q/2}\Big)\\
&\leq T^{-q}\Big\|\sum_{n\geq 0} (w_n-w_{n-1})\Big\|_q^q=T^{-q}\|x\|_{H_p^c}^p,
\end{align*}
and
$$\tau(1-e_T^{(2)})=\tau\Big(\1\Big(B^{1-\frac{p}{2}}> T\Big)\Big) 
\leq T^{-q}\tau\big(B^{q(1-\frac{p}{2})}\big)
=T^{-q}\|x\|_{H_p^c}^p.
$$
This proves \eqref{1-e}. 
As for $a_T$ we can write for $\frac{1}{p_0}=\frac{1}{2}+\frac{1}{q_0}$
$$\|y_T\|_{h_{p_0}^{1_c}}
\leq C_p \Big(\sum_{n\geq 0}\|b_n\|_2^2\Big)^{1/2} \Big\|\Big(\sum_{n\geq 0} |(w_n-w_{n-1})\E_n(1-e_T)|^2\Big)^{1/2}\Big\|_{q_0}.$$
Let $s=\frac{4}{2-p}$. Since $p_0 <\frac{4}{4-p}$ we have $q_0<s$. 
Thus we can consider $q_0<r_0<\infty$ such that $\frac{1}{q_0}=\frac{1}{s}+\frac{1}{r_0}$. 
By Stein's inequality we have 
\begin{align*}
\Big\|\Big(\sum_{n\geq 0} |(w_n-w_{n-1})\E_n(1-e_T)|^2\Big)^{1/2}\Big\|_{q_0}
&=\Big\|\Big(\sum_{n\geq 0} |\E_n((w_n-w_{n-1})(1-e_T))|^2\Big)^{1/2}\Big\|_{q_0}\\
&\leq \gamma_{q_0}\Big\|\Big(\sum_{n\geq 0} |(w_n-w_{n-1})(1-e_T)|^2\Big)^{1/2}\Big\|_{q_0} \\
&=\gamma_{q_0}\Big\|(1-e_T)\Big(\sum_{n\geq 0} (w_n-w_{n-1})^2\Big)(1-e_T)\Big\|_{q_0/2}^{1/2}\\
&\leq \gamma_{q_0}\|1-e_T\|_{r_0} \Big\|\sum_{n\geq 0} (w_n-w_{n-1})^2\Big\|_{s/2}^{1/2}.
\end{align*}
Lemma \ref{Lql1subLql2c} implies
$$\Big\|\sum_{n\geq 0} (w_n-w_{n-1})^2\Big\|_{s/2}^{1/2}
\leq \Big\|\sum_{n\geq 0} (w_n-w_{n-1})\Big\|_{s} =\|B^{1-\frac{p}{2}}\|_s=\|x\|_2^{2/s}.$$
By using \eqref{1-e} we obtain 
$$\|y_T\|_{h_{p_0}^{1_c}}\leq C_p\gamma_{q_0}\Big(\frac{2}{p}\Big)^{1/2}2^{1/r_0}T^{-q/r_0}\|x\|_{H_p^c}^{p(\frac{1}{2}+\frac{1}{r_0})}\|x\|_2^{2/s}.$$
We estimate $z_T$ as we did for $b_T$ by  
$$\|z_T\|_{h_{p_0}^{c}}
\leq \Big(\sum_{n\geq 0}\|b_n\|_2^2\Big)^{1/2}\|{\sup_n}^+ |w_{n-1}\E_{n-1}(1-e_T)|^2\|_{q_0/2}^{1/2}.$$
Since $(1-e_T)w_{n-1}^2(1-e_T)\leq (1-e_T)B^{2-p}(1-e_T)$, we have 
$$ |\E_{n-1}(w_{n-1}(1-e_T))|^2 \leq \E_{n-1}|w_{n-1}(1-e_T)|^2\leq  \E_{n-1}|B^{1-\frac{p}{2}}(1-e_T)|^2$$ 
and the noncommutative Doob inequality gives
\begin{align*}
\|{\sup_n}^+ |w_{n-1}\E_{n-1}(1-e_T)|^2\|_{q_0/2} 
&=\|{\sup_n}^+ |\E_{n-1}(w_{n-1}(1-e_T))|^2\|_{q_0/2}\\
&\leq \|{\sup_n}^+ \E_{n-1}|B^{1-\frac{p}{2}}(1-e_T)|^2\|_{q_0/2}\\
&\leq \delta_{q_0/2}\||B^{1-\frac{p}{2}}(1-e_T)|^2\|_{q_0/2}
=\delta_{q_0/2}\|(1-e_T)B^{2-p}(1-e_T)\|_{q_0/2}\\
&\leq \|1-e_T\|_{r_0}^2\|B^{2-p}\|_{s/2}
\leq 2^{2/r_0}T^{-2q/r_0}\|x\|_{H_p^c}^{2p/r_0}\|x\|_2^{4/s}.
\end{align*}
Thus 
$$\|z_T\|_{h_{p_0}^{c}}\leq  \Big(\frac{2}{p}\Big)^{1/2}2^{1/r_0}T^{-q/r_0}\|x\|_{H_p^c}^{p(\frac{1}{2}+\frac{1}{r_0})}\|x\|_2^{2/s}.$$
By \eqref{limT}, we obtain that 
$$\|x-(a_T+b_T)\|_{H_{p_0}^c}
\leq C(p,p_0)T^{-q/r_0}\|x\|_{H_p^c}^{p(\frac{1}{2}+\frac{1}{r_0})}\|x\|_2^{2/s},$$
and taking the limit as $T$ tends to $\infty$ yields (i).
\qd

\re
It is important to note that for all $T\geq 0$ we obtained a uniform bound
$$\|x-(a_T+b_T)\|_{H_{p_0}^c} \leq C(p,p_0) T^{-q/r_0}\|x\|_{H_p^c}^{p(\frac{1}{2}+\frac{1}{r_0})}\|x\|_2^{2/s},$$
where $s=\frac{4}{2-p}$ and $\frac{1}{r_0}=\frac{1}{p_0}-\frac{1}{2}-\frac{1}{s}$. 
\mar

Observe that we may deduce from the proof of Proposition \ref{BRDexpldicr-hp1c} an explicit decomposition of $H_p^c=h_p^{1_c}+h_p^c$. 
This gives a constructive proof of Theorem \ref{Davish1cdisc}. 
Indeed, for $x\in L_2(\M)$ we can set 
$$x^{1_c}=\sum_{n\geq 0} d_n(b_n(w_n-w_{n-1}))
\and 
x^{c}=\sum_{n\geq 0} d_n(b_nw_{n-1}),$$
where $w_n=\Big(\sum_{0\leq k\leq n}  |d_k(x)|^2 \Big)^{\frac{1}{2}-\frac{p}{4}}$ and $b_n=d_n(x)w_n^{-1}$ 
(here we assume that $\sum_{ n}  |d_n(x)|^2$ is invertible). 
Then $x=x^{1_c}+x^c$ and it follows from the proof of Proposition \ref{BRDexpldicr-hp1c} that 
$$\|x^{1_c}\|_{h_p^{1_c}}+\|x^c\|_{h_p^{c}}\leq C(p)\|x\|_{H_p^c}.$$
In fact, this explicit decomposition can be done at the level of the column $L_p$ spaces. 
More precisely, we  can define the space $L_p^{\cond \mbox{-}}(\M; \ell_2^c)$ by setting 
$$\|x\|_{L_p^{\cond \mbox{-}}(\M; \ell_2^c)}=\Big\|\Big(\sum_{n\geq 0}\E_{n-1}|x_n|^2\Big)^{1/2}\Big\|_p$$
for $1\leq p <\infty$ and $x=(x_n)_{n\geq 0}$ a finite sequence in $\M$. 
Then we might prove constructively that 
\begin{equation}\label{decLpcond}
L_p^{\cond}(\M; \ell_2^c)=L_p(\M; \ell_1^c)+L_p^{\cond\mbox{-}}(\M; \ell_2^c) \quad \mbox{ with equivalent norms for } 1\leq p \leq 2.
\end{equation} 
Even if we will not use it in this paper, it is worth mentioning that \eqref{decLpcond} implies 
$$H_p^c \mbox{ is complemented in } L_p^{\cond}(\M; \ell_2^c) \mbox{ for } 1\leq p <\infty.$$
Indeed, this is easy to see for $1<p<\infty$ by Stein's inequality. 
For $p=1$, this follows from \eqref{decLpcond} and from the fact that $h_1^c$ is complemented in $L_1^{\cond\mbox{-}}(\M; \ell_2^c)$ (by \eqref{En-1(dn)}) 
and $h_1^{1_c}$ is complemented in $L_1(\M; \ell_1^c)$ (by Lemma \ref{hp1comp}).

\subsection{Definition of diagonal spaces for $1\leq p <2$ and basic properties}\label{subsectdiag}

We fix an ultrafilter $\U$. 
For $x\in \M$ and $1\leq p < 2$, whenever the limits exist, we define 
$$\|x\|_{\h_p^{d}}= \lim_{\si,\U} \|x\|_{h_p^d(\si)}
\quad \mbox{and} \quad 
\|x\|_{\h_p^{1_c}}= \lim_{\si,\U} \|x\|_{h_p^{1_c}(\si)}.
$$
Observe that by interpolation between the cases $p=1$ and $p=2$ and Remark \ref{hp1reg} we have 
$$\frac12 \|x\|_p\leq  \|x\|_{\h_p^{d}} \leq  \|x\|_{\h_p^{1_c}}.$$
Hence $\|\cdot\|_{\h_p^{d}}$ and $\|\cdot\|_{\h_p^{1_c}}$ define two norms for $1\leq p <2$. \\
The discrete diagonal norms also satisfy some monotonicity properties. 

\begin{lemma}\label{monotondiag}
Let $1\leq p <  2$, $x\in \M$ and $\s \subset \s'$. Then  
\begin{enumerate}
  \item[(i)] 
$\|x\|_{h_p^d(\s)}\leq 2 \|x\|_{h_p^d(\s')}.$ 
Hence 
$$ \|x\|_{\h_p^{d}} \leq \sup_\s \|x\|_{h_p^d(\s)}\leq 2 \|x\|_{\h_p^{d}}.$$
 \item[(ii)] 
$\|x\|_{h_p^{1_c}(\s)}\leq \|x\|_{h_p^{1_c}(\s')}.$ 
Hence 
$$ \|x\|_{\h_p^{1_c}} = \sup_\s \|x\|_{h_p^{1_c}(\s)}.$$
\end{enumerate}
\end{lemma} 

\begin{proof} 
Let $\si\subset \si'$. By interpolation between the cases $p=1$ and $p=2$ we have for $1\leq p \leq 2$ and $t\in \s$ 
 $$ \|d_t^\s(x)\|_p =\Big\|\sum_{s \in I_t} d_s^{\s'}(x)\Big\|_p \leq 2 \Big( \sum_{s\in I_t} \|d_s^{\s'}(x)\|_p^p\Big)^{\frac1p},$$
where $I_t$ denotes the collection of $s\in \s'$ such that $t^-(\s)\leq s^-(\s')<s\leq t$. 
Thus
 $$ \|x\|_{h_p^d(\si)}\leq 2 \|x\|_{h_p^d(\si')} .$$
For (ii), we show that for $\si\subset \si'$ we have a contractive map
 $$
 \Si: \left\{\begin{array}{ccc}
L_p(\M;\ell_1^c(\si'))&\longrightarrow &L_p(\M;\ell_1^c(\si)) \\
(x_s)_{s\in\s'}&\longmapsto & (x_t)_{t\in\s}=\Big(\displaystyle\sum_{s\in I_t} x_{s}\Big)_{t\in\s}
\end{array}\right..
$$
Since for $x\in \M$ we have 
$\Si((d_s^{\s'}(x))_{s\in \s'})=(d_t^{\s}(x))_{t\in \s}$, this will yield the required result for $\h_p^{1_c}$.
Let $x=(x_s)_{s\in\s'}$ be in the unit ball of $L_p(\M;\ell_1^c(\si'))$, then by Lemma \ref{ballLp1} we may write 
$x_s= b_s a_s$ for all $s\in \s'$ with 
$$\Big(\sum_{s\in \si'}\|b_{s}\|_2^2\Big)^{1/2}\Big\|\Big(\sum_{s\in \si'}|a_s|^2\Big)^{1/2}\Big\|_q\le 1,$$ 
where $\frac1p=\frac12+\frac1q$. 
Then $\Si(x)=\Big(\displaystyle\sum_{s\in I_t} b_s a_s \Big)_{t\in\s}$ is of the form \eqref{decLp1} with
$$\Big(\sum_{t\in \si,s\in I_t }\|b_{s}^*\|_2^2\Big)^{1/2}\Big\|\Big(\sum_{t\in \si,s\in I_t }|a_s|^2\Big)^{1/2}\Big\|_q
=\Big(\sum_{s\in \si'}\|b_{s}\|_2^2\Big)^{1/2}\Big\|\Big(\sum_{s\in \si'}|a_s|^2\Big)^{1/2}\Big\|_q \le 1.$$ 
Hence $\Si(x)$ is in the unit ball of $L_p(\M;\ell_1^c(\si))$.
\qd

\begin{cor}
Let $1\leq p < 2$. Then the norms $\|\cdot\|_{\h_p^{d}}$ and $\|\cdot\|_{\h_p^{1_c}}$ do not depend on the choice of the ultrafilter $\U$, up to a constant.
\end{cor}

\begin{defi}
Let $1\leq p < 2$. We define
$$\tilde{\h}_p^d=\{x\in L_p(\M)  :  \|x\|_{\h_p^d}<\infty \}\quad \mbox{and} \quad 
\tilde{\h}_p^{1_c}=\{x\in L_p(\M)  :  \|x\|_{\h_p^{1_c}}<\infty \}.$$
\end{defi}

Adapting the proof of Proposition \ref{injHpc} we can show that these define two Banach spaces. 
By Remark \ref{hp1reg} (1) we have 
$$\tilde{\h}_p^{1_c}\subset \tilde{\h}_p^d \quad \mbox{contractively for } 1\leq p<2.$$ 
For technical reasons these spaces are too large. Hence we need to introduce their regularized versions as follows. 
Note that by the regularity property of the $h_p^{1_c}(\s)$-spaces stated in Remark \ref{hp1reg} and the fact that $\tilde{\h}_p^{1_c}$ 
is a subspace of $L_p(\M)$, we have 
$$\tilde{\h}_{\tp}^{1_c}\subset \tilde{\h}_{p}^{1_c} \quad \mbox{contractively for } 1\leq p \leq \tp <2 .$$

\begin{defi}
Let $1\leq p < 2$. We define
$$\h_p^d=\overline{L_2(\M) \cap \tilde{\h}_p^d}^{\|\cdot\|_{\h_p^d}}
\quad \mbox{and} \quad 
\h_p^{1_c}=\overline{\bigcup_{\tp>p}\tilde{\h}_{\tp}^{1_c}}^{\|\cdot\|_{\h_p^{1_c}}}.$$
\end{defi}

\re\label{defhpd} 
\begin{enumerate}
\item At this point  it is not obvious that the set $L_2(\M)\cap \tilde{\h}_p^d$ is non trivial. We will show later that this definition
of $\h_p^d$ actually makes sense. 
\item Note that for $1\leq p \leq 2$ we have bounded inclusions
$$\h_p^d \subset \tilde{\h}_p^d \subset L_p(\M)
\quad \mbox{and} \quad
\h_p^{1_c} \subset \tilde{\h}_p^{1_c} \subset L_p(\M).$$
Since by Proposition \ref{injHpc2} we have an injective map $\H_p^c \hookrightarrow L_p(\M)$, 
this implies that the natural bounded maps 
$$\h_p^d \hookrightarrow \H_p^c
\quad \mbox{and} \quad
\h_p^{1_c} \hookrightarrow \H_p^c$$
are injective. 
Similarly, since Proposition \ref{injhpc} implies that for $1<p<2$ the natural map $\h_p^c \hookrightarrow L_p(\M)$ is injective, 
we deduce that the map
$$\h_p^c \hookrightarrow \H_p^c$$
is injective for $1<p<2$. 
For $p=1$ we have $\lh_1^c \hookrightarrow \H_1^c$. 
Hence in what follows we will consider the spaces $\h_p^d, \h_p^{1_c}$ and $\h_p^c$ as subspaces of $\H_p^c$ for $1<p<2$ 
and $\h_1^d, \h_1^{1_c}$ and $\lh_1^c$ as subspaces of $\H_1^c$.  
\end{enumerate}
\mar

\subsection{Davis decomposition for $1\leq p <2$}

Equipped with the diagonal spaces $\h_p^d$ and $\h_p^{1_c}$ defined in subsection \ref{subsectdiag} , 
we can now extend the three versions of the Davis decomposition presented in subsection \ref{subsectDdiscr} to the continuous setting. 
Since we will consider the weak limit of the discrete case, we will need the following Lemma. 

\begin{lemma}\label{limdiag}
Let $1\leq p <2$.
\begin{enumerate}
\item[(i)]
Let $p\leq \tp \leq 2$ be such that $\tp >1$ and $(x_\s)_{\s}$ be an uniformly bounded family in $L_{\tp}(\M)$. 
Then
$$\|\w L_{\tp} \mbox{-} \lim_{\s,\U} x_\s\|_{\h_p^{1_c}}
\leq \lim_{\s,\U} \|x_{\s}\|_{h_p^{1_c}(\s)}.$$
\item[(ii)]
Let $(x_\s)_{\s}$ be an uniformly bounded family in $L_2(\M)$. 
Then
$$\|\w L_2 \mbox{-} \lim_{\s,\U} x_\s\|_{\h_p^{d}}
\leq 2 \lim_{\s,\U} \|x_{\s}\|_{h_p^{d}(\s)}.$$
\item[(iii)]
Let $(x_\s)_{\s}$ be an uniformly bounded family in $L_2(\M)$. 
Then
$$\|\w L_2 \mbox{-} \lim_{\s,\U} x_\s\|_{\h_p^{c}}
\leq 2^{1/p} \lim_{\s,\U} \|x_{\s}\|_{h_p^{c}(\s)}.$$
\end{enumerate}
\end{lemma}

\begin{proof}
We first consider assertion (i) and set $x=\w L_{\tp} \mbox{-} \lim_{\s,\U} x_\s$. 
We fix a partition $\s_0$ and $\eps>0$.
We can find a sequence of positive numbers $(\alpha_m)_{m=1}^M$ such that 
$\sum_m \al_m=1$, and partitions $\si^1,...,\si^M$ containing $\si_0$  such that
 $$ \Big\|x-\sum_m \al_m x_{\si^m}\Big\|_{\tp}<\eps $$
and
$$
\|x_{\si^m}\|_{h_{p}^{1_c}(\si^m)}\leq (1+\eps) \lim_{\s,\U} \|x_\s\|_{h_{p}^{1_c}(\s)}
\quad \mbox{for all } m=1,\cdots ,M.$$
We write
\begin{align*}
 \|x\|_{h_{p}^{1_c}(\si_0)}
&\leq \Big\|x-\sum_m \al_m x_{\si^m}\Big\|_{h_{p}^{1_c}(\si_0)}+ \Big\|\sum_m \al_m x_{\si^m}\Big\|_{h_{p}^{1_c}(\si_0)} \\
&\leq 2\eps |\si_0| + \sum_m \al_m \|x_{\si^m}\|_{h_{p}^{1_c}(\si_0)}.
\end{align*}
The last inequality comes from the fact that for $1\leq p <2, z\in L_p(\M)$ and $\s_0$ a finite partition we have
\begin{equation}\label{Lpinhp1}
 \|z\|_{h_p^{1_c}(\si_0)}\leq 2|\s_0|\|z\|_p\leq 2|\s_0|\|z\|_{\tp}.
\end{equation}
Indeed, by the triangle inequality in $h_p^{1_c}(\si_0)$ we have 
$\|z\|_{h_p^{1_c}(\si_0)}\leq \sum_{t \in \s_0}\|d_{t}^{\s_0}(z)\|_{h_p^{1_c}(\si_0)}.$
We can write $\big(\delta_{s,t}d_{t}^{\s_0}(z))_{s\in \s_0}=(b_{s}a_{s}\big)_{s\in \s_0}$
with 
$$b_{s}=\delta_{s,t}v_{t}|d_{t}^{\s_0}(z)|^{p/2}\quad \mbox{and} \quad a_{s}=\delta_{s,t}|d_{t}^{\s_0}(z)|^{p/q},$$ 
where $d_{t}^{\s_0}(z)=v_{t}|d_{t}^{\s_0}(z)|$ is the polar decomposition of $d_{t}^{\s_0}(z)$ 
and $\frac{1}{p}=\frac{1}{2}+\frac{1}{q}$.
Then we obtain
$$ \|d_{t}^{\s_0}(z)\|_{h_p^{1_c}(\si_0)}\leq\|v_{t}|d_{t}^{\s_0}(z)|^{p/2}\|_2\||d_{t}^{\s_0}(z)|^{p/q}\|_q
\leq \|d_{t}^{\s_0}(z)\|_p\leq 2\|z\|_p$$
and \eqref{Lpinhp1} follows.
Since $\s_0 \subset \s^m$ we get by Lemma \ref{monotondiag}
 $$ \|x\|_{h_{p}^{1_c}(\si_0)}\leq 2|\si_0|\eps+ \sum_m \al_m
 \|x_{\s^m}\|_{h_{p}^{1_c}(\si^m)}
\leq 2|\si_0|\eps + (1+\eps)\lim_{\s,\U} \|x_\s\|_{h_{p}^{1_c}(\s)} .$$
Sending $\eps$ to $0$ and taking the supremum over $\s_0$ yields (i). 
Assertion (ii) follows similarly from the fact that the H\"{o}lder inequality in $\ell_p(\s;L_p(\M))$ gives for $z\in L_2(\M)$ and a finite partition $\s$
$$\|z\|_{h_p^d(\s)}=\|(d_t^\s(z))_{t\in \s}\|_{\ell_p(\s;L_p(\M))}\leq |\s|^{1/q}\|z\|_2$$
for $1\leq p<2$ and $\frac{1}{p}=\frac{1}{2}+\frac{1}{q}$. 
The last point may be proved with the same kind of argument, by using the fact that $\|z\|_{\h_p^c}\leq \|z\|_2$ and Lemma \ref{convexityhpc}. 
\qd

\subsubsection{The ''regular`` version of the Davis decomposition} 

The continuous analogue of Theorem \ref{Davish1cdisc} for $1\leq p<2$ is

\begin{theorem}\label{Davishp1c}
Let $1\leq p<2$. Then 
\begin{enumerate}
\item[(i)]$\H_p^c=\h_p^{1_c}+\h_p^c$ for $1<p<2$, 
\item[(ii)]$\H_1^c=\h_1^{1_c}+\lh_1^c$,
\end{enumerate}
with equivalent norms.
\end{theorem}

\begin{proof}
Let $1\leq p <2$ and $x\in \M$ be such that $\|x\|_{\H_p^c}<1$. 
By Lemma \ref{normHpc} there exists $1\leq p <\tp  <2$ such that $\|x\|_{\H_{\tp}^c}<1$. 
We apply Theorem \ref{Davish1cdisc} to each partition $\s$ and $\tp$ 
and get a decomposition $x=a_\s+b_\s$ with $a_\s\in h_{\tp}^{1_c}(\s), b_\s \in h_{\tp}^c(\s)$ and 
$$\|a_\s\|_{h_{\tp}^{1_c}(\s)}+\|b_\s\|_{h_{\tp}^c(\s)}\leq C(\tp) \|x\|_{H_{\tp}^c(\s)}.$$
Here $C(\tp)$ denotes the constant in the equivalence $H_{\tp}^c(\s)=h_{\tp}^{1_c}(\s)+h_{\tp}^{c}(\s)$. 
Hence it does not depend on $\s$ and is bounded as $\tp \to 1$. 
For each $\s$ we have 
$$\|a_\s\|_{\tp}\leq 2 \|a_\s\|_{h_{\tp}^{1_c}(\s)} \leq 2C(\tp) \|x\|_{H_{\tp}^c(\s)}.$$ 
Thus the family $(a_\s)_\s$ is uniformly bounded in the reflexive space $L_ {\tp}(\M)$ and we can consider 
$$a=\w L_{\tp} \mbox{-} \lim_{\s,\U} a_\s \in L_{\tp}(\M).$$
By Lemma \ref{limdiag} (i) we obtain   
$$
\|a\|_{\h_{\tp}^{1_c}}\leq \lim_{\s,\U} \|a_\s\|_{h_{\tp}^{1_c}(\s)}.
$$
Then we deduce that $a\in \th_{\tp}^{1_c} \subset \h_p^{1_c}$. 
We now turn to the $b$-terms. 
Since the $h_{\tp}^c(\s)$-norms are decreasing in $\s$ by Lemma \ref{convexityhpc}, 
for each $\s$ we have 
\begin{equation}\label{estimate-b}
b_\s \in h_{\tp}^c(\s)\subset \h_{\tp}^c \quad \mbox{with} \quad 
\|b_\s\|_{\h_{\tp}^c}\leq 2^{1/\tp}\|b_\s\|_{h_{\tp}^c(\s)}.
\end{equation}
Indeed, by the density of $L_2(\M)$ in $ h_{\tp}^c(\s)$ there exists a sequence $(b_\s^n)_n$ in $L_2(\M)$ which converges in $ h_{\tp}^c(\s)$ to $b_\s$. 
By Lemma \ref{convexityhpc}, $(b_\s^n)_n$ is also a Cauchy sequence in $\h_{\tp}^c$, hence converges in $\h_{\tp}^c$ to $b_\s'$. 
We get two operators $b_\s$ and $b_\s'$ in $L_{\tp}(\M)$ thanks to Proposition \ref{injhpc}, and we can easily check that 
$\tau(y^*b_\s)=\tau(y^*b_\s')$ for all $y\in L_{(\tp)'}(\M)$. Then $b_\s=b_\s' \in \h_{\tp}^c$ with 
$$\|b_\s\|_{\h_{\tp}^c}=\lim_n \|b_\s^n\|_{\h_{\tp}^c}
\leq 2^{1/\tp}\lim_n \|b_\s^n\|_{h_{\tp}^c(\s)}= 2^{1/\tp} \|b_\s\|_{h_{\tp}^c(\s)}.$$
Hence the family $(b_\s)_\s$ is uniformly bounded in the reflexive space $\h_{\tp}^c$ and we can consider 
$$b=\w \h_{\tp}^c \mbox{-} \lim_{\s,\U} b_\s \in \h_{\tp}^c.$$ 
Moreover we have 
$$\|b\|_{\h_{\tp}^c}\leq \lim_{\s,\U}\|b_\s\|_{\h_{\tp}^c}\leq 2^{1/\tp}\lim_{\s,\U}\|b_\s\|_{h_{\tp}^c(\s)}.$$
Since the family $(b_\s)_\s$ is also uniformly bounded in the reflexive space $L_{\tp}(\M)$, 
the weak-limit of the $b_\s$'s in $L_{\tp}(\M)$ exists and coincide with $b$ for $ \h_{\tp}^c\subset L_{\tp}(\M)$. 
Then we obtain $x=a+b$ with $a\in \h_p^{1_c}$ and $b\in \h_{\tp}^c \subset \h_p^c$ for $1<p<2$, $b\in \h_{\tp}^c \subset \lh_1^c$ by Remark \ref{hpcregular}. 
The above estimates give 
\begin{align*}
\|a\|_{\h_p^{1_c}}+\|b\|_{\h_p^{c}}
&\leq\|a\|_{\h_{\tp}^{1_c}}+\|b\|_{\h_{\tp}^{c}} \\
&\leq \lim_{\s,\U} \|a_\s\|_{h_{\tp}^{1_c}(\s)} +2^{1/\tp}\lim_{\s,\U}\|b_\s\|_{h_{\tp}^c(\s)}\\
& \leq 2^{1/\tp}C(\tp) \lim_{\s,\U}\|x\|_{H_{\tp}^c(\s)}
\leq 2^{1/\tp}C(\tp). 
\end{align*}
Since $2^{1/\tp}C(\tp)$ is bounded as $\tp\to 1$ we may obtain a bound independant of the choice of $\tp$, 
say $\sup_{p<\tp<1+p/2} 2^{1/\tp}C(\tp)$. 
This concludes the proof of the Theorem. 
\qd

We can now deduce the continuous analogue of Theorem \ref{Davishddisc}, i.e., the Davis decomposition involving the space $\h_p^d$. 
To do this we need to extend Remark \ref{hp1reg} (1) to the continuous setting. 
This is not trivial, it comes from the following density result based on the notion of $p$-equiintegrability. 

\begin{lemma}\label{densityhp1c}
 Let $1\leq p <2$. Then $L_2(\M)\cap\h_p^{1_c}$ is dense in $\h_p^{1_c}$.
\end{lemma}

\begin{proof}
Let $x\in \h_p^{1_c}$ and $\eps>0$. 
By definition it suffices to consider $x\in \th^{1_c}_{\tp}$ for some $1\leq p<\tp<2$. 
We suppose that $\|x\|_{\h^{1_c}_{\tp}}<C$. 
Let $\tq>q$ be such that $\frac{1}{\tp}=\frac{1}{2}+\frac{1}{\tq}$ and $\frac{1}{p}=\frac{1}{2}+\frac{1}{q}$.
By Lemma \ref{monotondiag}, for each $\s$ we can decompose $d_t^\s(x)=b_\s(t)a_\s(t)$ with
$$\Big(\sum_{t\in \s}\|b_\s(t)\|_2^2\Big)^{1/2}\Big\|\Big(\sum_{t\in \s}|a_\s(t)|^2\Big)^{1/2}\Big\|_{\tq} <C.$$
We may assume that 
$$
\Big(\sum_{t\in \s}\|b_\s(t)\|_2^2\Big)^{1/2}<1
\quad\mbox{and}\quad 
\Big\|\Big(\sum_{t\in \s}|a_\s(t)|^2\Big)^{1/2}\Big\|_{\tq} <C.$$
We set 
$$\tilde{a}_\s(t)=a_\s(t)\1\Big(\sum_{t\in \s}|a_\s(t)|^2\leq T \Big)$$
with
\begin{equation}\label{T}
T \geq \eps^{\frac{2q}{q-\tq}}C^{\frac{2\tq}{\tq-q}}.
\end{equation}
Then we have
\begin{equation}\label{atildeinfty}
\Big\|\sum_{t\in \s}|\tilde{a}_\s(t)|^2\Big\|_\infty
=\Big\|\Big(\sum_{t\in \s}|a_\s(t)|^2\Big)\1\Big(\sum_{t\in \s}|a_\s(t)|^2\leq T \Big)\Big\|_\infty \leq T
\end{equation}
and
\begin{equation}\label{a-atilde}
\begin{array}{cl}
\Big\|\Big(\displaystyle\sum_{t\in \s}|a_\s(t)-\tilde{a}_\s(t)|^2\Big)^{1/2}\Big\|_{q}
&=\Big\|\Big(\displaystyle\sum_{t\in \s}|a_\s(t)|^2\Big)\1\Big(\displaystyle\sum_{t\in \s}|a_\s(t)|^2> T \Big)\Big\|_{q/2}^{1/2} \\
&\leq \Big\|\Big(\displaystyle\sum_{t\in \s}|a_\s(t)|^2\Big)T^{1-\frac{\tq}{q}}
\Big(\displaystyle\sum_{t\in \s}|a_\s(t)|^2\Big)^{\frac{\tq}{q}-1}\Big\|_{q/2}^{1/2}\\
&= T^{\frac{q-\tq}{2q}}\Big\|\displaystyle\sum_{t\in \s}|a_\s(t)|^2\Big\|_{\tq/2}^{\frac{\tq}{2q}}\\
&\leq T^{\frac{q-\tq}{2q}}C^{\frac{\tq}{q}}<\eps.
\end{array}
\end{equation}
We set 
$$y_\s=\sum_{t\in \s} d_t^\s(b_\s(t)\tilde{a}_\s(t)).$$
By \eqref{atildeinfty} and the H\"{o}lder inequality in $L_2(\M;\ell_2^c(\s))$ we get for each $\s$
$$\|y_\s\|_2
\leq 2\Big(\sum_{t\in \s}\|b_\s(t)\tilde{a}_\s(t)\|_2^2\Big)^{1/2}\leq 
2\Big(\sum_{t\in \s}\|b_\s(t)\|_2^2\Big)^{1/2}\Big\|\Big(\sum_{t\in \s}|\tilde{a}_\s(t)|^2\Big)^{1/2}\Big\|_\infty
\leq 2T^{1/2}.$$
Hence the family $(y_\s)_\s$ is uniformly bounded in $L_2(\M)$, and we can consider 
$$y=\w L_2 \mbox{-} \lim_{\s,\U} y_\s \in L_2(\M).$$
Lemma \ref{limdiag} implies 
$$\|y\|_{\h_{\tp}^{1_c}}\leq \lim_{\s,\U} \|y_\s\|_{h_{\tp}^{1_c}(\s)}.$$
By the definition of $y_\s$ and \eqref{atildeinfty} we get 
\begin{align*}
\|y_\s\|_{h_{\tp}^{1_c}(\s)}
& \leq \Big(\sum_{t\in \s}\|b_\s(t)\|_2^2\Big)^{1/2}\Big\|\Big(\sum_{t\in \s}|\tilde{a}_\s(t)|^2\Big)^{1/2}\Big\|_{\tq}\\
& \leq \Big(\sum_{t\in \s}\|b_\s(t)\|_2^2\Big)^{1/2}\Big\|\Big(\sum_{t\in \s}|\tilde{a}_\s(t)|^2\Big)^{1/2}\Big\|_{\infty}
\leq T^{1/2},
\end{align*}
and we deduce that $y\in \th_{\tp}^{1_c}\subset \h_p^{1_c}$. 
We may adapt the proof of Lemma \ref{limdiag} (i) to show that 
\begin{equation}\label{limx-yhp1c}
\|x-y\|_{\h_p^{1_c}}\leq \lim_{\s,\U} \|x-y_\s\|_{h_{p}^{1_c}(\s)}.
\end{equation}
Indeed, for a fixed partition $\s_0$ and $\delta>0$ 
we can find a sequence of positive numbers $(\alpha_m)_{m=1}^M$ such that 
$\sum_m \al_m=1$, and partitions $\si^1,...,\si^M$ containing $\si_0$  such that
 $$ \Big\|y-\sum_m \al_m y_{\si^m}\Big\|_{2}<\delta $$
and
$$
\|x-y_{\si^m}\|_{h_{p}^{1_c}(\si^m)}\leq (1+\delta) \lim_{\s,\U} \|x-y_\s\|_{h_{p}^{1_c}(\s)}
\quad \mbox{for all } m=1,\cdots ,M.$$
Lemma \ref{monotondiag} and \eqref{Lpinhp1} give
\begin{align*}
 \|x-y\|_{h_{p}^{1_c}(\si_0)}
&\leq  \Big\|x-\sum_m \al_m y_{\si^m}\Big\|_{h_{p}^{1_c}(\si_0)} +\Big\|\sum_m \al_m y_{\si^m}-y\Big\|_{h_{p}^{1_c}(\si_0)}\\
&\leq  \Big\|\sum_m \al_m (x-y_{\si^m})\Big\|_{h_{p}^{1_c}(\si_0)} + 2 |\si_0|\Big\|y-\sum_m \al_m y_{\si^m}\Big\|_{p}\\
&\leq  \sum_m \al_m \|x-y_{\si^m}\|_{h_{p}^{1_c}(\si_0)} + 2|\si_0|\Big\|y-\sum_m \al_m y_{\si^m}\Big\|_{2} \\
&\leq  \sum_m \al_m \|x-y_{\si^m}\|_{h_{p}^{1_c}(\si^m)} + 2\delta |\si_0|\\
& \leq (1+\delta) \lim_{\s,\U} \|x-y_\s\|_{h_{p}^{1_c}(\s)}+ 2\delta |\si_0| .
\end{align*}
Sending $\delta$ to $0$ and taking the supremum over $\s_0$ we obtain \eqref{limx-yhp1c}. 
For each $\s$ we have by \eqref{a-atilde}
\begin{align*}
\|x-y_\s\|_{h_{p}^{1_c}(\s)}
&=\Big\|\sum_{t\in\s}b_\s(t)(a_\s(t)-\tilde{a}_\s(t))\Big\|_{h_{p}^{1_c}(\s)}\\
&\leq \Big(\sum_{t\in \s}\|b_\s(t)\|_2^2\Big)^{1/2}\Big\|\Big(\sum_{t\in \s}|a_\s(t)-\tilde{a}_\s(t)|^2\Big)^{1/2}\Big\|_{q}\\
&< \eps.
\end{align*}
Hence $\|x-y\|_{\h_{p}^{1_c}} < \eps$ and this ends the proof of the Lemma. 
\qd

We can now define by density a contractive map from $\h_p^{1_c}$ to $\h_p^{d}$, 
which is clearly injective for $\h_p^{1_c}$ and $\h_p^{d}$ are subspaces of $L_p(\M)$. 

\begin{cor}
Let $1\leq p< 2$. Then we have a contractive inclusion
$$\h_p^{1_c}\subset \h_p^d .$$
\end{cor}

We deduce from Theorem \ref{Davishp1c} and Remark \ref{defhpd} (2) the desired Davis decomposition. 

\begin{theorem}\label{Davishpd}
Let $1\leq p<2$. Then 
\begin{enumerate}
\item[(i)]$\H_p^c=\h_p^{d}+\h_p^c$ for $1<p<2$, 
\item[(ii)]$\H_1^c=\h_1^{d}+\lh_1^c$,
\end{enumerate}
with equivalent norms.
\end{theorem}

\subsubsection{The version of the Davis decomposition in Randrianantoanina's style}

The continuous analogue of Corollary \ref{RaDdiscr} is stated as follows 

\begin{prop}\label{RaD}
Let $1<p<2$ and $x\in L_2(\M)$. 
Then there exist $a,b \in L_2(\M)$ such that 
\begin{enumerate}
\item[(i)] $x=a+b$,
\item[(ii)]$ \|a\|_{\h_p^d}+\|b\|_{\h_p^c}\leq C(p) \|x\|_{\H_p^c} $,
\item[(iii)] $\max\{\|a\|_2,\|b\|_2\} \leq f(p,\|x\|_{\H_p^c},\|x\|_2)$,
\end{enumerate}
where $C(p)\leq C(p-1)^{-1}$ as $p\to 1$. 
\end{prop}

\begin{proof}
We again use the limit argument detailed in the previous proofs of decompositions in the continuous setting. 
We start by applying Corollary \ref{RaDdiscr} to $x\in L_2(\M)$ and obtain a decomposition $x=a_\s+b_\s$ with 
\begin{equation}\label{eqRaDdiscr}
 \|a_\s\|_{h_p^d(\s)}+\|b_\s\|_{h_p^c(\s)} \leq C(p) \|x\|_{H_p^c(\s)} \quad \mbox{and} \quad 
\max\{\|a_\s\|_2,\|b_\s\|_2\} \leq f(p,\|x\|_{H_p^c(\s)},\|x\|_2).
\end{equation}
Hence the families $(a_\s)_\s$ and $(b_\s)_\s$ are uniformly bounded in $L_2$, and we can consider 
$$a=\w L_2 \mbox-\lim_{\s,\U} a_\s \quad \mbox{and} \quad b=\w L_2 \mbox-\lim_{\s,\U} b_\s.$$
We obtain $x=a+b$ where $a,b \in L_2(\M)$ satisfy 
$$\|a\|_2\leq  \lim_{\s,\U} \|a_\s\|_2 \leq \lim_{\s,\U} f(p,\|x\|_{H_p^c(\s)},\|x\|_2)=f(p,\|x\|_{\H_p^c},\|x\|_2)$$
and similarly
$$\|b\|_2 \leq f(p,\|x\|_{\H_p^c},\|x\|_2).$$
Lemma \ref{limdiag} (ii) and (iii) give 
$$
\|a\|_{\h_p^d}\leq 2\lim_{\s,\U} \|a_\s\|_{h_p^d(\s)}
\and
\|b\|_{\h_p^c}\leq 2^{1/p}\lim_{\s,\U} \|b_\s\|_{h_p^c(\s)}.
$$
Combining with \eqref{eqRaDdiscr} we get
$$\|a\|_{\h_p^d}+\|b\|_{\h_p^c}
\leq 2(\lim_{\s,\U} \|a_\s\|_{h_p^d(\s)}+ \lim_{\s,\U} \|b_\s\|_{h_p^c(\s)})
\leq 2C(p) \lim_{\s,\U} \|x\|_{H_p^c(\s)}
\leq 2 C(p)\|x\|_{\H_p^c}.$$
\qd

\begin{cor}\label{Hpcboxplus}
Let $1<p<2$. Then 
$$\H_p^c=\h_p^d\boxplus \h_p^c = \h_p^d + \h_p^c \quad \mbox{with equivalent norms}.$$
Moreover, the constant remains bounded as $p\to 1$. 
\end{cor}

\begin{proof}
On the one hand, 
if we consider $A_0=L_2(\M), X=\h_p^d, Y=\h_p^c$ and $A_1=L_p(\M)$, 
then we may translate Proposition \ref{RaD} in terms of $\boxplus$-sum as follows  
\begin{equation}\label{Dboxplus}
\H_p^c=\h_p^d\boxplus \h_p^c  \quad \mbox{with equivalent norms}.
\end{equation}
But this holds with a constant $C(p)$ which does not remain bounded as $p\to 1$. 
On the other hand, we know by Theorem \ref{Davishpd} that 
\begin{equation}\label{D+}
\H_p^c=\h_p^d+ \h_p^c  \quad \mbox{with equivalent norms},
\end{equation}
where the constant remains bounded as $p\to 1$. 
We deduce that $\h_p^d+ \h_p^c =\h_p^d\boxplus \h_p^c $ with equivalent norms for $1<p<2$. 
Hence the two sums coincide isometrically by Lemma \ref{sumeq}. 
This means that the constant in \eqref{Dboxplus} is the same than the one in \eqref{D+}, 
hence remains bounded as $p\to 1$.
\qd


\subsubsection{The ''mixed`` version of the Davis decomposition}\label{subsectDmixed}

Lemma \ref{densityhp1c} allows us to define the sum $\h_p^{1_c}\boxplus \h_p^c$, and we may extend Theorem \ref{BRDboxplusdiscr-hp1c} to the continuous setting. 

\begin{theorem}\label{BRDboxplus-hp1c}
Let $1\leq p<2$. Then 
$$\H_p^c=\h_p^{1_c} \boxplus \h_p^c\quad 
\mbox{with equivalent norms.}$$ 
\end{theorem}

We first need the continuous analogue of Proposition \ref{BRDexpldicr-hp1c}. 

\begin{prop}\label{BRDexpl-hp1c}
Let $1 \leq p<2$, $p<p_0<\frac{4}{4-p}$ and $x\in L_2(\M)$. 
Then there exist two families $(a_T)_{T\geq 0}$ and $(b_T)_{T\geq 0}$ in $L_2(\M)$ such that  
\begin{enumerate}
\item[(i)] $x=\displaystyle\lim_{T\to \infty} a_T+b_T $ in $\H_{p_0}^c$,
\item[(ii)] $\|a_T\|_{\h_p^{1_c}}+\|b_T\|_{\h_p^{c}}\leq C(p)\|x\|_{\H_p^c}$ for all $T\geq 0$,
\item[(iii)] $\max\{\|a_T\|_2,\|b_T\|_2\}\leq C(p,\|x\|_{\H_p^c},T)$ for all $T\geq 0$.
\end{enumerate} 
\end{prop}

\begin{proof}
Let $x\in L_2(\M)$. 
By Proposition \ref{BRDexpldicr-hp1c}, for $T\geq 0$ fixed and each $\s$ we can find $a_T(\s), b_T(\s) \in L_2(\M)$ such that 
\begin{itemize}
\item $\|x-( a_T(\s)+b_T(\s))\|_{H_{p_0}^c(\s)}\leq C(p,p_0) T^{-q/r_0}\|x\|_{H_p^c(\s)}^{p(\frac{1}{2}+\frac{1}{r_0})}\|x\|_2^{2/s},$
\item $\|a_T(\s)\|_{h_p^{1_c}(\s)}+\|b_T(\s)\|_{h_p^{c}(\s)}\leq C(p)\|x\|_{H_p^c(\s)}$,
\item $\max\{\|a_T(\s)\|_2,\|b_T(\s)\|_2\}\leq g(p,\|x\|_{H_p^c(\s)}, T)\|x\|_{H_p^c(\s)}$,
\end{itemize} 
where $s=\frac{4}{2-p}$ and $\frac{1}{r_0}=\frac{1}{p_0}-\frac{1}{2}-\frac{1}{s}$. 
Since $(a_T(\s))_\s$ and $(b_T(\s))_\s$ are uniformly bounded in $L_2(\M)$, we can consider  
$$a_T=\w L_2 \mbox{-}\lim_{\s,\U} a_T(\s)
\and
b_T=\w L_2 \mbox{-}\lim_{\s,\U} b_T(\s).$$
Then the point (iii) is clear, and (ii)  follows directly from Lemma \ref{limdiag}. 
Since $x-( a_T+b_T)=\w L_2 \mbox{-}\lim_{\s,\U}(x-( a_T(\s)+b_T(\s)))$, by using Lemma \ref{convexityHpc} we can easily show that 
$$\|x-( a_T+b_T)\|_{\H_{p_0}^c}\leq \beta_p  \lim_{\s,\U} \|x-( a_T(\s)+b_T(\s))\|_{H_{p_0}^c(\s)}
\leq  C(p,p_0) T^{-q/r_0}\|x\|_{\H_p^c}^{p(\frac{1}{2}+\frac{1}{r_0})}\|x\|_2^{2/s}.$$
This gives (i) and ends the proof of the Proposition. 
\qd

\begin{proof}[Proof of Theorem \ref{BRDboxplus-hp1c}]
It suffices to prove that if $x=(x_1,x_2)\in (\h_p^{1_c}\boxplus \h_p^c)^*$, 
then 
\begin{equation}\label{x2Lpcmo}
x_2 \in L_{p'}^c\MO \quad \mbox{ with } \quad \|x_2\|_{L_{p'}^c\MO}\leq C(p)\|x\|_{(\h_p^{1_c}\boxplus \h_p^c)^*}.
\end{equation}
We will conclude by using the fact that $L_{p'}^c\MO=(\H_p^c)^*=(\h_p^{1_c}+ \h_p^c)^*$ by Theorem \ref{fsduality} and Theorem \ref{Davishp1c}. 
Let $x=(x_1,x_2)\in (\h_p^{1_c}\boxplus \h_p^c)^*$. 
Then by Lemma \ref{dualboxplus} we have $x_1\in (\h_p^{1_c})^*$, 
$x_2 \in (\h_p^c)^*=L_{p'}^c\mo \subset L_2(\M)$ and $\langle x_1,y\rangle =\langle x_2,y\rangle$ for all $y\in L_2(\M)\cap \h_p^{1_c}$. 
Furthermore 
$$\|x\|_{(\h_p^{1_c}\boxplus \h_p^c)^*}=\max\{\|x_1\|_{(\h_p^{1_c})^*} , \|x_2\|_{(\h_p^{c})^*}\}.$$
For $R\geq 0$ and $\s$ fixed we consider the projection 
$$f_R(\s)=\1\Big(\Big|\sum_{t\in \s} e_{t,t}\ten d_t^\s(x_2)^*\Big|\leq R\Big)\in B(\ell_2(\s))\overline{\ten} \M.$$ 
Then $f_R(\s)=\sum_{t\in\s} e_{t,t}\ten f_R^t(\s)$, where $f_R^t(\s)=\1(|d_t^\s(x_2)^*|\leq R)\in \M$. 
We set
$$x_2^R(\s)=\sum_{t\in \s} d_t^\s(f_R^t(\s)d_t^\s(x_2)).$$
Since $(x_2^R(\s))_\s$ is uniformly bounded in $L_2(\M)$, we can define
$$x_2^R=\w L_2 \mbox{-}\lim_{\s,\U} x_2^R(\s).$$
We will show that
\begin{enumerate}
\item[(i)] $x_2^R \in L_{p'}^c\MO$ with $\|x_2^R\|_{L_{p'}^c\MO}\leq C(p)\|x\|_{(\h_p^{1_c}\boxplus \h_p^c)^*}$ for all $R\geq 0$,
\item[(ii)] $x_2=\w L_2 \mbox{-}\displaystyle\lim_{R\to \infty} x_2^R$.  
\end{enumerate}  
Since $L_{p'}^c\MO=(\H_p^c)^*$ by Theorem \ref{fsduality} and $L_2(\M)$ is dense in $\H_p^c$, we will deduce \eqref{x2Lpcmo}. 
Let $p<p_0<\frac{4}{4-p}$. 
On the one hand, by \eqref{En-1(dn)} we get for each $\s$ 
\begin{align}\label{x2Rhpc}
\|x_2^R(\s)\|_{h_{p'}^c(\s)}
&=\Big\|\sum_{t\in \s} \E_{t^-(\s)}|d_t^\s(f_R^t(\s)d_t^\s(x_2))|^2\Big\|_{p'/2}^{1/2} \\ \nonumber
&\leq\Big\|\sum_{t\in \s} \E_{t^-(\s)}|f_R^t(\s)d_t^\s(x_2)|^2\Big\|_{p'/2}^{1/2} 
\leq \|x_2\|_{h_{p'}^c(\s)}. 
\end{align}
Proposition \ref{hpcLpcmo} implies $\|x_2^R(\s)\|_{L_{p'}^cmo(\s)}\leq C(p')\|x_2\|_{L_{p'}^cmo(\s)}$. 
On the other hand, by definition of $f_R(\s)$ we can write
\begin{align*}
\|x_2^R(\s)\|_{h_{p'}^d(\s)}
&\leq 2 \Big(\sum_{t\in \s} \|f_R^t(\s)d_t^\s(x_2)\|_{p'}^{p'}\Big)^{1/p'}
 = 2 \Big(\sum_{t\in \s} \|f_R^t(\s)|d_t^\s(x_2)^*|^2f_R^t(\s)\|_{p'/2}^{p'/2}\Big)^{1/p'}\\
&=\Big\| f_R(\s) \Big(\sum_{t\in \s} e_{t,t}\ten |d_t^\s(x_2)^*|^2\Big)f_R(\s)\Big\|_{L_{p'/2}(B(\ell_2(\s)\oten \M)}^{1/2}
\leq R.
\end{align*}
Thus $x_2^R(\s)\in h_{p'}^d(\s) \cap L_{p'}^cmo(\s) =L_{p'}^cMO(\s)$, and we can control its $L_{p'}^cMO(\s)$-norm uniformly in $\s$. 
We deduce that $x_2^R \in L_{p'}^c\MO$. 
Now we want to estimate its $L_{p'}^c\MO$-norm. 
Let $p<p_0<\frac{4}{4-p}$ and $y\in  L_2(\M)$ be such that $\|y\|_{\H_p^c}\leq 1$. 
By Proposition \ref{BRDexpl-hp1c} we can approximate $y$ in $\H_{p_0}^c$ by a family $a_T+b_T$ such that 
$a_T,b_T \in L_2(\M)$ and 
$\|a_T\|_{\h_p^{1_c}}+\|b_T\|_{\h_p^{c}}\leq C(p)\|x\|_{\H_p^c}$ for all $T\geq 0$. 
Since $x_2^R \in L_{p'}^c\MO \subset L_{p_0'}^c\MO=(\H_{p_0}^c)^*$ and $x_2^R, a_T, b_T \in L_2(\M)$, we can write
$$\langle x_2^R,y\rangle 
=\lim_{T\to \infty} \langle x_2^R,a_T+b_T\rangle 
=\lim_{T\to \infty} \langle x_2^R,a_T\rangle+\langle x_2^R,b_T\rangle .$$
By \eqref{x2Rhpc} we clearly have 
$$|\langle x_2^R,b_T\rangle |
\leq \|x_2^R\|_{(\h_p^{c})^*}\|b_T\|_{\h_p^{c}}
\leq C_p\|x_2\|_{(\h_p^{c})^*}\|b_T\|_{\h_p^{c}}.$$
Observe that for $z\in L_2(\M)$ we have 
$\langle x_2^R,z\rangle=\langle x_2,z^R\rangle$, where 
$$z^R=\w L_2 \mbox{-}\lim_{\s,\U} z^R(\s)
\and
z^R(\s)=\sum_{t\in \s} d_t^\s(f_R^t(\s)d_t^\s(z)).
$$
If in addition $z \in L_2(\M)\cap \h_p^{1_c}$, then $z^R \in L_2(\M)\cap \h_p^{1_c}$ with 
\begin{equation}\label{zR}
\|z^R\|_{\h_p^{1_c}}\leq C_p \|z\|_{\h_p^{1_c}}.
\end{equation}
Indeed, for $\s$ fixed, let $\eps>0$ and $d_t^\s(z)=b_{\s}(t)a_{\s}(t)$ be a decomposition such that 
$$\Big(\sum_{t\in \s} \|b_{\s}(t)\|_2^2\Big)^{1/2} \Big\|\Big(\sum_{t\in \s}|a_{\s}(t)|^2\Big)^{1/2}\Big\|_q 
\leq \|z\|_{h_p^{1_c}(\s)}+\eps,$$ 
where $\frac{1}{p}=\frac{1}{2}+\frac{1}{q}$. 
Then $f_R^t(\s)d_t^\s(z)=f_R^t(\s)b_{\s}(t)a_{\s}(t)$ and 
\begin{align*}
\|z^R(\s)\|_{h_p^{1_c}}
&\leq C_p\|(f_R^t(\s)d_t^\s(z))_n\|_{L_p(\M;\ell_1^c(\s))}
\leq C_p\Big(\sum_{t\in \s} \|f_R^t(\s)b_\s(t)\|_2^2\Big)^{1/2} \Big\|\Big(\sum_{t\in \s}|a_{\s}(t)|^2\Big)^{1/2}\Big\|_q\\
&\leq C_p(\|z\|_{h_p^{1_c}(\s)}+\eps).
\end{align*}
Taking the limit in $\s$ and as $\eps \to 0$ we get \eqref{zR}. 
Then applying this to $z=a_T\in L_2(\M)\cap \h_p^{1_c}$ we get 
$$|\langle x_2^R,a_T\rangle |=|\langle x_2,a_T^R\rangle |
=|\langle x_1,a_T^R\rangle |
\leq \|x_1\|_{(\h_p^{1_c})^*}\|a_T^R\|_{\h_p^{1_c}}
\leq C_p\|x_1\|_{(\h_p^{1_c})^*}\|a_T\|_{\h_p^{1_c}}.
$$
Finally we obtain
\begin{align*}
|\langle x_2^R,y\rangle |
&\leq C_p \lim_{T\to \infty}\big(\|a_T\|_{\h_p^{1_c}}+\|b_T\|_{\h_p^{c}}\big)
\max\{\|x_1\|_{(\h_p^{1_c})^*} , \|x_2\|_{(\h_p^{c})^*}\}\\
&\leq C'_p \|y\|_{\H_p^c}
\max\{\|x_1\|_{(\h_p^{1_c})^*}, \|x_2\|_{(\h_p^{c})^*}\}.
\end{align*}
Hence by density of $L_2(\M)$ in $\H_p^c$ we deduce (i). 
It remains to prove the convergence (ii). 
We start by proving that $x_2=\w L_r \mbox{-}\lim_{R\to \infty} x_2^R$ for all $r>2$. 
Since the family $(x_2^R)_R$ is uniformly bounded in $L_2(\M)$, the weak limit exists in $L_2$, and necessarily coincides with the weak limit in $L_r$. 
Hence we will deduce (ii). 
Let $r>2$ and $y\in L_r(\M)$. 
We prove that for $\s$ fixed,
\begin{equation}\label{cvx2Rs}  
|\langle x_2-x_2^R(\s),y\rangle|\leq  C_r R^{-2/u}\|x_2\|_{2}^{1+2/u}\|y\|_r,
\end{equation}
where $\frac{1}{2}=\frac{1}{r}+\frac{1}{u}$. 
We will conclude that 
$$|\langle x_2-x_2^R,y\rangle|\leq \lim_{\s,\U}|\langle x_2-x_2^R(\s),y\rangle|\leq  C_r R^{-2/u}\|x_2\|_{2}^{1+2/u}\|y\|_r,$$
which trivially tends to $0$ as $R$ goes to $\infty$. 
To prove \eqref{cvx2Rs} we use 
$$\tr \circ \tau(1-f_R(\s))\leq R^{-2}
\tr \circ \tau \Big(\sum_{t\in \s} e_{t,t}\ten |d_t^\s(x_2)^*|^2\Big)=R^{-2}\|x_2\|_2^2.$$
We write 
\begin{align*}
|\langle x_2-x_2^R(\s),y\rangle|
&= \sum_{t\in \s} \tau \big(d_t^\s(x_2)^*d_t^\s((1-f_R^t(\s))d_t^\s(y))\big)\\
&\leq 2\|x_2\|_2 \Big(\sum_{t\in \s} \|(1-f_R^t(\s))d_t^\s(y)\|_2^2\Big)^{1/2} \\
&=2\|x_2\|_2 \Big\|(1-f_R(\s))\Big(\sum_{t\in \s} e_{t,t}\ten |d_t^\s(y)^*|^2\Big)(1-f_R(\s))\Big\|_{L_1(\B(\ell_2(\s))\overline{\ten}\M)}^{1/2} \\
&\leq  2\|x_2\|_2\|1-f_R(\s)\|_{{L_{u}(\B(\ell_2(\s))\overline{\ten}\M)}}
\Big\|\sum_{t\in \s} e_{t,t}\ten |d_t^\s(y)^*|^2\Big\|_{L_{r/2}(\B(\ell_2(\s))\overline{\ten}\M)}^{1/2} \\
&\leq 2R^{-2/u}\|x_2\|_2^{1+2/u}\|y\|_{h_r^d(\s)} 
\leq 2C_r R^{-2/u}\|x_2\|_2^{1+2/u}\|y\|_{r},
\end{align*}
where the last inequality comes from the continuous inclusion $L_r(\M) \subset h_r^d(\s)$. 
\qd

Since $\h_p^{1_c}\subset \h_p^d$ contractively, we can deduce from Theorem \ref{BRDboxplus-hp1c} a new proof of Corollary \ref{Hpcboxplus} which allows us to extend it to the case $p=1$. 

\begin{cor}\label{DboxplusH1c}
We have 
$$\H_1^c=\h_1^d \boxplus \h_1^c= \h_1^d + \lh_1^c \quad \mbox{with equivalent norms.}$$
\end{cor}

\begin{proof}
We consider the following bounded maps
$$\H_1^c=\h_1^{1_c}\boxplus \h_1^c
\stackrel{\varphi_1}{\longrightarrow } \h_1^{d}\boxplus \h_1^c
\stackrel{\varphi_2}{\longrightarrow} \h_1^{d}\boxplus \lh_1^c
\stackrel{\varphi_3}{\longrightarrow} \h_1^{d}+ \lh_1^c=\H_1^c.$$
Here the first equality comes from Theorem \ref{BRDboxplus-hp1c}, the last one from Theorem \ref{Davishpd} (ii), 
the map $\varphi_1$ comes from the contractive inclusion $\h_1^{1_c}\subset \h_1^d$, $\varphi_2$ from the quotient map $\h_1^c\to \lh_1^c$ and 
$\varphi_3$ is the quotient map described in Lemma \ref{quot1}. 
Since this composed map coincides with the identity on $L_2(\M)$, the result follows by density. 
\qd

\re
In fact Corollary \ref{DboxplusH1c} could be proved directly. Indeed, in the very recent paper \cite{RandX}, Randrianantoanina and Xu give a constructive proof of Corollary \ref{RaDdiscr} for $p=1$ in the discrete setting. 
Then Proposition \ref{RaD} can be easily extended to the case $p=1$ with slight modifications, and Corollary \ref{DboxplusH1c} follows directly. 
\mar 

Corollary \ref{DboxplusH1c} leads naturally to the definition

\begin{defi}
We define 
$$\h_1=\h_1^d \boxplus \h_1^c \boxplus \h_1^r.$$
\end{defi}

Eventually, combining Corollary \ref{DboxplusH1c} and Theorem \ref{BRDboxplus-hp1c} with Proposition \ref{boxplusH1}, 
we obtain a continuous analogue of \eqref{H1h1discr}. 

\begin{theorem}\label{H1-h1}
We have 
$$\H_1=\h_1= \h_1^{1_c} \boxplus\h_1^{1_r} \boxplus \h_1^c \boxplus \h_1^r
\quad \mbox{with equivalent norms}.$$
\end{theorem}

\re
The decomposition given by Theorem \ref{H1-h1} yields an ``atomic" characterization of the space $\H_1$ which could be useful for applications. 
Indeed, this provides a nice way for proving that an operator $x\in L_2(\M)$ is in $\BMO$. 
It suffices to test $x$ against the ``infinitesimal atoms" given by the definition of $\h_1^{1_c}$, $\h_1^{1_r}$ 
and the discrete atoms of $\h_1^c, \h_1^r$ introduced in \cite{bcpy} (since the $\h_1^c$ and $\h_1^r$-norms are infimum by Lemma \ref{convexityhpc}). 
\mar

\subsection{Davis inequalities for $2\leq p <\infty$}

We now want to extend Theorem \ref{Davishddisc} to the continuous setting for $2\leq p <\infty$. 
As we did in subsection \ref{BGp>2} for proving the noncommutative Burkholder-Gundy inequalities for $2\leq p <\infty$, we will use a dual approach. 
This is why we need, as in this latter case, the version of the Davis decomposition in  Randrianantoanina's style proved in Proposition \ref{RaD}. 
Moreover, we need to discuss the dual space of the diagonal space $\h_p^d$ for $1\leq p<2$. 
That is a very delicate point, and actually we won't describe this dual. 
However, we define a smaller space $\J_p^d$ for $2< p \leq \infty$ which will play the role of the diagonal space in the Davis inequalities. 

\begin{defi}
Let $2<p\leq \infty$. 
We define the space $\J_p^d$ as the space whose closed unit ball is given by the absolute convex set
$$B_{\J_p^d}= \overline{\{x \in L_2(\M) :  \lim_{\s,\U} \|x\|_{h_p^d(\s)}\leq 1,
\|x\|_{2}\leq 1\}}^{\|\cdot\|_2}.$$
Then the norm in $\J_p^d$ is given by 
$$\|x\|_{\J_p^d}=\inf\{ C \geq 0  :  x \in CB_{\J_p^d}\}.$$
\end{defi}

Lemma \ref{banach} ensures that this defines well a Banach space. 
We may naturally introduce the seminorm
$$\|x\|_{\h_p^{d}}=\lim_{\s,\U} \|x\|_{h_p^d(\s)}$$
for $2<p\leq \infty$ and $x\in L_p(\M)$. 
By interpolation between the cases $p=2$ and $p=\infty$, we have
$$ \|x\|_{\h_p^{d}}\leq 2 \|x\|_p.$$
In this situation we also have some monotonicity properties. 

\begin{lemma}\label{monotondiag2}
Let $2<p\leq \infty$, $x\in L_p(\M)$ and $\s \subset \s'$. Then  
$$\|x\|_{h_p^d(\s')}\leq 2 \|x\|_{h_p^d(\s)}.$$
Hence 
$$\frac12 \|x\|_{\h_p^{d}}\leq \inf_\s \|x\|_{h_p^d(\s)} \leq \|x\|_{\h_p^{d}} .$$
\end{lemma}

\begin{proof}
The proof is similar to that of Lemma \ref{monotondiag}, hence we omit the details. 
\qd

As a direct consequence, we see that the seminorm $\|\cdot\|_{\h_p^{d}}$ and the space $\J_p^d$ 
do not depend on the choice of the ultrafilter $\U$, up to a constant. 
We can now state our continuous version of the Davis inequalities for $2\leq p <\infty$. 	

\begin{theorem}\label{Dp>2}
Let $2\leq p<\infty$. Then
$$\H_p^c=\J_{p}^d\cap \h_p^c \quad \mbox{with equivalent norms}.$$
Moreover, the constant remains bounded as $p\to \infty$.
\end{theorem}

\begin{proof}
We clearly have a continuous inclusion $\H_p^c\subset \h_p^c$ and a contractive inclusion $\H_p^c\subset \J_p^d$. 
Indeed, let $x\in \M$ be such that $\|x\|_{\H_p^c}\leq 1$. 
Then 
$$\lim_{\s,\U} \|x\|_{\h_p^d(\s)} \leq \lim_{\s,\U} \|x\|_{\H_p^c(\s)} \leq 1 
\quad \mbox{and} \quad 
\|x\|_2\leq \|x\|_{\H_p^c}\leq 1.$$
This means that $x\in B_{\J_p^d}$. 
Conversely, let $x\in \J_{p}^d\cap \h_p^c$ be of norm $\leq 1$. 
We can write
$$x=L_2\mbox{-} \lim_n x_n = \h_p^c\mbox{-} \lim_n x'_n,$$
where the sequence $(x_n)_n$ satisfies $\|x_n\|_{\h_p^d}\leq 1, \|x_n\|_2\leq 1$ for all $n$, 
and  $(x'_n)_n$ is a sequence in $L_p(\M)$. 
Recall that by Corollary \ref{dual-intpolHpc} and Lemma \ref{dualHpc-direct} we have 
$$(\H_p^c)^*=\H_{p'}^c=\{x\in L_{p'}(\M) : \|x\|_{\H_p^c}<\infty\}$$
with equivalent norms. 
Hence by the density of $L_2(\M)$ in $\H_{p'}^c$ it suffices to estimate 
$|\tau(x^*y)|$ for $y\in L_2(\M), \|y\|_{\H_{p'}^c}\leq 1$. 
By Proposition \ref{RaD} we may decompose $y=a+b$ with $a,b \in L_2(\M)$ and 
$$\|a\|_{\h_p^d}+\|b\|_{\h_p^c}\leq C(p).$$
Then
$$|\tau(x^*y)|
\leq |\tau(x^*a)|+|\tau(x^*b)| 
 \leq \lim_n |\tau(x_n^*a)|+\lim_n|\tau((x'_n)^*b)| .$$
For each $\s$ we have
$$|\tau(x_n^*a)| \leq \|x_n\|_{h_p^d(\s)}\|a\|_{h_{p'}^d(\s)}
\quad \mbox{and} \quad
|\tau((x'_n)^*b)|\leq \|x'_n\|_{h_p^c(\s)}\|b\|_{h_{p'}^c(\s)}.$$
Taking the limit over $\s$ yields 
$$|\tau(x_n^*a)| \leq \|x_n\|_{\h_p^d}\|a\|_{\h_{p'}^d}\leq \|a\|_{\h_{p'}^d}
\quad \mbox{and} \quad
|\tau((x'_n)^*b)|\leq \|x'_n\|_{\h_p^c}\|b\|_{\h_{p'}^c}. $$
Hence we get
\begin{align*}
|\tau(x^*y)|
&\leq \|a\|_{\h_{p'}^d}+ \lim_n \|x'_n\|_{\h_p^c}\|b\|_{\h_{p'}^c}
=\|a\|_{\h_{p'}^d}+ \|x\|_{\h_p^c}\|b\|_{\h_{p'}^c}\\
&\leq \|a\|_{\h_p^d}+\|b\|_{\h_p^c}\leq C(p).
\end{align*}
Moreover the constant $C(p)$ remains bounded as $p\to 1$ thanks to Corollary \ref{Hpcboxplus}. 
\qd

We presented above a direct proof of Theorem \ref{Dp>2}, but this does not explain where does the space $\J_p^d$ come from. 
This is why we detail below the whole argument, which highlights the construction of the space $\J_p^d$. 
Moreover, we will use this construction in the sequel. 


The delicate point here is to describe the dual space of the diagonal space $\h_p^d$ for $1<p<2$. 
Since we are only interested in the dual of the sum $\h_p^d\boxplus \h_p^c$, 
the key trick is to replace $\h_p^d$ in this sum by a nicer space, without changing the $\boxplus$-sum. 
We first observe that since $L_2(\M)$ is dense in $\h_p^c$, we have

\begin{lemma}\label{obshpc}
Let $1\leq p <2$. Then 
\begin{enumerate}
\item[(i)] $\h_p^c=L_2(\M)\boxplus \h_p^c$ for $1<p<2$,
\item[(ii)] $\lh_1^c=L_2(\M)\boxplus \lh_1^c$ 
\end{enumerate}
isometrically.
\end{lemma}

\begin{proof}
For $1<p<2$, we consider
$$A_0=L_2(\M),  X=L_2(\M), Y=\h_p^c  \mbox{ and }  A_1=L_p(\M).$$
By the density of $L_2(\M)$ in $\h_p^c$ it suffices to see that $\|x\|_{\h_p^c}=\| x\|_{L_2(\M)\boxplus \h_p^c }$ 
for all $x\in L_2(\M)$. 
Let $x\in L_2(\M)$. It is clear that $\| x\|_{L_2(\M)\boxplus \h_p^c }\leq \|x\|_{\h_p^c}$. 
Conversely, we assume $\| x\|_{L_2(\M)\boxplus \h_p^c }<1$. Then there exist $a,b \in L_2(\M)$ such that 
$$x=a+b \quad \mbox{and} \quad \|a\|_2+\|b\|_{\h_p^c}<1.$$
By the H\"{o}lder inequality we get
$$\|x\|_{\h_p^c}\leq \|a\|_{\h_p^c}+\|b\|_{\h_p^c} \leq \|a\|_{2}+\|b\|_{\h_p^c} <1.$$ 
Since $L_2(\M)$ is dense in $\lh_1^c$ and $\lh_1^c$ embeds into $L_1(\M)$, the proof for $p=1$ is similar. 
\qd

The idea is to add the space $L_2(\M)$ to $\h_p^d$ to obtain a new larger diagonal space, in which $L_2(\M)$ will be dense, 
and which will preserve the $\boxplus$-sum with $\h_p^c$. 
Hence we introduce the following space, which will play the role of $\h_p^d$ in the sequel. 

\begin{defi}
Let $1\leq p <2$. We define 
$$\K_p^d=\h_p^d\boxplus L_2(\M),$$
i.e., $\K_p^d$ is the completion of $L_2(\M)$ with respect to the norm
$$\|x\|_{\K_p^d}=\inf_{x=a+b, a \in L_2(\M)\cap \h_p^d, b \in L_2(\M)} \|a\|_{\h_p^d}+\|b\|_2.$$
\end{defi}

Note that in this application we consider
$$A_0=L_2(\M),  X=\h_p^d, Y=L_2(\M)  (\mbox{and }  A_1=L_p(\M)).$$ 
By the definition of $\h_p^d$, these spaces satisfy the density assumption \eqref{density} (moreover $X$ and $Y$ embed continuously into $A_1$). 
By working a little bit more we can prove that the space $\K_p^d$ embeds into $L_p(\M)$. 
The discrete analogue of $\K_p^d$ is the space $K_p^d(\s)=h_p^d(\s)\boxplus L_2(\M)$, defined as the completion of $L_2(\M)$ with respect to the norm
$$\|x\|_{K_p^d(\s)}=\inf_{x=a+b, a \in L_2(\M), b \in L_2(\M)} \|a\|_{h_p^d(\s)}+\|b\|_2.$$
Observe that since we consider finite partitions, the norm $\|\cdot\|_{h_p^d(\s)}$ is equivalent to the norm $\|\cdot\|_p$ for $1\leq p <2$. 
Hence for a finite partition $\s$, $K_p^d(\s)$ is $L_p(\M)$ equipped with the norm $\|\cdot\|_{K_p^d(\s)}$. 
 
\begin{lemma}\label{eqKpd}
Let $1\leq p <2$ and $x\in L_2(\M)$. Then 
$$ \frac12 \|x\|_{\K_p^d}\leq  \lim_{\s,\U}\|x\|_{K_p^d(\s)}\leq  \|x\|_{\K_p^d}.$$
Moreover the map $i_\U:x\in L_2(\M) \mapsto (x)^\bullet $ extends to a contractive injective map
$$i_\U:K_p^d\to \prodd_\U K_p^d(\s).$$
\end{lemma}
 
\begin{proof}
Let $x\in L_2(\M)$. 
It is obvious that 
$$\lim_{\s,\U}\|x\|_{K_p^d(\s)}\leq  \|x\|_{\K_p^d}.$$
Conversely, we assume $\lim_{\s,\U}\|x\|_{K_p^d(\s)}<1$. We may suppose that $\|x\|_{K_p^d(\s)}<1$ for all $\s$. 
Then for each $\s$ there exist $a(\s),b(\s) \in L_2(\M)$ such that 
$$x=a(\s)+b(\s) \quad \mbox{and} \quad \|a(\s)\|_{h_p^d(\s)}+\|b(\s)\|_2<1.$$
Note that
$$\|a(\s)\|_2=\|x-b(\s)\|_2 \leq \|x\|_2+1.$$
Hence the families $(a(\s))_\s$ and $(b(\s))_\s$ are uniformly bounded in $L_2(\M)$, and we can consider 
$$a=\w L_2 \mbox{-}\lim_{\s,\U} a(\s) \quad \mbox{and} \quad b=\w L_2 \mbox{-}\lim_{\s,\U} b(\s).$$
Then we may write
$$x=a+b,$$
where $a \in L_2(\M)\cap \h_p^d, b \in L_2(\M)$ satisfy by Lemma \ref{limdiag} (ii)
$$\|a\|_{\h_p^d}+\|b\|_{2}
\leq 2 \lim_{\s,\U} \big( \|a(\s)\|_{h_p^d(\s)}+\|b(\s)\|_{2}\big)\leq 2.$$
We obtain
$$ \|x\|_{\K_p^d}\leq  2\lim_{\s,\U}\|x\|_{K_p^d(\s)}.$$
\qd

Note that by Lemma \ref{monotondiag} we have
$$\|x\|_{K_p^d(\s)}\leq 2\|x\|_{K_p^d(\s')}$$
for $\s \subset \s'$ and $x\in L_2(\M)$. Hence 
$$\lim_{\s,\U}\|x\|_{K_p^d(\s)}\leq \sup_\s \|x\|_{K_p^d(\s)} \leq 2\lim_{\s,\U}\|x\|_{K_p^d(\s)}.$$
This means that the norm $\|\cdot\|_{\K_p^d}$ is equivalent to $\sup_\s \|\cdot\|_{K_p^d(\s)}$. 
Thus adapting the proof of Proposition \ref{injHpc} and using Lemma \ref{eqKpd} we can show that

\begin{lemma}\label{injKpd}
Let $1\leq p<2$. Then 
\begin{enumerate}
\item[(i)] $\{x\in L_p(\M) : \|x\|_{\K_p^d}<\infty \}$ is complete.
\item[(ii)] $\K_p^d$ embeds injectively into $L_p(\M)$. 
\end{enumerate}
\end{lemma}

Observe that by Lemma \ref{sumeq}, we deduce that in fact $\K_p^d=\h_p^d+L_2(\M)$ isometrically. 
We can now consider 
$$A_0=L_2(\M),  X=\K_p^d, Y=\h_p^c  (\mbox{and }  A_1=L_p(\M)).$$ 
The associativity of $\boxplus$ combined with Lemma \ref{obshpc} yield that 
$\K_p^d$ preserves the $\boxplus$-sum with $\h_p^c$ in the following sense. 

\begin{lemma}\label{boxplusKpd}
Let $1\leq p <2$. Then 
\begin{enumerate}
\item[(i)] $\h_p^d\boxplus \h_p^c=\K_p^d \boxplus \h_p^c $ for $1<p<2$,
\item[(ii)] $\h_1^d\boxplus \lh_1^c=\K_1^d \boxplus \lh_1^c $
\end{enumerate}
isometrically.
\end{lemma}

\begin{proof}
By associativity, Lemma \ref{obshpc} gives for $1<p<2$ 
$$\h_p^d\boxplus (L_2(\M) \boxplus \h_p^c)
=(\h_p^d\boxplus L_2(\M)) \boxplus \h_p^c
=\K_p^d \boxplus \h_p^c.$$
The proof for $p=1$ is the same. 
\qd

At this point we have our new candidate $\K_p^d$ for the diagonal space. 
Indeed, interchanging $\h_p^d$ to $\K_p^d$ does not affect the $\boxplus$-sum with $\h_p^c$. 
Moreover $L_2(\M)$ is dense in $\K_p^d$, and this will help us for describing its dual space as the space $\J_p^d$ introduced previously. 
We first need to give another description of $\J_p^d$. 
In the discrete case, for a finite partition $\s$ and $2<p\leq \infty$ we define   
$J_p^d(\s)$ as the space $L_p(\M)$ equipped with the norm
$$\|x\|_{J_p^d(\s)}=\max(\|x\|_{h_p^d(\s)},\|x\|_2).$$
By Lemma \ref{monotondiag2}, it is clear that for $2<p\leq \infty, x\in L_p(\M)$ and $\s\subset \s'$ we have 
$$\|x\|_{J_p^d(\s')}\leq 2\|x\|_{J_p^d(\s)}.$$
For $1\leq p<2$, the discrete duality $h_p^d(\s)$-$h_{p'}^d(\s)$ implies
$$(K_p^d(\s))^*=J_{p'}^d(\s)\quad \mbox{with equivalent norms}.$$
Moreover, 
$$\frac12 \|x\|_{J_{p'}^d(\s)}\leq \|x\|_{(K_p^d(\s))^*}\leq \|x\|_{J_{p'}^d(\s)}.$$
Observe that the space $\J_p^d$ may be characterized similarly to the space $L_p^c\MO$ as follows. 

\begin{lemma}\label{ballJhpd}
Let $2<p\leq \infty$. 
\begin{enumerate}
\item[(i)] For $2<p<\infty$, the unit ball of $\J_p^d$ is equivalent to 
$$\mathbb{B}_p=\{x=\w L_2 \mbox{-} \lim_{\s,\U} x_\s  :  \lim_{\s,\U} \|x_\s\|_{J_p^d(\s)}\leq 1\}.$$   
\item[(ii)] The unit ball of $J_\infty^d$ is equivalent to 
$$\mathbb{B}_\infty=\overline{\{x=\w L_2\mbox{-} \lim_{\si,\U} x_\si \mbox{ in } L_2  :  \lim_{\s,\U} \|x_{\si}\|_{J_\infty^d(\s)} \leq 1\}}^{\|\cdot\|_2}.$$
\end{enumerate}
\end{lemma}

\begin{proof}
Since the discrete $J_p^d(\s)$-norms are decreasing in $\s$ (up to a constant $2$), we may adapt the proof of Proposition \ref{ballLpcMO} 
and obtain that $\mathbb{B}_p$ is equivalent to 
$$\overline{\{x \in L_2(\M) :  \lim_{\s,\U} \|x\|_{J_p^d(\s)}\leq 1 \}}^{\|\cdot\|_2}.$$
Moreover, it is clear that for $x\in L_p(\M)$ 
$$ \lim_{\s,\U} \|x\|_{J_p^d(\s)}\simeq_2 \max(\|x\|_{\h_p^d},\|x\|_2) .$$
We obtain that $\mathbb{B}_p$ is equivalent to $B_{\J_p^d}$ for $2<p\leq \infty$. 
\qd

This characterization describes the dual space of $\K_p^d$. 

\begin{lemma}\label{dualKpd}
Let $1\leq p <2$. Then 
$$(\K_p^d)^*=\J_{p'}^d \quad \mbox{with equivalent norms}.$$
\end{lemma}

\begin{proof}
The proof is similar to that of Theorem \ref{fsduality}. 
Indeed the description of the space $\J_{p'}^d$ given in Lemma \ref{ballJhpd} is similar to that of the space $L_{p'}^c\MO$. 
The contractive inclusion $ \J_{p'}^d  \subset (\K_p^d)^*$ follows easily from the discrete duality $(K_p^d(\s))^*=J_{p'}^d(\s)$ 
and the density of $L_2(\M)$ in $\K_p^d$. 
For the reverse inclusion, recall that by Lemma \ref{eqKpd} the space $\K_p^d$ embeds into $\prodd_\U K_p^d(\s)$, 
and $ \|x\|_{\K_p^d}\leq  2\lim_{\s,\U}\|x\|_{K_p^d(\s)}$. 
Hence by the Hahn-Banach Theorem we may extend a linear functional on $\K_p^d$ of norm less than one 
to a linear functional on $\prodd_\U K_p^d(\s)$ of norm less than two. 
Then we use the same argument as in the proof of Theorem \ref{fsduality}. 
The crucial point here is that 
\begin{equation}\label{propKpd}
L_2(\M) \mbox{ is dense in } \K_p^d \quad \mbox{and} \quad \|x\|_2\leq \|x\|_{J_p^d(\s)}.
\end{equation}
\qd

\re
The same argument does not work if the observation \eqref{propKpd} is not verified. 
This explains why we cannot easily describe similarly the dual space of $\h_p^d$ for $1\leq p <2$, and justifies the introduction of the spaces $\K_p^d$. 
\mar

We obtain another proof of Theorem \ref{Dp>2}.

\begin{proof}[Proof of Theorem \ref{Dp>2}]
Combining Corollary \ref{dual-intpolHpc} (i) with Corollary \ref{Hpcboxplus} and Lemma \ref{boxplusKpd} and  we get for $2<p< \infty$
$$\H_p^c=(\H_{p'}^c)^*=(\h_{p'}^d \boxplus \h_{p'}^c)^*
=(\K_{p'}^d \boxplus \h_{p'}^c)^*.$$
Then Lemma \ref{dualboxplus}, Lemma \ref{dualKpd} and Corollary \ref{dualhpc} (i) yield
$$\H_p^c=(\K_{p'}^d)^* \cap (\h_{p'}^c)^*=\J_{p}^d \cap \h_{p}^c.$$
\qd

\re
This argument can be extended to the case $p=\infty, p'=1$. Then by duality Corollary \ref{DboxplusH1c} implies that 
$$\BMO^c=\J_{\infty}^d\cap \bmo^c \quad \mbox{with equivalent norms}.$$
\mar

\subsection{Burkholder-Rosenthal inequalities}

We may now extend the noncommutative Burkholder-Rosenthal inequalities recalled in Theorem \ref{Bdiscr} to the continuous setting. 
We introduce the conditioned Hardy space $\h_p$ as follows. 

\begin{defi}
Let $1 < p<\infty$. We define
$$\h_p=\left\{\begin{array}{ll}
\h_p^d+\h_p^c+\h_p^r& \quad \mbox{for} \quad 1< p< 2 \\
\J_p^d\cap \h_p^c\cap \h_p^r& \quad \mbox{for}\quad 2\leq p<\infty
\end{array}\right.,$$
where the sum is taken in $L_p(\M)$ and the intersection in $L_2(\M)$.  
\end{defi}

Combining the Davis inequalities (Theorem \ref{Davishpd} and Theorem \ref{Dp>2}) with the Burkholder-Gundy inequalities (Theorem \ref{BG1} and Theorem \ref{BG2}) 
we get

\begin{theorem}
Let $1<p<\infty$. Then 
$$L_p(\M)=\h_p \quad \mbox{with equivalent norms}.$$
\end{theorem} 

\begin{cor}
Let $1<p<2$. Then 
$$\h_p=\h_p^d\boxplus\h_p^c \boxplus \h_p^r
=\h_p^{1_c}\boxplus \h_p^{1_r}\boxplus\h_p^c \boxplus \h_p^r
 \quad \mbox{isometrically}.$$
\end{cor}

\begin{proof}
Combining Corollary \ref{Hpcboxplus} with Corollary \ref{Hpboxplus} we obtain that the two sums   
$\h_p^d+\h_p^c+\h_p^r$ and $\h_p^d\boxplus\h_p^c \boxplus \h_p^r$ coincide with equivalent norms. 
Hence they coincide isometrically by Lemma \ref{sumeq}. 
The second equality follows similarly from Theorem \ref{BRDboxplus-hp1c}. 
\qd

\section*{Appendix}

We end this paper with some problems which are still open at the time of this writing. 
They concern the more difficult case $p=1$. 
For $1<p<2$, Corollary \ref{HpXp} gives a nice description of the space $\H_p^c$. 
However, we do not know if this characterization still holds true for $p=1$. 

\begin{pb}
Do we have $\H_{1}^c=\{x\in L_1(\M)  :  \|x\|_{\H_1^c}<\infty \}$ ?
\end{pb}

On the dual side, by Remark \ref{normLpcMO} (iii) we know that the $L_p^c\MO$-norm is the limit of the discrete $L_p^cMO$-norms for $2<p<\infty$. 
For $p=\infty$, we only established one estimate in Corollary \ref{normBMOc} (iii).

\begin{pb}
For $x\in \M$, do we have  
$$\|x\|_{\BMO^c} \simeq  \lim_{\si,\U} \|x\|_{BMO^c(\si)} ?$$
\end{pb}

The two last problems concern the delicate point of injectivity of the spaces. 
The first one concerns the space $\H_1$ defined in paragraph \ref{subsectH1}. 

\begin{pb}
Does $\H_1$ embed injectively into $L_1(\M)$ ? Or, equivalently, do we have $\H_1=\H_1^c+\H_1^r$ ?
\end{pb}

A way of solving this problem could be by finding a ``Randrianantoanina's type" explicit decomposition in $L_2(\M)$ of the discrete space $H_1$ with a simultaneous control of the norms. 
The second injectivity question deals with the column conditional Hardy space $\h_1^c$ studied in Section \ref{secthp}. 

\begin{pb}\label{h1csubsetL1}
Does $\h_1^c$ embed injectively into $L_1(\M)$ ? Or, equivalently, do we have $\h_1^c=\lh_1^c$ ?
\end{pb}

Observe that these two last problems are somehow related. Indeed, 
for $x\in \M$ we can consider 
$$\|x\|_{\th_1}=\lim_{\s,\U} \|x\|_{h_1(\s)}.$$
This defines a norm on $\M$, and we denote by $\th_1$ the corresponding completion. 
With the notations of Section \ref{secthp}, we have seen in the proof of Lemma \ref{embeddinghpc} that for $x\in \M$ we have $v_\U(x)=(v_\s(x))^\bullet \in L_1(\N_\U)$. 
Moreover, Proposition \ref{vextendhp} (i) yields
$$\|v_\U(x)\|_{L_1(\N_\U)}
=\lim_{\s,\U}\|v_\s(x)\|_{L_1(\N(\s))}
\leq C \lim_{\s,\U}\|x\|_{h_1(\s)}=C\|x\|_{\th_1}.$$
This means that $v_\U$ extends to a bounded map from $\th_1$ to $L_1(\N_\U)$. 
Since $L_1(\N_\U,\E_{\M_\U})$ embeds into $L_1(\N_\U)$ by Remark \ref{embeddingL1c(M,E)}, 
the following commuting diagram shows that the natural map $\psi:\h_1^c\to \tilde{\h}_1$ is injective:
 $$\xymatrix{
    \h_1^c \ar@{^{(}->}[d]^{v_\U} \ar@{->}[r]^{\psi}     & \tilde{\h}_1 \ar@{->}^{v_\U}[d]  \\
    L_1^c(\N_\U,\E_{\M_\U}) \ar@{^{(}->}[r] & L_1(\N_\U)
  }$$
Moreover, \eqref{H1h1discr} implies that 
$$\th_1=\tilde{\H}_1 \quad \mbox{with equivalent norms,} $$
where $\tilde{\H}_1$ denotes the completion of $\M$ with respect to the norm 
$$\|x\|_{\tilde{\H}_1}=\lim_{\s,\U} \|x\|_{H_1(\s)}.$$
Hence the problem of the injectivity of $\h_1^c$ into $L_1(\M)$ is related to the problem of the injectivity of $\tilde{\H}_1$ into $L_1(\M)$. 
More precisely, the commuting diagram
 $$\xymatrix{
    \h_1^c \ar@{->}_{\varphi}[rd]  \ar@{^{(}->}[r]^{\psi} & \tilde{\H}_1 \ar@{->}^{\tilde{\varphi}}[d]  \\
    & L_1(\M)
  }$$
means that if $\tilde{\varphi}$ is injective then $\varphi$ is also injective.

\backmatter

\bibliographystyle{plain}



\end{document}